\documentclass[11pt]{amsart}

\usepackage{amssymb}
\usepackage{amsmath}
\usepackage{amsfonts}
\usepackage{amsthm}
\usepackage{stmaryrd}
\usepackage[all]{xy}
\usepackage{mathrsfs}
\usepackage{graphicx}
\usepackage{color}
\usepackage{multirow}
\usepackage{pstricks}

\usepackage[margin=1.0in]{geometry}

\numberwithin{equation}{subsection}
\theoremstyle{plain}
\newtheorem{theorem}[subsection]{Theorem}

\newtheorem{lemma}[subsection]{Lemma}
\newtheorem{prop}[subsection]{Proposition}

\theoremstyle{definition}

\newtheorem{construction}[subsection]{Construction}

\newtheorem{example}[subsection]{Example}

\newtheorem{notation}[subsection]{Notation}

\newtheorem{remark}[subsection]{Remark}

\newcommand{\ra}{\rightarrow}
\newcommand{\xra}{\xrightarrow}
\newcommand{\hra}{\hookrightarrow}

\def\AAA{\mathbb{A}}

\def\CC{\mathbb{C}}

\def\FF{\mathbb{F}}
\def\GG{\mathbb{G}}
\def\HH{\mathbb{H}}

\def\PP{\mathbb{P}}
\def\QQ{\mathbb{Q}}
\def\RR{\mathbb{R}}
\def\SSS{\mathbb{S}}

\def\ZZ{\mathbb{Z}}

\def\bfi{\mathbf{i}}

\def\bfA{\mathbf{A}}
\def\bfB{\mathbf{B}}

\def\calD{\mathcal{D}}
\def\calE{\mathcal{E}}
\def\calF{\mathcal{F}}
\def\calG{\mathcal{G}}
\def\calH{\mathcal{H}}

\def\calL{\mathcal{L}}

\def\calO{\mathcal{O}}

\def\calS{\mathcal{S}}
\def\calT{\mathcal{T}}

\def\calX{\mathcal{X}}

\def\gotha{\mathfrak{a}}
\def\gothb{\mathfrak{b}}
\def\gothc{\mathfrak{c}}

\def\gothe{\mathfrak{e}}

\def\gothh{\mathfrak{h}}

\def\gothp{\mathfrak{p}}
\def\gothq{\mathfrak{q}}

\def\gothB{\mathfrak{B}}

\def\gothF{\mathfrak{F}}
\def\gothG{\mathfrak{G}}
\def\gothH{\mathfrak{H}}

\def\gothS{\mathfrak{S}}

\def\gothV{\mathfrak{V}}

\def\scrA{\mathscr{A}}

\def\rmn{\mathrm{n}}

\def\rmM{\mathrm{M}}

\def\ttS{\mathtt{S}}
\def\ttT{\mathtt{T}}

\DeclareMathOperator{\Art}{Art}

\DeclareMathOperator{\Gal}{Gal}
\DeclareMathOperator{\Ker}{Ker}

\DeclareMathOperator{\Hom}{Hom}
\DeclareMathOperator{\id}{id}

\DeclareMathOperator{\Ind}{Ind}

\DeclareMathOperator{\Rec}{Rec}
\DeclareMathOperator{\Res}{Res}
\DeclareMathOperator{\Spec}{Spec}

\newcommand{\Q}{\mathbb{Q}}
\newcommand{\Z}{\mathbb{Z}}
\newcommand{\R}{\mathbb{R}}
\newcommand{\C}{\mathbb{C}}
\newcommand{\G}{\mathbb{G}}
\newcommand{\cO}{\mathcal{O}}
\newcommand{\kb}{\underline{k}}

\newcommand{\Fpb}{\overline{\FF}_p}

\newcommand{\univ}{\mathrm{univ}}

\newcommand{\Sym}{\mathrm{Sym}} 
\newcommand{\pr}{\mathrm{pr}}

\newcommand{\Tate}{\mathrm{Tate}} 

\newcommand{\Gys}{\mathrm{Gys}}

\newcommand{\es}{\mathrm{es}}
\newcommand{\cl}{\mathrm{cl}}

\newcommand{\bbalpha}{\boldsymbol{\alpha}}

\newcommand{\cusp}{\mathrm{cusp}}

\newcommand{\et}{\mathrm{et}}

\newcommand{\Fr}{\mathrm{Fr}}
\newcommand{\Frob}{\mathrm{Frob}}

\newcommand{\GL}{\mathrm{GL}}

\newcommand{\Iw}{\mathrm{Iw}}

\newcommand{\Qp}{\QQ_p}
\newcommand{\res}{\mathrm{res}}

\newcommand{\Sh}{\mathrm{Sh}}

\newcommand{\Tr}{\mathrm{Tr}}

\newcommand{\ur}{\mathrm{ur}}

\newcommand{\Qlb}{\overline{\Q}_{l}}
\newcommand{\Qpur}{\Q_{p}^{\mathrm{ur}}}
\newcommand{\tcD}{\tilde\calD}
\newcommand{\As}{\mathrm{As}}
\newcommand{\uk}{\underline k}

\newcommand{\longto}{\longrightarrow}

\begin{document}

\title{Tate cycles on some quaternionic Shimura varieties mod $p$}
\author{Yichao Tian}
\address{Yichao Tian,  Chinese Academy of Sciences,  55 Zhong Guan Cun East Road, Beijing, 100190, China.}

\address{
Mathematics Institute,
University of Bonn,
Endenicher Allee 60,
53115, Bonn, Germany
}
\email{yichaot@math.ac.cn, tian@math.u-bonn.de}
\author{Liang Xiao}
\address{Liang Xiao, Department of Mathematics, University of Connecticut, 341 Mansfield Road, Unit 1009, Storrs, CT 06269-1009, U.S.}
\email{liang.xiao@uconn.edu}
\date{\today}

\begin{abstract}
Let $F$ be a totally real field in which a prime number $p>2$ is inert.
We continue the study of the (generalized) Goren--Oort strata on quaternionic Shimura varieties over  finite extensions of $\FF_p$.
We prove that, when the dimension of the quaternionic Shimura variety is even, the Tate conjecture for the special fiber of the quaternionic Shimura variety holds for the cuspidal $\pi$-isotypical component, as long as the two unramified Satake parameters at $p$ are not differed by a root of unity. 
\end{abstract}


\maketitle
\tableofcontents

\section{Introduction}
One of the most important conjectures in algebraic geometry is the Tate conjecture   on algebraic cycles \cite{Tate}.   The general case of this conjecture is far from being proved. In this paper, we will consider the Tate conjecture for Hilbert modular varieties modulo an inert  prime.  

Let $F$ be a totally real field of degree $g=[F:\QQ]$, and $p>2$ be a prime number inert in $F$.  Let  $\AAA_{F}$ be  the ad\`ele ring, $\AAA_F^{\infty}$ (resp. $\AAA_F^{\infty,p}$) the subring of  finite ad\`eles (resp. prime-to-$p$ finite ad\`eles) of $F$. 
Fix a sufficiently small open compact subgroup $K=K^pK_p\subseteq \GL_2(\AAA_F^{\infty})$, where  $K_p=\GL_2(\cO_{F_p})$ and $K^p\subseteq \GL_2(\AAA_F^{\infty,p})$.
Let $X$ be the  Hilbert modular scheme of level $K$. This is a  quasi-projective smooth scheme over $\ZZ_{(p)}$.    For a fixed prime $\ell$, the $\ell$-adic \'etale cohomology group 
$H^g_\et(X_{\overline \QQ}, \overline{\QQ}_{\ell})$
 is equipped with a commuting action of $\Gal_{\QQ}:=\Gal(\overline \QQ/\QQ)$ and the Hecke algebra $\calH_{K}:=\overline \QQ[K\backslash \GL_2(\AAA_F^{\infty})/K]$.
Let $\pi=\pi^{\infty} \otimes \pi_{\infty}$ be a  cuspidal automorphic representation of $\GL_2(\AAA_F)$ such that the archimedean component  $\pi_{\infty}$ is a  discrete series of parallel weight 2, and that the $K$-fixed vectors  $\pi^{\infty,K}\neq 0$. 
We put 
\[H^g_\et(X_{\overline \QQ}, \overline{\QQ}_{\ell})[\pi]
:  =\Hom_{\calH_K}\big(\pi^{\infty,K}, H^g_\et(X_{\overline \QQ}, \overline{\QQ}_{\ell})\big).
\]
Let $\rho_{\pi}: \Gal(\overline \QQ/F)\ra \GL_2(\overline\QQ_{\ell})$ denote the Galois representation attached to $\pi$ (see e.g. \cite{Taylor}). 
Then the main result of \cite{brylinski-labesse} essentially says that
the semi-simplification of the $\Gal_{\QQ}$-module $H^g_\et(X_{\overline \QQ}, \overline{\QQ}_{\ell})[\pi]$ is isomorphic to the Asai representation $\As(\rho_{\pi}):=\otimes\mathrm{Ind}_{\Gal_{F}}^{\Gal_{\QQ}}(\rho_{\pi})$, which is the tensor induction of $\rho_{\pi}$ from $\Gal_F$ to $\Gal_\QQ$.\footnote{Conjecturally, $H^g_\et(X_{\overline \QQ}, \overline{\QQ}_{\ell})[\pi]$ is semi-simple so that it is isomorphic to $\As(\rho_{\pi})$. This conjecture is true if  $\As(\rho_{\pi})$ is irreducible.} 
By our assumption on $p$, both  $\rho_{\pi}$ and $H^g_\et(X_{\overline \QQ}, \overline{\QQ}_{\ell})$ are unramified at $p$. It makes sense to view $\As(\rho_{\pi})$ and $H^g_\et(X_{\overline \QQ}, \overline{\QQ}_{\ell})[\pi]$ as a $\Gal_{\FF_p}$-module.

Assume that $g$ is even. For $q$ a power of $p$, we write $\Frob_q\in \Gal_{\FF_p}$ for the geometric $q$-Frobenius.  We put
\[\calT(\pi, \overline\FF_p)\colon  =\bigcup_{j\geq 1}\As(\pi)(g/2)^{ \Frob_{p^{j}}=1},\]
for the space of  Tate classes of $\As(\pi)(g/2)$ defined over a finite extension of $\FF_p$.
It is easy to see that  $\dim_{\overline \QQ_{\ell}}\calT(\pi, \overline\FF_p)\geq \binom{g}{g/2}$, and the equality holds if  the two  eigenvalues of $\rho_{\pi}(\Frob_{p^g})$ do not differ by any roots of unity. Therefore, the Tate conjecture suggests that there should exist $\binom{g}{g/2}$ series algebraic cycles on $X_{\overline \FF_p}$ that contributes to 
$ \calT(\pi, \overline\FF_p) $. \emph{This is expected to hold for every $\pi$}.


\medskip
In this paper, we take a purely characteristic $p$ approach to construct the desired algebraic cycles on $X_{\overline \FF_p}$, and show  that these cycles contribute to all the geometric Tate classes in $H^g_\et(X_{\overline \QQ}, \overline{\QQ}_{\ell}(g/2))[\pi]$ if the eigenvalues of  $\rho_{\pi}(\Frob_{p^g})$ are sufficiently general. Here are our main results.

\begin{theorem}\label{T:introduction Tate}
Assume that $g=[F:\QQ]$ is even. 
Let $B_\infty$ denote the  quaternion algebra over $F$ ramified exactly at all archimedean places, and let $\Sh_K(B_\infty): = B^\times_\infty \backslash B^\times_\infty(\AAA_f) / K$ be the associated Shimura set over $\overline \FF_p$. We fix an isomorphism $(B_{\infty}\otimes \AAA_{F}^{\infty})^{\times}\cong \GL_2(\AAA_F^{\infty})$ so that the Hecke algebra $\calH_K$ acts on $H^0(\Sh_{K}(B_{\infty}), \overline \QQ_{\ell})$. 
\begin{enumerate}
\item There exist naturally defined correspondences
$$
\Sh_K(B_\infty) \xleftarrow{p_i} X_i  \xrightarrow{q_i} X_{\overline \FF_p}, \qquad i = 1, \dots, \tbinom{g}{g/2},
$$
such that $p_i$ is an  iterated $\PP^1$-bundles (so each connected component of $X_i$ is isomorphic to an iterated $\PP^1$-bundle), $q_i$ is a locally closed immersion, and both $p_i$ and $\pi_i$ are equivariant for prime-to-$p$ Hecke correspondences.

\item 
Let $\pi$ be a cuspidal automorphic  representation of $\GL_2(\AAA_F)$ associated  to a holomorphic Hilbert cuspidal eigenform of parallel weight $2$, and let $\pi_B$ be the Jacquet--Langlands transfer of $\pi$ to an automorphic representation of $(B_{\infty}\otimes\AAA_F)^{\times}
$. 
Denote by $\alpha_\pi$ and $\beta_\pi$ the two eigenvalues of $\rho_{\pi}(\Frob_{p^g})$.
 We put similarly
 \[H^0(\Sh_{K}(B_{\infty}), \overline \QQ_{\ell})[\pi_B]:=\Hom_{\calH_K}\big(\pi_{B}^{\infty,K}, H^0(\Sh_{K}(B_{\infty}), \overline \QQ_{\ell})\big).\]
Then the natural map
\begin{align*}
\bigoplus_{1\leq i \leq \binom g{g/2}} H^0\big(\Sh_{K}(B_{\infty}), \overline \QQ_{\ell}\big)[\pi_B] &\xra{\oplus p_i^*}\bigoplus_{1\leq i \leq \binom g{g/2}}   H^0_\et(X_{i,\overline \FF_p}, \overline \QQ_{\ell})[\pi_B]
\\ & \xrightarrow{\textrm{Gysin}}\calT(\pi,\overline \FF_p) \subseteq H^g_\et\big(X_{\overline \FF_p}, \overline \QQ_{\ell}(g/2)\big)[\pi]
\end{align*}
is injective if $\alpha_\pi \neq \beta_\pi$, and is isomorphic to $ \calT(\pi, \overline\FF_p)$ if $\alpha_\pi / \beta_\pi$ is not a root of unity.
In particular, if $\alpha_\pi / \beta_\pi$ is not a root of unity,  the Tate conjecture is true for the $\pi$-isotypic component of $H^g_\et\big(X_{\overline \FF_p}, \overline \QQ_{\ell}(g/2)\big)$ over all finite extensions of $\FF_p$. 
\end{enumerate} 
\end{theorem}

In fact, we prove a result stronger than the one stated here. A full statement will be stated later in Theorem~\ref{T:introduction-weight cycles}.

\begin{remark}
\begin{enumerate}
\item These cycles $X_1, \dots, X_{\binom{g}{g/2}}$ realize the Jacquet--Langlands correspondence geometrically and at the same time, they give the Tate classes for the $\pi$-isotypical component when $\pi_p$ is sufficiently general. We expect the union of them to be the supersingular locus; this will be  proved in a later paper \cite{Liu-Tian} of Yifeng Liu and the first author.

The geometric realization of Jacquet--Langlands correspondence was first studied by Ribet \cite{ribet, ribet2} and Helm \cite{helm, helm-PEL}. They gave some examples of the cycles in the case of modular or Shimura curves and unitary Shimura varieties that realizes the Jacquet--Langlands correspondence geometrically.
The geometric aspect of this technique is further developed by the authors in \cite{tian-xiao1}.
From this point of view, the theorem above may be understood as:
geometric Jacquet-–Langlands correspondence can give ``generic" Tate classes in the special
fiber of the Hilbert modular varieties.

\item Our construction does not give sufficient algebraic  cycles on $X_{\overline \FF_p}$ when  $\alpha_\pi=\beta_\pi$. For instance, for $g=2$, it follows from the Hodge index theorem and our computation of the intersection matrix of $X_1$ and $X_2$ on $X_{\overline\FF_p}$ (see Example~\ref{Ex:intersection-matrix})  that   the contribution of $X_1$ and $X_2$ to $\calT(\pi, \overline \FF_p)$ is one-dimensional if $\alpha_\pi=\beta_\pi$. It is an interesting question to find extra ``exotic'' algebraic cycles that are not cohomologically equivalent to our cycles.

\item
If one instead considers the Tate conjecture of Hilbert modular varieties over the generic fiber, namely over $\QQ$, this topic has a long history dated back to 1980s. But the situation is very different from the discussion in this paper. For a general $\pi$ that is not CM or the base change from a smaller field, the space of Tate classes $\As(\rho_{\pi})^{\Gal_{\QQ}}$ is zero. In contrast, the Tate classes over the special fiber at an inert prime always have dimension at least $\binom{g}{g/2}$. \emph{So the Tate conjecture of $X$ over $\QQ$ is a very different question from the Tate conjecture of $X_{\FF_p}$ over $\FF_p$.} We list below some known results for the Tate conjecture of Hilbert modular varieties over $\QQ$.

\begin{itemize}
\item If $\pi$ is non-CM,
this conjecture was proved by  Harder, Langlands, and Rapoport in \cite{HLR} when $g=2$.
In fact, they show that $\As(\rho_\pi)(1)^{\Gal_\QQ}$ is non-zero only if $\pi$ is the base change of a cuspidal automorphic representation of $\GL_{2}(\AAA_{\QQ})$, in which case Hirzebruch--Zagier cycles are account for the Tate classes.
Similar but partial results are obtained by Ramakrishnan \cite{ram} and Getz--Hahn \cite{Getz-Hahn} in the higher dimensional cases.

\item 
When $g=2$ and $\pi$ is CM, more algebraic cycles are expected to contribute to $\As(\rho_\pi)(1)^{\Gal_{\QQ}}$, and it was solved  independently by  Murty--Ramakrishnan  \cite{murty-ram} and Klingenberg \cite{Klingenberg} by reducing to the Lefschetz (1,1)-theorem for Hodge classes.

\item When $g=2$ and $\pi$ is the base change of a cuspidal automorphic representation of $\GL_{2,\QQ}$, Langer \cite{Langer} constructed a variant of the Hirzebruch--Zagier cycle in characteristic $0$, and showed that its reduction modulo $p$ contributes to a one-dimensional subspace of $\Tate_p(\pi)$. His cycles  are strictly contained in the union of our cycles $X_1\cup X_2$. But due to the reason we just said earlier, his construction seems to be hard to generalize to general $\pi$. 
\end{itemize}

\item 
Despite the difference of the Tate conjecture over $\QQ$ and over finite fields, it is an interesting question to study the interrelation between the reduction of cycles in characteristic zero and cycles in characteristic $p$ that we constructed. Such study has interesting corollaries in arithmetics geometric applications, e.g. bounding Selmer groups. See a series of works of Yifeng Liu and the first author \cite{liu1, liu2, Liu-Tian}.

\item 
Analogues of this theorem for other Shimura varieties have already appeared in recent works, e.g. \cite{HTX, XZ}.
\end{enumerate}
\end{remark}

\subsection{Generalized Goren--Oort cycles} 
We now explain the construction of the cycles. We allow $g$ to be of arbitrary parity, and let $r$ be an integer with $1\leq r\leq \lfloor g/2\rfloor $.  In this paper,  we will  construct explicitly $\binom{g}{r}$ so-called \emph{generalized Goren--Oort cycles}  $X_1,\dots, X_{\binom{g}{r}}$  of codimension $r$ in $X_{\FF_{p^{g}}}$ such that each $X_{i}$ is isomorphic to a fibration of $r$-th iterated $\PP^1$-bundle over (the characteristic $p$ fiber of) some $(g-2r)$-dimensional quaternionic Shimura variety.  Moreover, the construction is compatible with prime-to-$p$ Hecke correspondences when the tame level $K^p$ changes.  
As pointed out by Xinwen Zhu,  the union of these  $X_i$'s  should be  the Zariski closure of the Newton stratum of $X_{\FF_{p^g}}$ with slope $\big(\frac rg, \dots, \frac rg, \frac{g-r}{g}, \dots, \frac{g-r}{g}\big)$, where both $\frac rg$ and $\frac{g-r}{g}$ appear with $g$ times. In particular, if $g$ is even and $r=\frac g2$, the union of $X_i$'s  should be exactly the supersingular locus of $X_{ \FF_{p^{g}}}$.

Fix an isomorphism $\iota_p: \overline\QQ_p\xra{\sim} \CC$. Composing with $\iota_p$ defines a bijection between the set of $p$-adic embeddings of $F$ with that of its archimedean places. 
 Since  $p$ is inert in $F$, the image of every $p$-adic embedding of $F$  lies in the maximal unramified extension of $\QQ_p$, hence the $p$-Frobenius $\sigma$ acts naturally on the set of $p$-adic embeddings.
  We label the $p$-adic embeddings of $F$ by $\{\tau_i: i\in \ZZ/g\ZZ\}$ such that $\tau_{i+1}=\sigma\circ\tau_i$, and  let   $\infty_i=\iota_p\circ \tau_i$ denote the corresponding  archimedean place  of $F$.
  For an even subset $\ttS$ of archimedean  places of $F$, we denote by $B_{\ttS}$ the quaternion algebra over $F$ which ramifies exactly at $\ttS$. 
  When ${\ttS}$ is the set of all archimedean places, we also write  $\ttS=\infty$.
 Fix an isomorphism $(B_{\ttS}\otimes_F \AAA_F^{\infty})^{\times}\cong\GL_2(\AAA_F^{\infty})$, so that $K$ can be viewed as an open compact subgroup of $(B_{\ttS}\otimes_F\AAA_F^{\infty})^{\times}$. 
For a subset $\ttT\subseteq \ttS$,  we will define in \S \ref{N:Shimura varieties}  a quaternionic Shimura variety $\Sh_K(B_{\ttS,\ttT})$ over $\FF_{p^{g}}$  attached to the reductive group $\Res_{F/\QQ}(B_{\ttS}^{\times})$ of level $K$. This is a      $g-\#\ttS$ dimensional smooth variety,  which is proper if $\ttS$ is non-empty. Here, the subset $\ttT$ means some modification on the usual Deligne homomorphism in the definition of $\Sh_{K}(G_{\ttS,\ttT})$ (see \S \ref{S:quaternoinic Shimura varieties}).  The Shimura varieties $\Sh_K(G_{\ttS,\ttT})$  with the same $\ttS$ but different choices of $\ttT$ will have  the same geometry, but the Galois actions on the geometric connected components of $\Sh_{K}(G_{\ttS,\ttT})$ will be different.

 The basic idea  of the constructions of the Goren--Oort cycles is as follows. Recall that there are exactly $g$ divisors, say $Y_1, \dots, Y_{g}$,  in the Goren--Oort stratification (or Ekedahl--Oort stratification) of $X_{ \FF_{p^{g}}}$.   The main result of \cite{tian-xiao1} shows that, if $g\geq 1$, \emph{each  $Y_{i}$ is isomorphic to a $\PP^1$-fibration over $\Sh_K(B_{\ttS_i, \ttT_i})$ with $\ttS_i=\{\infty_i,\infty_{i-1}\}$ and $\ttT_i=\{\infty_i\}$ (if $K^p$ is sufficiently small)}.  Actually, the results of \cite{tian-xiao1} apply to more general quaternionic Shimura varieties.
 For $r=1$, the generalized Goren--Oort cycles of codimension 1 are defined to be these Goren--Oort divisors $Y_i$'s.
If $r\geq 2$,  we  consider the $g-2$ Goren--Oort divisors $Z_{j}$ of $\Sh_K(B_{\ttS_i,\ttT_i})$ for $j\in \{i-2, \dots, i-{g}+1\}$ (See Proposition~\ref{P:GO-fibration}). 
Taking the inverse image of $Z_{j}$ in $Y_i$, we get a codimension 2 cycle $Y_{i,j}$ in $X_{ \FF_{p^{2g}}}$ which is isomorphic to a fibration of twice iteration of $\PP^1$-bundles over some quaternionic Shimura variety of dimension $g-4$. This gives the construction for $r=2$. In the general  case, the  codimension $r$ generalized Goren--Oort cycles on $X_{\overline \FF_p}$ are obtained  by repeating this process $r$ times.

 \begin{example}\label{Ex:introduction}

 (1) When $g=2$, there are  $2$ Goren--Oort divisors $X_1, X_2$ on $X_{ \FF_{p^2}}$, and each $X_{i}$ is isomorphic to a $\PP^1$-bundle over the discrete Shimura set $\Sh_K(B_{\infty, \ttT_{i}}^{\times})_{\overline \FF_p}$ with $\ttT_i=\{\infty_i\}$. We remark that the cycle constructed by Langer in \cite{Langer} is completely contained in (but not equal to) the union $X_1\cup X_2$. 
 
 (2) When $g=3$, there are $3$ Goren--Oort divisors on $X_{ \FF_{p^3}}$, say $Y_1, Y_2, Y_3$.     For $i\in \ZZ/3\ZZ$, each $Y_{i}$ is a $\PP^1$-fibration over $\Sh_{K}(B_{\ttS_{i}, \ttT_{i}})$ as discussed above. 
 
(3) When $g=4$, there are $6$  Goren--Oort cycles  of codimension 2 on $X_{\FF_{p^{4}}}$.
We start with the $4$ Goren--Oort divisors $Y_1,\dots, Y_{4}$  of $X_{ \FF_{p^4}}$. Then for each $i\in \ZZ/4\ZZ$, we have a $\PP^1$-fibration   $\pi_i: Y_i\to \Sh_{K}(B_{\ttS_i,\ttT_{i}})$. On each quaternionic Shimura surface $\Sh_{K}(B_{\ttS_i,\ttT_{i}})$,  there are 2 Goren--Oort divisors, say $Z_{i-2}$ and $Z_{i-3}$,   corresponding to  $\infty_{i-2}$ and $\infty_{i-3}$ respectively.  Then each of $Z_j$ with $j\in \{i-2, i-3\}$ is again  isomorphic to a $\PP^1$-fibration over  the zero-dimensional  Shimura variety $\Sh_K(B_{\infty, \ttT_i})$ with $\ttT_i=\{\infty_{i},\infty_{i-2}\}$. Put $X_{i,j}:= \pi_i^{-1}(Z_j)\subseteq Y_i$.  This is a codimension $2$ cycle on $X_{ \FF_{p^4}}$. In Theorem~\ref{T:GO description}, we will see that $X_{1,3}=X_{3,1}$ and $X_{2,4}=X_{4,2}$, so the 6 Goren--Oort cycles of codimension 2 are exactly $X_{1,3}, X_{2,4}, X_{1, 2}, X_{2, 3}, X_{3,4}, X_{4, 1}$. Note that the geometry of these 6 cycles are not the same:  each irreducible component of $X_{1,3}$ and $X_{2,4}$ are isomorphic to $(\PP^1)^2$, while that of other $4$ Goren--Oort cycles is isomorphic to the $\PP^1$-bundle $\PP(\cO_{\PP^1}(p)\oplus \cO_{\PP^1}(-p))$ over $\PP^1$ (see Example~\ref{Ex:GO cycles}).

After the paper is written, we were informed that when $g=4$, the geometry of these cycles were already known to Yu \cite{yu}, using a different method.
 \end{example}
The best way (so far) to parametrize the generalized Goren--Oort cycles is to use some combinatorial data, called \emph{ periodic semi-meanders} (mostly for the benefit of later computation of the Gysin-restriction matrix).   
A \emph{periodic semi-meander} of $g$ nodes is a graph where $g$ nodes are aligned equidistant on a section of a vertical cylinder, and are either connected pairwise by non-intersecting curves (called \emph{arcs}) drawn above the section, or connected by a straight line (called \emph{semi-lines}) to $+\infty$ at the top of the cylinder.  We use $r$ to denote the number of arcs.
For example, 
$\psset{unit=0.6}
\begin{pspicture}(-0.3,-0.1)(2.8,0.8)
\psset{linewidth=1pt}
\psset{linecolor=red}
\psarc(-0.25,0){0.25}{0}{90}
\psarc(2.75,0){0.25}{90}{180}
\psarc(1.25,0){0.25}{0}{180}
\psline(0.5,0)(0.5,0.75)
\psline(2,0)(2,0.75)
\psset{linecolor=black}
\psdots(0,0)(0.5,0)(1,0)(1.5,0)(2,0)(2.5,0)
\end{pspicture}$
 and
 $\psset{unit=0.6}
\begin{pspicture}(-0.3,-0.1)(2.8,0.8)
\psset{linewidth=1pt}
\psset{linecolor=red}
\psarc(0.75,0){0.25}{0}{180}
\psarc(1.75,0){0.25}{0}{180}
\psbezier{-}(0,0)(0,0.75)(2.5,0.75)(2.5,0)
\psset{linecolor=black}
\psdots(0,0)(0.5,0)(1,0)(1.5,0)(2,0)(2.5,0)
\end{pspicture}$ are both semi-meanders of $6$ points with $r=2$ and $3$ respectively.
An elementary computation shows that there are $\binom gr$ semi-meanders of $g$ nodes with $r$ arcs  for $r \leq \frac g2$. For the detailed discussion, see \ref{S:semi-meanders}.

To each periodic semi-meander $\gotha$ with $g$ nodes and $r$ arcs, one can associate a generalized Goren--Oort cycle $X_{\gotha}$ of codimension $r$ in $X_{ \FF_{p^{2g}}}$. We refer the reader to \S \ref{S:GO cycles} for the  precise definition.  The $g$ nodes  of a periodic semi-meander $\gotha$ correspond to the $g$ archimedean places $\infty_1,\dots, \infty_g$ from the left to the right.  
By construction, $X_{\gotha}$ is an $r$-th iterated $\PP^1$-bundle over the quaternionic Shimura variety $\Sh_K(B_{\ttS_{\gotha}, \ttT_{\gotha}})$, where $\ttS_{\gotha}$ consists of all archimedean places of $F$ corresponding to the end nodes of all $r$ arcs, and $\ttT_{\gotha}$ consists of those corresponding to the \emph{right} ends of the $r$ arcs.
We will denote by 
\[\pi_{\gotha}: X_{\gotha}\to \Sh_K(B_{\ttS_{\gotha}, \ttT_{\gotha}})\] the projection map. 
For instance, when $g=4$ and $r=2$, the cycles $X_{1,3}$ and $X_{2,4}$ in Example~\ref{Ex:introduction} correspond to the semi-meanders 
$\psset{unit=0.6}
\begin{pspicture}(-0.3,-0.1)(1.8,0.8)
\psset{linewidth=1pt}
\psset{linecolor=red}
\psarc(0.75,0){0.25}{0}{180}
\psarc(-0.25,0){0.25}{0}{90}
\psarc(1.75,0){0.25}{90}{180}
\psset{linecolor=black}
\psdots(0,0)(0.5,0)(1,0)(1.5,0)
\end{pspicture}$
 and 
$\psset{unit=0.6}
\begin{pspicture}(-0.3,-0.1)(1.8,0.8)
\psset{linewidth=1pt}
\psset{linecolor=red}
\psarc(0.25,0){0.25}{0}{180}
\psarc(1.25,0){0.25}{0}{180}
\psset{linecolor=black}
\psdots(0,0)(0.5,0)(1,0)(1.5,0)
\end{pspicture}$, and the other $4$ cycles $X_{1, 2}, X_{2, 3}, X_{3,4}, X_{4, 1}$ correspond respectively to the semi-meanders 
\[\psset{unit=0.6}
\begin{pspicture}(-0.3,-0.1)(1.8,0.8)
\psset{linewidth=1pt}
\psset{linecolor=red}
\psarc(-0.25,0){0.25}{0}{90}
\psbezier{-}(-1,0)(-1,0.75)(0.5,0.75)(0.5,0)
\psarc(1.75,0){0.25}{90}{180}
\psbezier{-}(1,0)(1,0.75)(2.5,0.75)(2.5,0)
\psset{linecolor=black}
\psdots(0,0)(0.5,0)(1,0)(1.5,0)
\end{pspicture},
\quad
\begin{pspicture}(-0.3,-0.1)(1.8,0.8)
\psset{linewidth=1pt}
\psset{linecolor=red}
\psarc(0.25,0){0.25}{0}{180}
\psbezier{-}(-0.5,0)(-0.5,0.75)(1,0.75)(1,0)
\psbezier{-}(1.5,0)(1.5,0.75)(3,0.75)(3,0)
\psset{linecolor=black}
\psdots(0,0)(0.5,0)(1,0)(1.5,0)
\end{pspicture},
\quad 
\begin{pspicture}(-0.3,-0.1)(1.8,0.8)
\psset{linewidth=1pt}
\psset{linecolor=red}
\psarc(0.75,0){0.25}{0}{180}
\psbezier{-}(0,0)(0,0.75)(1.5,0.75)(1.5,0)
\psset{linecolor=black}
\psdots(0,0)(0.5,0)(1,0)(1.5,0)
\end{pspicture},
\quad
\begin{pspicture}(-0.3,-0.1)(1.8,0.8)
\psset{linewidth=1pt}
\psset{linecolor=red}
\psarc(1.25,0){0.25}{0}{180}
\psbezier{-}(-1.5,0)(-1.5,0.75)(0,0.75)(0,0)
\psbezier{-}(0.5,0)(0.5,0.75)(2,0.75)(2,0)
\psset{linecolor=black}
\psdots(0,0)(0.5,0)(1,0)(1.5,0)
\end{pspicture},
\]
 
 \subsection{Main theorem revisited}
We now describe the main results of this paper. We consider a regular multiweight $(\underline k, w)\in \ZZ^{g+1}$ with $\underline k = (k_1, \dots, k_g)$, that is a collection of integers such that $k_i \geq 2$ and $k_i \equiv w \mod 2$. There is an automorphic \'etale  local system $\calL^{(\uk, w)}$ on $X$, which is a lisse $\overline \QQ_{\ell}$-sheaf of rank $\prod_{i=1}^{g}(k_i-1)$ pure of Deligne weight $g(w-1)$ (see \S \ref{S:automorphic sheaves}). 
We fix a cuspidal automorphic representation $\pi=\pi^{\infty} \otimes \pi_{\infty}$ of $\GL_2(\AAA_F)$ associated to a  holomorphic Hilbert modular forms of  weight $(\underline k, w)$ such that $\pi^{\infty, K}\neq 0$. 
Let $\rho_{\pi}$ be the Galois representation attached to $\pi$, and let $\As(\rho_{\pi})=\otimes\Ind_{\Gal_{\QQ}}^{\Gal_F}(\rho_{\pi})$ be the Asai representation of $\rho_{\pi}$.  
Let $\calH_{K^p}:=\overline \QQ_{\ell}[K^p\backslash \GL_2(\AAA_{F}^{\infty,p})/K^p]$ denote the prime-to-$p$ Hecke algebra, and $\pi^{\infty,p}$ be the prime-to-$p$ part of $\pi$. We put
\[
H^g_\et\big(X_{\overline \FF_p}, \calL^{(\underline k, w)}\big)[\pi]: = \Hom_{\calH_{K^p}}\big((\pi^{\infty, p})^{K^p}, H^g_\et(X_{\overline \FF_p}, \calL^{(\uk,w)})\big).\footnote{Note that by strong multiplicity one theorem, one has an isomorphism $\Hom_{\calH_{K^p}}\big(\pi^{\infty, p, K^p}, H^g_\et(X_{\overline \FF_p}, \calL^{(\uk,w)})\big)\cong \Hom_{\calH_K}\big(\pi^{\infty, K}, H^g_\et(X_{\overline \FF_p}, \calL^{(\uk,w)})\big)$}
\]
According to \cite{brylinski-labesse},  the $\Gal_{\FF_p}$-module $ H^g_\et\big(X_{\overline \FF_p}, \calL^{(\uk,w)}\big)[\pi]$ has the same semi-simplification  as
 $$ \As(\rho_{\pi})|_{\Gal_{\FF_p}}=\otimes\Ind_{\Gal_{\FF_{p^g}}}^{\Gal_{\FF_{p}}}(\rho_{\pi}|_{\Gal_{\FF_p^{g}}}).
$$

 Fix an integer $r$ with $1\leq r\leq  g/2$. We denote by $\gothB^r_{\emptyset}$  the set of periodic semi-meanders of $g$ nodes and $r$ arcs. 
 As explained above, for each $\gotha\in \gothB^r_{\emptyset}$, one has a generalized Goren--Oort cycle $X_{\gotha}$ in $X_{ \FF_{p^{2g}}}$ of codimension $r$, which admits an $r$-th iterated $\PP^1$-bundle morphism $\pi_{\gotha}: X_{\gotha}\to \Sh_K(B_{\ttS_\gotha, \ttT_{\gotha}})$. 
 One can also define an automorphic \'etale local system $\calL^{(\uk,w)}_{\ttS_{\gotha}, \ttT_{\gotha}}$ on $\Sh_K(B_{\ttS_\gotha, \ttT_{\gotha}})$ (see \S \ref{S:automorphic sheaves}), which is compatible with the local system  $\calL^{(\uk,w)}$ on $X$ in the sense that $\pi_{\gotha}^*\calL^{(\uk,w)}_{\ttS_\gotha,\ttT_\gotha}\cong \calL^{(\uk,w)}|_{X_{\gotha}}$. When $(\uk,w)=(2,\dots, 2)$,  both $\calL_{\ttS_{\gotha}, \ttT_{\gotha}}^{(\uk, w)}$ and $\calL^{(\uk, w)}$ are the constant sheaf $\overline \QQ_{\ell}$.
 We consider  the  composite  map 
 \[
 \mathrm{Gys}_{\gotha}: H^{g-2r}_\et\big(\Sh_K(B_{\gotha})_{\overline \FF_p}, \calL^{(\uk,w)}_{(\ttS_{\gotha}, \ttT_{\gotha})}\big)\xra{\pi^*_{\gotha}} H^{g-2r}_\et\big(X_{\gotha,\overline \FF_p}, \calL^{(\uk,w)}|_{X_{\gotha}}\big)\xra{\mathrm{Gysin}} H^g_\et\big(X_{\overline \FF_p}, \calL^{(\uk,w)}(r)\big),
 \]
 where the second arrow is the Gysin map.
 Since the construction of $X_{\gotha}$ is compatible with prime-to-$p$ Hecke correspondence, $\Gys_{\gotha}$ is equivariant under the action by $\calH_{K^p}$.
 The main result of this paper is the following.
    \begin{theorem}[Theorem~\ref{T:Tate}]
\label{T:introduction-weight cycles}
Let $\alpha,\beta$ denote the two eigenvalues of $\rho_{\pi}(\Frob_{p^{g}})$.
Consider the map induced by the direct sum of Gysin maps
\begin{equation}
\label{E:Gysin map}
\Gys: \bigoplus_{\gotha\in \gothB^r_{\emptyset}} H^{g-2r}_\et\big(\Sh_K(B_{\ttS_\gotha, \ttT_{\gotha}})_{\overline \FF_p}, \calL^{(\uk,w)}_{\ttS_\gotha, \ttT_{\gotha}}\big)[\pi] \xra{\sum_{\gotha}\Gys_{\gotha}} H^{g}_\et\big(X_{\overline \FF_p}, \calL^{(\underline k, w)}(r)\big)[\pi]
\end{equation}
 on the $\pi$-isotypic components.
Then the following statements hold:
\begin{enumerate}
\item If $\alpha\neq\beta$, then the morphism $\Gys$ is injective. 
\item If $\alpha / \beta$ is not a $2n$-th root of unity for any $n \leq g$,\footnote{The reason that we have $2n$-th (as opposed to $n$-th) root of unity here is purely technical. See Remark~\ref{Remark after the theorems}(3).} then $\Gys$ induces an isomorphism when restricted to the generalized eigenspaces of $\Frob_{p^{2g}}$ on both source and target with eigenvalues $\alpha^{2r}\beta^{2(g-r)}/p^{2rg}$.
\end{enumerate}
\end{theorem}
This Theorem can be viewed as some sort of geometric Jacquet--Langlands transfer  from the quaternionic  Shimura varieties $\Sh_K(B_{\ttS_\gotha, \ttT_{\gotha}})$'s to $X$.

For the applications to the Tate conjecture, we assume that $g$ is even. Then all periodic semi-meander $\gotha$ with $g$ nodes and $\frac g2$ arcs, we have $\ttS_{\gotha}=\infty$, and the Goren--Oort cycle $X_{\gotha,\overline \FF_p}$ is a collection of $(g/2)$-th iterated $\PP^1$-bundles parametrized by the common discrete  Shimura set\footnote{Our previous notation for this Shimura set should be $\Sh_K(B_{\infty,\ttT_{\gotha}})_{\overline \FF_p}$. Since they are all canonically isomorphic for all $\gotha$, we omit $\ttT_\gotha$ from the notation.}  
  \begin{equation}
  \Sh_K(B_{\infty})_{\overline \FF_p}=  B^{\times}_{\infty}\backslash (B_{\infty}\otimes_F\AAA_{F}^{\infty} )^{\times}/K
\end{equation}
Applying Theorem~\ref{T:introduction-weight cycles} to the case $(\uk,w)=(2,\dots, 2)$ gives Theorem~\ref{T:introduction Tate}.

\subsection{Overview of the proof of Theorem~\ref{T:introduction-weight cycles}}
We consider the restriction map 
\[
\Res_{\gotha}: H^{g}_\et\big(X_{\overline \FF_p}, \calL^{(\underline k, w)}(r)\big)[\pi]\to H^{g}_\et\big(X_{\gotha, \overline \FF_p}, \calL^{(\underline k, w)}|_{X_\gotha}(r)\big)[\pi]\xra{\Tr_{\pi_{\gotha}},\cong} H^{g-2r}_\et\big(\Sh_K(B_{\ttS_\gotha, \ttT_\gotha})_{ \overline \FF_p}, \calL^{(\underline k, w)}_{\ttS_\gotha, \ttT_\gotha}\big)[\pi],
\]
where the second map is the trace isomorphism. We get thus a composite map 
\begin{equation}\label{E:intersection-matrix}
\xymatrix{
\bigoplus_{\gothb\in \gothB^r_{\emptyset}} H^{g-2r}_\et\big(\Sh_K(B_{\ttS_\gothb, \ttT_\gothb})_{\overline \FF_p}, \calL^{(\uk,w)}_{\ttS_\gothb,\ttT_\gothb}\big)[\pi] \ar[r]^-{\Gys} & H^{g}_\et\big(X_{\overline \FF_p}, \calL^{(\underline k, w)}(r)\big)[\pi]\ar[d]^{\Res:=\oplus_{\gotha}\Res_{\gotha}}\\
&\bigoplus_{\gotha\in \gothB^r_{\emptyset}}H_\et^{g-2r}\big(\Sh_K(B_{\ttS_\gotha, \ttT_\gotha})_{\overline \FF_p}, \calL^{(\uk,w)}_{\ttS_\gotha, \ttT_\gotha}\big)[\pi].}
\end{equation}
When $(\uk, w)$ is of parallel weight 2, this is essentially the intersection matrix of the cycles $X_{\gotha}$'s in $X_{\overline \FF_p}$. The upshot is  that each ``matrix  entry'' $\Res_{\gotha}\circ\Gys_{\gothb}$ can be read off from the periodic semi-meanders $\gotha$ and $\gothb$ (see Theorem~\ref{T:intersection combinatorics}), and the  determinant of the intersection matrix is closely  related to the determinant of the Gram matrix of the link representation of  periodic Templey--Lieb algebras, which has been computed in  \cite{xxz-model}. Using this result, one can compute explicitly the determinant of  $\Res\circ \Gys$, which does not vanish as long as $\alpha\neq \beta$.   Theorem~\ref{T:introduction-weight cycles}(1) follows immediately, and statement (2) is obtained from (1) plus a direct computation of the dimensions of the generalized eigenspaces of $\Frob_{p^{2g}}$ with the given eigenvalue. 

\begin{example}\label{Ex:intersection-matrix}
(1) If  $g=2$ and  $r=1$,   the intersection matrix $(\Res_{\gotha}\circ\Gys_{\gothb})_{\gotha,\gothb\in \gothB^1_{\emptyset}}$  (under certain basis)  writes as 
\[
\begin{pmatrix}
-2p & \alpha+\beta\\
p^2\frac{\alpha+\beta}{\alpha\beta} & -2p
\end{pmatrix},
\]
whose determinant is $p^2(\alpha-\beta)^2/(\alpha\beta)$. 

(2) Assume $g=3$ and $r=1$. Even though the Shimura varieties $\Sh_K(G_{\ttS_{\gotha}, \ttT_\gotha})$ for $\gotha\in \gothB_{\emptyset}^1$ are not exactly the same, but  we have an isomorphism (see Proposition~\ref{P:cohomology of sh appendix})  
\[
H^1_\et\big(\Sh_K(B_{\ttS_{\gotha}, \ttT_{\gotha}})_{\overline \FF_p}, \overline\QQ_{\ell}\big)[\pi]\cong [\rho_{\pi}\otimes\det(\rho_{\pi})(1)]|_{\Gal_{\FF_p}}
\]
 for  $\gotha\in \gothB^1_{\emptyset}$ as $\Gal_{\FF_{p^3}}$-modules if $\alpha\neq \beta$.  The intersection matrix  (under a suitable basis)  is  
 \[
 \begin{pmatrix}
 -2p & \eta^{-1} & \eta\\
 \eta &-2p &\eta^{-1}\\
 \eta^{-1} & \eta &-2p
 \end{pmatrix},
 \]
 where $\eta$ is some operator which acts as scalar multiplication by $(\alpha/\beta)^{1/3}$ (resp. by $(\beta/\alpha)^{1/3}$) on the  eigenspace of $\Frob_{p^3}=\alpha\beta^2/p^3$ (resp. $\Frob_{p^3}=\alpha^2\beta/p^3$) in $[\rho_{\pi}\otimes\det(\rho_{\pi})(1)]|_{\Gal_{\FF_p}}$. The determinant of the above matrix is $- p^3(\alpha-\beta)^2/(\alpha\beta)$. 
 \end{example}

\if false
\bigskip
This paper is the third in a series, in which we study the Goren--Oort stratification of quaternionic Shimura varieties  associated to a totally real field $F$ in  characteristic $p$ and its applications.
In the first paper \cite{tian-xiao1}, we gave a global description of each stratum, in terms of a $\PP^1$-product bundle over another quaternionic Shimura variety.
In  \cite{tian-xiao2}, we apply this result to prove the classicality of overconvergent Hilbert modular forms of small slopes.
The aim of this paper is to investigate the Goren--Oort strata by viewing them as special cycles on the special fiber of quaternionic Shimura varieties. 
 This leads to a construction of algebraic cycles of middle codimension on the special fiber of quaternionic Shimura varieties, which,  under genericity conditions, coincide with  the prediction by  the Tate conjecture over finite fields.  We start by explaining the underlying philosophy by some examples.

 

\subsection{Hilbert modular surface}
Let $F$ be a real quadratic field and $p>2$ be a prime number that is inert in $F/\QQ$. Let $\AAA^{\infty}_F$ be the ring of finite adeles of $F$, and $K\subset \GL_2(\AAA_F^\infty)$ be an open compact subgroup hyperspecial at $p$. We consider the Hilbert modular variety $\calX$ of level $K$; it admits a smooth integral model $\calX$ over $\ZZ_{(p)}$, with $X$ as its special fiber.
Fix a prime number $\ell \neq p$.
The main result of Brylinski and Labesse \cite{brylinski-labesse} says that, up to Frobenius semisimplification, the cuspidal part of the $\ell$-adic \'etale cohomology of $X$ is given as follows
\[
H^2_\et(X_{\overline \FF_p}, \overline \QQ_\ell)_\cusp \cong \bigoplus_\pi (\pi^\infty)^K \otimes \rho_\pi|_{\Gal(\overline \FF_p/ \FF_{p^2})}^{\otimes 2},
\]
where the direct sum is taken over all cuspidal automorphic representations whose archimedean components are discrete series  of parallel weight 2 and whose $p$-component is unramified, and $\rho_\pi$ is the Galois representation associated to $\pi$. In particular, the representation $\rho_\pi$ is unramified at $p$; so the tensorial induction is just simply a self tensor product.

We fix an automorphic cuspidal representation $\pi$ of $\GL_2(\AAA_F)$ as above. We write   $H^2_\et(X_{\overline \FF_p}, \overline \QQ_\ell)_\pi:=(\pi^{\infty})^K\otimes \rho_{\pi}|_{\Gal(\Fpb/\FF_{p^2})}^{\otimes 2}$ for the $\pi$-isotypical component.
Let $\Frob_{p^2}\in \Gal(\Fpb/\FF_{p^2})$ denote the geometric Frobenius element. 
We assume that
\begin{itemize}
\item
the two  eigenvalues $\alpha$ and $\beta$ of $\rho_\pi(\Frob_{p^2})$ are distinct,
\item
$(\pi^\infty)^K$ is one-dimensional so that $H^2_\et(X_{\overline \FF_p}, \overline \QQ_\ell)_\pi \cong \rho_{\pi}|_{\Gal(\Fpb/\FF_{p^2})}^{\otimes 2}$, and
\item
the central character of $\pi$ is trivial.
\end{itemize}
Here, the first condition is an essential hypothesis, whereas the last two conditions are made to simplify the discussion.
The action of $\Frob_{p^2}$ on $H^2_\et(X_{\overline \FF_p}, \overline \QQ_\ell)_\pi$ has (generalized) eigenvalues $\alpha^2, \beta^2$, and $\alpha\beta = p^2$ which has multiplicity two.  In other words, $H^2_\et(X_{\overline \FF_p}, \overline \QQ_\ell(1))_\pi$ has a two-dimensional subspace on which the $p^2$-Frobenius acts trivially (or more rigorously speaking, unipotently).  According to the prediction of the famous Tate Conjecture, this subspace should be generated by  cycle classes of $X$ defined over $\FF_{p^2}$.

The main theorem of this paper shows that the cycle classes of Goren--Oort strata (\cite{goren-oort}) of $X$ span this two dimensional subspace of $H^2_\et(X_{\overline \FF_p}, \overline \QQ_\ell(1))_\pi$, and hence we  verify the Tate Conjecture in this setting by explicitly exhibiting cycles in $X$.  We explain this with some details now.  Recall that the main result of the first paper in this series (\cite[Theorem~1.5.1]{tian-xiao1}) implies that there are two collections $X_1, X_2$ of $\PP^1$'s on $X$ parameterized by the discrete Shimura variety $\Sh_{\infty_1, \infty_2}$  associated to the quaternion algebra $B_{\infty_1, \infty_2}$ over $F$, which ramifies exactly at the two archimedean places of $F$,  of the same level $K$.\footnote{As pointed out by the referee, the existence of two families of $\PP^1$ on Hilbert modular surfaces was first observed by Pappas in his Ph.D. thesis.}
From this, we get a natural homomorphism
\begin{equation}
\label{E:cycle map HMsurface}
\bigoplus_{i=1}^2 H^0_\et(\Sh_{\infty_1, \infty_2, \overline \FF_p}, \overline \QQ_\ell) \xrightarrow{\cong}
\bigoplus_{i=1}^2 H^0_\et(X_{i, \overline \FF_p}, \overline \QQ_\ell) \xrightarrow{\mathrm{Gysin}}
H^2_\et(X_{\overline \FF_p}, \overline \QQ_\ell(1))^{\mathrm{Frob}_{p^2} =1},
\end{equation}
where the second morphism is the Gysin map associated to the closed immersion $X_{i,\overline \FF_p}\hra X_{\overline\FF_p} $ defined in \eqref{E:Gysin-general}.\footnote{If we pretend that $X$ is proper, then the Gysin map is dual to the restriction map $H^2_\et(X_{\overline \FF_p}, \overline \QQ_\ell) \to H^2_\et(X_{i, \overline \FF_p}, \overline \QQ_\ell)$ under the Poincar\'e dualities on $X$ and $X_i$.}
All homomorphisms are  equivariant for the prime-to-$p$ Hecke actions. 
 By Jacquet-Langlands correspondence, $(\pi^{\infty})^K$ appears on each term of the left hand side of \eqref{E:cycle map HMsurface} with multiplicity one.
Taking the $\pi$-isotypical parts (or more precisely the $(\pi^\infty)^K$-isotypical parts) of  \eqref{E:cycle map HMsurface} gives rise to a homomorphism
\[
\bigoplus_{i=1}^2 \overline \QQ_\ell \xrightarrow{\mathrm{Gysin}} H^2_\et(X_{\overline \FF_p}, \overline \QQ_\ell(1))_\pi^{\mathrm{Frob}_{p^2} =1}.
\]
Our main result of this paper says that this is an isomorphism.  To show this, we consider    the natural restriction map:
\[
H^2_\et(X_{\Fpb}, \Qlb(1))\xrightarrow{\mathrm{res}} \bigoplus_{i=1}^2 H^2_\et(X_{i,\Fpb},\Qlb(1))\cong\bigoplus_{i=1}^2 H^0_\et(\Sh_{\infty_1,\infty_2, \Fpb},\Qlb).
\] 
Its composition with \eqref{E:cycle map HMsurface} gives an endomorphism of $\bigoplus_{i=1}^2 H^0(\Sh_{\infty_1,\infty_2, \Fpb},\Qlb)$. 
 The key point is to show that the projection to the $\pi$-isotypical component of this endomorphism is given by 
\[
\begin{pmatrix}
-2p & \alpha + \beta \\ \alpha + \beta & -2p
\end{pmatrix},
\]
whose determinant is equal to $-(\alpha-\beta)^2$.  Here, the entries $\alpha + \beta$ come from the fact that the morphisms mapping the intersection $X_1 \cap X_2$ to $\Sh_{\infty_1, \infty_2}$ using the two $\PP^1$-parametrizations exactly give the Hecke correspondence $T_p$ of $\Sh_{\infty_1, \infty_2}$,\footnote{In general, we arrive at certain twist of the Hecke correspondence $T_p$. But the twist disappears under the assumption on the trivialness of the central character. } and hence we see the evaluation of $T_p$ at $\pi$, which is $\alpha + \beta$.   This matrix should be viewed as the $\pi$-projection of the intersection matrix of the Goren--Oort strata $X_1$ and $X_2$. 
 From this, we see that when $\alpha \neq \beta$, these Goren--Oort strata give rise to all Tate cycles in the $\pi$-isotypical component of the special fiber of the Hilbert modular surface.

In contrast, if  $\alpha = \beta$, $H^2_\et(X_{\overline \FF_p}, \overline \QQ_\ell(1))_\pi$ is expected to have four dimensional Tate classes.  However, the image of the classes of Goren--Oort strata in $H^2_\et(X_{\overline \FF_p}, \overline \QQ_\ell(1))_\pi$ probably only contributes to a one-dimensional subspace (see Example~\ref{Ex:alpha=beta} for the discussion).

Kartik Prasanna suggested to us that one might be able to obtain finer information when $\alpha = \beta\in \{p, -p\}$. In this case, if we consider the action of $\Frob_p$ on $H^2_\et(X_{\overline \FF_p}, \overline \QQ_\ell(1))_\pi$ (as opposed to the $\Frob_{p^2}$-action), the eigenvalues are $\alpha$ with multiplicity three and $-\alpha$ with multiplicity one.
One can show that the image of the Goren--Oort strata contains (and is expected to be equal to) the $(-\alpha)$-eigenspace.
It is the $\alpha$-eigenspace that is ``larger than usual", which causes the cycle map to be zero.\footnote{One should compare this for example with the Gross--Zagier formula in the case when the elliptic curve has rank $\geq 2$. In that case, the Mordell--Weil rank is ``larger than usual". So the Heegner point is ``unwilling to pick out a canonical rank one subgroup of the Mordell--Weil group", and hence has to be torsion.
For the same philosophy, in our case, the cycle class from the Goren--Oort strata is expected to be zero in the ``unusually large" $(-\alpha)$-eigenspace.}
See also Remark~\ref{Ex:alpha=beta}.

\subsection{Hilbert modular four-folds}
As one tries to generalize the result above to higher dimensional cases, one finds quickly that the Goren--Oort stratification does not provide enough interesting cycles.
Take a Hilbert modular four-fold $\calX$ as an example, where the totally real field $F$ defining it has degree $4$ and $p$ is inert in $F/\QQ$, and the level structure $K \subseteq \GL_2(\AAA_F^\infty)$ is hyperspecial at $p$.  Hence $\calX$ has an integral model over $\ZZ_{(p)}$, and let $X$ denote its special fiber.
Let $\pi$ be a cuspidal automorphic representation of $\GL_2(\AAA_F)$ with trivial central character whose archimedean components are discrete series of parallel weight $2$ and whose $p$-adic component is unramified.
 We assume again that $(\pi^\infty)^K$ is one-dimensional. 
Then the $\pi$-isotypical component $H^4_\et(X_{\overline \FF_p}, \overline \QQ_\ell)_\pi$ is isomorphic to $\rho_\pi|_{\Gal(\overline \FF_p / \FF_{p^4})}^{\otimes 4}$ up to Frobenius semisimplification.
Let $\alpha$ and $\beta$ be the (generalized) eigenvalues of $\rho_\pi(\Frob_{p^4})$ so that  $\alpha\beta=p^4$. The essential hypothesis in this case is that $\alpha/\beta \neq \pm 1$.
The same computation as before shows that the multiplicity of $\Frob_{p^4}$-eigenvalue $1$ in $H^4_\et(X_{\overline \FF_p}, \overline \QQ_\ell(2))_\pi$ is $\binom 42 =6$.
As a result, we should expect 6 collections of certain ($2$-dimensional) strata of $X$ which contribute to this $6$-dimensional subspace.

We recall   the following description of the Goren--Oort strata in  \cite[\S 1.5.3]{tian-xiao1}.  There are indeed 6 collections of two-dimensional Goren--Oort strata $X_{ij}$ for $\{i,j\} \subset \{0, \dots, 3\}$.  Unfortunately, only two of them $X_{02}$ and $X_{13}$ contribute to the correct Tate classes; they are $(\PP^1)^2$-bundles parametrized by the discrete Shimura varieties for the quaternion algebra $B_{\infty_0, \dots, \infty_3}$ over $F$ (which ramifies exactly at all archimedean places). 
Other strata are $\PP^1$-bundles over certain Shimura curves, and they  do not contribute to any Tate cycles in  $H^4_\et(X_{\overline \FF_p}, \overline \QQ_\ell(2))_\pi$, since Shimura curves do not have interesting $H^0$'s.

To get enough strata that contribute to the Tate cycles of $H^4_\et(X_{\overline \FF_p}, \overline \QQ_\ell(2))_\pi$, we need to look at all codimension one strata $X_0, \dots, X_3$ first.  Each $X_i$ is a $\PP^1$-bundle over the Shimura surface for the quaternion algebra $B_{\infty_i, \infty_{i-1}}$.  If we consider the Goren--Oort stratification of the Shimura surface for $B_{\infty_i, \infty_{i-1}}$ and take the corresponding $\PP^1$-bundle, we will get two $2$-dimensional subvarieties $X_{i,1}$ and $X_{i,2}$ of $X_i$.  In fact one of them is equal to a Goren--Oort stratum (either $X_{02}$ or $X_{13}$) and the other one is a completely new collection of subvarieties of $X$, which is a family of $\PP^1$-bundles over $\PP^1$ parametrized by the discrete Shimura variety for $B_{\infty_0, \dots, \infty_3}$. 
We call them \emph{(generalized) Goren--Oort cycles}.
In fact, each irreducible component is isomorphic to $\PP(\calO_{\PP^1}(p) \oplus \calO_{\PP^1}(-p))$ (see Example~\ref{Ex:GO cycles}). One can characterize these new subvarieties by looking at the $p^2$-torsion of the universal abelian variety.
To sum up, we obtain $4$ new two-dimensional subvarieties.  Our main theorem says that they together with the two Goren--Oort strata, span the $6$-dimensional Tate classes in  $H^4_\et(X_{\overline \FF_p}, \overline \QQ_\ell(1))_\pi$.\footnote{We recently learned that the existence of these $6$ collections of two-dimensional subvarieties of $X$ (which form the supersingular locus of $X$) were previously known by Chia-Fu Yu \cite{yu}.}

More generally, we prove the following.
\begin{theorem}[corollary of Theorem~\ref{T:Tate}]
\label{T:introduction-Tate}
Let $F$ be a totally real field of \emph{even} degree $g=[F:\Q]$,  and let $p$ be a rational prime inert in $F$. 
 Let $X$ denote the special fiber over $\FF_{p^g}$ of the Hilbert modular variety of some level $K\subseteq \GL_2(\AAA_F^\infty)$ hyperspecial at $p$. 
  Fix an automorphic representation $\pi$ of $\GL_2(\AAA_F)$ whose archimedean components are all discrete series of  weight $2$ (the lowest weight) with trivial central character. 
Suppose that $\pi_p$ is an unramified principal series whose two Satake parameters do not differ by an $n$-th root of unity for any $n \leq g/2$. 

Assume that $(\pi^\infty)^K$ is one-dimensional. 
  Then Tate Conjecture holds for the $\pi$-isotypical component $H^g_\et(X_{\overline \FF_p}, \overline \QQ_\ell(\frac g2))_\pi$. 
  More precisely,   the generalized invariant subspace of  $H^g_\et(X_{\overline \FF_p}, \overline \QQ_\ell(\frac g2))_\pi$ under the action of $\Gal(\Fpb/\FF_{p^g})$ has dimension  $\binom{g}{g/2}$, and it is generated by the cycle classes of   Goren--Oort cycles on $X$.
\end{theorem}

The condition on the Satake parameters is equivalent to saying that all $\Frob_{p^g}$-eigenvalues on $\rho_\pi^{\otimes g}$ are distinct.  We emphasize once again that this condition is essential.  The contribution from the (generalized) Goren--Oort cycles can be degenerate if the two parameters are the same.  See Remark~\ref{Ex:alpha=beta}.

One might wonder whether this result adds new knowledge on the semi-simplicity of the Frobenius action on the cohomology groups of Hilbert modular varieties. 
Unfortunately, our theorem applies only when the two Frobenius eigenvalues $\alpha$ and $\beta$ are distinct. 
In this case, $\rho_\pi(\Frob_{p^g})$ is already semisimple.
If the tensorial induction of $\rho_\pi$ is an irreducible representation of $\Gal_{\QQ}$ (as opposed to $\Gal_{\FF_{p^g}}$), then the semisimplicity of $\Frob_{p^g}$ on the $\pi$-isotypical component is already known.
 Note also that Nekov\'a\v r \cite{nekovar} proved recently the semi-simplicity of Galois representation appearing in $H^g_\et(\calX_{\overline \QQ}, \overline \QQ_\ell)$, using Eichler--Shimura relations.

\subsection{Framework for Goren--Oort cycles}
The explicit verification of Tate Conjecture is merely an ``extremal" case of the following study of Goren--Oort cycles.
We still start with the Galois  side.
Let $F$ be a totally real field of degree $g$ over $\QQ$ and let $p$ be a prime inert in $F/\QQ$.  We fix a natural number $r \leq g/2$.
We consider a regular multiweight $(\underline k, w)$, that is a collection of integers, where $\underline k = (k_1, \dots, k_g)$ with $k_i \geq 2$ and $k_i \equiv w \mod 2$.

As before, we consider the special fiber $X$ over $\FF_p$ of the Hilbert modular variety, taking the limit over all prime-to-$p$ level structures, but fixing the open compact subgroup at $p$ to be hyperspecial. 
There is an automorphic local system $\calL^{(\underline k, w)}$ over $X$.
We fix a cuspidal automorphic representation $\pi$ of $\GL_2(\AAA_F)$ associated to holomorphic Hilbert modular forms of  weight $(\underline k, w)$.
We focus on the $\pi$-isotypical component  \[H^g_\et(X_{\overline \FF_p}, \calL^{(\underline k, w)})[\pi]: = \Hom_{\GL_2(\AAA_F^{\infty,p})}(\pi^{\infty,p}, H^g_\et(X_{\overline \FF_p}, \calL^{(\underline k, w)})),\] which is (up to Frobenius semisimplification) isomorphic to $\rho_\pi|_{\Gal(\overline \FF_p / \FF_{p^g})}^{\otimes g}$, where $\rho_\pi$ is the Galois representation associated to $\pi$.
Let $\alpha$ and $\beta$ denote the two eigenvalues of $\rho_\pi(\Frob_{p^g})$.  We assume that $\alpha / \beta$ is \emph{not a $2n$th root of unity for $n \leq g$}.\footnote{The reason that we have $2n$th (as opposed to $n$th) root of unity here is purely technical. See Remark~\ref{Remark after the theorems}(3).}

 We consider the $p^g$-Frobenius eigenvalue $\alpha^{g-r}\beta^r$ for some positive integer $r \leq \frac g2$.
The associated generalized eigenspace has dimension $\binom gr$. The case $r=\frac g2$ gives rise to Tate cycles discussed above. 
Our main result of this paper is the following.

\begin{theorem}[Theorem~\ref{T:Tate}]
\label{T:introduction-weight cycles}
Suppose that $\alpha / \beta$ is not a $2n$-th root of unity for $n \leq g$.
There exist $\binom gr$ subvarieties $X_1, \dots, X_{\binom gr}$ (explicitly defined later) of $X_{ \FF_{p^g}}$ with the following properties:
\begin{itemize}
\item[(1)] Each $X_i$ is an $r$-times iterated $\PP^1$-bundle over the special fiber of a Shimura variety\footnote{In fact, we need a slightly funny choice of Deligne homomorphism for these quaternionic Shimura varieties. We refer to the main context of the paper for details.} associated with  a quaternion algebra over $F$ which  ramifies exactly at $2r$ archimedean places.
\item[(2)] The direct sum of the Gysin maps
\begin{equation}
\label{E:Gysin map}
\bigoplus_{i=1}^{\binom gr} H^{g-2r}_\et(X_{i, \overline \FF_p}, \calL^{(\underline k, w)}|_{X_i})[\pi] \longrightarrow H^{g}_\et(X_{\overline \FF_p}, \calL^{(\underline k, w)}(r))[\pi]
\end{equation}
is an isomorphism on the generalized eigenspace for $\Frob_{p^g} = \alpha^{g-r} \beta^r / p^{gr}$.
\end{itemize}
\end{theorem}
Theorem~\ref{T:introduction-Tate} is a special (and degenerate) case of this theorem (except for a less restricted condition on $\alpha/\beta$).
We in fact prove in Theorem~\ref{T:Tate} a stronger result than stated here for general quaternionic Shimura varieties.

\begin{remark}
\label{R:union of GO cycle is the Newton stratum}
The name Goren--Oort cycles is slightly misleading. In fact, when $r = g/2$, one can prove that the union of the $X_i$'s considered in Theorem~\ref{T:introduction-weight cycles} is the supersingular locus of $X$.
More generally, as Xinwen Zhu pointed out to us, the union of the $X_i$'s considered in 
Theorem~\ref{T:introduction-weight cycles} is the closure of the \emph{Newton} stratum of the Hilbert modular variety $X$, with slopes $(\frac rg, \dots, \frac rg, \frac{g-r}g, \dots, \frac{g-r}g)$.
\end{remark}
\begin{remark}
In the view of the previous remark, we point out the key philosophy suggested by Theorem~\ref{T:introduction-weight cycles}: \emph{the irreducible components of the Newton strata of the Hilbert modular varieties (generically) contribute to certain Frobenius eigenspace of the cohomology of the Hilbert modular varieties. In particular, the supersingular locus contributes to the Tate classes in the cohomology.
Moreover, the (generic) dimension of the Frobenius eigenspace determines the number of irreducible components in the corresponding Newton stratum.}

We expect this philosophy to continue to hold for more general Shimura varieties, as elaborated in the later joint work of the authors and David Helm \cite{HTX}, as well as the forthcoming joint work of Xinwen Zhu and the second author \cite{XZ}.
\end{remark}

\subsection{Intersection matrix}
Before giving the construction of the subvarieties $X_i$, let us first explain the essential idea in the proof of Theorem~\ref{T:introduction-weight cycles}(2). 
A straightforward computation shows that the generalized eigenspace for $\Frob_{p^g} = \alpha^{g-r} \beta^r / p^{gr}$ on both sides of \eqref{E:Gysin map} have the same dimension.   
 Hence, to prove our Theorem, it is enough to show that  the left vertical homomorphism  in the following diagram is an isomorphism when restricted to the  generalized eigenspace for $\Frob_{p^g} = \alpha^{g-r} \beta^r / p^{gr}$.
\[
\xymatrix@C=60pt{
\displaystyle\bigoplus_{i=1}^{\binom gr} H^{g-2r}_\et(X_{i, \overline \FF_p}, \calL^{(\underline k, w)}|_{X_i})[\pi] \ar[r]^-{\eqref{E:Gysin map}}_{\textrm{Gysin}}
\ar@{-->}[d]& H^{g}_\et(X_{\overline \FF_p}, \calL^{(\underline k, w)}(r))[\pi]
\ar[d]^-{\mathrm{restriction}} \\
\displaystyle\bigoplus_{i=1}^{\binom gr} H^{g-2r}_\et(X_{i, \overline \FF_p}, \calL^{(\underline k, w)}|_{X_i})[\pi]
\ar[r]^-{\cup (c_1(\PP^1))^r}_-{\cong}
&
\displaystyle\bigoplus_{i=1}^{\binom gr} H^{g}_\et(X_{i, \overline \FF_p}, \calL^{(\underline k, w)}|_{X_i}(r))[\pi],
}
\]
In other words, we are reduced to showing that a certain $\binom gr \times \binom gr$-matrix is invertible. When $g =2r$ and $(\underline k, w)$ is of parallel weight $2$, this is the intersection matrix of $X_i$'s.
When both $g$ and $r$ are relatively small, it is not extremely difficult to compute this matrix and its determinant. 
However,  it appears to be a non-trivial problem to prove the invertibility of such a matrix  for general $g$ and $r$.

In this paper, we will show that the matrix can be computed in a completely combinatorial way, and it is related to a version of Gram determinant for periodic semi-meanders, which has been well studied by mathematical physicists \cite{xxz-model, graham-lehrer}. 
 We confess that this potential link with mathematical physics is probably a coincidence.  It only suggests a strong relation to representation theory, to which the mathematics physics problem is also related.   See Remark~\ref{R:quantization}.

\subsection{Goren--Oort cycles: construction}\label{GO-cycles:introduction}
The best way (so far) to parametrize the Goren--Oort cycles is to use periodic semi-meanders (mostly for the benefit of later computation of the Gysin-restriction matrix).  As before, we take $F$ to be a totally real field of degree $g$ in which $p$ is inert.  

A \emph{periodic semi-meander} of $g$ nodes is a graph where $g$ nodes are aligned equidistant on a section of a vertical cylinder, and are either connected pairwise by non-intersecting curves (called \emph{arcs}) drawn above the section, or connected by a straight line (called \emph{semi-lines}) to $+\infty$ on the top of the cylinder.  We use $r$ to denote the number of arcs.
For example, 
$\psset{unit=0.6}
\begin{pspicture*}(-0.3,-0.1)(2.8,0.8)
\psset{linewidth=1pt}
\psset{linecolor=red}
\psarc(-0.25,0){0.25}{0}{90}
\psarc(2.75,0){0.25}{90}{180}
\psarc(1.25,0){0.25}{0}{180}
\psline(0.5,0)(0.5,0.75)
\psline(2,0)(2,0.75)
\psset{linecolor=black}
\psdots(0,0)(0.5,0)(1,0)(1.5,0)(2,0)(2.5,0)
\end{pspicture*}$ and $\psset{unit=0.6}
\begin{pspicture*}(-0.3,-0.1)(2.8,0.8)
\psset{linewidth=1pt}
\psset{linecolor=red}
\psarc(0.75,0){0.25}{0}{180}
\psarc(1.75,0){0.25}{0}{180}
\psbezier{-}(0,0)(0,0.75)(2.5,0.75)(2.5,0)
\psset{linecolor=black}
\psdots(0,0)(0.5,0)(1,0)(1.5,0)(2,0)(2.5,0)
\end{pspicture*}$ are both semi-meanders of $6$ points with $r=2$ and $3$ respectively.
An elementary computation shows that there are $\binom gr$ semi-meanders of $g$ nodes with $r$ arcs ($r \leq \frac g2$).

We label the set of $p$-adic embeddings $\calO_F \hookrightarrow W(\overline \FF_p) = \widehat \ZZ_p^\ur$ by $\tau_1, \dots, \tau_g$ so that $\sigma \circ \tau_i = \tau_{i+1}$, where $\sigma$ is the absolute Frobenius $W(\overline \FF_p) \to W(\overline \FF_p)$ and $\tau_i = \tau_{i \bmod g}$.
Let $X$ denote the special fiber of the Hilbert modular variety.
For each $\overline \FF_p$-point $x \in X$, we denote by  $A_x$  the fiber at $x$ of  the universal abelian scheme.
Let $\calD_x$ denote the covariant Dieudonn\'e module of the $p$-divisible group of $A_x$. It is a free $W(\overline \FF_p)$-module of rank $2g$.  The $\calO_F$-action on $A_x$ induces a 
natural direct sum decomposition
$\calD_x \cong \bigoplus_{i=1}^g \calD_{x, i}$, where $\calD_{x,i}$ is the direct summand of $\calD_x$ on which $\calO_F$ acts  through $\tau_i$.  Each $\calD_{x, i}$ is free of rank two over $W(\overline \FF_p)$.

The Verschiebung induces a $\sigma^{-1}$-semilinear map $V_i: \calD_{x, i+1} \to \calD_{x, i}$.  The image $V_i(\calD_{x, i+1}) / p\calD_{x, i}$ is canonically isomorphic to $\omega_{A_x^\vee, i}$, the $\tau_i$-component of the invariant $1$-differentials on $A_x^\vee$.  The latter is a one-dimensional $\overline \FF_p$-vector space.

We fix $r \leq g/2$.
To each semi-meander $\gotha$ considered above, we associate a subvariety $X_\gotha$ of $X$ whose $\overline \FF_p$-points $x$ are characterized as follows:
\begin{itemize}
\item
For each curve connecting $a$ and $a+d$ for some odd number $d < g$, we require that 
\[
V_a\circ V_{a+1} \circ \cdots \circ V_{a+d} (\calD_{x, a+d+1}) \subseteq p^{(d+1)/2}\calD_{x, a} \textrm{ or equivalently } p^{(d+1)/2}|V_a\circ V_{a+1} \circ \cdots \circ V_{a+d}.
\]
In fact, the inclusion forces an equality.
\end{itemize}
For example, the condition for the semi-meander $\psset{unit=0.6}
\begin{pspicture*}(-0.25,-0.1)(3.25,0.8)
\psset{linewidth=1pt}
\psset{linecolor=red}
\psarc{-}(0.25,0){0.25}{0}{180}
\psarc{-}(1.25,0){0.25}{0}{180}
\psline{-}(2.5,0)(2.5,0.75)
\psbezier{-}(-0.5,0)(-0.5,0.75)(2,0.75)(2,0)
\psbezier{-}(3,0)(3,0.75)(5.5,0.75)(5.5,0)
\psset{linecolor=black}
\psdots(0,0)(0.5,0)(1,0)(1.5,0)(2,0)(2.5,0)(3,0)
\end{pspicture*}$ (with $g =7$) is
\[
p\,|\,V_0V_1, \quad p\,|\, V_2 V_3, \quad\textrm{and } p \,|\,V_6 \frac{V_0V_1}{p} \frac{V_2V_3}{p}V_4.
\]

Iteratively applying (the divisor case of) the main theorem of \cite{tian-xiao1} (see \S 1.5 of {\it loc. cit.}),  it is not difficult to see that $X_\gotha$ is an $r$-times iterated $\PP^1$-bundle over the special fiber of quaternionic Shimura variety\footnote{More rigorously, we have a funny choice of Deligne homomorphism for this quaternionic Shimura variety.} for the quaternion algebra which ramifies at all archimedean places corresponding to those nodes which are linked to arcs.

The main computation of this paper is to show that the Gysin-restriction matrix is roughly given by the natural bilinear form on the vector space spanned by the periodic semi-meanders.
See Subsection~\ref{S:gram matrix} for the definition and Theorem~\ref{T:intersection combinatorics} for the precise statement.

Finally, we make a technical remark: our results involve transferring constructions from the unitary setup to the quaternion case. This is the origin of most notational complications.  Moreover, certain descriptions of the Goren--Oort cycles have ambiguity, but the ambiguity does not affect the proof of the main theorem.
\fi

\subsection*{Structure of the paper}
In Section 2, we recall necessary facts about Goren--Oort stratification from \cite{tian-xiao1}.
Some of the proofs are mostly book-keeping but technical. The readers may skip them for the first time reading.
In Section 3, we first recall the combinatorics about semi-meanders and then give the definition of the Goren--Oort cycles associated to  periodic semi-meanders.
In Section 4, we state our main  Theorem~\ref{T:Tate}  and prove it   assuming  Theorem~\ref{T:intersection combinatorics}, which says that the Gysin-restriction matrix for Goren--Oort cycles is roughly the same as the Gram matrix of the corresponding periodic semi-meanders.
Theorem~\ref{T:intersection combinatorics} is proved in Section 5.
The appendix includes a proof of the description of the cohomology of quaternionic Shimura varieties. This is well known to the experts but we include it there for completeness. 

\subsection*{Acknowledgements}
This paper is in debt to the help and the encouragement of many people: Ahmed Abbes, Jennifer Balakrishnan, Phillip Di Francesco, Matthew Emerton, Kiran  Kedlaya,  Zhibin Liang, Yifeng Liu, Kartik Prasanna, Richard Stanley, Chia-Fu Yu, Zhiwei Yun, Wei Zhang, and Xinwen Zhu.  
We especially thank Zhibin Liang, who helped us with a number of computations using computers to verify our conjectural formula in an early stage of this project.

We started working on this project when attending a workshop held at the Institute of Advance Study at Hong Kong University of Science and Technology (HKUST) in December 2011.
We thank the organizers Jianshu Li and Shou-wu Zhang, as well as the staff at IAS of HKUST.
We also met during various workshops held at University of Tokyo, Fields Institute at Toronto, and Morningside Center of Mathematics, to discuss this project. 
We especially thank the organizers as well as the staff members at these workshops.

The first author is partially supported by  the National Natural Science Foundation of China (No. 11321101).
The second author is partially supported by Simons collaboration grant \#278433 and NSF grant DMS-1502147. 

\subsection{Notation}
\label{N:totally real field} 
For a field $L$, we use $\Gal_L$ to denote its absolute Galois group. 
For a number field $L$, we write $\AAA_L$ (resp. $\AAA_L^\infty$, $\AAA_L^{\infty, p}$) for its ring of adeles (resp. finite adeles, finite adeles away from a rational prime $p$).
When $L=\QQ$, we suppress the subscript $L$, e.g. by writing $\AAA^\infty$.
Let $\underline p_L$ denote the idele of $\AAA_L^\infty$ which is $p$ at all $p$-adic places and trivial elsewhere. We also normalize the Artin reciprocity map $\Art: \AAA_{L}^{\times}/L^{\times}\ra \Gal_{L}^{\mathrm{ab}}$ so that a local uniformizer at a finite place $v$ corresponds to a \emph{geometric} Frobenius element at $v$.

Throughout this paper, we fix $F$ a totally real field of degree $g>1$ over $\QQ$.
Let $\Sigma$ denote the set of places of $F$, and $\Sigma_\infty$  the subset of all real places.
We fix a prime number $p>2$ \emph{inert} in the extension $F/\QQ$.\footnote{Although most of our argument works equally well when $p$ is only assumed to be unramified, we insist to assume that $p$ is inert which largely simplifies the notation so that the proof of the main result is more accessible.
But see Remark \ref{Remark after the theorems}(1).}
We put $\gothp = p\calO_F$, $F_\gothp$ the completion of $F$ at $\gothp$, $\calO_\gothp$ the valuation ring, and $k_\gothp$ the residue field.

We fix an isomorphism $\iota_p: \CC \xra{\sim} \overline \QQ_p$.
Let $\QQ_{p^g}$ denote the unramified extension of $\QQ_p$ of degree $g$ in $\overline \QQ_p$; let $\ZZ_{p^g}$ be its valuation ring.
Post-composition with $\iota_p$ induces an bijection between the set of archimedean places and 
$
\Sigma_\infty = \Hom(F, \RR) $ and  the set of $p$-adic embeddings  $\Hom(F, \QQ_{p^g}) \cong \Hom(\calO_F,  \FF_{p^g}).
$  In particular, the absolute Frobenius $\sigma$ acts on $\Sigma_\infty$ by sending $\tau \in \Sigma_\infty$ to $\sigma\tau:=\sigma \circ \tau$; this makes $\Sigma_\infty$ into one cycle. 
Let $\Qpur$ denote the maximal unramified extension of $\Q_p$, and $\Z_p^{\mathrm{ur}}$ denote its valuation ring.

For a finite field $\FF_{q}$, we denote by $\Frob_{q}\in \Gal_{\FF_{q}}$ the \emph{geometric} Frobenius element.

\section{Goren--Oort stratification}
\label{Sec:GO stratification}
We first recall the Goren--Oort stratification of the special fiber of quaternionic Shimura varieties and their descriptions, following \cite{tian-xiao1}.  We tailor our discussion to later application and hence we will focus on certain special cases discussed in {\it loc. cit.}

\subsection{Quaternionic Shimura varieties}
\label{S:quaternoinic Shimura varieties}
Let $\ttS$ be a set of places of $F$ of even cardinality such that   $\gothp\notin \ttS$.
  Put $\ttS_\infty = \ttS \cap \Sigma_\infty$ and $\ttS_\infty^{c} = \Sigma_\infty - \ttS_\infty$,\footnote{Note that the upper script $c$ was used to denote complex conjugation in \cite{tian-xiao1}. In this paper, we however use it to mean taking the set theoretic complement.} and $d=\#\ttS_{\infty}^c$. 
We also fix a subset $\ttT$ of $\ttS_\infty$.
We denote by $B_\ttS$ the quaternion algebra over $F$ ramified exactly at $\ttS$.
Let $G_{\ttS,\ttT} = \Res_{F/\QQ}(B_\ttS^\times)$ be the associated $\QQ$-algebraic group.  Here we inserted the subscript $\ttT$ because we use  the following   \emph{Deligne homomorphism} 
\[
\xymatrix@R=0pt{
h_{\ttS, \ttT}:\ \SSS(\RR) = \CC^\times \ar[rr] &&
G_{\ttS,\ttT}(\RR) \cong (\HH^\times)^{\ttS_\infty - \ttT} \times (\HH^\times)^{\ttT} \times \GL_2(\RR)^{\ttS_\infty^{c}}
\\
x+y \bfi \ar@{|->}[rr] &&
\bigg((\ 1, \dots, 1\ ), (\ x^2+y^2, \dots, x^2+y^2\ ), (\ \big({
\begin{smallmatrix}
x&y\\-y&x
\end{smallmatrix}\big), 
 \dots,  \big(
\begin{smallmatrix}
x&y\\-y&x
\end{smallmatrix}\big)}
\ 
)\bigg).
}
\]
When $\ttT=\emptyset$, the Deligne homomorphism $h_{\ttS,\emptyset}$ is the same as $h_{\ttS}$ considered in \cite[\S 3.1]{tian-xiao1}. 
The $G_{\ttS,\ttT}(\R)$-conjugacy class of $h_{\ttS,\ttT}$ is   independent of $\ttT$ and  is isomorphic to $\gothH_{\ttS}:=(\gothh^{\pm} )^{\ttS^{c}_{\infty}}$, where  $\gothh^\pm=\PP^1(\C)-\PP^1(\RR)$. 
Consider the Hodge cocharacter 
\[
\mu_{\ttS, \ttT}: \GG_{m,\C}\xra{z\mapsto (z,1)} \SSS_{\C}\cong \G_{m,\C}\times \G_{m,\C}\xra{h_{\ttS,\ttT}} G_{\ttS,\ttT,\C}.
\]
Here, the composition of the natural inclusion $\C^{\times}=\SSS(\R)\hra \SSS(\C)$ with the first (resp. second) projection $\SSS(\C)\ra\C^{\times}$ is the identity map (resp. the complex conjugation).

The reflex field $F_{\ttS, \ttT}$, i.e. the field of definition of the conjugacy class of $\mu_{\ttS,\ttT}$, is a finite extension of $\QQ$ sitting inside $\CC$ and hence inside $\overline \QQ_p$ via $\iota_p$.
It is clear that the $p$-adic closure of $F_{\ttS, \ttT}$ in $\overline \QQ_p$ is contained in  $\QQ_{p^g}$, the unramified extension of $\Qp$ of degree $g$ in $\overline \QQ_p$.
Instead of working with occasional smaller reflex field, we are content with working with Shimura varieties over $\QQ_{p^g}$. 

We fix an isomorphism $G_{\ttS,\ttT}(\Qp) \simeq \GL_2(F_\gothp)$ and put $K_p = \GL_2(\calO_\gothp)$. We will only consider open compact subgroups $K \subseteq G_{\ttS,\ttT}(\AAA^\infty)$\footnote{In earlier papers of this series, the open compact subgroup $K$ was denoted by $K_\ttS$. We choose to drop the subscript because for all $\ttS$ we encounter later, the group $G_\ttS(\AAA^\infty)$ are isomorphic, and hence we can naturally identify the $K_\ttS$'s for different $\ttS$'s.}
 of the form $K = K_pK^p$ with $K^p$ an open compact subgroup of $G_{\ttS,\ttT}(\AAA^{\infty, p})$, or occasionally $K=\Iw_pK^p$ with $\Iw_p:  = \big( \begin{smallmatrix} \calO_\gothp^\times & \calO_\gothp \\ p\calO_\gothp & \calO_\gothp^\times \end{smallmatrix} \big)$ when $\ttS_\infty^c = \emptyset$. 
 For such a $K$, we have a Shimura variety  $\calS h_{K}(G_{\ttS,\ttT})$ defined over $\QQ_{p^g}$, whose $\C$-points (via $\iota_p$) are given by
\[
\calS h_{K}(G_{\ttS,\ttT})(\C)=G_{\ttS,\ttT}(\Q)\backslash \gothH_{\ttS}\times G_{\ttS,\ttT}(\AAA^{\infty})/K.
\] 
We put $\calS h_{K_p}(G_{\ttS,\ttT}):=\varprojlim_{K^p}\calS h_{K^pK_p}(G_{\ttS,\ttT})$. This Shimura variety has dimension $d = \#\ttS_\infty^c$.
There is a natural morphism of geometric connected components
\begin{equation}\label{E:connected-components}
\pi_0(\calS h_{K_p}(G_{\ttS,\ttT})_{\overline \Q_p} )\longrightarrow F^{\times, \mathrm{cl}}_{+} \backslash \AAA_F^{\infty,\times} / \calO_\gothp^\times ,
\end{equation}
where $F_{+}^{\times}$ is the subgroup of totally positive elements of $F^{\times}$, and the superscript cl stands for taking closure in the corresponding topological space.
The morphism \eqref{E:connected-components} is an isomorphism if $\ttS_\infty^c \neq \emptyset$ by  \cite[Th\'eor\`eme 2.4]{deligne1}.
Following the convention in \cite[\S 2.11]{tian-xiao1},
\emph{we shall call the preimage of an element $\boldsymbol x \in F^{\times, \mathrm{cl}}_{+} \backslash \AAA_F^{\infty,\times} / \calO_\gothp^\times$ under the map~\eqref{E:connected-components} a geometric connected component, 
although it is not geometrically connected when $\ttS_\infty^c = \emptyset$.
The preimage of $\boldsymbol 1$ is called the neutral geometric connected component, which we denote by $\calS h_{K_p}(G_{\ttS, \ttT})_{\overline \QQ_p}^\circ$.}

Note that, for different choices of $\ttT$, the Shimura varieties $\calS h_{K}(G_{\ttS,\ttT})$ are isomorphic over $\overline\QQ_p$ (in fact over $\overline \QQ$ if we have not $p$-adically completed the reflex field), but the actions of $\Gal(\overline\QQ_p/\QQ_{p^g})$ depend on $\ttT$.
By Shimura reciprocity law (cf. \cite{deligne1} or \cite[\S 2.7]{tian-xiao1}), the action of $\Gal(\overline \Q_p/\Q_{p^g})$ on $\pi_0(\calS h_{K_p}(G_{\ttS,\ttT})_{\overline \Q_p} )$  factors through  
$\Gal_{\FF_{p^g}} \cong \Gal( \ZZ_p^\ur / \ZZ_{p^g})$, so that the  connected components of  $\calS h_{K_p}(G_{\ttS,\ttT})_{\overline \Q_p}$ are actually defined over $\Q_{p}^{\ur}$, the maximal unramified extension of $\QQ_p$. 
More precisely, the action of the geometric Frobenius of $\FF_{p^{g}}$ on $F^{\times, \mathrm{cl}}_{+} \backslash \AAA_F^{\infty,\times} / \calO_\gothp^\times$, induced through the homomorphism \eqref{E:connected-components}, is given by multiplication by the finite idele
\begin{equation}
\label{E:Shimura reciprocity for ShGST}
(\underline p_F)^{(2\#\ttT + \#\ttS_\infty^c)} \in F^{\times, \mathrm{cl}}_{+} \backslash \AAA_F^{\infty,\times} / \calO_\gothp^\times.\footnote{When $\ttS_\infty^c = \emptyset$ or equivalently when $\calS h_{K_p}(G_{\ttS, \ttT})$ is a zero-dimensional Shimura variety, the action of $\Frob_{p^g}$ is given by multiplication by the finite idele $(\underline p_F)^{\#\ttT}$ in the center $\Res_{F/\QQ}\GG_m$ of $G_{\ttS, \ttT}$.  This gives the canonical model for the discrete Shimura variety in the sense of \cite[2.8]{tian-xiao1}.}
\end{equation}
This determines a reciprocity map:
\begin{equation*}
\Rec_p\colon  \Gal_{\FF_{p^g}} \longrightarrow F^{\times, \mathrm{cl}}_{+} \backslash \AAA_F^{\infty,\times} / \calO_\gothp^\times.
\end{equation*}
Write $\nu: G_{\ttS, \ttT} \to \Res_{F/\QQ}(\GG_m)$ for the reduced norm homomorphism.
Following Deligne's recipe \cite{deligne2} of connected Shimura varieties, we put 
\begin{equation}\label{E:Deligne-group}
\calG_{\ttS, \ttT, p}: = \big(G_{\ttS,\ttT}(\AAA^{\infty, p}) / \calO_{F, (p)}^{\times, \mathrm{cl}} \big) \times \Gal_{\FF_{p^g}} \ \footnote{Comparing with \cite[(2.11.3)]{tian-xiao1}, we dropped the star extension because the center of $G_{\ttS, \ttT}$ is $\Res_{F/\QQ}\GG_m$, which  has trivial first cohomology. 
We also include the Galois part into the definition of $\calG$ to simplify notation here.}
\end{equation}
and define   $\calE_{G_{\ttS, \ttT}}$  to be the  subgroup  of $\calG_{\ttS,\ttT,p}$
consisting of pairs $(x, \sigma)$ such that $\nu(x)$
is equal to $\Rec_p(\sigma)^{-1}$. Here, $\cO_{F,(p)}^{\times, \mathrm{cl}}$ denotes the closure of $\cO_{F,(p)}^{\times}$ in $G_{\ttS,\ttT}(\AAA^{\infty,p})$.

The limit $\calS h_{K_p}(G_{\ttS,\ttT})_{ \QQ^{\ur}_p}$
carries an action by $\calG_{\ttS,\ttT,p}$,  and $\calE_{G_{\ttS,\ttT}}$ is the stablizer of each geometric connected component.   
Conversely, if $\calS h_{K_p}(G_{\ttS,\ttT})^{\bullet}_{ \QQ^{\ur}_p}$ is a geometric connected component, one can recover $\calS h_{K_p}(G_{\ttS, \ttT})$ from $\calS h_{K_p}(G_{\ttS, \ttT})^\bullet_{\QQ^{\ur}_p}$ by first forming the product
\[
\calS h_{K_p}(G_{\ttS,\ttT})^\bullet_{\QQ_p^\ur} \times_{\calE_{G_{\ttS, \ttT}}} \calG_{\ttS, \ttT, p}
\]
and then take the Galois descend to $\QQ_{p^g}$. 


\begin{notation}
\label{N:G(A infty)}
Note that $G_{\ttS,\ttT}(\AAA^{\infty})$ depends only on the finite places contained in $\ttS$.
In later applications, we will consider only pairs of subsets $(\ttS',\ttT')$ such that  $\ttS'$ contains the same finite places as $\ttS$. In that case, we will fix an isomorphism $G_{\ttS',\ttT'}(\AAA^{\infty})\cong G_{\ttS,\ttT}(\AAA^{\infty})$, and denote them uniformly by $G(\AAA^\infty)$ when no confusions arise. 
Similarly,  we have its subgroup $G(\AAA^{\infty, p})\subseteq G(\AAA^{\infty})$ consisting of elements whose $p$-component is trivial.
Thus, we may view $K$ (resp. $K^p$) as an open compact subgroup of $G(\AAA^{\infty})$ (resp. $G(\AAA^{\infty, p})$).

Under this identification, the group $\calG_{\ttS, \ttT, p}$ is independent of $\ttS, \ttT$, and we henceforth write $\calG_{p}$ for it.
Its subgroup $\calE_{G_{\ttS, \ttT}}$ in general depends on the choice of $\ttS$, and $\ttT$.  
However, the key  point is that, if $\ttS'$ and $\ttT'$ is another pair of subsets satisfying similar conditions and $\#\ttS_\infty - 2\#\ttT = \#\ttS'_\infty - 2\#\ttT'$ (which will be the case we consider later in this paper), then the subgroup $\calE_{G_{\ttS, \ttT}}$ is the same as $\calE_{G_{\ttS', \ttT'}}$.
\end{notation}

\begin{remark}
\label{R:ST compatible with GO cycles}
Using Proposition~\ref{P:integral model of Sh} and Construction~\ref{S:from unitary to quaternionic} later, we have access to most of the statements in \cite{tian-xiao1} which were initially proved for unitary groups and interpreted using connected Shimura varieties.
The key point mentioned in Notation~\ref{N:G(A infty)} has the additional benefit that the description of the Goren--Oort strata actually descends to quaternionic Shimura varieties because now the subgroups $\calE_{G_{\ttS, \ttT}}$'s are compatible for different $\ttS$ and $\ttT$'s.
\end{remark}

\subsection{An auxiliary CM field}\label{A:CM}
To use the results in \cite{tian-xiao1}, we fix a CM extension $E/F$ such that
\begin{itemize}
\item every place in $\ttS$ is inert in $E/F$, and
\item the place $\gothp$ splits as $\gothq \bar\gothq$ in $E/F$ if $\#\ttS_\infty^c$ is even, and it is inert in $E/F$ if $\#\ttS_\infty^c$ is odd.
\end{itemize}
These conditions  imply that $B_{\ttS}$ splits over $E$. 
In later applications, we will need to consider several  subsets $\ttS$ at the same time.  We remark that,  for all subsets $\ttS$ involved later,  the finite places contained in $\ttS$ are the same,   and  $\#\ttS_\infty^c$ will have the same parity. In particular, this means that we can fix for the rest of this paper one CM field $E$ that satisfies the above conditions (for the initial $B_\ttS$). 

We shall frequently use the following two finite idele elements:
\begin{enumerate}
\item
$\underline p_F$ denotes the finite idele in $\AAA_F^\infty$ which is $p$ at $\gothp$ and is $1$ elsewhere (which we have already introduced in \ref{N:totally real field});
\item
when $\gothp$ splits into $\gothq \bar \gothq$ in $E$, 
$\underline \gothq$ denotes the finite idele in $\AAA_E^\infty$ which is $p$ at $\gothq$, $p^{-1}$ at $\bar \gothq$, and $1$ elsewhere.
\end{enumerate}

Let $\Sigma_{E,\infty}$ denote the set of complex embeddings of $E$. We fix a choice of subset  $\tilde \ttS_{\infty}\subseteq \Sigma_{E,\infty}$ such that
the natural restriction map $\Sigma_{E,\infty}\ra \Sigma_{\infty}$ induces an isomorphism $\tilde\ttS_{\infty}\xra{\sim} \ttS_{\infty}$.
When $\gothp$ splits into $\gothq \bar\gothq$, we use $ \tilde \ttS_{\infty/\gothq}$  (resp. $ \tilde \ttS_{\infty/\bar\gothq}$) to denote the subset of places in $\tilde \ttS_\infty$ inducing $\gothq$ (resp. $\bar \gothq$) through the isomorphism $\iota_p$. 
We put 
\begin{equation}
\label{E:definition of Delta S}
 \Delta_{\tilde \ttS_{\infty}} := \#\tilde \ttS_{\infty/\bar \gothq} -\# \tilde \ttS_{\infty/ \gothq}.
\end{equation} 
We remark that all the $\tilde \ttS_{\infty}$'s we encounter later in this paper will all have the same $\Delta_{\tilde \ttS_{\infty}}$.

We write $E_\gothp$ for $F_\gothp  \otimes_F E$. It is the quadratic unramified extension of $F_\gothp$ if $\gothp$ is inert and it is $E_\gothq \times E_{\bar \gothq}$ if $\gothp$ splits.  We set $\calO_{E_{\gothp}}: = \calO_\gothp \otimes_{\calO_F} \calO_E$.

We put $\tilde \ttS = (\ttS, \tilde \ttS_\infty)$.
Put $T_{E,\tilde\ttS,\ttT}=T_{E}=\Res_{E/\Q}(\G_m)$, where the subscript $(\tilde\ttS,\ttT)$ means that we take the following Deligne homomorphism
\[
\xymatrix@R=0pt{
h_{E,\tilde\ttS, \ttT}\colon  \SSS(\RR) = \CC^\times \ar[r] &
T_{E,\tilde\ttS,\ttT}(\RR) = \bigoplus_{\tau\in \Sigma_{\infty}}(E\otimes_{F,\tau}\R)\cong (\C^{\times})^{\ttS_{\infty}-\ttT}\times (\C^{\times})^{\ttT}\times (\C^{\times})^{\ttS_{\infty}^c}
\\
z=x+y \bfi \ar@{|->}[r] & \bigg( (\bar z,\dots, \bar z),({z}^{-1},\dots, {z}^{-1}), (1,\dots, 1)\bigg).
}
\]
Here, the isomorphism $E\otimes_{F,\tau}\R\simeq \C$ for $\tau\in \ttS_{\infty}$ is given by the chosen embedding $\tilde\tau\in \tilde\ttS_{\infty}$ lifting $\tau$. 
One has the system of zero-dimensional Shimura varieties $\calS h_{K_{E}}(T_{E,\tilde\ttS,\ttT})$ with $\C$-points given by:
\[
\calS h_{K_{E}}(T_{E,\tilde\ttS,\ttT})(\C)=E^{\times,\cl}\backslash  T_{E,\tilde\ttS,\ttT}(\AAA^{\infty})/K_{E},
\] 
for any open compact subgroup $K_{E}\subseteq T_{E,\tilde\ttS,\ttT}(\AAA^{\infty})\cong \AAA_{E}^{\infty,\times}$. 
We put $K_{E,p}=\cO_{E, \gothp}^{\times}\subseteq T_{E,\tilde\ttS,\ttT}(\Q_p)$, and write 
$\calS h_{K_{E,p}}(T_{E,\tilde\ttS,\ttT})=\varprojlim_{K_{E}^p} \calS h_{K^p_{E}K_{E,p}}(T_{E,\tilde\ttS,\ttT})$ as the inverse limit over all open compact subgroups $K_E^p \subseteq T_{E, \tilde \ttS, \ttT}(\AAA^{\infty,p})$. 

As in Notation~\ref{N:G(A infty)}, we identify $T_{E, \tilde \ttS, \ttT}(\AAA^\infty)$ for all $\tilde \ttS$ and $ \ttT$, and write $T_E(\AAA^\infty)$ for it; so $K_E$ is naturally its subgroup.

Under the isomorphism $\iota_p:\CC\cong \overline \Q_p$, the image of the reflex field of  $\calS h_{K_{E}}(T_{E,\tilde\ttS,\ttT})$  is contained in $\Q_{p^{2g}}$. 
It makes sense to talk about $\calS h_{K_{E}}(T_{E,\tilde\ttS,\ttT})_{\Q_{p^{2g}}}$.
As $K_{E,p}$ is hyperspecial,  the action of $\Gal_{\Q_{p^{2g}}}$ on $\calS h_{K_{E}}(T_{E,\tilde\ttS,\ttT})(\overline\Q_p)$ is unramified. So $\calS h_{K_{E}}(T_{E,\tilde\ttS,\ttT})_{\Q_{p^{2g}}}$ is the disjoint union of the spectra of some finite unramified extension of $\Q_{p^{2g}}$, and it has an  integral canonical model over $\Z_{p^{2g}}$ by taking the spectra of the corresponding rings of integers. 
   Denote by $\Sh_{K_{E}}(T_{E,\tilde\ttS,\ttT})$ its special fiber. By Shimura's reciprocity law, the action of the  geometric Frobenius $\Frob_{p^{2g}}$ of $\FF_{p^{2g}}$ on $\Sh_{K_{E}}(T_{E,\tilde\ttS,\ttT})(\overline\FF_p)$ is given by
\begin{itemize}
\item[(i)] when  $\gothp$ is inert in $E/F$,
multiplication by $(\underline p_F)^{\#(\ttS_{\infty}-\#\ttT) - \#\ttT} = (\underline p_F)^{\#\ttS_{\infty}-2\#\ttT}$
 and
\item[(ii)] when $\gothp$ splits into $\gothq\bar \gothq$,
multiplication by 
\[
\underline \varpi_\gothq^{2(\# \tilde \ttS_{\infty /\bar  \gothq} - \#\ttT)} \underline \varpi_{\bar \gothq}^{2 (\# \tilde \ttS_{\infty / \gothq} - \#\ttT)} =
(\underline p_F)^{\#\ttS_{\infty}-2\#\ttT} (\underline \gothq)^{\Delta_{\tilde \ttS_{\infty}}},
\]
where $\underline \varpi_\gothq$ (resp. $\underline \varpi_{\bar \gothq}$) is the finite idele in $\AAA_E^\infty$ which is $p$ at the place $\gothq$ (resp. $\bar \gothq$) and is $1$ elsewhere, $\underline \gothq$ is the idele defined in Subsection \ref{A:CM}(2) above, and $\Delta_{\tilde \ttS_{\infty}}$ is defined in \eqref{E:definition of Delta S}.
\end{itemize} 
In particular, if $(\tilde\ttS',\ttT')$ is another  pair above such that $\#\ttS'_{\infty}-2\#\ttT'=\#\ttS_{\infty}-2\#\ttT$ and $\Delta_{\tilde \ttS_{\infty}} = \Delta_{\tilde \ttS'_{\infty}}$ if $\gothp$ splits, then there exists an isomorphism  of Shimura varieties over $\FF_{p^{2g}}$:
    \begin{equation}\label{E:isom-Shimura-var-E}
   \Sh_{K_E}(T_{E,\tilde\ttS,\ttT})\xra{\cong}\Sh_{K_E}(T_{E,\tilde\ttS',\ttT'})
    \end{equation}
    compatible with the Hecke action of $T_{E}(\AAA^{\infty,p})$ on both sides as $K_{E}^p$ varies.

\subsection{A unitary Shimura variety}
\label{S:unitary Shimura variety}
Let $Z=\Res_{F/\Q}(\G_m)$ be the center of $G_{\ttS,\ttT}$. Put $G''_{\tilde\ttS}=G_{\ttS,\ttT}\times_{Z} T_{E,\tilde\ttS,\ttT}$, which is the quotient of $G_{\ttS,\ttT}\times T_{E,\tilde\ttS,\ttT}$ by $Z$ embedded anti-diagonally as $z\mapsto (z,z^{-1})$.
Consider the product Deligne homomorphism 
$$
h_{\ttS,\ttT}\times h_{E,\tilde\ttS,\ttT}\colon \SSS(\RR)=\C^{\times}\ra (G_{\ttS,\ttT}\times T_{E,\tilde\ttS,\ttT})(\R),
$$
which can be further composed with the quotient map to $G''_{\tilde \ttS}$ to get
\[
h''_{\tilde \ttS}\colon  \SSS(\RR)=\C^{\times}\ra (G_{\ttS,\ttT}\times_Z T_{E,\tilde\ttS,\ttT})(\R) \cong G''_{\tilde \ttS}(\R).
\]
Note that $h''_{\tilde\ttS}$ does not depend on the choice of $\ttT\subseteq \ttS_{\infty}$ (hence the notation), and its
 conjugacy class is identified with $\gothH_\ttS=(\gothh^{\pm})^{\ttS_{\infty}^c}$.
Let $K''_p$ denote the (maximal) open compact subgroup $\GL_2(\calO_\gothp) \times_{\calO_\gothp^\times} \calO_{E, \gothp}^\times$ of $G''_{\tilde \ttS}(\Qp)$.
We will consider open compact subgroups of the form $K'' = K''_p K''^p \subset G''_{\tilde \ttS}(\AAA^{\infty})$ with $K''^p\subset G''_{\tilde \ttS}(\AAA^{\infty,p})$.
These data give rise to a Shimura variety $\calS h_{K''}(G''_{\tilde \ttS})$ (defined over $\QQ_{p^{2g}}$), whose $\CC$-points (via $\iota_p$) are given by
\[
\calS h_{K''}(G''_{\tilde \ttS})(\CC) = G''_{\tilde \ttS}(\QQ) \backslash (\gothH_\ttS \times G''_{\tilde \ttS}(\AAA^\infty)) / K''.
\]
We put $\calS h_{K''_p}(G''_{\tilde \ttS}): = \varprojlim_{K''^p}\calS h_{K''}(G''_{\tilde \ttS})$. 
Its geometric connected components admit a natural map
\begin{equation}\label{E:geometric-connected-unitary}
\pi_0\big(\calS h_{K''_p}(G''_{\tilde \ttS})_{\overline \QQ_p} \big) \longrightarrow 
\big( F_+^{\times, \mathrm{cl}} \backslash \AAA_F^{\infty, \times} / \calO_\gothp^\times \big) \times 
\big( 
E^{1} \backslash \AAA_{E}^{1} / \calO_{E_{\gothp}}^{N_{E/F} = 1}
\big),
\end{equation}
where $N_{E/F}$  is the norm from $E$ to $F$, and $E^1$ (resp. $\AAA_E^1$) is the subgroup of $E^{\times}$ (resp. $\AAA_E^{\times}$) with norm $1$ in $F^{\times}$ (resp. $\AAA_F^{\times}$).  
As in the quaternionic case, this is an isomorphism if $\ttS_{\infty}^c\neq \emptyset$.
 
We write 
 $\calS h_{K''_p}(G''_{\tilde \ttS})_{\overline \QQ_p}^\circ$ for the preimage of $\boldsymbol 1 \times \boldsymbol 1$, and 
call it  \emph{the neutral geometric connected component} of the unitary Shimura variety.

We can define the group $\calE_{G''_{\tilde \ttS}}$ and $\calG''_{\tilde \ttS, p}$ for the Shimura data $(G''_{\tilde \ttS}, h''_{\tilde \ttS})$ as in Subsection~\ref{S:quaternoinic Shimura varieties} (see e.g. \cite[\S 2.11]{tian-xiao1} for the recipe). First, we spell out the Shimura reciprocity map 
\begin{equation}\label{E:Shimura-rec-unitary}
\Rec''_p\colon \Gal_{\FF_{p^{2g}}} \longrightarrow \big( F_+^{\times, \mathrm{cl}} \backslash \AAA_F^{\infty, \times} / \calO_\gothp^\times \big) \times 
\big( 
E^{1} \backslash \AAA_{E}^{\infty, N_{E/F}= 1} / \calO_{E_{\gothp}}^{N_{E/F} = 1}.
\big),
\end{equation}
The Frobenius image  $\Rec''_p(\Frob_{p^{2g}})$ is given as follows
 \begin{itemize}
\item
when $\gothp$ is inert in $E/F$, $\Rec''_p(\Frob_{p^{2g}}) = (\underline{p}_F)^{2g} \times 1$, and
\item
when $\gothp$ splits in $E/F$, $\Rec''_p(\Frob_{p^{2g}}) = (\underline p_F)^{2g} \times (\underline \gothq)^{2\Delta_{\tilde \ttS_{\infty}}}$.
\end{itemize}
We put $\calG''_{\tilde \ttS, p}= \big( G''_{\tilde \ttS}(\AAA^{\infty, p}) / \calO_{E, (p)}^{\times, \mathrm{cl}} \big) \times \Gal_{\FF_{p^{2g}}}$\footnote{As in the footnote to \eqref{E:Deligne-group}, we omitted the star product in the definition of this group comparing to \cite[(2.11.3)]{tian-xiao1} because the center $\Res_{E/\QQ}(\GG_m)$ of $G''_{\tilde \ttS, \ttT}$ has trivial first cohomology.} and define $\calE_{G''_{\tilde \ttS}}$ to be its subgroup of pairs $(x, \sigma)$ such that $\nu''(x)$ is equal to $\Rec''_p(\sigma)^{-1}$, where
\[
\xymatrix@R=0pt@C=40pt{
\nu'': G''_{\tilde \ttS} = G_{\ttS, \ttT} \times_Z T_{E, \tilde \ttS, \ttT} \ar[r]
& \Res_{F/\QQ}(\GG_m) \times \Res_{E/\QQ}(\GG_m)^{N_{E/F} = 1}
\\
(g,t) \ar@{|->}[r] & \big(\nu(g) N_{E/F}(t), \, t/\bar t\big)
}
\]
 is the natural morphism from $G''_{\tilde \ttS}$ to its maximal abelian quotient.


\begin{remark}
\label{R:E and G'' are independent of S }
Similar to  Notation~\ref{N:G(A infty)}, if $\ttS'$ is another subset of places of $F$ containing the same finite places as $\ttS$ (together with a choice of $\tilde\ttS'_{\infty}$), then  $G''_{\tilde\ttS'}(\AAA^{\infty})$ is isomorphic to $ G''_{\tilde\ttS}(\AAA^{\infty})$. We fix such an isomorphism, and denote them uniformly as $G''(\AAA^\infty)$.
Hence we naturally identify groups $\calG''_{\tilde \ttS,p}$ for different $\tilde \ttS$'s.

When $\#\ttS_\infty = \# \ttS'_\infty$ and $\Delta_{\tilde \ttS_{\infty}} = \Delta_{\tilde \ttS'_{\infty}}$ if $\gothp$ splits in $E/F$, the subgroup $\calE_{G''_{\tilde \ttS'}} \subset \calG''_{\tilde \ttS',p}$ can be also identified with $\calE_{G''_{\tilde \ttS}} \subset \calG''_{\tilde \ttS,p}$.  Indeed, in this case the reciprocity map $\Rec''_p$ for $\tilde \ttS$ and $\tilde \ttS'$ are the same.
\end{remark}

\begin{prop}
\label{P:integral model of Sh}

\emph{(1)}
We have a canonical isomorphism $\calE_{G_{\ttS, \ttT}} \cong \calE_{G''_{\tilde \ttS}}$, and   that $\calS h_{K_p}(G_{\ttS,\ttT})^\circ_{\Q_p^\ur}$ together with the action of $\calE_{G_{\ttS, \ttT}}$ is isomorphic to 
$\calS h_{K''_p}(G''_{\tilde \ttS})^\circ_{ \Qpur}$ together with the action of $\calE_{G''_{\tilde \ttS}}$.

\emph{(2)} The  Shimura varieties $\calS h_K(G_{\ttS, \ttT})$ (resp.  $\calS h_{K''}(G''_{\tilde\ttS})$) admit integral canonical models over $\Z_{p^g}$ (resp. over $\Z_{p^{2g}}$), and the connected Shimura variety $\calS h_{K_p}(G_{\ttS,\ttT})^\circ_{ \Qpur}\cong\calS h_{K''_p}(G''_{\tilde\ttS})^\circ_{ \Qpur}$ admits a canonical integral model over $\ZZ_p^{\ur}$.
\end{prop}
\begin{proof}

For (1), the case when $\ttT = \emptyset$ is treated in \cite{tian-xiao1}.  In general, 
note that the sequence of morphisms
\[
G''_{\tilde \ttS}  \leftarrow G_{\ttS, \ttT} \times T_{E, \tilde \ttS, \ttT} \ra G_{\ttS, \ttT}
\]
is compatible with the associated Deligne homomorphism, and the conjugacy classes of Deligne homomorphisms into various algebraic groups defined above are canonically identified.
Standard facts (e.g. \cite[Corollary~2.17]{tian-xiao1}) about Shimura varieties imply that  the series of morphisms of Shimura varieties
\[
\calS h_{K''_p}(G''_{\tilde \ttS}) 
\leftarrow \calS h_{K_p}(G_{\ttS, \ttT}) \times_{\ZZ_{p^g}} \calS h_{K_{E,p}} ( T_{E, \tilde \ttS, \ttT}) \ra \calS h_{K_p}(G_{\ttS,\ttT})
\]
induce isomorphisms on the neutral connected components. Hence, by \cite[Theorem~3.14]{tian-xiao1}, there exists an integral canonical model for $\calS h_{K''_p}(G''_{\tilde \ttS}) $ over $\ZZ_{p^{2g}}$, and thus the neutral connected component $\calS h_{K''_p}(G''_{\tilde \ttS})^{\circ}\cong \calS h_{K_p}(G_{\ttS,\ttT})^{\circ}$ admits an integral canonical model over $\ZZ_{p}^{\ur}$. 
Applying $\times_{\calE_{G_{\ttS,\ttT}}} \calG_{\ttS,\ttT,p}$, the latter induces an integral canonical model of $\calS h_{K_p}(G_{\ttS,\ttT})$ over $\ZZ_{p}^{\ur}$, which descends to $\ZZ_{p^g}$ (cf. \cite[Corollary~2.17]{tian-xiao1}).
\end{proof}
\begin{remark}
Statement (2) of Proposition~\ref{P:integral model of Sh} is a consequence of the much more general theory of Kisin \cite{kisin}. 
However, in this paper, we will need essentially this explicit relationship between the integral models of $\calS h_{K}(G_{\ttS,\ttT})$  and those of $\calS h_{K''}(G''_{\tilde\ttS})$.
\end{remark}

\begin{notation}\label{N:Shimura varieties}
We use $\calS h_{K_p}(G_{\ttS,\ttT})$, $\calS h_{K_{E,p}}(T_{E,\tilde \ttS ,\ttT})$,  $\calS h_{K''_p}(G''_{\tilde\ttS}), ...$ to denote the integral model over $\Z_{p^g}$ or $\Z_{p^{2g}}$ of the corresponding Shimura variety, and 
use systematically Roman letters to denote the special fibers of Shimura varieties:
\[
\Sh_{?}(G_{\ttS, \ttT})^\star_{\overline \FF_p}: = \calS h_{?}(G_{\ttS, \ttT})^\star_{ \ZZ_p^\ur}\otimes_{\ZZ_p^\ur} \overline \FF_p, \quad \textrm{and} \quad \Sh_?(G_{\ttS, \ttT}): = \calS h_?(G_{\ttS, \ttT}) \times_{\ZZ_{p^g}} \FF_{p^g}.
\]
for $? = K$ or $K_p$, and $\star = \circ$ or $\emptyset$, and 
\[
\Sh_{K_{E,p}}(T_{E,\tilde \ttS ,\ttT})_{\FF_{p^{2g}}} := \calS h_{K_{E,p}}(T_{E,\tilde \ttS ,\ttT}) \otimes_{\ZZ_{p^{2g}}} \FF_{p^{2g}}
, \quad \Sh_{K''_p}(G''_{\tilde \ttS})_{\FF_{p^{2g}}}: = \calS h_{K''_p}(G''_{\tilde \ttS }) \otimes_{\ZZ_{p^{2g}}} \FF_{p^{2g}}
\]
and similarly with open compact subgroups $K_E = K_{E, p} K_E^p \subset T_E(\AAA^{\infty})$ and $K'' = K''_p K''^p \subset G''(\AAA^{\infty})$.  We put $\Sh_{K_p}(G''_{\tilde \ttS})_{\overline \FF_p}^\circ = \calS h_{K_p}(G''_{\tilde \ttS})_{\ZZ_p^\ur}^\circ \otimes_{\ZZ_p^\ur} \overline \FF_p$.
\end{notation}

\subsection{Automorphic local systems}\label{S:automorphic sheaves}
We now study the automorphic sheaves on these Shimura varieties.
Fix a prime $\ell\neq p$ and an isomorphism $\iota_\ell: \CC\simeq \overline{\Q}_\ell$.
Let $(\underline k, w)$ be   a \emph{regular multiweight}, which means a tuple $(\underline k, w) \in \ZZ^{\Sigma_\infty} \times \ZZ$ such that $k_\tau \equiv w \pmod 2$ and $k_\tau \geq 2$ for all $\tau\in \Sigma_\infty$.
Consider the algebraic representation 
\[
\rho^{(\underline k, w)}_{\ttS, \ttT} = \boxtimes_{\tau \in \Sigma_\infty} \big( \Sym^{k_\tau-2}(\mathrm{std}^\vee) \otimes \det{}^{\frac{k_\tau - w}{2}}\big)
\]
of $G_{\ttS, \ttT} \otimes_{\Q} \C \cong \prod_{\tau \in \Sigma_\infty} \GL_2(\CC)$, where std is the standard representation of $\GL_2(\CC)$. 
As explained in \cite{milne}, we have an automorphic $\overline{\Q}_\ell$-lisse sheaf $\calL_{\ttS, \ttT}^{(\underline k, w)}$ on $\calS h_{K_p}(G_{\ttS, \ttT})$ associated to $\rho^{(\kb, w)}_{\ttS, \ttT}$.
Note that  $\calL_{\ttS,\ttT}^{(\kb,w)}$ is pure of weight $(w-2)(g-\#\ttS_{\infty}+2\#\ttT)$.

We fix a section $\tilde \Sigma  \subset\Sigma_{E,\infty}$ of the natural restriction map $\Sigma_{E, \infty} \to \Sigma_\infty$ (which is independent of the choices $\tilde \ttS_\infty$).
 Consider the following one-dimensional representation of $T_{E, \tilde \ttS, \ttT} \otimes_\Q \CC \cong \prod_{\tilde \tau \in \tilde \Sigma}\GG_{m, \tilde \tau} \times \GG_{m, \overline{\tilde \tau}}$:
\[
\rho_{E, \tilde \Sigma}^w = \otimes_{\tilde \tau \in \tilde \Sigma}\; x^{2-w} \circ \pr_{E, \tilde \tau},
\]
where $\overline{\tilde \tau}$ is the complex conjugate embedding of $\tilde \tau$, $\pr_{E, \tilde \tau}$ is the projection to the $\tilde \tau$-component, and $x^{2-w}$ is the character of $\CC^\times$ given by raising to the $(2-w)$th power.
This representation gives rise to a lisse $\overline \QQ_\ell$-\'etale sheaf $\calL^{w}_{E, \tilde \ttS, \ttT, \tilde \Sigma}$ pure of weight $(w-2)(\#\ttS_{\infty}-2\#\ttT)$ on $\calS h_{K_{E,p}}(T_{E, \tilde \ttS, \ttT})$. 
If $\Sh_{K_E}(T_{E,\tilde\ttS',\ttT'})$ is another Shimura variety with $\#\ttS'_{\infty}-2\#\ttT'=\#\ttS_{\infty}-2\#\ttT$ and $\gamma: \Sh_{K_{E}}(T_{E,\tilde\ttS,\ttT})\cong \Sh_{K_E}(T_{E,\tilde\ttS',\ttT'})$ is the isomorphism \eqref{E:isom-Shimura-var-E}, then  we have a natural isomorphism 
\begin{equation}
\label{E:gamma identifies sheaves}
\gamma^*(\calL^w_{E,\tilde\ttS',\ttT',\tilde \Sigma})\xra{\sim} \calL^w_{E,\tilde\ttS,\ttT,\tilde \Sigma}.
\end{equation}

Let $\alpha_{\ttT}: G_{\ttS,\ttT} \times T_{E, \tilde \ttS, \ttT} \to G''_{\tilde \ttS}$ denote the natural quotient morphism. We have the following diagram.
\begin{equation}\label{E:diag-Sh-var}
\xymatrix{
\calS h_{K_p}(G_{\ttS,\ttT}) &
  \calS h_{K_p}(G_{\ttS,\ttT}) \times_{\ZZ_{p^g}} \calS h_{ K_{E,p}}( T_{E,\tilde\ttS,\ttT})\ar[l]_-{\pr_1}\ar[d]^-{\pr_2} \ar[r]^-{\boldsymbol\alpha_{\ttT}} & \calS h_{K''_p}(G''_{\tilde\ttS})\\
&\calS h_{K_{E,p}}(T_{E,\tilde\ttS,\ttT}).
}
\end{equation}

By our definition, the tensor product representation $\rho^{(\underline k,w)}_{\ttS, \ttT} \otimes \rho_{E,  \tilde \Sigma}^w$ of $G_{\ttS, \ttT} \times T_{E, \tilde \ttS, \ttT}$ factors through $G''_{\tilde \ttS}$.
This  defines a $\overline \QQ_\ell$-lisse sheaf   $\calL''^{(\underline k, w)}_{\tilde \ttS, \tilde \Sigma}$ on $\calS h_{K''_p}(G''_{\tilde \ttS})$ such that we have a canonical isomorphism
\begin{equation}\label{E:pull back along alpha = tensor}
\boldsymbol \alpha^*_{\ttT}(\calL''^{(\underline k, w)}_{\tilde \ttS, \tilde \Sigma}) \cong \pr_1^*(\calL_{\ttS, \ttT}^{(\underline k, w)}) \otimes \pr_2^*(\calL_{E, \tilde \ttS, \ttT, \tilde \Sigma}^{ w}).
\end{equation}

 Put $D=B_{\ttS}\otimes_{F}E$.
  Then our choice of $E/F$ in \ref{A:CM} implies that $D\cong \rmM_{2\times 2}(E)$, which explains the  omission  of $\ttS$ in our notation.
  We fix such an isomorphism, and take a maximal order $\cO_{D}\cong \rmM_{2\times 2}(\cO_E)$. 
Recall that  there exists a versal family of abelian varieties of dimension $4g$ $a:\bfA''_{\tilde\ttS, K''_p}\ra \calS h_{K_{p}''}(G''_{\tilde\ttS})$ \cite[3.20]{tian-xiao1} equipped with a natural action by $\cO_{D}$.
 Here, ``versal'' means that the Kodaira--Spencer map for the family $\bfA''_{\tilde\ttS}$  is an isomorphism.
 Using $\bfA''_{\tilde\ttS}$, $\calL''^{(\kb,w)}_{\tilde\ttS,\tilde\Sigma}$ can be reinterpreted as follows. 
  Put $H_\ell(\bfA''_{\tilde\ttS})=R^1a_{*}(\overline\QQ_{\ell})$, which is an $\ell$-adic local system on $\calS h_{K''_p}(G''_{\tilde\ttS})$ equipped with an induced action by $\rmM_{2\times 2}(E)$. 
  For each $\tilde\tau\in \Sigma_{E,\infty}$, let $H_\ell(\bfA''_{\tilde\ttS})_{\tilde\tau}$ denote the direct summand of $H_\ell(\bfA''_{\tilde\ttS})$ on which $E$ acts via  $E\xra{\tilde\tau} \CC\xra{\iota_\ell} \overline\QQ_\ell$.
  Consider the idempotent $\gothe=\big(\begin{smallmatrix}1 &0\\ 0&0\end{smallmatrix}\big)\in \rmM_{2\times 2}(E)$. We put $H_\ell(\bfA''_{\tilde\ttS})_{\tilde\tau}^{\circ}=\gothe \cdot H_\ell(\bfA''_{\tilde\ttS})_{\tilde\tau}$, which is an $\ell$-adic local system on $\calS h_{K''_p}(G''_{\tilde\ttS})$ of rank $2$.
   We have a canonical decomposition
 \[
 H_\ell(\bfA''_{\tilde \ttS})=\bigoplus_{\tilde\tau\in \tilde\Sigma}\big(H_\ell(\bfA''_{\tilde\ttS})_{\tilde\tau}\oplus H_\ell(\bfA''_{\tilde\ttS})_{\overline{\tilde \tau}}\big)=\bigoplus_{\tilde\tau\in \tilde\Sigma} \big(H_\ell(\bfA''_{\tilde\ttS})_{\tilde\tau}^{\circ,\oplus2 }\oplus H_{\ell}(\bfA''_{\tilde\ttS})_{\overline{\tilde \tau}}^{\circ,\oplus 2}\big).\]
 Then one has
 \begin{equation}\label{E:etale-sheaf-AV}
 \calL''^{(\underline k, w)}_{\tilde\ttS,\tilde\Sigma}=\bigotimes_{\tilde \tau\in \tilde \Sigma} \bigg(\Sym^{k_{\tau}-2}H_\ell(\bfA''_{\tilde\ttS})^{\circ}_{\tilde\tau}\otimes (\wedge^2 H_{\ell}(\bfA''_{\tilde\ttS})^{\circ}_{\tilde\tau})^{\frac{w-k_{\tau}}{2}}\bigg).
 \end{equation}

\begin{remark}
We shall introduce a general construction below to relate the unitary Shimura varieties and the quaternionic Shimura varieties. 
 We point out beforehand that the entire construction is modeled on the following question:
By Hilbert 90 theorem, we have an exact sequence:
\[
1 \to F^{\times,\cl}\backslash \AAA^{\infty,\times}_{F}/ \calO_\gothp^\times \to E^{\times,\cl}\backslash \AAA_{E}^{\infty,\times}/\cO_{E_{\gothp}}^{\times}\xrightarrow{z \mapsto z/\bar z}E^{1} \backslash \AAA_{E}^{ \infty, 1} / \calO_{E_{\gothp}}^{N_{E/F} = 1} \to 1.
\]
The construction involves picking a preimage of some element in the target of the surjective map above.  In general, there is no canonical choice of this preimage, and all choices form a torsor under the group $F^{\times,\cl}\backslash \AAA^{\infty,\times}_{F}/ \calO_\gothp^\times$.  In very special case when the element in the target of the surjective map is trivial, one can have a canonical choice of its preimage, namely the identity element $1$.
\end{remark}
 
\begin{construction}
\label{S:from unitary to quaternionic}
We now discuss a very important process that allows us to transfer certain correspondences on the unitary Shimura varieties $\Sh_{K''_p}(G''_{\tilde \ttS})$ to the quaternionic Shimura varieties $\Sh_{K_p}(G_{\ttS, \ttT})$.
Throughout this subsection, we assume that we are given two sets of data: $\tilde \ttS,\ttT$, and $\tilde \ttS',\ttT'$ as above, and that they satisfy the following conditions:
\begin{equation}
\label{E:transfer condition}
\#\ttS_\infty - 2\#\ttT = \#\ttS'_\infty - 2\#\ttT', \quad \quad \Delta_{\tilde \ttS_{\infty}} = \Delta_{\tilde \ttS'_{\infty}} \textrm{ if }\gothp\textrm{ splits in }E/F,
\end{equation}
and the finite places contained in $\ttS$ and  those in $\ttS'$ are  the same. By \eqref{E:isom-Shimura-var-E}, this implies that the Shimura varieties $\Sh_{K_{E,p}}(T_{E, \tilde \ttS, \ttT}) $ and $\Sh_{K_{E,p}}(T_{E, \tilde \ttS', \ttT'}) $ are isomorphic.

Suppose that we are given a correspondence between the two unitary Shimura varieties
\begin{equation}\label{E:correspondence-Sh}
\Sh_{K''_p}(G''_{\tilde \ttS}) \xleftarrow{\ \pi''\ } X \xrightarrow{\ \eta''\ } \Sh_{K''_p}(G''_{\tilde \ttS'}),
\end{equation}
where the group $\calG''_{\tilde \ttS,p} \cong \calG''_{\tilde \ttS',p}$ acts on all three spaces and the morphisms are equivariant for the actions.
We further assume that the fibers of  $\pi''$  are  geometrically connected.

{\bf Step I:} We will complete the correspondence above into the following commutative diagram
\begin{equation}
\label{E:lifting process}
\xymatrix{
\Sh_{K_p}(G_{\ttS, \ttT}) \times_{\Spec(\FF_{p^{d}})} \Sh_{K_{E,p}}(T_{E, \tilde \ttS, \ttT}) \ar[d]^{\boldsymbol \alpha_\ttT} & \ar[l]_-{\pi^\times} Y \ar[r]^-{\eta^\times} \ar[d]^{\boldsymbol \alpha''_\ttT} & \Sh_{K_p}(G_{\ttS', \ttT'}) \times_{\Spec(\FF_{p^{d}})} \Sh_{K_{E,p}}(T_{E, \tilde \ttS', \ttT'})\ar[d]^{\boldsymbol \alpha'_{\ttT'}}
\\
\Sh_{K''_p}(G''_{\tilde \ttS}) & \ar[l]_-{\pi''} X  \ar[r]^-{\eta''} & \Sh_{K''_p}(G''_{\tilde \ttS'}),
}
\end{equation}
so that $Y$ is defined as the Cartesian product of the left square, and the top line is equivariant for the actions of $\calG_{\ttS,\ttT,p} \times \AAA^{\infty, \times}_E \cong \calG_{\ttS',\ttT',p} \times \AAA^{\infty, \times}_E$.  For this, it suffices to lift the morphism $\eta''$ to $\eta^\times$.
We point out that both $\boldsymbol \alpha_{\ttT}$ and $\boldsymbol \alpha'_{\ttT'}$ map every geometric connected component isomorphically to another geometric connected component.

We now separate the discussion (but not in an essential way) depending on whether $\ttS_\infty^c$ is empty.
\begin{itemize}

\item
When $\ttS_\infty^c \neq \emptyset$,
let  $Y^\circ$ denote the preimage $(\pi^{\times})^{-1}(\Sh_{K_p}(G_{\ttS, \ttT})^\circ_{\overline \FF_p} \times \{ \boldsymbol 1\})$, where $\boldsymbol 1$ denotes the neutral point, namely the image of $1\in \AAA_{E}^{\infty,\times}$ in $\Sh_{K_{E,p}}(T_{E, \tilde \ttS, \ttT})_{\overline \FF_p}$. 
By our assumption on $\pi''$, $Y^\circ$ is a geometric connected component of $Y$.
Its image under $\eta'' \circ \bbalpha''_\ttT$ lies in a geometric connected component of $\Sh_{K''_p} (G''_{\tilde \ttS'})$, say $\Sh_{K''_p} (G''_{\tilde \ttS'})^\bullet_{\overline \FF_p}$, corresponding to some  $(\boldsymbol x, \boldsymbol s)\in \big( F_+^{\times, \mathrm{cl}} \backslash \AAA_F^{\infty, \times} / \calO_\gothp^\times \big) \times 
\big( 
E^{1} \backslash \AAA_{E}^{\infty, 1} / \calO_{E_{\gothp}}^{N_{E/F} = 1}
\big)$ via the map \eqref{E:geometric-connected-unitary}. By Hilbert's 90th Theorem, 
there exists $\boldsymbol t\in E^{\times,\cl} \backslash \AAA_{E}^{\infty, \times} / \calO_{E_{\gothp}}^{\times}$ with $\boldsymbol t/\bar {\boldsymbol t}=\boldsymbol s$, and the choice of $\boldsymbol t$ is unique up to $F^{\times,\cl} \backslash \AAA_{F}^{\infty, \times} / \calO_{{\gothp}}^{\times}$.
We claim that giving a $(\calG_{\ttS,\ttT,p} \times \AAA^{\infty, \times}_E)$-equivariant map $\eta^\times$  as above is \emph{equivalent} to choosing such a $\boldsymbol t$.

Indeed, let $\Sh_{K_p}(G_{\ttS',\ttT'})^{\bullet}_{\overline \FF_p}$ be the connected component of $\Sh_{K_p}(G_{\ttS',\ttT'})_{\overline \FF_p}$ corresponding to $\boldsymbol y=\boldsymbol x N_{E/F}(\boldsymbol t)^{-1}$ via the map \eqref{E:connected-components}.
 Then $\boldsymbol\alpha'_{\ttT'}$ sends $\Sh_{K_p}(G_{\ttS',\ttT'})^{\bullet}_{\overline \FF_p}\times \{\boldsymbol t\}$ isomorphically to $\Sh_{K''_p}(G''_{\tilde \ttS'})^\bullet_{\overline \FF_p}$.
Note that $Y$ (resp. $\Sh_{K_p}(G_{\ttS', \ttT'}) \times_{\Spec(\FF_{p^{d}})} \Sh_{K_{E,p}}(T_{E, \tilde \ttS', \ttT'})$) can be recovered from $Y^{\circ}$ (resp. $\Sh_{K_p}(G_{\ttS', \ttT'})^\bullet_{\overline \FF_p} \times \{ \boldsymbol t\}$ for any $\boldsymbol t\in E^{\times, \cl}\backslash \AAA_E^{\infty, \times} / \calO^\times_{E_\gothp} $ \footnote{We point out that $E^{\times, \cl}\backslash \AAA_E^{\infty, \times} / \calO^\times_{E_{\gothp}}$ is canonically isomorphic to  $\calO_{E, (p)}^{\times, \cl} \backslash \AAA_E^{\infty,p, \times}$.}) by applying  $-\times_{\calE_{G_{ \ttS,\ttT}}} \big(\calG_{\ttS, \ttT, p} \times E^{\times, \cl}\backslash \AAA_E^{\infty, \times} / \calO^\times_{E_{\gothp}} \big)$. 
Here, recall that $\calE_{G_{\ttS,\ttT}}\cong \calE_{G''_{\tilde \ttS}}$ by Proposition~\ref{P:integral model of Sh}(1), and it embeds into the product $\calG_{\ttS, \ttT, p} \times E^{\times, \cl}\backslash \AAA_E^{\infty, \times} / \calO^\times_{E_{\gothp}}$ as follows: 
the morphism from  $ \calE_{G_{\ttS,\ttT}}$ to  $\calG_{\ttS, \ttT, p} $ is  the natural embedding and that to $ E^{\times, \cl}\backslash \AAA_E^{\infty, \times} / \calO^\times_{E_{\gothp}}$ is given by first projecting to the Galois factor and then applying the Shimura reciprocity map as specified in Subsection~\ref{A:CM}(i) and (ii). 
Therefore,  once such a $\boldsymbol t$  is chosen, we can define $\eta^\times$ as the morphism obtained by applying $-\times_{\calE_{G_{\ttS,\ttT}}} \big(\calG_{\ttS, \ttT, p} \times E^{\times, \cl}\backslash \AAA_E^{\infty, \times} / \calO^\times_{E_{\gothp}} \big)$ to the composite map 
\[
 Y^{\circ}\xra{\eta''\circ \boldsymbol \alpha_\ttT''} \Sh_{K''_p}(G''_{\tilde\ttS'})_{\overline \FF_p}^{\bullet}\xra{\sim} \Sh_{K_p}(G_{\ttS', \ttT'})^\bullet_{\overline \FF_p} \times \{ \boldsymbol t\},
\]
where the last isomorphism is the inverse of the restriction of  $\boldsymbol\alpha'_{\ttT'}$ to $\Sh_{K_p}(G_{\ttS', \ttT'})^\bullet_{\overline \FF_p} \times \{ \boldsymbol t\}$. Conversely, it is also clear that such a $\boldsymbol t$ is determined by $\eta^{\times}$.

\item
When $\ttS_\infty^c =\emptyset$,
a slight rewording is needed. 
Let $X^\circ$ denote the preimage under $\pi''$ of the $\overline \FF_p$-point $\boldsymbol 1 \in \Sh_{K''_p}(G''_{\tilde \ttS})_{\overline \FF_p}$.  So it is mapped under $\eta''$ to a point $\boldsymbol g'' \in \Sh_{K''_p}(G''_{\tilde \ttS})_{\overline \FF_p}$. Let $Y^\circ$ denote the preimage under $\pi^\times$ of the $\overline \FF_p$-point $(\boldsymbol 1 , \boldsymbol 1)  \in \Sh_{K_p}(G_{\ttS, \ttT})_{\overline \FF_p} \times \Sh_{K_{E,p}} (T_{E, \tilde \ttS, \ttT})_{\overline \FF_p}$.
Then the map $\eta^\times$ must take $Y^\circ$  to an $\overline \FF_p$-point $(\boldsymbol x, \boldsymbol t)$ in $\boldsymbol \alpha'^{-1}_{\ttT'}(\boldsymbol g'')$, and conversely, $\eta^\times$ is determined by this choice of such a  point by the same argument as above using the fact that $\eta^\times$ is equivariant for the $(\calG_{\ttS,\ttT,p} \times \AAA^{\infty, \times}_E)$-action.\end{itemize}

In summary, one can always define such a lift $\eta^\times$, depending on a choice of a certain element $\boldsymbol t\in E^{\times,\cl} \backslash \AAA_{E}^{\infty, \times} / \calO_{E_{\gothp}}^{\times} $ which is  unique up to multiplication by an element of $F^{\times, \cl}\backslash \AAA_F^{\infty, \times} / \calO^\times_{\gothp}$. 
In this case, we say that $\eta^{\times}$ is constructed with \emph{shift $\boldsymbol t$}. In general, we do not have a canonical choice for $\boldsymbol t$, and hence neither for $\eta^\times$.  However, in the special case when  
$\Sh_{K''_p}(G''_{\tilde \ttS'})^\bullet_{\overline \FF_p}$
 is the neutral connected component $\Sh_{K''_p}(G''_{\tilde \ttS'})^\circ_{\overline \FF_p}$ in the former case and $\boldsymbol g'' = \boldsymbol 1$ in the later case,  there is a canonical choice of such lift, namely, the neutral connected component $\Sh_{K_p}(G_{\ttS', \ttT'})^\circ_{\overline \FF_p} \times \{ \boldsymbol 1\}$ in the former case and $(\boldsymbol 1, \boldsymbol 1)$ in the latter case.  So under this additional hypothesis, we do have a canonical map $\eta^\times$.\\

{\bf Step II:} Suppose that we have  constructed the diagram~\eqref{E:lifting process} with shift $\boldsymbol t$ (which is canonical up to an element of $F^{\times, \cl}\backslash \AAA_F^{\infty, \times} / \calO^\times_{\gothp}$), we want to obtain a correspondence 
\begin{equation}
\label{E:correspondence on quaternionic Sh var}
\Sh_{K_p}(G_{\ttS, \ttT})_{\FF_{p^{2g}}} \xleftarrow{\ \pi\ } Z \xrightarrow{\ \eta\ } \Sh_{K_p}(G_{\ttS', \ttT'})_{\FF_{p^{2g}}}.
\end{equation}
For this, it suffices to construct \eqref{E:correspondence on quaternionic Sh var} over $\overline \FF_p$ which carries equivariant action of $\Gal_{\FF_{p^{2g}}}$.
Starting with the top row of \eqref{E:lifting process}, composing $\eta^\times$ with multiplication by $\boldsymbol t^{-1}$ (note that $\Sh_{K_{E,p}}(T_{E, \tilde \ttS', \ttT'})$ is in fact a group scheme), we get a correspondence\footnote{Once again, this correspondence depend on the choice of $\boldsymbol t$, which is  unique up to multiplication by an element of $\calO_{F,(p)}^{\times, \cl} \backslash\AAA_F^{\times, \infty, p}$.}
\begin{equation}
\label{E:after times x-1}
\Sh_{K_p}(G_{\ttS, \ttT})_{\overline \FF_p} \times \Sh_{K_{E,p}}(T_{E, \tilde \ttS, \ttT})_{\overline \FF_p} \xleftarrow{\ \pi^\times\ } Y \xrightarrow{\ \boldsymbol t^{-1} \circ \eta^\times\ } \Sh_{K_p}(G_{\ttS', \ttT'})_{\overline \FF_p} \times \Sh_{K_{E,p}}(T_{E, \tilde \ttS', \ttT'})_{\overline \FF_p},
\end{equation}
which respects the projection to $\Sh_{K_{E,p}}(T_{E, \tilde \ttS, \ttT})_{\overline \FF_p} \stackrel \gamma \cong \Sh_{K_{E,p}}(T_{E, \tilde \ttS', \ttT'})_{\overline \FF_p}$.
Taking the fiber of \eqref{E:after times x-1} over $\boldsymbol 1$ of $\Sh_{K_{E,p}}(T_{E, \tilde \ttS, \ttT})_{\overline \FF_p}$ would gives \eqref{E:correspondence on quaternionic Sh var} (based changed to $\overline \FF_p$), but to descend we need to modify the Galois action above (so that the Galois action preserves the fiber over $\boldsymbol 1$) as follows:
we change the action of $\Frob_{p^{2g}}$ on \eqref{E:after times x-1} by further composing with a Hecke action given by $1 \times (\underline p_F)^{2\#\ttT - \#\ttS_\infty}
\in G(\AAA^\infty) \times \AAA_E^{\infty, \times}$ if $\gothp$ is inert in $E/F$, and $1 \times (\underline p_F)^{2\#\ttT - \#\ttS_\infty } (\underline \gothq)^{-\Delta_{\tilde \ttS_{\infty}}}$ if $\gothp$ splits in $E/F$.
This way, we have constructed a new Galois action on the factor $\Sh_{K_{E,p}}(T_{E,\tilde \ttS, \ttT})_{\overline \FF_p}$. By usual Galois descent, we get \eqref{E:correspondence on quaternionic Sh var}.

{\bf Step III:}  we will obtain a sheaf version of the construction above, namely, if in addition, we are given an isomorphism of sheaves
\begin{equation}
\label{E:sheaf morphism for ''}
\eta''^\sharp: \pi''^*(\calL''^{(\underline k, w)}_{\tilde \ttS, \tilde \Sigma}) \xrightarrow{\ \cong\ } \eta''^*(\calL''^{(\underline k, w)}_{\tilde \ttS', \tilde \Sigma}),
\end{equation}
then we will construct an isomorphism of sheaves 
\begin{equation}
\label{E:sheaf morphism on quaternionic}
\eta^\sharp: \pi^*(\calL^{(\underline k, w)}_{\ttS, \ttT})\xrightarrow{\ \cong\ } \eta^*(\calL^{(\underline k, w)}_{\ttS', \ttT'}),
\end{equation}
which again depends on the choice of $\boldsymbol t$ in Step I.
First, pulling back \eqref{E:sheaf morphism for ''} along $\boldsymbol \alpha''_{\ttT}$ in the commutative diagram \eqref{E:lifting process}, we get
\[
\boldsymbol \alpha''^*_\ttT(\eta''^\sharp): (\pi^\times)^* \big( \boldsymbol \alpha_\ttT^*(\calL''^{(\underline k, w)}_{\tilde \ttS, \tilde \Sigma}) \big) \xrightarrow{\ \cong\ } (\eta^\times)^*\big( \boldsymbol \alpha'^*_{\ttT'}(\calL''^{(\underline k, w)}_{\tilde \ttS', \tilde \Sigma}) \big)
\]
Taking into account the isomorphism \eqref{E:pull back along alpha = tensor}, we have
\[
\boldsymbol \alpha''^*_\ttT(\eta''^\sharp): (\pi^\times)^* \big( \pr_1^*(\calL_{\ttS, \ttT}^{(\underline k, w)}) \otimes \pr_2^*(\calL_{E, \tilde \ttS, \ttT, \tilde \Sigma}^{ w} )\big) \xrightarrow{\ \cong\ } (\eta^\times)^*\big( \pr'^*_1(\calL_{\ttS', \ttT'}^{(\underline k, w)}) \otimes \pr'^*_2(\calL_{E, \tilde \ttS', \ttT', \tilde \Sigma}^{ w}) \big).
\]
Composing this with the action of $\boldsymbol t^{-1}$, we get an isomorphism
\[
(\pi^\times)^*\big( \pr_1^*(\calL_{\ttS, \ttT}^{(\underline k, w)}) \otimes \pr_2^*(\calL_{E, \tilde \ttS, \ttT, \tilde \Sigma}^{ w} )\big) \xrightarrow{\ \cong\ }
(\boldsymbol t^{-1} \circ \eta^\times)^*\big( \pr'^*_1(\calL_{\ttS', \ttT'}^{(\underline k, w)}) \otimes \pr'^*_2(\calL_{E, \tilde \ttS', \ttT', \tilde \Sigma}^{ w}) \big).
\]
Since we may also identify the sheaves $\calL_{E, \tilde \ttS, \ttT, \tilde \Sigma}^{ w}$ with $\calL_{E, \tilde \ttS', \ttT', \tilde \Sigma}^{ w}$ using \eqref{E:gamma identifies sheaves}, we may restrict the morphism above to the fiber over the neutral point $\boldsymbol 1$ and get a morphism of sheaves \eqref{E:sheaf morphism on quaternionic} we want over $\overline \FF_p$. (Once again, this morphism is unique up to multiplication with an element of $F^{\times, \cl}\backslash \AAA_F^{\infty, \times} / \calO^\times_{\gothp}$.)
To descend it back down to $\FF_{p^{2g}}$, we modify the action of the  Frobenius by composing it with a central Hecke action as in Step II above.
This concludes the needed construction.

{\bf Step IV:}
Understand the ambiguity appeared in the construction.
\emph{We call $\eta$ the morphism associated to $\eta''$ with shift $\boldsymbol t$}, where $\boldsymbol t \in E^{\times, \cl}\backslash \AAA_E^{\infty, \times} / \calO^\times_{E,\gothp}$ is the element appeared in Step I, and is  determined only up to multiplication by an element of $F^{\times, \cl}\backslash \AAA_F^{\infty, \times} / \calO^\times_{\gothp}$.
When $\Sh_{K''_p}(G''_{\tilde \ttS})^\bullet_{\overline \FF_p} = \Sh_{K''_p}(G''_{\tilde \ttS})^\circ_{\overline \FF_p} $ in Step I, we can take $\boldsymbol t=\boldsymbol 1$ and we get a canonically determined $\eta$ with shift $\boldsymbol 1$.

Finally, let us mention where the choice made in Step I is specified later in this paper. 
In Subsection~\ref{S:GO cycles}, we invoke this construction to define the Goren--Oort cycles; this is where the choice will be fixed.  Moreover, this choice will retroactively determine the choice we make when applying this construction to define link morphisms in the earlier
 Subsection~\ref{S:link morphism}, whenever this subsection is quoted.  The shift will allow us to keep track of the choices we made.

\end{construction}


\begin{remark}
\label{R:composition of shifts}
Suppose that we are given two correspondences as in Construction~\ref{S:from unitary to quaternionic}. Namely, we have
\begin{itemize}
\item subsets of $\tilde \ttS_i$, $\ttT_i$ for $i = 1,2,3$ such that
$\#\ttS_{i, \infty} - 2 \#\ttT_i$, the subset of $\ttS_i$ of finite places, and $\Delta_{\tilde \ttS_{i,\infty}}$ are independent of $i$, and
\item
two $\calG''_{\tilde \ttS_i,p}$-equivariant correspondences between Shimura varieties
\[
\Sh_{K''_p}(G''_{\tilde \ttS_i}) \xleftarrow{\pi''_i} X_i \xrightarrow{\eta''_i}
\Sh_{K''_p}(G''_{\tilde \ttS_{i+1}})
\]
with $i=1,2$
 such that $\pi''_i$ is a fiber bundle with geometric connected fibers.
 \end{itemize}
Then we can compose these two correspondences to get a correspondence
\[
\Sh_{K''_p}(G''_{\tilde\ttS_1}) \xleftarrow{ \pi''_3} X_3: =X_1 \times_{\eta''_1, \Sh_{K''_p}(G''_{\tilde \ttS_2}), \pi''_2} X_2 \xrightarrow{\eta''_3} \Sh_{K''_p}(G''_{\tilde \ttS_{3}}).
\]
Thus we may apply Construction~\ref{S:from unitary to quaternionic} to get correspondences $(\pi_1, \eta_1)$ and $(\pi_2, \eta_2)$ on the quaternionic Shimura varieties:
\begin{equation}
\label{E:composition of correspondences}
\xymatrix@R=0pt{
&& X_3 \ar[dl] \ar@/_15pt/[ddll]_-{\pi_3} \ar[dr] \ar@/^15pt/[ddrr]^-{\eta_3} \\
&X_1 \ar[dl]_-{\pi_1} \ar[dr]^-{\eta_1} && X_2 \ar[dl]_-{\pi_2}\ar[dr]^-{\eta_2}
\\
\Sh_{K_p}(G_{\ttS_1,\ttT_1}) &&
\Sh_{K_p}(G_{\ttS_2,\ttT_2}) &&
\Sh_{K_p}(G_{\ttS_3,\ttT_3}),
}
\end{equation}
with shifts $\boldsymbol t_1, \boldsymbol t_2$. Then their composition $(\pi_3, \eta_3) =(\pi_2, \eta_2) \circ (\pi_1, \eta_1)$ is the correspondence of quaternionic Shimura varieties associated to $(\pi''_3, \eta''_3)$ with shift $\boldsymbol t_1 \boldsymbol t_2$.

Conversely, if we apply Construction~\ref{S:from unitary to quaternionic} to $(\pi''_i, \eta''_i)$ to get three correspondences $(\pi_i, \eta_i)$ for $i=1,2,3$ such that $(\pi_3, \eta_3) = (\pi_2, \eta_2)\circ (\pi_1, \eta_1)$, then their shifts satisfy the equality $\boldsymbol t_3 = \boldsymbol t_1 \boldsymbol t_2$.
\end{remark}

\medskip


\subsection{Hecke operators at $p$}\label{S:Hecke-at-p}

In this subsection, we consider the case $\ttS_\infty^c =\emptyset$, namely when the Shimura varieties are discrete. We want to relate the Hecke operators at $p$ for the unitary and quaternionic Shimura varieties in a manner similar as above.
We assume that $\gothp$ splits in $E/F$ which is the case we need encounter later.

Let $\Iw_{p}\subseteq \GL_2(\cO_{\gothp})$ denote the subgroup consisting of matrices which are upper triangular when modulo $\gothp$.  The discussion in this section is designed to cover this case and give an integral canonical model $\calS h_{\Iw_p}(G_{\ttS,\ttT})$ of the Shimura variety with Iwahori level structure.
We denote by $T_\gothp$ the Hecke correspondence given by the following diagram:
  \begin{equation}\label{E:Tp operator correspondence}
  \xymatrix{&\calS h_{\Iw_p}(G_{\ttS,\ttT})\ar[rd]^{\pi_2}\ar[ld]_{\pi_1}\\
  \calS h_{K_p}(G_{\ttS,\ttT}) && \calS h_{K_p}(G_{\ttS,\ttT}),}
  \end{equation}
 where $\pi_1$ is the natural projection, and $\pi_2$ sends the double coset of $x\in G(\AAA^{\infty})$ to that of $x\big(\begin{smallmatrix}\underline p_F^{-1} &0\\ 0&1\end{smallmatrix}\big)$.
 
For the unitary Shimura variety, we have $G''(\Q_p)\cong\GL_2(F_\gothp)\times_{F^{\times}_{\gothp}}(E_\gothq^\times \times E_{\bar \gothq}^\times)$ and we use $\Iw''_p$ to denote the subgroup $\Iw_p \times_{\calO_\gothp^\times}(\calO_{E_\gothq}^\times \times \calO_{E_{\bar \gothq}}^\times)$.  Similarly, we have an integral model $\calS h_{\Iw''_p}(G''_{\tilde \ttS})$ of the unitary Shimura variety with this Iwahori level structure.
The element $\gamma''_\gothq = \big( \big( \begin{smallmatrix} p^{-1} &0\\0&1
\end{smallmatrix} \big), (1, p)\big) \in G''(\QQ_p)$ gives rise to a Hecke operator $T_\gothq$ corresponding to the double coset $K''_p \gamma''_\gothq K''_p$.  Geometrically, it is given by the the following diagram
\begin{equation}
\label{E:Tq correspondence}
\xymatrix{&\calS h_{\Iw''_p}(G''_{\tilde \ttS})\ar[rd]^{\pi''_2}\ar[ld]_{\pi''_1}\\
  \calS h_{K''_p}(G''_{\tilde \ttS}) && \calS h_{K''_p}(G''_{\tilde \ttS}),}
\end{equation}
where $\pi''_1$ is the natural projection, and $\pi''_2$ sends the double coset of $x\in G''(\AAA^{\infty})$ to that of $x\gamma''_\gothq$.

In a language similar to the previous subsection (except that we cannot quote it directly because the morphism $\pi''$ therein would not have geometric connected fibers), we may phrase the relation between the Hecke correspondences $T_\gothp$ and $T_\gothq$ in terms of the following commutative diagram (with $T_\gothq$ vertical on the left and $T_\gothp$ vertical on the right)
\[
\xymatrix@C=30pt{
\Sh_{K''_p}(G''_{\tilde \ttS}) & \ar[l] _-{\alpha_\ttT}
\Sh_{K_p}(G_{\ttS, \ttT}) \times \Sh_{K_{E,p}}(T_{E, \tilde \ttS, \ttT}) && \ar[ll] _-{\textrm{filber over }\boldsymbol 1}
\Sh_{K_p}(G_{\ttS, \ttT})
\\
\Sh_{\Iw''_p}(G''_{\tilde \ttS})\ar[u]^{\pi''_1} \ar[d]_{\pi''_2} & \ar[u]_{\textrm{natural}} \ar[l] _-{\alpha_\ttT}
\Sh_{\Iw_p}(G_{\ttS, \ttT}) \times \Sh_{K_{E,p}}(T_{E, \tilde \ttS, \ttT}) \ar[d]^{x \mapsto x \big( \big( {\begin{smallmatrix} \underline p_F^{-1} &0\\0&1 \end{smallmatrix}} \big), \varpi_{\bar \gothq} \big)} && \ar[ll] _-{\textrm{filber over }\boldsymbol 1} \ar[u]_{\pi_1} \ar[dd]^{\pi_2}
\Sh_{\Iw_p}(G_{\ttS, \ttT})
\\
\Sh_{K''_p}(G''_{\tilde \ttS}) & \ar[l] _-{\alpha_\ttT} \ar[d]^{x \mapsto x \big(1, \varpi_{\bar \gothq}^{-1}\big)}
\Sh_{K_p}(G_{\ttS, \ttT}) \times \Sh_{K_{E,p}}(T_{E, \tilde \ttS, \ttT})
\\
&
\Sh_{K_p}(G_{\ttS, \ttT}) \times \Sh_{K_{E,p}}(T_{E, \tilde \ttS, \ttT}) && \ar[ll] _-{\textrm{filber over }\boldsymbol 1}
\Sh_{K_p}(G_{\ttS, \ttT}).
}
\]
So we may view $T_\gothp$ as the correspondence associated to $T_\gothq$ in a similar fashion to the previous subsection, with shift $\varpi_{\bar \gothq}^{-1} \in E^{\times, \cl}\backslash \AAA_E^{\infty, \times} / \calO^\times_{E_{\gothp}}$.

\subsection{Links}
\label{S:links}
We recall briefly the notion of links introduced in \cite[\S 7]{tian-xiao1}. 
We put $g=[F:\QQ]$ points aligned equi-distantly on a section of a vertical cylinder, one point corresponding to an archimedean place in  $\Sigma_\infty$ (also identified with a $p$-adic embedding of $F$ via $\iota_p: \CC\cong \overline\Q_p$) so that the Frobenius action is equivalent to shifting the points in one direction.
For a subset $\ttS$  of places of $F$ as above, we label places in $\ttS_\infty$ by a \emph{plus sign} and places in $\ttS_\infty^c$ by a \emph{node}. We call the entire picture a \emph{band} corresponding to $\ttS$.
We often draw the picture in the $2$-dimensional $xy$-plane by thinking of $x$-coordinate modulo $g$. 
 We present the points $\tau_0, \dots, \tau_{g-1}$ on the $x$-axis with coordinates $x=0, \dots, g-1$, such that the Frobenius shifts the points to the right by $1$, and shifts $\tau_{g-1}$ back to $\tau_0$ (by first shifting to $x=g$ and thinking of the $x$-coordinate modulo $g$).
  For example, if $F$ has degree $6$ over $\QQ$ and $\ttS_\infty = \{\tau_1, \tau_3,\tau_4\}$, then we draw the band as
$
\psset{unit=0.3}
\begin{pspicture*}(-.5,-0.3)(5.5,0.3)
\psset{linecolor=black}
\psdots(0,0)(2,0)(5,0)
\psdots[dotstyle=+](1,0)(3,0)(4,0)
\end{pspicture*}$.

Suppose that $\ttS'$ is another  set of places of $F$ with even cardinality  such that it contains exactly the same finite places of $F$ as $\ttS$ and satisfies $\#\ttS_{\infty}=\#\ttS'_{\infty}$.
We put  the band for $\ttS$ above  the band for $\ttS'$ on the same cylinder. 
In the 2-dimensional picture, we draw the band for $\ttS$ on the line $y=1$ and the band for $\ttS'$ on the line $y=-1$.  For each of the nodes of $\ttS$, we  draw a curve starting from it and go monotonically downwards linking to a node  of $\ttS'$ (and ignore the plus signs) such that all the curves do not intersect with each other.
Such a graph is called a \emph{link} $\eta: \ttS \to \ttS'$.
Two links are considered the same if the curves can be continuously deformed  to each other while keeping all curves from intersecting.
We say a curve is \emph{turning to the left (resp. right)} if it can be deformed into a smooth curve which has positive (resp. negative) tangent slopes in the $2$-dimensional picture. 
The \emph{displacement} of a curve in $\eta$ is the number of points it travels to the right, that is the difference between the $x$-coordinates of the ending and starting points of the curve (adding an appropriate multiple of $g$ according to the times the curve wraps around the cylinder). The displacement is negative if the curve turns to the left.
The \emph{total displacement} $v(\eta)$ is the sum of displacements of all curves. 
For example, if $g=5$,  $\ttS_\infty = \{ \tau_1, \tau_3\}$ and $\ttS'_\infty = \{ \tau_2, \tau_4\}$, the link given by
\begin{equation}
\label{E:left turn link}
\psset{unit=0.3} \eta = 
\begin{pspicture*}(-.5,-0.3)(5,2.3)
\psset{linecolor=red}
\psset{linewidth=1pt}
\psbezier(0,2)(1,1)(3,1)(3,0)
\psbezier(2,2)(2,1)(3.5,1.5)(4.5,0.5)
\psbezier(-0.5,1.3)(0.5,.3)(1,1)(1,0)
\psarc{-}(-0.5,0){0.5}{0}{90}
\psarc{-}(4.5,2){0.5}{180}{270}
\psset{linecolor=black}
\psdots(0,2)(2,2)(4,2)
\psdots(0,0)(1,0)(3,0)
\psdots[dotstyle=+](1,2)(3,2)
\psdots[dotstyle=+](2,0)(4,0)
\end{pspicture*}.
\end{equation}
has total displacement $v(\eta) = 3+3+2 = 8$.
For another example, the action of Frobenius $\sigma$ twists the band and gives rise to a link $\sigma: \ttS \to \sigma(\ttS)$, called the \emph{Frobenius link}, for which every curve is turning to the right with displacement $1$, where $\sigma(\ttS)$ is the set of places containing the same finite places as $\ttS$ but $\sigma(\ttS)_\infty = \sigma(\ttS_\infty)$.  Its total displacement is $v(\sigma) = d =\#\ttS_\infty^c$.

For a link $\eta: \ttS \to \ttS'$, we use $\eta^{-1}: \ttS' \to \ttS$ to denote the link obtained by flipping the picture about the equator of the cylinder.
For two links $\eta: \ttS \to \ttS'$ and $\eta': \ttS' \to \ttS''$, we have a natural \emph{composition} of links $\eta'\circ \eta: \ttS \to \ttS''$ given by putting the picture of $\eta$ on top of the picture of $\eta'$ and joint the nodes corresponding to $\ttS'$.
It is obvious that $v(\eta^{-1}) = -v(\eta)$ and $v(\eta'\circ \eta) = v(\eta') + v(\eta)$.

When discussing the relative positions of two nodes of the band associated to $\ttS$,  it is convenient to use the following
\begin{notation}\label{N:tau-plus}
For $\tau \in \ttS_\infty^c$,  let $n_\tau$ be the minimal positive integer such that $\sigma^{-n_\tau}\tau \in \ttS_\infty^c$.  We put $\tau^- := \sigma^{-n_\tau}\tau$.
We use $\tau^+$ to denote the place in $\ttS_\infty^c$ such that $(\tau^+)^- = \tau$.
\end{notation}
\begin{example}\label{Ex:fundamental-link}
A link from $\ttS$ to itself can only be an integer power of the \emph{fundamental link} $\eta_\ttS$, that is to link each $\tau$ to $\tau^+$ by shifting to the right with displacement $n_{\tau^+}$.  For example, 
$
\psset{unit=0.3}
\begin{pspicture*}(-.5,-0.3)(4.4,2.3)
\psset{linecolor=red}
\psset{linewidth=1pt}
\psbezier(0,2)(0,1)(2,1)(2,0)
\psbezier(2,2)(2,1)(3,1)(3,0)
\psbezier(-2,2)(-2, 1)(0,1)(0,0)
\psbezier(3,2)(3, 1)(5,1)(5,0)
\psset{linecolor=black}
\psdots(0,2)(2,2)(3,0)
\psdots(0,0)(2,0)(3,2)
\psdots[dotstyle=+](1,2)(4,0)
\psdots[dotstyle=+](1,0)(4,2)
\end{pspicture*}.
$
The total displacement of a fundamental link is exactly $v(\eta_\ttS) = g= [F:\QQ]$.
\end{example}


\subsection{Link morphisms I}\label{S:link morphism}
Let $\ttS$ and $\ttS'$ be two even subsets of places of $F$ consisting of the same finite places and $\#\ttS_\infty = \#\ttS'_\infty$. Let  $\eta: \ttS \to \ttS'$ be a link. We say that  $\eta$ is a \emph{right-turning link} if all its curves (if there are any) are turning to the right.
We allow the case $\ttS_{\infty}=\Sigma_{\infty}$ (so that there are no curves in the link $\eta$ at all). In that case we say $\eta$ is the \emph{trivial link}.
In this section, we suppose $\eta$ is right-turning. 
 For each node $\tau \in \ttS^c_\infty$, we use $m(\tau)$ to denote the displacement of the curve connected to $\tau$ in $\eta$.
Let   $\tilde \ttS_\infty$ and $\tilde \ttS'_\infty$ be (any) lifts of $\ttS_{\infty}$ and $\ttS'_{\infty}$ as in Subsection~\ref{A:CM}. 
We have two unitary  Shimura varieties $\Sh_{K''_p}(G''_{\tilde\ttS})$ and $\Sh_{K''_p}(G''_{\tilde\ttS'})$ as defined in \ref{S:unitary Shimura variety}.

We now recall the definition of the link morphism on  $\Sh_{K''_p}(G''_{\tilde\ttS})$  associated to the right-turning link $\eta$ as  in \cite[Definition 7.5]{tian-xiao1}.
Let $n$ be an integer, which is always taken to be $0$ if $\gothp$ is inert in $E$.
A \emph{link morphism} of \emph{indentation degree} $n$ associated to $\eta$ on $\Sh_{K''_p}(G''_{\tilde\ttS})$ is a pair $(\eta''_{(n),\sharp}, \eta''^{\sharp}_{(n)})$, where

\begin{enumerate}
\item 
$\eta''_{(n), \sharp}: \Sh_{K''_p}(G''_{\tilde \ttS}) \to \Sh_{K''_p}(G''_{\tilde \ttS'})$ is a morphism of Shimura varieties that induces a bijection on geometric points;
\item $\eta''^\sharp_{(n)}: \bfA''_{\tilde \ttS} \to \eta''^*_{(n), \sharp} (\bfA''_{\tilde \ttS'})$ is a $p$-quasi-isogeny of abelian varieties compatible with the $\cO_D$-actions, the polarizations, and the tame level structures;

\item[(3)] for each geometric point $x$ of $\Sh_{K''_p}(G''_{\tilde \ttS})$ with image $x' = \eta''_{(n), \sharp}(x)$, if we write $\tilde \calD(\bfA''_{\tilde \ttS, x})_{\tilde\tau}$ for the $\tilde\tau$-component of the \emph{covariant} Dieudonn\'e module of $\bfA''_{\tilde\ttS,x}$ for each $\tilde\tau \in \Sigma_{E,\infty}$, 
then there exists, for  each $\tilde\tau\in \ttS^c_{E,\infty}$, some $t_{\tilde\tau}\in \Z$ \emph{independent of the point $x$} such that 
   \[
  \eta'^{\sharp}_{(n), *}\big(F_{\es,\bfA''_{\tilde\ttS,x}}^{m(\tau)}(\tcD(\bfA''_{\tilde \ttS,x})_{\tilde\tau})\big)=p^{t_{\tilde\tau}}\tcD(\bfA''_{\tilde\ttS',x'})_{\sigma^{m(\tau)}\tilde\tau},
   \]
   where $F_{\es,\bfA''_{\tilde\ttS,x}}^{m(\tau)}: \tilde \calD(\bfA''_{\tilde\ttS,x})_{\tilde\tau}\ra \tcD(\bfA''_{\tilde\ttS,x})_{\sigma^{m(\tau)}\tilde\tau}$ is $m(\tau)$-th iteration of the essential Frobenius for $\bfA''_{\tilde\ttS,x}$  defined in \cite[\S 4.2]{tian-xiao1}; and

\item[(4)] when $\gothp$ splits as $\gothq\bar\gothq$ in $E$, the degree of the quasi-isogeny
\[
\eta''^\sharp_{(n),\gothq}: \bfA''_{\tilde \ttS}[\gothq^\infty] \to \eta''^*_{(n), \sharp}(\bfA''_{\tilde \ttS'}[\gothq^\infty])
\] of the $\gothq$-divisible groups
is $p^{2n}$. 
\end{enumerate}
When the indentation degree $n$ is clear by the context, we  write simply $(\eta_{\sharp},\eta^{\sharp})$ for $(\eta_{(n),\sharp},\eta^{\sharp}_{(n)})$.

For our purpose, the most important property  we need is the following
\begin{lemma}[\cite{tian-xiao1}, Proposition~7.8]\label{L:uniqueness-link}
Let $\eta: \ttS\ra \ttS'$ be a link as above. 
Then there exists at most one link morphism $(\eta''_{(n),\sharp}, \eta''^{\sharp}_{(n)})$ with indentation degree $n$ from $\Sh_{K''}(G''_{\tilde\ttS})$ to $\Sh_{K''}(G''_{\tilde\ttS'})$.
\end{lemma}

\begin{example}
\label{Ex:twisted Frobenius}
Let $\ttS$ and $\tilde\ttS$ be as in the previous subsection. 
Let  $\sigma^2: \ttS\ra \sigma^2(\ttS)$ be the second iteration of the Frobenius link on $\ttS$. 
Put $\sigma^{2}\tilde\ttS=(\sigma^2(\ttS), \sigma^2(\tilde\ttS_{\infty}))$.
 In  \cite[\S 3.22]{tian-xiao1}, we introduced natural morphisms called the \emph{twisted (partial) Frobenius}
\[
\gothF''_{\gothp^2}: \Sh_{K''_p}(G''_{\tilde \ttS}) \to \Sh_{K''_p} (G''_{\sigma^2\tilde \ttS})
\]
together with a quasi-isogeny of abelian varieties
\[
 \eta''_{\gothp^2}: \bfA''_{\tilde \ttS} \to (\gothF''_{\gothp^2})^* \bfA''_{\sigma^2\tilde \ttS}.
\]
Such a morphism is characterized by the fact that the morphism $p\eta''_{\gothp^2}$ is given by the $p^2$-relative Frobenius.
Then in the language of link morphism introduced above, $(\gothF''_{\gothp^2},\eta''_{\gothp^2})$ is the link morphism on $\Sh_{K''_p}(G''_{\tilde \ttS}) $ associated to the link $\eta = \sigma^2$ of indentation degree $0$ if $\gothp$ is inert in $E/F$ and of indentation degree $2\Delta_{\tilde \ttS}$ if $\gothp$ splits in $E/F$ (\cite[Example 7.7(1)]{tian-xiao1}).
\end{example}

\begin{example}
\label{Ex:Sq}
When $\gothp$ splits into $\gothq\bar \gothq$ in $E/F$, consider the Hecke operator $S_\gothq$ given by multiplication by $1\times\underline \gothq^{-1}\in G''(\AAA^{\infty})=G(\AAA^{\infty})\times_{\AAA_F^{\infty,\times}}\AAA_{E}^{\infty,\times}$  on the unitary Shimura variety. We start with the versal family of  abelian varieties $\bfA''_{\tilde \ttS}$ on $\Sh_{K''_p}(G''_{\tilde \ttS})$, and put $\bfB\colon  = \bfA''_{\tilde \ttS} \otimes_{\calO_E} \bar \gothq \cdot  \gothq^{-1}$ equipped with the induced action by  $\cO_D$.
Let $\phi_{\gothq}\colon \bfA''_{\tilde\ttS}\ra \bfB$ denote the natural $p$-quasi-isogeny induced by $\cO_E\ra \bar\gothq\gothq^{-1}$. 
We equip  $\bfB$ with the natural prime-to-$p$ level structure compatible with $\phi_{\gothq}$. 
The  polarization $\lambda_{\bfA''_{\tilde\ttS}}$ on $\bfA''_{\tilde \ttS}$ naturally induces a polarization $\lambda_{\bfB}$ on $\bfB$ such that $\lambda_{\bfA''_{\tilde\ttS}}=\phi_{\gothq}^{\vee}\circ\lambda_{\bfB}\circ\phi_{\gothq}$.  
 There is a unique morphism
\[
S_\gothq: \Sh_{K''_p}(G''_{\tilde \ttS}) \to \Sh_{K''_p}(G''_{\tilde \ttS}),
\]
which, together with $\phi_{\gothq}$, 
gives a link morphism for the trivial link $\id: \ttS \to \ttS$ of indentation degree $2g$. 

If we apply Construction~\ref{S:from unitary to quaternionic} to the morphism $S_\gothq$ (with the morphism $\pi$ there being trivial), we can lift it to a endomorphism
of $
\Sh_{K_p}(G_{\ttS, \ttT}) \times \Sh_{K_{E,p}}(T_{\tilde \ttS, \ttT}) 
$
given by multiplication by $((\underline p_F)^{-1}, \varpi_{\bar \gothq}^{-2}) \in G(\AAA^{\infty}) \times_{\AAA_F^{\infty, \times}} \AAA_E^{\infty, \times}$.
So the endomorphism $S_\gothp$ given by multiplication by central element $(\underline p_F)^{-1}$ may be viewed as the morphism on the quaternionic Shimura variety obtained by applying Construction~\ref{S:from unitary to quaternionic} to the morphism  $\eta''=S_\gothq$ with shift $\varpi_{\bar \gothq}^2$.
\end{example}

\subsection{Normalizations of link morphisms}
\label{S:normalization of link morphisms}
Keep the notation of Subsection~\ref{S:link morphism} and assume moreover that 
\begin{itemize}
\item
the link morphism $(\eta''_{(n),\sharp}, \eta''^{\sharp}_{(n)})$  on $\Sh_{K''_p}(G''_{\tilde\ttS})$ exists,
\item
$\Delta_{\tilde \ttS_\infty} = \Delta_{\tilde \ttS'_{\infty}}$ if $\gothp$ splits in $E/F$, and
\item
we are given two subsets $\ttT \subseteq \ttS_\infty$ and $\ttT' \subseteq \ttS'_\infty$ such that $\#\ttT = \#\ttT' $.
\end{itemize}
Let $(\underline k, w)\in \ZZ^{\Sigma_\infty}\times \ZZ$ be a regular multi-weight as in Subsection~\ref{S:automorphic sheaves}, and  let $\calL''^{(\underline k, w)}_{\tilde \ttS, \tilde \Sigma}$ and $\calL''^{(\underline k, w)}_{\tilde \ttS', \tilde \Sigma}$ be the corresponding  $\ell$-adic \'etale local systems on $\Sh_{K''_p}(G''_{\tilde\ttS})$ and $\Sh_{K''_p}(G''_{\tilde\ttS'})$ respectively. Then the $p$-quasi-isogeny $\eta''^{\sharp}_{(n)}$ induces an  isomorphism of \'etale local systems
\[\calL''^{(\underline k, w)}_{\tilde \ttS, \tilde \Sigma}\xra{\cong}\eta''^*_{(n),\sharp}\calL''^{(\underline k, w)}_{\tilde \ttS', \tilde \Sigma}.\]
Applying Construction~\ref{S:from unitary to quaternionic} to the link morphism $(\eta''_{(n),\sharp}, \eta''^{\sharp}_{(n)})$ (with $X= \Sh_{K''_p}(G''_{\tilde\ttS})$  and $\pi''$ in \eqref{E:correspondence-Sh} equal to the identity), we get   a pair of morphisms
\[
\eta_{(n), \sharp}: \Sh_{K_p}(G_{\ttS,\ttT}) \to \Sh_{K_p}(G_{\ttS',\ttT'})\quad \textrm{and} \quad \eta_{(n)}^\sharp: \calL_{\ttS, \ttT}^{(\underline k, w)} \xrightarrow{\cong} \eta_{(n),\sharp}^* (\calL_{\ttS', \ttT'}^{(\underline k, w)}),
\]
depending on  some  $\boldsymbol t \in E^{\times, \cl} \backslash \AAA_{E}^{\infty, \times} / \calO_{E,\gothp}^\times$ (See Construction~\ref{S:from unitary to quaternionic} Step IV).
In the sequel, we call $(\eta_{(n),\sharp},\eta^{\sharp}_{(n)})$ (or simply $\eta_{(n),\sharp}$ for short) \emph{the link morphism with indentation degree $n$ on the quaternionic Shimura variety $\Sh_{K_p}(G_{\ttS,\ttT})$ with shift $\boldsymbol t $.} 
Note that by  Lemma~\ref{L:uniqueness-link} and Construction~\ref{S:from unitary to quaternionic}, for a fixed lifting $\tilde\ttS$ of $\ttS$, indentation degree $n$ and shift $\boldsymbol t$,  there exists at most one link morphism $(\eta_{(n),\sharp},\eta_{(n)}^{\sharp})$ on $\Sh_{K}(G_{\ttS,\ttT})$.

The link morphism $(\eta_{(n),\sharp}, \eta_{(n)}^\sharp)$ induces a homomorphism of the cohomology groups:
\begin{align*}
\tilde \eta_{(n)}^{\star}: H^\star_{\et}(\Sh_{K_{p}}(G_{\ttS', \ttT'})_{\overline \FF_p}, \calL_{\ttS',\ttT'}^{(\underline k, w)}) &\longrightarrow H^\star_{\et}(\Sh_{K_{p}}(G_{\ttS, \ttT})_{\overline \FF_p}, \eta_{(n),\sharp}^*(\calL_{\ttS',\ttT'}^{(\underline k, w)}))
\\
& \xrightarrow{(\eta_{(n)}^\sharp)^{-1}}  H^\star_{\et}(\Sh_{K_{p}}(G_{\ttS, \ttT})_{\overline \FF_p}, \calL_{\ttS,\ttT}^{(\underline k, w)})
\end{align*}
which is equivariant under the Hecke action by $G(\AAA^{\infty})$ \footnote{Here, $G(\AAA^{\infty})$ denotes the common finite adelic points of  $G_{\ttS,\ttT}$ and $ G_{\ttS',\ttT'}$ according to Notation~\ref{N:G(A infty)}.} and the Galois action by $\Gal_{\FF_{p^{2g}}}$. 
We fix a square root $p^{1/2}\in \overline \Q_{\ell}$ of $p$. We put 
   \begin{equation}\label{E:normalized-link}
   \eta_{(n)}^{\star}=\frac{1}{p^{v(\eta)/2}}\tilde \eta_{(n)}^{\star},
   \end{equation} 
   and call it the \emph{normalized link morphism} on the cohomology groups of quaternionic Shimura varieties associated to $\eta$ with indentation degree $n$ and shift $\boldsymbol t$. This normalization will be justified in Lemma~\ref{L:inverse of link morphism}(2). 
When the link morphism $\eta''_{(n),\sharp}: \Sh_{K''_p}(G''_{\tilde\ttS})\ra \Sh_{K''_p}(G''_{\tilde\ttS'})$ preserves the neutral connected components, $\boldsymbol t=\boldsymbol 1$ is a canonical choice, and in that case, $\eta^{\star}_{(n)}$ is canonically defined. 
 
Let $\eta_1: \ttS_1\ra \ttS_2$ and $\eta_2: \ttS_2\ra \ttS_3$ be two links with all curves turning to the right, satisfying the conditions above, i.e. all $\ttS_i$ have the same set of finite places, $\#\ttS_{1, \infty} = \#\ttS_{2, \infty}=\#\ttS_{3, \infty}$,  $\#\ttT_1 = \#\ttT_2 = \#\ttT_3$, and $\Delta_{\tilde \ttS_{1,\infty}} = \Delta_{\tilde \ttS_{2,\infty}} = \Delta_{\tilde \ttS_{3,\infty}}$ if $\gothp$ splits in $E/F$.
Suppose that there are link morphisms $(\eta''_{i,(n_i),\sharp}, \eta''^{\sharp}_{i,(n_i)})$ for $i=1,2$ on unitary Shimura varieties with indentation degree $n_i$. 
 Then the composed morphism 
 \[
 \eta''_{12,(n_{12}),\sharp}: \Sh_{K''_p}(G''_{\tilde\ttS_1})\xra{\eta''_{1,(n_1),\sharp}} \Sh_{K''_p}(G''_{\tilde\ttS_2})\xra{\eta''_{2,(n_2),\sharp}} \Sh_{K''_p}(G''_{\tilde\ttS_3})
 \]
 together with the composed quasi-isogeny 
 \[
 \eta''^{\sharp}_{12, (n_{12})}: \bfA''_{\tilde\ttS_1}\xra{\eta''^{\sharp}_{1,(n_1)}}\eta''^*_{1,(n_1),\sharp}(\bfA''_{\tilde\ttS_2})\xra{\eta''^*_{1,(n_1),\sharp}(\eta''^{\sharp}_{2, (n_2)})}\eta''^*_{1, (n_1),\sharp}\eta''^{*}_{2, (n_2),\sharp}(\bfA''_{\tilde\ttS_3})
 \]
 gives the (unique) link morphism on the unitary Shimura varieties with indentation degree $n_{12}: = n_1+n_2$ associated to the composed link $\eta''_{12, (n_{12})}:=\eta''_{2, (n_2)}\circ\eta''_{1,(n_1)}$. 
From this, we get a link morphism of quaternionic Shimura varieties of indentation degree $n_{12}$:
\[
\eta_{12, (n_{12}), \sharp}: \Sh_{K_p}(G_{\ttS_1, \ttT_1}) \xrightarrow{\eta_{1, (n_1), \sharp}}
 \Sh_{K_p}(G_{\ttS_2, \ttT_2}) \xrightarrow{\eta_{2, (n_2),\sharp}}
  \Sh_{K_p}(G_{\ttS_3, \ttT_3}),
\]
such that the shift of $\eta_{12, (n_{12}), \sharp}$ is the product of the shifts of $\eta_{1, (n_1), \sharp}$ and $\eta_{2, (n_2), \sharp}$.
Moreover, we have $\eta_{12,(n_{12})}^\star=\eta_{1,(n_1)}^\star\circ\eta_{2,(n_2)}^\star$  on the cohomology groups of quaternionic Shimura varieties. 

\subsection{Automorphic representations}
\label{S:automorphic-representations}
Following \cite[\S 5.10]{tian-xiao2}, for a regular multi-weight $(\underline k, w)$, we use $\scrA_{(\underline k, w)}$
to denote the set of irreducible  automorphic cuspidal representations 
$\pi$ of $\GL_2(\AAA_F)$ such that 
\begin{itemize}
\item its archimedean component $\pi_\tau$ for each $\tau \in \Sigma_\infty$ is a discrete series of weight $k_\tau -2$ with central character $x\mapsto x^{w-2}$, 
\item and  its $\gothp$-component $\pi_\gothp$ is \emph{spherical}.
\end{itemize}
We write $\pi^{\infty, \gothp}$ to denote the prime-to-$\gothp$ finite part of $\pi$. 

Let $\pi\in \scrA_{(\underline k, w)}$. We denote by  $\rho_\pi: \Gal_F \to \GL_2(\overline \QQ_\ell)$ the Galois representation attached to $\pi$ normalized so that if $v$ is  a finite place of $F$ at which $\pi$ is spherical, the action of  a \emph{geometric} Frobenius at $v$ has trace equal to the eigenvalue of usual Hecke operator $T_{v}$ on $\pi_v^{\GL(\cO_{F_v})}$, where $\pi_v$ denotes the local component at $v$ of $\pi$. Let  $\rho_{\pi, \gothp}$ be the restriction of $\rho_\pi$ to $\Gal_{\FF_{p^g}}$ (note that $\rho_\pi$ is unramified at $\gothp$ since $\pi_{\gothp}$ is spherical).

Let  $T_\gothp$ and $S_\gothp$ be the Hecke operators  given by the action on $\pi_\gothp^{K_p}$ of the double cosets $K_p \big(
\begin{smallmatrix} p^{-1} &0\\0 &1 \end{smallmatrix}  \big)K_p$ and $K_p p^{-1} K_p$, respectively. 
Then, for $\pi\in \scrA_{(\kb,w)}$,  the characteristic polynomial of $\rho_{\pi, \gothp}(\Frob_{p^{g}})$ is given by 
\begin{equation}\label{E:Eichler-Shimura}
X^2-T_\gothp(\pi) X+S_\gothp(\pi) p^g,
\end{equation}
where $T_\gothp(\pi)$ and $S_\gothp(\pi)$ denote respectively the eigenvalues of $T_\gothp$ and $S_\gothp$ on $\pi_\gothp^{K_p}$.

We say that an automorphic representation $\pi \in\scrA_{(\underline k, w)}$ \emph{appears} in the cohomology of the Shimura variety $\Sh_K(G_{\ttS, \ttT})$ if for each finite place $v$ of $\ttS$, the local component $\pi_v$ of $\pi$ is square integrable so that $\pi$ is the image, under  the  Jacquet--Langlands correspondence, of a unique automorphic representation $\pi_{B_\ttS}$ of $G_{\ttS,\ttT}(\AAA)=(B_{\ttS}\otimes_{\Q}\AAA)^{\times}$, and $(\pi^\infty_{B_\ttS})^K$ is nonzero.


\begin{notation}
For $\pi\in \scrA_{(\uk, w)}$ and a $\overline \QQ_\ell[G(\AAA^{\infty,p})]$-module $M$, we write 
\[
M[\pi] : = \Hom_{\overline \QQ_\ell[G(\AAA^{\infty,p})]}(\pi_{B_{\ttS}}^{\infty, \gothp}, M)
\]
for its \emph{$\pi$-isotypical component}.
By the strong multiplicity one theorem for quaternion algebras, $\pi_{B_\ttS}$ is determined by $\pi_{B_\ttS}^{\infty,\gothp}$; this justifies the notation for $M[\pi]$.
There is also a finite version: let $K^p \subset G(\AAA^{\infty, p})$ be an open compact subgroup so that $(\pi_{B_\ttS}^{\infty, \gothp})^{K^p}$ is an irreducible module over the prime-to-$p$ Hecke  algebra $\calH_{K^p} :  = \overline \QQ_\ell [ K^p \backslash G(\AAA^{\infty, p}) / K^p]$.  Then $\pi_{B_\ttS}$ is determined by the  $\calH_{K^p}$-module $(\pi^{\infty,\gothp}_{B_\ttS})^{K^p}$.
For an $\calH_{K^p}$-module $M$, we put 
\[
M[\pi]:  = \Hom_{\calH_{K^p}} ((\pi_{B_\ttS}^{\infty, \gothp})^{K^p}, M).
\]
\end{notation}
 
\begin{prop}
\label{P:cohomology of sh}
Let $\pi \in\scrA_{(\underline k, w)}$ be an automorphic representation appearing in the cohomology of the Shimura variety $\calS h_K(G_{\ttS, \ttT})$.
Then we have
\begin{equation}
\label{E:description of quaternioinic cohomology}
H^{i}_\et(\Sh_K(G_{\ttS, \ttT})_{\overline \FF_p}, \calL_{\ttS, \ttT}^{(\underline k, w)})[\pi] =
\begin{cases}
 \rho_{\pi, \gothp}^{\otimes d} \otimes [\det(\rho_{\pi,\gothp})(1)]^{\otimes \#\ttT} & \textrm{if }i =d,\\
 0 & \textrm{if }i \neq d;
 \end{cases}
\end{equation}
it is compatible, up to semisimplification,  with the action of the geometric Frobenius $\Frob_{p^{2g}}$.
Explicitly, if $\alpha_\pi, \beta_\pi$ are the two  eigenvalues of $\rho_{\pi, \gothp}(\Frob_{p^g})$, then the (generalized) eigenvalues of the action of $\Frob_{p^{2g}}$ on \eqref{E:description of quaternioinic cohomology} are $p^{-2g\#\ttT}\alpha_\pi^{2(i+\#\ttT)}\beta_\pi^{2(d-i+\#\ttT)}$ with multiplicity $\binom di$ for $0 \leq i\leq d$.
\end{prop}
\begin{proof}
The first part of the proposition is well known to experts. We defer its proof to the Appendix (see Proposition~\ref{P:cohomology of sh appendix}).
The explicit description of the action of $\Frob_{p^{2g}}$ is straightforward.
\end{proof}

\begin{prop}
\label{P:self composition of eta}
Assume that $d = \#\ttS_\infty^c \neq 0$.

\begin{enumerate}
\item[(1)]
The $2g$-th iteration of the Frobenius link $\sigma^{2g}: \ttS \to \ttS$ coincides with the $2d$-fold self-composition of the fundamental link $\eta_\ttS$ \eqref{Ex:fundamental-link}.

\item[(2)]
The link morphism  on  $\Sh_{K''_p}(G''_{\tilde\ttS})$ with indentation degree $0$ associated to $\sigma^{2g} = \eta_\ttS^{2d}$ exists.  It is given by
\begin{enumerate}
\item[(a)]
$g$-fold self-composition $(\gothF''_{\gothp^2})^g$ if $\gothp $ is inert in $E/F$;
\item[(b)]
$(\gothF''_{\gothp^2})^g \cdot S_\gothq^{-\Delta_{\tilde \ttS_{\infty}}}$ if $\gothp$ splits in $E/F$, where $S_{\gothq}$ is defined in Example~\ref{Ex:Sq}.
\end{enumerate}
Moreover, this link morphism preserves the neutral geometric connected component $\Sh_{K''_p}(G''_{\tilde \ttS})^\circ_{\overline \FF_p}$ and hence induces a \emph{canonical} link morphism $(\eta^{2d}_{\ttS,(0),\sharp},\eta^{2d,\sharp}_{\ttS,(0)})$ on the quaternionic Shimura variety $\Sh_{K_p}(G_{\ttS, \ttT})$ with shift $\boldsymbol 1$ for any fixed subset $\ttT \subset \ttS_\infty$.  

\item[(3)] Let
 \[
(\eta^{2d}_\ttS)^{\star}_{(0)}\colon H_\et^d(\Sh_{K}(G_{\ttS, \ttT})_{\overline \FF_p}, \calL^{(\underline k,w)}_{\ttS, \ttT}) \to H_\et^d(\Sh_{K}(G_{\ttS, \ttT})_{\overline \FF_p}, \calL^{(\underline k,w)}_{\ttS, \ttT})
\]
be the normalized link morphism  \eqref{E:normalized-link} induced by $(\eta^{2d}_{\ttS,(0),\sharp},\eta^{2d,\sharp}_{\ttS,(0)})$. 
Then we  have an equality of operators on cohomology groups:
\begin{equation}
\label{E:etaS equality as operator}
(\eta^{2d}_{\ttS})^{\star}_{(0)}=p^{-dg}\cdot\Frob_{p^{2g}}\circ S_\gothp^{-d-2\#\ttT},
\end{equation}
  where $S_{\gothp}$ is the Hecke operator given by the central element $\underline p_F^{-1}\in G(\AAA^{\infty})$.
   In particular, for each  $\pi\in \scrA_{(\kb,w)}$ and  each integer $i$ with $0\leq i\leq d$,   the (generalized) eigenspace of $\Frob_{p^{2g}}$ on  $H_\et^d(\Sh_{K}(G_{\ttS, \ttT})_{\overline \FF_p}, \calL^{(\underline k,w)}_{\ttS, \ttT})[\pi] $ with eigenvalue $p^{-2g\#\ttT}\alpha_\pi^{2(i+\#\ttT)}\beta_\pi^{2(d-i+\#\ttT)}$ is the same as  the  (generalized) eigenspace of $(\eta^{2d}_\ttS)^{\star}_{(0)}$ with eigenvalue
$ (\alpha_\pi/\beta_\pi)^{2i-d}$.
\end{enumerate}
\end{prop}
\begin{proof}
Statement (1) is evident. For (2), we first check   that the maps given by (a) and (b) are  link morphisms with indentation degree $0$ associated to the link $\eta_\ttS^{2d}$.
This follows easily  from Examples~\ref{Ex:twisted Frobenius} and \ref{Ex:Sq}. By the uniqueness of link morphisms (Lemma~\ref{L:uniqueness-link}), they are the link morphisms we sought for.

We next show that the link morphism in the unitary case preserves the neutral geometric connected component $\Sh_{K''_p}(G''_{\tilde \ttS})^\circ_{\overline \FF_p}$.
This is a direct computation using the Shimura reciprocity map (in Subsection~\ref{S:unitary Shimura variety}), which we spell out now. 
Denote by $\Phi^{2g}$  the Frobenius endomorphism of $\Sh_{K''_p}(G''_{\tilde\ttS})$ relative to $\FF_{p^{2g}}$.
 Then $(\gothF''_{\gothp^2})^g$ is  nothing but the composition of $\Phi^{2g}$ with  the Hecke operator $S_p^{-g}$, where $S_p$ is the Hecke correspondence given by  the central element $(\underline p^{-1}_F, 1)\in  G(\AAA^{\infty})\times_{\AAA_{F}^{\infty,\times}}\AAA_{E}^{\infty,\times}\cong G''(\AAA^{\infty})$.
Recall that the set of geometric connected components of $\Sh_{K''_p}(G''_{\tilde\ttS})$ is given by 
$$
\pi_0(\Sh_{K''_p}(G''_{\tilde\ttS})_{\overline \FF_p})\cong \big( F_+^{\times, \mathrm{cl}} \backslash \AAA_F^{\infty, \times} / \calO_\gothp^\times \big) \times 
\big( 
E^{\times, N_{E/F}=1,\cl} \backslash \AAA_{E}^{\infty, N_{E/F} = 1} / \calO_{E_{\gothp}}^{N_{E/F} = 1}
\big).
$$
The action of $\Phi^{2g}$ on $\pi_0(\Sh_{K''_p}(G''_{\tilde\ttS})_{\overline \FF_p})$ coincides with the arithmetic Frobenius  $\Frob_{p^{2g}}^{-1}\in \Gal_{\FF_{p^{2g}}}$, which is computed already by \eqref{E:Shimura-rec-unitary}.
We now list the actions of these operators on the geometric connected components.
\begin{center}
\begin{tabular}{|c|c|c|}
\hline
Operator & When $\gothp$ splits & When $\gothp$ is inert
\\
\hline
$\Phi^{2g}$ & $(\underline p_F)^{-2g} \times (\underline \gothq)^{-2\Delta_{\tilde \ttS_{\infty}}}$& $(\underline p_F)^{-2g}\times 1 $
\\
\hline
$S_p$ &$(\underline p_F)^{-2}\times 1$&$(\underline p_F)^{-2}\times 1$
\\
\hline
$S_\gothq$ &$1\times \underline \gothq^{-2}$ & N/A
\\
\hline
\end{tabular} 
\end{center}
It is now clear that the link morphisms given in (1) and (2) preserve the neutral geometric connected component. This verifies (2).

We now turn to the proof of (3).
It suffices to verify \eqref{E:etaS equality as operator} because $S_\gothp$ acts on the $\pi$-component by the scalar $\omega_\pi(\underline p^{-1}) = \alpha_\pi \beta_\pi / p^g$ according to \eqref{E:Eichler-Shimura},  and then statement (3) would follow immediately from the following easy computation:
\[  p^{-dg} \times
p^{-2g\#\ttT}\alpha_\pi^{2(i+\#\ttT)}\beta_\pi^{2(d-i+\#\ttT)} \times \big(\alpha_\pi \beta_\pi /p^g\big)^{-(d +2\#\ttT)} = (\alpha_\pi / \beta_\pi)^{2i-d}.
\]

To prove \eqref{E:eta a to b when r=1}, we first compute the canonical lift of the link morphism $((\eta^{2d}_{\ttS,(0)})''_\sharp, (\eta^{2d}_{\ttS,(0)})''^\sharp)$ to an endomorphism of $\Sh_{K_p}(G_{\ttS, \ttT}) \times \Sh_{K_{E, p}}(T_{E,\tilde \ttS, \ttT})$ appearing in Construction~\ref{S:from unitary to quaternionic} Step I (and the shift in Step II is trivial in our case).
This lift is clearly a composition of the Frobenius endomorphism relative to $\FF_{p^{2g}}$, which we denote by $\Phi_\times^{2g}$, and the action of a Hecke operator given by a central element $\boldsymbol x$ in $G(\AAA^\infty) \times \AAA_E^{\infty, \times}$.
This central element $\boldsymbol x$ is characterized by (and uniquely determined by) the following two conditions: 
\begin{itemize}
\item[(a)] the resulting link morphism on $\Sh_{K_p}(G_{\ttS,\ttT})\times \Sh_{K_{E,p}}(T_{E,\tilde\ttS,\ttT})$ preserves the neutral connected component, and
\item[(b)] under the natural projection $G(\AAA^\infty) \times \AAA_E^{\infty, \times}\ra G(\AAA^{\infty})\times_{\AAA_{F}^{\infty,\times}}\AAA^{\infty,\times}_E\cong G''(\AAA^{\infty})$,  $\boldsymbol x$  is mapped to the central element $((\underline p_F)^g, 1)$ if $\gothp$ is inert in $E/F$ and to $((\underline p_F)^{g},(\underline \gothq)^{ \Delta_{\tilde \ttS_{\infty}}})$ if $\gothp$ splits in $E/F$.
\end{itemize}
 
We claim that $\boldsymbol x = \big((\underline p_F)^{\#\ttS_\infty^c + 2 \#\ttT} ,\, (\underline p_F)^{\#\ttS_\infty- 2 \#\ttT}\big)$ if $\gothp$ is inert in $E/F$, and $\boldsymbol x = \big((\underline p_F)^{\#\ttS_\infty^c + 2 \#\ttT}$, $ (\underline p_F)^{\#\ttS_\infty- 2 \#\ttT}(\underline \gothq)^{ \Delta_{\tilde \ttS_{\infty}}}\big)$ if $\gothp$ splits in $E/F$.
Clearly, this element satisfies (b) above.  To see (a), we note that the action of $\Phi_\times^{2g}$ on the geometric connected component is the image of the \emph{arithmetic Frobenius} $\Frob_{p^{2g}}^{-1}$ under the Shimura reciprocity maps in Subsections~\ref{S:quaternoinic Shimura varieties} and \ref{A:CM}, namely
\[
\begin{cases}
\big((\underline p_F)^{-2\#\ttS_\infty^c - 4\#\ttT}, \  (\underline p_F)^{2\#\ttT - \#\ttS_\infty} \big)& \textrm{ if } \gothp\textrm{ is inert in }E/F,
\\
\big((\underline p_F)^{-2\#\ttS_\infty^c - 4\#\ttT} , \ (\underline p_F)^{2\#\ttT - \#\ttS_\infty} (\underline \gothq)^{-\Delta_{\tilde \ttS_{\infty}}}\big) & \textrm{ if } \gothp\textrm{ splits in }E/F.
\end{cases}
\]
But this element is exactly $(\nu\times \id)(\boldsymbol{x}^{-1})$.

Now, taking the fiber over $\boldsymbol{1} \in \Sh_{K_{E,p}}(T_{E, \tilde \ttS, \ttT})$ tells us that the (canonical) link morphism $(\eta^d_{\ttS})_{ (0),\sharp}$ is the Frobenius endomorphism $\Phi^{2g}_{G_{\ttS,\ttT}}$ on $\Sh_{K_p}(G_{\ttS, \ttT})$ relative to $\FF_{p^{2g}}$ composed with the Hecke operator given by multiplication by the first coordinate of $\boldsymbol x$, namely,  $S_\gothp^{-\#\ttS_\infty^c - 2\#\ttT}=S_\gothp^{-d - 2\#\ttT}$.
For the action of $(\eta^d_{\ttS})^{\star}_{(0)}$ on $H_\et^d(\Sh_{K}(G_{\ttS, \ttT})_{\overline \FF_p}, \calL^{(\underline k,w)}_{\ttS, \ttT})$, we note that the induced action of the Frobenius endomorphism $\Phi^{2g}_{G_{\ttS,\ttT}}$ on cohomology co\"incides with $\Frob_{p^{2g}}$ (as opposed to the arithmetic Frobenius). So we have $(\eta^{2d}_{\ttS})^{\star}_{(0)}=p^{-dg}\cdot\Frob_{p^{2g}}\circ S_\gothp^{-d-2\#\ttT}$, where  $p^{-dg}$ is the normalization factor in \eqref{E:normalized-link}. This proves \eqref{E:etaS equality as operator} and hence the Proposition.
\end{proof}

\subsection{Link morphisms II}
\label{S:link morphism II}
Let $\eta: \ttS\ra\ttS' $ be a general link. Then there exists an integer $N\geq 0$ such that the composition of the links $\xi: = \eta\circ \sigma^{2gN}=\eta\circ(\eta_{\ttS}^{2d})^N=(\eta_{\ttS'}^{2d})^{N}\circ \eta: \ttS \to \ttS'$ is right-turning, where $\eta_{\ttS}$ is the fundamental link for $\ttS$ \eqref{Ex:fundamental-link}.
Suppose that the link morphism on $\Sh_{K''_p}(G''_{\tilde\ttS})$ associated to $\xi$ with indentation degree $n$ exists.
 Then we put, for each $\pi\in \scrA^{(\kb,w)}$,
\begin{align*}
\eta^\star_{(n)}\colon H^d_\et(\Sh_{K_p}(G_{\ttS', \ttT'})_{\overline \FF_p}, \calL_{\ttS', \ttT'}^{(\underline k, w)} )[\pi] &\xrightarrow{((\eta_\ttS^{2d})^\star_{(0)})^{-N}}
H^d_\et(\Sh_{K_p}(G_{\ttS', \ttT'})_{\overline \FF_p}, \calL_{\ttS', \ttT'}^{(\underline k, w)} ) [\pi]\\&
\xrightarrow{\ \xi^\star_{(n)} \ }
H^d_\et(\Sh_{K_p}(G_{\ttS,\ttT})_{\overline \FF_p}, \calL^{(\underline k, w)}_{\ttS, \ttT})[\pi],
\end{align*}
and  call it the \emph{normalized link morphism} on the cohomology group of quaternionic Shimura varieties associated to $\eta$ with indentation degree $n$. 
Here the link morphism $(\eta_\ttS^{2d})^\star_{(0)}$ is taken to be the canonical one, so that it is invertible by  Proposition~\ref{P:self composition of eta}.
The shift of $\eta^\star_{(n)}$ is defined to be the same as that of $\xi^\star_{(n)}$ (as $(\eta_\ttS^{2d})^\star_{(0)}$ has shift $\boldsymbol 1$).
By Lemma~\ref{L:uniqueness-link} on the uniqueness of link morphisms, this definition does not depend on the choice of $N$ (but on the shift of $\xi^\star_{(n)}$) and is compatible with compositions.

\begin{lemma}
\label{L:inverse of link morphism}
\begin{enumerate}
\item
For any link $\eta: \ttS \to \ttS'$,  there exist an integer $N>0$ and another right-turning link $\xi: \ttS' \to \ttS$ such that $\xi \circ \eta: \ttS \to \ttS$ is the same as $\sigma^{2gN}:\ttS\ra \ttS$.

\item
If $\eta: \ttS \to \ttS'$ is a right-turning link, and the link morphism $(\eta''_{(n),\sharp},\eta''^{\sharp}_{(n)})$ on  $\Sh_{K''_p}(G''_{\tilde \ttS})$ with indentation degree $n$ associated to $\eta$ exists, then there exists $N >0$ such that the link morphism associated to $\eta^{-1} \circ (\eta_{\ttS'}^{2d})^N: \ttS' \to \ttS$ of indentation $-n$ exists.

\item
Let $\eta: \ttS\ra \ttS'$ be the link as in $(2)$, and $\eta_{(n), \sharp}: \Sh_{K_p}(G_{ \ttS, \ttT}) \to \Sh_{K_p}(G_{ \ttS', \ttT'})$  be the link morphism with some shift $\boldsymbol t$ obtained by applying Construction~\ref{S:from unitary to quaternionic} to $\eta''_{(n), \sharp}$.   
If $\eta^{-1}: \ttS' \to \ttS$ denotes the inverse link, then the morphism
\[
(\eta^{-1})^\star_{(-n)}:   
H^d_\et(\Sh_{K_p}(G_{\ttS,\ttT})_{\overline \FF_p}, \calL^{(\underline k, w)}_{\ttS, \ttT})
 \longto 
H^d_\et(\Sh_{K_p}(G_{\ttS',\ttT'})_{\overline \FF_p}, \calL^{(\underline k, w)}_{\ttS', \ttT'})
\]
with some shift $\boldsymbol t^{-1}$
is the same as the inverse of $\eta^\star_{(n)}$, which has shift $\boldsymbol t$.
Moreover, if  $\eta''_{(n),\sharp}$ (or equivalently $\eta_{(n),\sharp}$) is finite flat of degree $p^{v(\eta)}$, where $v(\eta)$ denotes the total displacement of $\eta$, we have also  $(\eta^{-1})^\star_{(-n)}=p^{-v(\eta)/2}\Tr_{\eta_{(n),\sharp}}$, where $\Tr_{\eta_{(n),\sharp}}$ is  the trace map on cohomology induced by the finite flat morphism $\eta_{(n),\sharp}$.
\end{enumerate}
\end{lemma}
\begin{proof}
(1) is obvious.  
For (2), we may first find $N$ so that $\xi : = \eta^{-1} \circ (\eta_{\ttS'}^{2d})^N$ has all curves turning to the right.
Then we consider the two morphisms
\[
\xymatrix{
\Sh_{K''_p}(G''_{\tilde \ttS}) \ar[dr]_{\eta_{(n),\sharp}} && \Sh_{K''_p}(G''_{\tilde \ttS'})
\ar[dl]^{(\eta_{\ttS'}^{2d})^N_{(0),\sharp}}
\\
&
\Sh_{K''_p}(G''_{\tilde \ttS'})
}
\]
Since the link morphism $\eta_{(n),\sharp}$ induces a bijection on the closed points, \cite[Proposition~4.8]{helm} implies that after possibly enlarging $N$, the map $(\eta_{\ttS'}^{2d})^N_{(0),\sharp}$ factors through $\eta_{(n),\sharp}$, as $\eta_{(n),\sharp} \circ \xi_\sharp$.  It is easy to see that $\xi_\sharp$ gives the required link morphism.

The first part of (3) follows from the uniqueness of link morphism (Lemma~\ref{L:uniqueness-link}).
For the second part of  (3),   note that the composed morphism 
\[
H^{d}_{\et}(\Sh_{K_p}(G_{\ttS',\ttT'})_{\overline \FF_p},\calL^{(\kb,w)}_{\ttS', \ttT'})\xra{p^{v(\eta)/2}\eta^{\star}_{(n)}} H^{d}_{\et}(\Sh_{K_p}(G_{\ttS,\ttT})_{\overline \FF_p},\calL^{(\kb,w)}_{\ttS, \ttT})\xra{\Tr_{\eta_{(n),\sharp}}}
H^{d}_{\et}(\Sh_{K_p}(G_{\ttS',\ttT'})_{\overline \FF_p},\calL^{(\kb,w)}_{\ttS', \ttT'}),
\] 
is nothing but the multiplication by $p^{v(\eta)}$, according to our normalization of $\eta_{(n),\sharp}^{\star}$ \eqref{E:normalized-link}. 
It follows immediately that $(\eta^{-1})^\star_{(-n)}=p^{-v(\eta)/2}\Tr_{\eta_{(n),\sharp}}$.
\end{proof}

\subsection{Goren--Oort divisors}
We recall the definition of the Goren--Oort stratification from \cite[Section~4]{tian-xiao1}.
We will  make essential use of the case of divisors. 
Let $\Sh_{K_p}(G_{\ttS,\ttT})$ be the special fiber of a quaternionic Shimura variety of  type considered in Subsection~\ref{S:quaternoinic Shimura varieties}. 
We fix throughout this paper a choice of lifting $\tilde\ttS_{\infty}$ of $\ttS_{\infty}$,  and let  $\Sh_{K''_p}(G''_{\tilde\ttS})$ be the associated unitary Shimura variety.

In \cite[Definition~4.6 and \S 4.9]{tian-xiao1}, we defined, for each $\tau\in \ttS_{\infty}^c$, the \emph{Goren--Oort divisor} $\Sh_{K''_p}(G''_{\tilde\ttS})_{\tau}$ of $\Sh_{K''_p}(G''_{\tilde\ttS})$ at $\tau$ as the vanishing locus of the $\tau$-th  partial Hasse invariant of the versal family $\bfA''_{\tilde\ttS}$. Each $\Sh_{K''_p}(G''_{\tilde\ttS})_{\tau}$ is projective and  smooth by \cite[Proposition~4.7]{tian-xiao1}.
Transferring these structures to the quaternionic Shimura varieties using Proposition~\ref{P:integral model of Sh}, one gets  a Goren--Oort divisor $\Sh_{K_p}(G_{\ttS,\ttT})_{\tau}$ on $\Sh_{K_p}(G_{\ttS,\ttT})$ for each $\tau\in \ttS_{\infty}^c$.
 When $\ttT=\emptyset$, this is done in \cite[4.9]{tian-xiao1}, and the general case is the same. 
 
For a subset $J \subset \ttS_\infty^c$, we put $\Sh_{K_p}(G_{\ttS,\ttT})_J = \cap _{\tau \in J} \Sh_{K_p}(G_{\ttS, \ttT})_\tau$ and $\Sh_{K''_p}(G''_{\tilde\ttS})_{J}=\cap_{\tau\in J}\Sh_{K''_p}(G''_{\tilde\ttS})_{\tau}$.
The closed subvarieties  $\Sh_{K_p}(G_{\ttS,\ttT})_{J}$ (resp. $\Sh_{K''_p}(G''_{\tilde\ttS})_{J}$) with $J$ running through the subsets of $\ttS_{\infty}^c$ form the Goren--Oort stratification of $\Sh_{K_p}(G_{\ttS,\ttT})$ (resp. $\Sh_{K''_p}(G''_{\tilde\ttS})$).

The main results of \cite{tian-xiao1} give an explicit description of all closed Goren--Oort strata $\Sh_{K''_p}(G''_{\tilde\ttS})_{J}$ (resp.  $\Sh_{K_p}(G_{\ttS,\ttT})_{J}$) as  iterated $\PP^1$-bundles  over another unitary (resp. quaternionic) Shimura variety of the same type.
We list results from \cite{tian-xiao1} that we will make use in this paper.  (One more result will be used later in proving Lemma~\ref{L:restriction gysin compatibility}.)

\begin{prop}\label{P:GO-fibration}
Let $\tau \in \ttS_\infty^c$.
 Assume that $\tau^-:=\sigma^{-n_{\tau}}\tau$ is different from $\tau$ (See \ref{N:tau-plus} for the notation). 
We put $\ttS_\tau = \ttS \cup \{\tau, \tau^-\}$ and $\ttT_\tau = \ttT\cup \{\tau\}$. 
Let  $\tilde\ttS_{\tau,\infty}$  be the lifting of $\ttS_{\tau,\infty}$ derived from $\tilde\ttS_{\infty}$  according to  the rule of \cite[5.3]{tian-xiao1}, and  put $\tilde\ttS_{\tau}=(\ttS_{\tau},\tilde\ttS_{\tau,\infty})$. In particular, $\Delta_{\tilde \ttS_{\tau,\infty}} = \Delta_{\tilde \ttS_{\infty}}$ when $\gothp $ splits in $E/F$. 
\begin{enumerate}
\item[(1)]
There exists a  $\PP^1$-bundle fibration
\[
\pi_{\tau}'': \Sh_{K''_p}(G''_{\tilde\ttS})_{\tau}\ra \Sh_{K''_p}(G''_{\tilde\ttS_{\tau}})_{\tau}
\]
equivariant for the action of $\calG''_{\tilde \ttS, p} = \calG''_{\tilde \ttS_\tau, p}$, and a $p$-quasi-isogeny of abelian schemes on $\Sh_{K''}(G''_{\tilde\ttS})_{\tau}$
\[
\Phi_{\pi_{\tau}''}: \bfA''_{\tilde\ttS}\ra \pi''^*_{\tau}(\bfA''_{\tilde\ttS'}).
\]
By Construction~\ref{S:from unitary to quaternionic}, this gives rise to a $\PP^1$-bundle fibration 
\[
\pi_\tau: \Sh_{K_p}(G_{\ttS,\ttT})_{\tau} \longto \Sh_{K_p}(G_{\ttS_\tau, \ttT_\tau})
\]
with some shift $\boldsymbol t_\tau = \boldsymbol t_\tau(\ttS, \ttT) \in E^{\times, \cl}\backslash \AAA_E^{\infty, \times} / \calO^\times_{E_{\gothp}}$, which is unique up to $F^{\times, \cl}\backslash \AAA_{F}^{\infty,\times}/\cO^{\times}_{F_{\gothp}}$ and 
compatible with the Hecke action of $G(\AAA^{\infty,p})$, together with an isomorphism of \'etale sheaves for each given  regular multiweight $(\underline k, w)$
\[
\pi_\tau^\sharp\colon \calL^{(\underline k, w)}_{\ttS, \ttT}|_{\Sh_{K_p}(G_{\ttS,\ttT})_{\tau}} \xra{\cong} \pi_\tau^*(\calL^{(\underline k, w)}_{\ttS_\tau, \ttT_\tau}).
\]
The morphisms $\pi_\tau$ and $\pi_\tau^\sharp$ are unique up to a central Hecke action by an element of $F^{\times, \cl} \backslash \AAA_F^{\infty,  \times} /\calO_\gothp^\times$.

\item[(2)]
Let $\cO(1)$ be the tautological quotient line bundle  on $\Sh_{K_p}(G_{\ttS,\ttT})_{\tau}$ for the $\PP^1$-bundle given by $\pi_\tau$.
If $\tau^-$ is different from $\tau$, then the normal bundle of the closed immersion  $\Sh_{K_p}(G_{\ttS, \ttT})_\tau \hookrightarrow \Sh_{K_p}(G_{\ttS, \ttT})$ is, up to tensoring a line bundle which is a torsion class in the Picard group of $\Sh_{K_p}(G_{\ttS, \ttT})_\tau$, the same as $\calO(-2p^{n_\tau}) = \calO(1)^{\otimes(-2p^{n_\tau})}$.

 
\end{enumerate}

\end{prop}
\begin{proof}
In statement (1), the existence of $\pi''_{\tau}$ is a special case of \cite[Corollary 5.9]{tian-xiao1}.
 Roughly speaking, this $\PP^1$-bundle $\pi_{\tau}''$ parametrizes the lines  (the Hodge filtration) in the reduced $\tilde\tau^-:=\sigma^{-n_{\tau}}\tilde\tau$-component of the relative de Rham homology of the versal family  $\bfA''_{\tilde\ttS_{\tau}}$ on $\Sh_{K''_p}(G''_{\tilde\ttS_{\tau}})_{\tau}$.
It is straightforward to check that the condition~\eqref{E:transfer condition} is satisfied for the pairs $(\tilde \ttS, \ttT)$ and $(\tilde \ttS_\tau, \ttT_\tau)$.
We apply Construction~\ref{S:from unitary to quaternionic} to deduce the existence of $(\pi_{\tau}, \pi_\tau^\sharp)$ from that of $(\pi''_{\tau}, \Phi_{\pi''_\tau})$.
 
 Statement (2) follows from \cite[Proposition~6.4]{tian-xiao1}, when noting that the quaternionic Shimura varieties and the unitary Shimura varieties have isomorphic geometric connected components. 
\end{proof}

 Proposition~\ref{P:cohomology of sh}(1) implies that we have a  morphism     
\[
\pi_\tau^*\colon 
H^\star_\et \big(\Sh_K(G_{\ttS_\tau, \ttT_\tau})_{\overline \FF_p}, \calL^{(\underline k, w)}_{\ttS_\tau, \ttT_\tau}\big) \longto 
H^\star_\et \big(\Sh_K(G_{\ttS, \ttT})_{\tau, \overline \FF_p}, \calL^{(\underline k, w)}_{\ttS, \ttT}\big)
\]
equivariant under the actions of the prime-to-$p$ Hecke algebra 
$\calH_{K^p}$.
It is canonical up to the action of the central Hecke character, which comes from the ambiguity of choosing the shift in Construction~\ref{S:from unitary to quaternionic}.

\begin{theorem}
\label{T:GO description}
\
\begin{itemize}
\item[(1)]
Let $\tau_1, \tau_2 \in \ttS_\infty^c$ be two places such that $\tau_1, \tau_2, \tau_1^-, \tau_2^-$ are distinct. We have a Cartesian diagram
\[
\xymatrix{
	\Sh_{K_p}(G_{\ttS,\ttT})_{\{\tau_1,\tau_2\}} \ar[r]^-{\pi_{\tau_1}} \ar[d]^{\pi_{\tau_2}}
	& \Sh_{K_p}(G_{\ttS _{\tau_1}, \ttT_{\tau_1}})_{\tau_2} \ar[d]_{\pi_{\tau_2}}
	\\
	\Sh_{K_p}(G_{\ttS_{\tau_2}, \ttT_{\tau_2}})_{\tau_1}
	\ar[r]^-{\pi_{\tau_1}}
	&
	\Sh_{K_p}(G_{\ttS_{\tau_1,\tau_2}, \ttT_{\tau_1,\tau_2}}).
}
\]
If we use the notation of shifts of these $\pi_{\tau_i}$ as in Proposition~\ref{P:GO-fibration}(1), then we have an equality
 \[
\boldsymbol {t}_{\tau_1}(\ttS,\ttT) \boldsymbol t_{\tau_2}(\ttS_{\tau_1}, \ttT_{\tau_2}) =\boldsymbol {t}_{\tau_2}(\ttS,\ttT) \boldsymbol{t}_{\tau_1}(\ttS_{\tau_2},\ttT_{\tau_2}).
\]
Moreover, we have a commutative diagram of induced morphisms on  the cohomology groups:
\[
\xymatrix{
H^\star_\et\big(\Sh_{K_p}(G_{\ttS_{\tau_1}\cup \ttS_{\tau_2}, \ttT_{\tau_1} \cup \ttT_{\tau_2}})_{\overline \FF_p}, \calL^{(\underline k,w)}_{\ttS_{\tau_1}\cup \ttS_{\tau_2}, \ttT_{\tau_1} \cup \ttT_{\tau_2}} \big)
 \ar[r]^-{\pi_{\tau_2}^*}
\ar[d]^{\pi_{\tau_1}^*}
&
H^\star_\et\big( \Sh_{K_p}(G_{\ttS_{\tau_1}, \ttT_{\tau_1}})_{\tau_2, \overline \FF_p}, \calL^{(\underline k, w)}_{\ttS_{\tau_1}, \ttT_{\tau_1}}\big)
\ar[d]^{\pi_{\tau_1}^*}
\\
H^\star_\et\big( \Sh_{K_p}(G_{\ttS_{\tau_2}, \ttT_{\tau_2}})_{\tau_1, \overline \FF_p}, \calL^{(\underline k, w)}_{\ttS_{\tau_2}, \ttT_{\tau_2}}\big)
\ar[r]^-{\pi_{\tau_2}^*}
&
H^\star_\et\big( \Sh_{K_p}(G_{\ttS, \ttT})_{\{\tau_1, \tau_2\},\overline \FF_p},
\calL_{\ttS, \ttT}^{(\underline k, w)}\big).
}
\]
\item[(2)]
Let $\tau \in \ttS_\infty^c$ be a place such that $\tau,\tau^+, \tau^-$ are distinct. Put $n = n_{\tau^+}- n_{\tau}$ if $\gothp$ splits in $E/F$ and $n=0$ if $\gothp$ is inert in $E/F$.
Let  $\eta: \ttS_{\tau^+} = \ttS \cup \{ \tau^+, \tau\} \to \ttS_{\tau} = \ttS \cup \{ \tau, \tau^-\}$ be the link  given by straight lines except sending $\tau^-$ to $\tau^+$ over $\tau$:
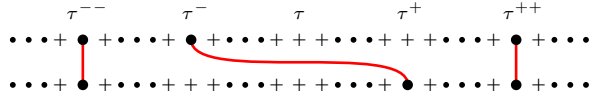
\[
\psset{unit=0.3}
\begin{pspicture*}(-5,-0.4)(23,4)
\psset{linecolor=red}
\psset{linewidth=1pt}
\psline{-}(0,0)(0,2)
\psline{-}(19.2,0)(19.2,2)
\psbezier(4.8,2)(4.8,0)(14.4,2)(14.4,0)
\psset{linecolor=black}
\psdots(0,0)(0,2)(4.8,2)(14.4,0)(19.2,0)(19.2,2)
\psdots[dotstyle=+](1,0)(-1,0)(-1,2)(1,2)(3.8,0)
(3.8,2)(4.8,0)(5.8,0)(5.8,2)(8.6,0)(8.6,2)(10.6,0)(9.6,0)(9.6,2)(10.6,2)(13.4,0)(13.4,2)(15.4,0)(14.4,2)(15.4,2)(18.2,0)(18.2,2)(20.2,0)(20.2,2)
\psset{linewidth=.1pt}
\psdots(1.7,0)(1.7,2)(2.4,0)(2.4,2)(3.1,0)(3.1,2)(6.5,0)(7.2,0)(7.9,0)(6.5,2)(7.2,2)(7.9,2)(11.3,0)(12,0)(12.7,0)(11.3,2)(12,2)(12.7,2)(16.1,0)(16.8,0)(17.5,0)(16.1,2)(16.8,2)(17.5,2)(-1.7,0)(-1.7,2)(-2.4,0)(-2.4,2)(-3.1,0)(-3.1,2)(20.9,0)(21.6,0)(22.3,0)(20.9,2)(21.6,2)(22.3,2)
\rput(.2,3.3){\psframebox*{\begin{tiny}$\tau^{--}$\end{tiny}}}
\rput(5,3.3){\psframebox*{\begin{tiny}$\tau^{-}$\end{tiny}}}
\rput(9.6,3.1){\psframebox*{\begin{tiny}$\tau$\end{tiny}}}
\rput(14.5,3.3){\psframebox*{\begin{tiny}$\tau^{+}$\end{tiny}}}
\rput(19.5,3.3){\psframebox*{\begin{tiny}$\tau^{++}$\end{tiny}}}
\end{pspicture*}
\]
Let  $\eta_{(n),\sharp}$ be  the morphism defined by the following commutative diagram 
\begin{equation}\label{D:link-morphism-Sh}
\xymatrix{
\Sh_{K_p}(G_{\ttS, \ttT})_{\tau^+} \ar[d]_{\pi_{\tau^+}}
&
\Sh_{K_p}(G_{\ttS, \ttT})_{\{\tau^+,\tau\}}
\ar@{_{(}->}[l]
\ar@{^{(}->}[r]
\ar@{->}[ld]_\cong
&
\Sh_{K_p}(G_{\ttS, \ttT})_{\tau}
\ar[d]^{\pi_\tau}
\\
\Sh_{K_p}(G_{\ttS_{ \tau^+}, \ttT_{\tau^+}}) \ar[rr]^{\eta_{(n), \sharp}}
&&
\Sh_{K_p}(G_{\ttS_{\tau}, \ttT_\tau}),
}
\end{equation}
Then the following statements hold:
\begin{enumerate}
\item[(a)] The map $\eta_{(n), \sharp}$ is the morphism obtained by applying  Construction~\ref{S:from unitary to quaternionic} to a link morphism on $\Sh_{K''_p}(G''_{\tilde\ttS_{\tau^+}})$ with indentation degree $n$. 

\item[(b)] If $\boldsymbol t_? \in E^{\times, \cl} \backslash \AAA_E^{\infty,\times}/\calO_{E, \gothp}^\times$ for $? = \tau, \tau^+$ denotes the shift  of the correspondence 
$$\Sh_{K_p}(G_{\ttS_?, \ttT_?}) \xleftarrow{\pi_?} \Sh_{K_p}(G_{\ttS, \ttT})_? \hookrightarrow \Sh_{K_p}(G_{\ttS, \ttT}),$$
then $\eta_{(n), \sharp}$ has shift $\boldsymbol t_{\tau^+} \boldsymbol t^{-1}_\tau$. 

\item[(c)] The morphism  $\eta_{(n),\sharp}$ is finite flat of degree $p^{v(\eta)}$.

\item[(d)] The $p$-quasi-isogenies of the versal families of abelian varieties on  $\Sh_{K''_p}(G''_{\tilde\ttS_{\tau^+}})$ given by 
\[
\pi''^*_{\tau^+}(\bfA''_{\tilde\ttS_{\tau^+}})|_{\Sh_{K''_p}(G''_{\tilde\ttS})_{\{\tau^+,\tau\}}}\xleftarrow{\Phi_{\pi''_{\tau^+}}} \bfA''_{\tilde\ttS}|_{\Sh_{K''_p}(G''_{\tilde\ttS})_{\{\tau^+,\tau\}}}\xra{\Phi_{\pi''_{\tau}}} \pi''^*_{\tau}(\bfA''_{\tilde\ttS_{\tau}})|_{\Sh_{K''_p}(G''_{\tilde\ttS})_{\{\tau^+,\tau\}}}
\]
induces a link morphism on the sheaves $\eta_{(n)}^\sharp: \calL^{(\underline k, w)}_{\ttS_{\tau^+}, \ttT_{\tau^+}} \longrightarrow \eta_{(n),\sharp}^*(\calL^{(\underline k, w)}_{\ttS_\tau, \ttT_\tau})$. 
 Then the induced normalized link morphism $\eta_{(n)}^\star$ on the cohomology groups constructed as in \ref{S:normalization of link morphisms}   fits into the following commutative diagram:
\begin{equation}\label{D:link-quaternionic-Sh}
\xymatrix{
H^\star_\et\big(\Sh_{K_p}(G_{\ttS, \ttT})_{\tau^+, \overline \FF_p}\big) \ar[r] &
H^\star_\et\big(\Sh_{K_p}(G_{\ttS, \ttT})_{\{\tau^+, \tau\}, \overline \FF_p}\big) & H^\star_\et\big(\Sh_{K_p}(G_{\ttS, \ttT})_{\tau,\overline \FF_p}\big)
\ar[l]
\\
H^\star_\et\big(\Sh_{K_p}(G_{\ttS _{ \tau^+}, \ttT_{\tau^+}})_{\overline \FF_p}\big) \ar[u]^{\pi_{\tau^+}^*} \ar@{-->}[ru]^\cong
&&
H^\star_\et\big(\Sh_{K_p}(G_{\ttS _{ \tau}, \ttT_\tau})_{\overline \FF_p}\big) \ar[u]^{\pi_\tau^*}
\ar[ll]_-{p^{(n_\tau + n_{\tau^+})/2} \eta^\star_{(n)}},
}
\end{equation}
where  the upper horizontal arrows are natural restriction maps.
Here,  for simplification, we have suppressed the sheaves from the notation. For instance, $ H^\star_\et(\Sh_{K_p}(G_{\ttS, \ttT})_{\tau^+, \overline \FF_p}) $ should be understood as $H^\star_\et(\Sh_{K_p}(G_{\ttS, \ttT})_{\tau^+, \overline \FF_p}, \calL_{\ttS,\ttT}^{(\kb,w)}|_{\Sh_{K_p}(G_{\ttS, \ttT})_{\tau^+}} )$.

\end{enumerate}

\item[(3)]
Assume that  $\ttS_\infty^c = \{\tau, \tau^-\}$ (and hence $\gothp$ splits in $E/F$).
Then $\Sh_{K_p}(G_{\ttS,\ttT})_{\{\tau,\tau^-\}}$ is isomorphic to the special fiber of the zero-dimensional Shimura variety $\Sh_{\Iw_p}(G_{\ttS_{\tau},\ttT_{\tau}})$ of Iwahori level at $\gothp$.
Let $\eta: \ttS_{\tau^-} \to \ttS_{\tau}$ denote the link map (with no curve). 
Then the link morphism $\eta_{(n_{\tau^-}),\sharp}:  \Sh_{K_p}(G_{\ttS_{\tau^-}, \ttT_{\tau^-}})\xra{\sim} \Sh_{K_p}(G_{\ttS_\tau, \ttT_\tau})$ of indentation degree $2n_{\tau^-}$ associated to $\eta$ exists, and the following diagram 
\[
\xymatrix{
& \Sh_{K_p}(G_{\ttS, \ttT})_{ \{\tau, \tau^-\}}\ar[dl]_{\pi_\tau} \ar[dr]^{\pi_{\tau^-}}
\\
\Sh_{K_p}(G_{\ttS_\tau, \ttT_\tau}) && 
\Sh_{K_p}(G_{\ttS_{\tau^-}, \ttT_{\tau^-}})\ar[r]_-{\eta_{(n_{\tau^-})}}^-\cong
&
\Sh_{K_p}(G_{\ttS_\tau, \ttT_\tau}).
}
\]
is (the base change to $\FF_{p^g}$ of)   the Hecke  correspondence $T_\gothp$ on $\Sh_{K_p}(G_{\ttS_{\tau},\ttT_{\tau}})$.
If $\boldsymbol t_?\in E^{\times, \cl} \backslash \AAA_E^{\infty,\times}/\calO_{E_{\gothp}}^\times$ for $? = \tau, \tau^+$ denotes the shift of the correspondence 
$$\Sh_{K_p}(G_{\ttS_?, \ttT_?}) \xleftarrow{\pi_?} \Sh_{K_p}(G_{\ttS, \ttT})_? \hookrightarrow \Sh_{K_p}(G_{\ttS, \ttT}),$$
then $\eta_{(n_{\tau^-}), \sharp}$ has shift $\varpi_{\bar \gothq}\boldsymbol t_{\tau}^{-1} \boldsymbol t_{\tau^-}$.
Moreover, 
the  map induced by the diagram above on cohomology groups
\[
H^0_\et \big(\Sh_K(G_{\ttS_{\tau}, \ttT_{\tau}})_{\overline \FF_p}\big)
\xrightarrow{(\eta_{(n_{\tau^-})}\circ\pi_{\tau^-})^*}H^0_\et \big(\Sh_K(G_{\ttS, \ttT})_{\{\tau,\tau^-\}, \overline \FF_p} \big) \xrightarrow{\mathrm{Tr}_{\pi_\tau}}
H^0_\et \big(\Sh_K(G_{\ttS_\tau, \ttT_\tau})_{\overline \FF_p}\big).
\]
is the usual Hecke action $T_\gothp$. Here, as in (2), we have suppressed the sheaves from the notation.
\end{itemize}
\end{theorem}
\begin{proof}
The analogs of (1),    (2)(a)  and (2)(c) for unitary Shimura varieties were  proved in
\cite[Proposition~7.12 and Theorem~7.16]{tian-xiao1}.  The statements here follow  from Construction~\ref{S:from unitary to quaternionic}.

Statement (2)(b) regarding shifts follows directly from Remark~\ref{R:composition of shifts}. Statement (2)(d) is direct by the construction of $\eta_{(n)}^{\sharp}$ and $\eta^{\star}_{(n)}$.
 
For (3), the analogous statement for unitary Shimura varieties $\Sh_{K''}(G''_{\tilde\ttS'})$ (with $T_\gothp$ replaced by $T_{\gothq}$) was proved in \cite[Theorem~7.16(2)]{tian-xiao1}. One deduces (3) using the construction in Subsection~\ref{S:Hecke-at-p}, and computes the shifts by Remark~\ref{R:composition of shifts}.
\end{proof}

\section{Goren--Oort cycles}

In this section, we investigate  certain generalization of Goren--Oort strata, called Goren--Oort cycles.  They are parametrized by certain combinatorial data, called \emph{periodic semi-meanders}.
We will show later that the intersection matrix of the Goren--Oort cycles turns out to be closely related to the Gram matrix associated to these periodic semi-meanders (which explains our choice of the combinatorial model).

\subsection{Periodic semi-meanders}
\label{S:semi-meanders}
The combinatorial construction that we will use later is related to the so-called link representations of periodic Temperley--Lieb algebras, which appear naturally in the study of mathematical physics; see for example \cite{di-francesco, graham-lehrer, xxz-model}.
We will simply state here the main result with minimal input, and refer to \cite{xxz-model} for a detailed discussion of the mathematical physics background and the proofs.

We slightly modify the usual definition of periodic semi-meanders to adapt to  our situation.
Recall that $F$ is a totally real field of degree $g$ and $\ttS$, $\ttT$ are introduced as in Subsection~\ref{S:quaternoinic Shimura varieties}, and $d = \#\ttS_\infty^c$.
We consider the band associated to $\ttS$ defined as in Subsection~\ref{S:links}, and recall that the band is placed on a cylinder, but we often draw it over the $2$-dimensional $xy$-plane with $x$-coordinate taken modulo $g$.

A \emph{periodic semi-meander} for $\ttS$ is a collection of curves (called \emph{arcs}) that link two nodes of the band for $\ttS$, and straight lines (called \emph{semi-lines}) that links a node to the infinity ($+\infty$ in the $y$-direction) subject to the following conditions: 
\begin{itemize}
\item All the arcs and semi-lines lie on the cylinder above the band (that is to have positive $y$-coordinate in the $2$-dimensional picture). 
\item Each node of the band for $\ttS$ is exactly one end point of an arc or a semi-line.  
\item There are no intersection points among these arcs and semi-lines. 
\end{itemize}   
The number of arcs is denoted by $r$ (so $r \leq d/2$), and the number of semi-lines $d-2r$ is called the \emph{defect} of the periodic semi-meander.
Two periodic semi-meanders are considered as the same if they can be continuously deformed into each other while keeping the above three properties in the process.
We use $\gothB_\ttS^r$ denote the set of semi-meanders for $\ttS$ with $r$ arcs (up to the deformations).
For example, if $F$ has degree $7$ over $\QQ$, $r=2$, and $\ttS = \{\infty_1, \infty_4\}$, we have
\begin{eqnarray}
\label{E:BS2}\qquad\quad
\gothB_\ttS^2 &=\ 
\Big\{
\psset{unit=0.6}
\begin{pspicture*}(-0.25,-0.1)(3.25,0.8)
\psset{linewidth=1pt}
\psset{linecolor=red}
\psarc{-}(1.25,0){0.25}{0}{180}
\psarc{-}(2.75,0){0.25}{0}{180}
\psline{-}(0,0)(0,0.75)
\psset{linecolor=black}
\psdots(0,0)(1,0)(1.5,0)(2.5,0)(3,0)
\psdots[dotstyle=+](0.5,0)(2,0)
\end{pspicture*},
\
\begin{pspicture*}(-0.25,-0.1)(3.25,0.8)
\psset{linewidth=1pt}
\psset{linecolor=red}
\psarc{-}(-0.25,0){0.25}{0}{90}
\psarc{-}(3.25,0){0.25}{90}{180}
\psarc{-}(2,0){0.5}{0}{180}
\psline{-}(1,0)(1,0.75)
\psset{linecolor=black}
\psdots(0,0)(1,0)(1.5,0)(2.5,0)(3,0)
\psdots[dotstyle=+](0.5,0)(2,0)
\end{pspicture*},
\
\begin{pspicture*}(-0.25,-0.1)(3.25,0.8)
\psset{linewidth=1pt}
\psset{linecolor=red}
\psarc{-}(0.5,0){0.5}{0}{180}
\psarc{-}(2.75,0){0.25}{0}{180}
\psline{-}(1.5,0)(1.5,0.75)
\psset{linecolor=black}
\psdots(0,0)(1,0)(1.5,0)(2.5,0)(3,0)
\psdots[dotstyle=+](0.5,0)(2,0)
\end{pspicture*},\
\begin{pspicture*}(-0.25,-0.1)(3.25,0.8)
\psset{linewidth=1pt}
\psset{linecolor=red}
\psarc{-}(1.25,0){0.25}{0}{180}
\psarc{-}(-0.25,0){0.25}{0}{90}
\psarc{-}(3.25,0){0.25}{90}{180}
\psline{-}(2.5,0)(2.5,0.75)
\psset{linecolor=black}
\psdots(0,0)(1,0)(1.5,0)(2.5,0)(3,0)
\psdots[dotstyle=+](0.5,0)(2,0)
\end{pspicture*},\
\begin{pspicture*}(-0.25,-0.1)(3.25,0.8)
\psset{linewidth=1pt}
\psset{linecolor=red}
\psarc{-}(0.5,0){0.5}{0}{180}
\psarc{-}(2,0){0.5}{0}{180}
\psline{-}(3,0)(3,0.75)
\psset{linecolor=black}
\psdots(0,0)(1,0)(1.5,0)(2.5,0)(3,0)
\psdots[dotstyle=+](0.5,0)(2,0)
\end{pspicture*},\\
\nonumber& 
\psset{unit=0.6}
\begin{pspicture*}(-0.25,-0.1)(3.25,0.8)
\psset{linewidth=1pt}
\psset{linecolor=red}
\psarc{-}(2,0){0.5}{0}{180}
\psbezier{-}(1,0)(1,1)(3,1)(3,0)
\psline{-}(0,0)(0,0.75)
\psset{linecolor=black}
\psdots(0,0)(1,0)(1.5,0)(2.5,0)(3,0)
\psdots[dotstyle=+](0.5,0)(2,0)
\end{pspicture*},\
\begin{pspicture*}(-0.25,-0.1)(3.25,0.8)
\psset{linewidth=1pt}
\psset{linecolor=red}
\psarc{-}(2.75,0){0.25}{0}{180}
\psbezier{-}(1.5,0)(1.5,1)(3.5,1)(3.5,0)
\psbezier{-}(-1.5,0)(-1.5,0.75)(0,0.75)(0,0)
\psline{-}(1,0)(1,0.75)
\psset{linecolor=black}
\psdots(0,0)(1,0)(1.5,0)(2.5,0)(3,0)
\psdots[dotstyle=+](0.5,0)(2,0)
\end{pspicture*},\
\begin{pspicture*}(-0.25,-0.1)(3.25,0.8)
\psset{linewidth=1pt}
\psset{linecolor=red}
\psarc{-}(-0.25,0){0.25}{0}{90}
\psarc{-}(3.25,0){0.25}{90}{180}
\psbezier{-}(-2,0)(-2,1)(1,1)(1,0)
\psbezier{-}(2.5,0)(2.5,1)(4.5,1)(4.5,0)
\psline{-}(1.5,0)(1.5,0.75)
\psset{linecolor=black}
\psdots(0,0)(1,0)(1.5,0)(2.5,0)(3,0)
\psdots[dotstyle=+](0.5,0)(2,0)
\end{pspicture*},\
\begin{pspicture*}(-0.25,-0.1)(3.25,0.8)
\psset{linewidth=1pt}
\psset{linecolor=red}
\psarc{-}(0.5,0){0.5}{0}{180}
\psbezier{-}(-.5,0)(-.5,1)(1.5,1)(1.5,0)
\psbezier{-}(3,0)(3,1)(5,1)(5,0)
\psline{-}(2.5,0)(2.5,0.75)
\psset{linecolor=black}
\psdots(0,0)(1,0)(1.5,0)(2.5,0)(3,0)
\psdots[dotstyle=+](0.5,0)(2,0)
\end{pspicture*},\
\begin{pspicture*}(-0.25,-0.1)(3.25,0.8)
\psset{linewidth=1pt}
\psset{linecolor=red}
\psarc{-}(1.25,0){0.25}{0}{180}
\psbezier{-}(0,0)(0,1)(2.5,1)(2.5,0)
\psline{-}(3,0)(3,0.75)
\psset{linecolor=black}
\psdots(0,0)(1,0)(1.5,0)(2.5,0)(3,0)
\psdots[dotstyle=+](0.5,0)(2,0)
\end{pspicture*} \Big\}.
\end{eqnarray}
When drawing in the $xy$-plane, points are placed on the $x$-axis at points of coordinates $(0,0), \dots, (g-1,0)$ and the diagram for a periodic semi-meander is taken to be periodic in the $x$-direction of period $g$.
  The curves connecting the points can connect across the imaginary boundary lines at $x=-1/2$ and $x =g-1/2$ (which are identified). See for example \eqref{E:BS2}. 
An elementary calculation shows that $\#\gothB_\ttS^r = \binom dr$.

A \emph{standard presentation} of a semi-meander is where all the  arcs are monotonic in the $x$-direction, namely, it does not twist back and forth.
Using the $xy$-plane picture, we define the \emph{left} and \emph{right end-nodes} of an arc, as follows:
\begin{itemize}
\item
when the arc appears as one arc in the standard presentation, its left (resp. right) end-node is the left (resp. right) endpoint of the arc;
\item
when the arc  appears in two parts linked through the imaginary boundary  lines at $x = -1/2$ and $x= g-1/2$, its left (resp. right) end-node  is the right (resp. left) endpoint of the arc 
\end{itemize}

For $\gotha \in \gothB_\ttS^r$, we use $\ell(\gotha)$ to denote the \emph{total span} of $\gotha$, that is the sum of the span of all curves over the band, where the span takes into account of the periodicity at the imaginary boundary.  For example, the last element of $\gothB_\ttS^2$ in \eqref{E:BS2} has two arcs with spans $1$ and $5$, respectively, and hence its total span is $6$. The second element of $\gothB_{\ttS}^2$ in \eqref{E:BS2} has two arcs with spans $1$ and $2$, respectively, and hence its total span is $3$.

\subsection{Gram matrix}
\label{S:gram matrix}

For $\gotha, \gothb \in \gothB_\ttS^r$, we consider the drawing $D(\gotha, \gothb)$ obtained by taking mirror image of $\gothb$ reflected about the $x$-axis and then identifying the $d$ nodes of $\gothb$ with those of $\gotha$ according to their labellings.
\begin{itemize}
\item
We say a loop (namely, a closed curve) in $D(\gotha, \gothb)$ is \emph{contractible} if it can be continuously contracted to a point on the cylinder (ignoring all other curves and lines on the picture). We write $m_0(\gotha, \gothb)$ for the number of contractible loops in $D(\gotha, \gothb)$.
\item
We say a loop in $D(\gotha, \gothb)$ is \emph{non-contractible} if, ignoring other curves and lines on the picture, it can be continuously deformed into a loop wrapped around the cylinder. (Since all loops do not intersect itself, the loop can only wrap the cylinder for one round.) We write $m_T(\gotha, \gothb)$ for the number of non-contractible loops in $D(\gotha,\gothb)$. This number  can be zero   if $r = d/2$.
\item
We use $\ttS_\gotha$ to denote the union of $\ttS$ with the nodes that are connected to an arc of $\gotha$.
So the band of $\ttS_\gotha$ maybe obtained from the band of $\gotha$ by replacing the end-nodes of arcs in $\gotha$ by plus signs.
 We define $\ttS_\gothb$ similarly.
\item
Assume that $r<d/2$, that no semi-line of $\gotha$ is connected to another semi-line of $\gotha$ in $D(\gotha,\gothb)$, and that the same is true for $\gothb$.
We define the \emph{reduction} of $D(\gotha, \gothb)$ to be a link $\eta_{\ttS_\gotha, \ttS_\gothb}$ from the band of $\ttS_\gotha$ to the band of $\ttS_\gothb$ such that each node $\tau_\gotha$ of $\ttS_\gotha$ (corresponding to an element of $\ttS_{\gotha, \infty}^c$) is linked to a node $\tau_\gothb$ of $\ttS_\gothb$ in the same way as the semi-line at $\tau_\gotha$ is linked to the semi-line at $\tau_\gothb$ in $D(\gotha, \gothb)$.
In practice, this amounts to removing all the (contractible) loops in $D(\gotha, \gothb)$, and then continuously deforming the remaining curves into a link (with top and bottom extended by semi-lines).
We put $m_v(\gotha, \gothb)$ to be the total displacement of $\eta_{\ttS_\gotha, \ttS_\gothb}$.
\item
When $r = \frac d2$, $\ttS_\gotha = \ttS_\gothb$ contains all the archimedean places. For consistency,  we write $\eta_{\ttS_\gotha, \ttS_\gothb}$ for the trivial link from the band of $\ttS_\gotha$ to the band of $\ttS_\gothb$ (as there are no nodes on the bands).
\end{itemize}

We define the \emph{Gram product} to be the following pairing
\[
\langle \cdot\,|\,\cdot\rangle_\ttS: \gothB_\ttS^r \times \gothB_\ttS^r \xrightarrow{\quad\quad}
\begin{cases} \overline \QQ_\ell(v) & \textrm{ if } r< d/2, \\
\overline \QQ_\ell[T] & \textrm{ if }r =d/2.
\end{cases}
\]
\[
\langle \gotha \,|\, \gothb\rangle_\ttS
=
\begin{cases}
0 &  {\begin{array}{l}\textrm{if in the diagram }D(\gotha, \gothb), \textrm{ two semi-lines}\\ \textrm{of }\gotha \textrm{(or of }\gothb) \textrm{ are connected,}
\end{array}} \\
(-2)^{m_0(\gotha, \gothb)} v^{m_v(\gotha, \gothb)} & \textrm{ \;otherwise if } r< d/2 , \textrm{ and}\\
 (-2)^{m_0(\gotha, \gothb)} T^{m_T(\gotha, \gothb)} & \textrm{ \;otherwise if }r =d/2.
\end{cases}
\]
Note that only one of $m_v(\gotha, \gothb)$ and $m_T(\gotha, \gothb)$ can be non-zero by definition.
We use $\gothV_\ttS^r$ to denote the $\overline \QQ_\ell$-vector space with basis $\gothB_\ttS^r$ and extend the Gram product linearly to all of $\gothV_\ttS^r$.

\begin{example}
\label{Ex:periodic semi-meanders}
The following examples are copied from \cite{xxz-model}.

\begin{enumerate}
\item
$
\psset{unit=0.3}
\gotha =
\begin{pspicture*}(-1.5,-0.2)(9.5,3)
\psset{linewidth=1pt}
\psset{linecolor=red}
\psline{-}(9,0)(9,3)
\psbezier{-}(0,0)(0,3.5)(8,3.5)(8,0)
\psbezier{-}(1,0)(1,2)(5,2)(5,0)
\psarc{-}(3.5,0){0.5}{0}{180}
\psarc{-}(6.5,0){0.5}{0}{180}
\psset{linecolor=black}
\psdots(0,0)(1,0)(3,0)(4,0)(9,0)(8,0)(7,0)(6,0)(5,0)
\psdots[dotstyle=+](2,0)
\end{pspicture*}$,
$\psset{unit=0.3}
\gothb =
\begin{pspicture*}(-1.5,-0.2)(9.5,3)
\psset{linewidth=1pt}
\psset{linecolor=red}
\psline{-}(0,0)(0,3)
\psbezier{-}(4,0)(4,1.5)(7,1.5)(7,0)
\psarc{-}(2,0){1}{0}{180}
\psarc{-}(5.5,0){0.5}{0}{180}
\psarc{-}(8.5,0){0.5}{0}{180}
\psset{linecolor=black}
\psdots(0,0)(1,0)(3,0)(4,0)(9,0)(8,0)(7,0)(6,0)(5,0)
\psdots[dotstyle=+](2,0)
\end{pspicture*}$,
$\psset{unit=0.3}
D(\gotha,\gothb) = 
\begin{pspicture*}(-.5,-3)(9.5,3)
\psset{linewidth=1pt}
\psset{linecolor=red}
\psline{-}(-1,0)(-1,3)
\psline{-}(9,0)(9,3)
\psbezier{-}(0,0)(0,3.5)(8,3.5)(8,0)
\psbezier{-}(1,0)(1,2)(5,2)(5,0)
\psarc{-}(3.5,0){0.5}{0}{180}
\psarc{-}(6.5,0){0.5}{0}{180}
\psline{-}(0,0)(0,-3)
\psline{-}(-1,0)(-1,-3)
\psbezier{-}(4,0)(4,-1.5)(7,-1.5)(7,0)
\psarc{-}(2,0){1}{180}{360}
\psarc{-}(5.5,0){0.5}{180}{360}
\psarc{-}(8.5,0){0.5}{180}{360}
\psset{linecolor=black}
\psdots(-1,0)(0,0)(1,0)(3,0)(4,0)(9,0)(8,0)(7,0)(6,0)(5,0)
\psdots[dotstyle=+](2,0)
\end{pspicture*}$,
the reduction of the link is $
\psset{unit=0.3}
\eta_{\ttS_\gotha, \ttS_\gothb} = 
\begin{pspicture*}(-.5,-1)(9.5,3)
\psset{linewidth=1pt}
\psset{linecolor=red}
\psline{-}(9,2)(9,3)
\psline{-}(0,0)(0,-1)
\psbezier{-}(0,0)(0,1.5)(9,.5)(9,2)
\psset{linecolor=black}
\psdots(0,0)(9,2)
\psdots[dotstyle=+](2,0)(1,0)(3,0)(4,0)(8,0)(7,0)(6,0)(5,0)(2,2)(0,2)(1,2)(3,2)(4,2)(8,2)(7,2)(6,2)(5,2)(9,0)
\end{pspicture*}
$,
and $\langle \gotha| \gothb \rangle_\ttS = (-2) v^{-9}$.

\item
$
\psset{unit=0.3}
\gotha =
\begin{pspicture*}(-.5,-0.2)(9.5,3)
\psset{linewidth=1pt}
\psset{linecolor=red}
\psline{-}(0,0)(0,3)
\psline{-}(9,0)(9,3)
\psbezier{-}(1,0)(1,3.5)(8,3.5)(8,0)
\psbezier{-}(2,0)(2,1.5)(5,1.5)(5,0)
\psarc{-}(3.5,0){0.5}{0}{180}
\psarc{-}(6.5,0){0.5}{0}{180}
\psset{linecolor=black}
\psdots(0,0)(1,0)(2,0)(3,0)(4,0)(9,0)(8,0)(7,0)(6,0)(5,0)
\end{pspicture*}$,
$\psset{unit=0.3}\gothb =
\begin{pspicture*}(-.5,-0.2)(9.5,3)
\psset{linewidth=1pt}
\psset{linecolor=red}
\psline{-}(6,0)(6,3)
\psline{-}(7,0)(7,3)
\psbezier{-}(1,0)(1,1.5)(4,1.5)(4,0)
\psbezier{-}(0,0)(0,2.5)(5,2.5)(5,0)
\psarc{-}(2.5,0){0.5}{0}{180}
\psarc{-}(8.5,0){0.5}{0}{180}
\psset{linecolor=black}
\psdots(0,0)(1,0)(2,0)(3,0)(4,0)(9,0)(8,0)(7,0)(6,0)(5,0)
\end{pspicture*}$,
$\psset{unit=0.3} D(\gotha, \gothb) = 
\begin{pspicture*}(-.5,-3)(9.5,3)
\psset{linewidth=1pt}
\psset{linecolor=red}
\psline{-}(0,0)(0,3)
\psline{-}(9,0)(9,3)
\psbezier{-}(1,0)(1,3.5)(8,3.5)(8,0)
\psbezier{-}(2,0)(2,1.5)(5,1.5)(5,0)
\psarc{-}(3.5,0){0.5}{0}{180}
\psarc{-}(6.5,0){0.5}{0}{180}
\psline{-}(6,0)(6,-3)
\psline{-}(7,0)(7,-3)
\psbezier{-}(1,0)(1,-1.5)(4,-1.5)(4,0)
\psbezier{-}(0,0)(0,-2.5)(5,-2.5)(5,0)
\psarc{-}(2.5,0){0.5}{180}{360}
\psarc{-}(8.5,0){0.5}{180}{360}
\psset{linecolor=black}
\psdots(0,0)(1,0)(2,0)(3,0)(4,0)(9,0)(8,0)(7,0)(6,0)(5,0)
\end{pspicture*}$,
and $\langle\gotha\,|\,\gothb\rangle_\ttS = 0$.

\item
$\gotha =
\psset{unit=0.3}
\begin{pspicture*}(-.5,-0.2)(9.5,2)
\psset{linewidth=1pt}
\psset{linecolor=red}
\psbezier{-}(3,0)(3,1.5)(6,1.5)(6,0)
\psarc{-}(-0.5,0){0.5}{0}{90}
\psarc{-}(1.5,0){0.5}{0}{180}
\psarc{-}(4.5,0){0.5}{0}{180}
\psarc{-}(7.5,0){0.5}{0}{180}
\psarc{-}(9.5,0){0.5}{90}{180}
\psset{linecolor=black}
\psdots(0,0)(1,0)(2,0)(3,0)(4,0)(9,0)(8,0)(7,0)(6,0)(5,0)
\end{pspicture*}$,
$\gothb =\psset{unit=0.3}
\begin{pspicture*}(-.5,-0.2)(9.5,4.6)
\psset{linewidth=1pt}
\psset{linecolor=red}
\psbezier{-}(-4,0)(-4,3.5)(3,3.5)(3,0)
\psbezier{-}(-5,0)(-5,4.5)(4,4.5)(4,0)
\psbezier{-}(6,0)(6,3.5)(13,3.5)(13,0)
\psbezier{-}(5,0)(5,4.5)(14,4.5)(14,0)
\psarc{-}(-.5,0){0.5}{0}{90}
\psarc{-}(1.5,0){0.5}{0}{180}
\psarc{-}(7.5,0){0.5}{0}{180}
\psarc{-}(9.5,0){0.5}{90}{180}
\psset{linecolor=black}
\psdots(0,0)(1,0)(2,0)(3,0)(4,0)(9,0)(8,0)(7,0)(6,0)(5,0)
\end{pspicture*}$,
and $\psset{unit=0.3}
D(\gotha, \gothb) = 
\begin{pspicture*}(-.5,-4.6)(9.5,2)
\psset{linewidth=1pt}
\psset{linecolor=red}
\psbezier{-}(3,0)(3,1.5)(6,1.5)(6,0)
\psarc{-}(-0.5,0){0.5}{-90}{90}
\psarc{-}(1.5,0){0.5}{0}{360}
\psarc{-}(4.5,0){0.5}{0}{180}
\psarc{-}(7.5,0){0.5}{0}{360}
\psarc{-}(9.5,0){0.5}{90}{270}
\psbezier{-}(-4,0)(-4,-3.5)(3,-3.5)(3,0)
\psbezier{-}(-5,0)(-5,-4.5)(4,-4.5)(4,0)
\psbezier{-}(6,0)(6,-3.5)(13,-3.5)(13,0)
\psbezier{-}(5,0)(5,-4.5)(14,-4.5)(14,0)
\psset{linecolor=black}
\psdots(0,0)(1,0)(2,0)(3,0)(4,0)(9,0)(8,0)(7,0)(6,0)(5,0)
\end{pspicture*}$, and $\langle \gotha \,|\, \gothb\rangle_\ttS = (-2)^3 T^2$.
\end{enumerate}
\end{example}

\begin{remark}
\label{R:quantization}
When $\ttS = \emptyset$,
the vector space  $\gothV_\ttS^r$ is the \emph{link representation} of the so-called periodic Temperley--Lieb algebra $\calE TLP_N(T, -2)$ under the notation of \cite{xxz-model}.  (In particular, we specialize the theory to the case when the quantum variable $q = i$.)  With respect to the bilinear form we introduced earlier, the representation is $\dagger$-hermitian with respect to the natural involution $\dagger$ on the Temperley--Lieb algebra.  Since we will not use the structure of this representation, we simply refer to \cite[Section~2.3]{xxz-model} for further discussion.  
It seems that the mysterious relationship between this mathematical physics calculation and our Shimura variety calculation probably  comes from  some common representation theory feature. 
  It might be an intriguing question to ask what quantization could mean for  Shimura varieties (or its local analogues) so that the intersection matrix computed in a similar manner as we did later would have a chance to match with the quantized version of the Gram determinant in {\it loc. cit.}
\end{remark}

The following theorem is essentially the main theorem of \cite{xxz-model} (which seems to have been  known by  \cite{graham-lehrer} using a different argument).
\begin{theorem}
\label{T:determinant}
Put $t_{d,r} = \sum_{i = 0}^{r-1} \binom di$.
Let $\gothG_\ttS^r$ denote the Gram matrix $(\langle \gotha\,|\,\gothb\rangle)_{\gotha, \gothb \in \gothB_\ttS^r}$.  Then its determinant is given as follows.
\begin{enumerate}
\item When $d$ is even, $\det \gothG_\ttS^{d/2} = \pm (T^2 - 4)^{t_{d, d/2}}$.
\item For $r <d/2$, $\det \gothG_\ttS^r = \pm (v^g - v^{-g})^{2t_{d,r}}$.
\end{enumerate}
\end{theorem}
\begin{proof}
When $\ttS = \emptyset$ (so $d=g$), this is a special case of \cite[Theorem~4.1]{xxz-model}.
Indeed, the parameter $\alpha$ in {\it loc. cit.} is $T$ in our notation, and since the $\beta$ in {\it loc. cit.}  is $-2$, the $C_k$  in {\it loc. cit.} are  equal to $\pm 1$ for all $k$.  One easily simplifies their formula to the one stated in this theorem.

The general case requires little modification, but the method of the proof may be viewed as a toy model for the proof of Theorem~\ref{T:Tate} later.
When $r= \frac d2$, we just simply ignore all points corresponding to $\ttS_\infty$. This verifies (1).
So we assume $r < \frac d2$ from now on to prove (2). We use $\langle \gotha | \gothb \rangle_d$ to denote the pairing computed by removing all points from $\ttS_\infty$ (and shrink the cylinder accordingly) and hence with displacements computed with respect to only the $d$ nodes.
Let $\gothG_d^r$ denote the corresponding matrix.  We need to compare $\det \gothG_d^r$ with $\det \gothG_\ttS^r$, by showing that $\det \gothG_\ttS^r$ can be obtained by replacing all $v^d$ in the expression of $\det\gothG_d^r$ by $v^g$.

By the definition of determinant, $\det \gothG_\ttS^r$ is the sum over all permutations $\sigma$ of the set $\gothB_\ttS^r$, of the product of the signature of $\sigma$, and, for every cycle $(\gotha_1 \dots \gotha_t)$ of the permutation $\sigma$, the product
\begin{equation}
\label{E:product of links}
\langle \gotha_1 |\gotha_2\rangle_\ttS \cdot \langle \gotha_2 |\gotha_3\rangle_\ttS\cdots \langle \gotha_t|\gotha_1\rangle_\ttS.
\end{equation}
The same applies to $\det \gothG_d^r$ except that the product \eqref{E:product of links} are taken for the pairing $\langle \cdot |\cdot\rangle_d$.
The product \eqref{E:product of links}, if not zero, is equal to $(-2)^{m_0}v^{m_v}$, where $m_0 = m_0(\gotha_1, \gotha_2)+\cdots +m_0(\gotha_t,\gotha_1)$ is the sum of  total number of contractible loops in the diagrams $D(\gotha_{1},\gotha_2), D(\gotha_2, \gotha_3),\dots,D(\gotha_t,\gotha_1)$, and   $m_v = m_v(\gotha_1, \gotha_2)+\cdots +m_v(\gotha_t,\gotha_1)$ is equal to the total displacement of the composition of the link 
\begin{equation}
\label{E:composition of links}
\eta_{\ttS_ {\gotha_t}, \ttS_{ \gotha_1}}  \circ \cdots\circ \eta_{\ttS_ {\gotha_2}, \ttS_{ \gotha_3}} \circ
\eta_{\ttS_{ \gotha_1}, \ttS_ {\gotha_2}},
\end{equation}
by the additivity of total displacements as  remarked in Subsection~\ref{S:links}.
Note that \eqref{E:composition of links} is in fact a link from $\ttS_{ \gotha_1}$ to itself.  So it must be an integer power $n$ of the fundamental link $\eta_{\ttS_{ \gotha_1}}$ defined in Subsection~\ref{S:links}.
In particular, we have $m_v = n g$.
Making the same observation for computing the standard Gram determinant $\det \gothG_d^r$, the product \eqref{E:product of links} with $\langle \cdot |\cdot \rangle_d$ is instead equal to $(-2)^{m_0} v^{m'_v}$ with the same $m_0$ as above, and $m'_v$ is the total displacement of \eqref{E:composition of links} with all points corresponding to $\ttS_\infty$ removed.  By the same discussion above, we have $m'_v=  n d$ with the \emph{same} $n$ above.

In conclusion, each term of $\det \gothG_\ttS^r$ can be obtained from the corresponding term of $\det \gothG_d^r$ via replacing $v^d$ by $v^g$.
Therefore, $\det \gothG_\ttS^r = \pm (v^g - v^{-g})^{2t_{d,r}}$.
\end{proof}

\begin{notation}
\label{N:basic arc}
Using the illustration of periodic semi-meanders as in \eqref{E:BS2}, we say an arc $\delta$ \emph{lies over} another arc $\delta'$ if the contractible closed loop in the picture given by adjoining $\delta$ with the equator contains $\delta'$ inside.  
For example, in the list of $\gothB_\ttS^2$ in \eqref{E:BS2}, the last five periodic semi-meanders each has an arc lying over another.

In a periodic semi-meander for $\ttS$, a \emph{basic arc} is an arc $\delta$ which satisfies the following equivalent conditions
\begin{itemize}
\item
in the 2-dimensional picture, 
$\delta$ does not lie over any other arcs,
\item
in the 2-dimensional picture, the only points below $\delta$ are plus signs, or
\item
$\delta$ is an arc which links some $\tau$ to $\tau^-$ (See \ref{N:tau-plus} for the notation).
\end{itemize}
For example, in the list of $\gothB_\ttS^2$ in \eqref{E:BS2}, each of the five periodic semi-meanders in the first row  has two basic arcs, and each of the five periodic semi-meanders in the second row  has one basic arc.

It is clear that all periodic semi-meanders have a basic arc except the one with only semi-lines.
Given a periodic semi-meander $\gotha \in \gothB_\ttS^r$ for $\ttS$  with a basic arc $\delta$ linking two nodes $\tau,\tau^- \in \ttS_\infty^c$, we can delete the arc and replace its end-nodes by $+$ to get a periodic semi-meander $\gotha\backslash \delta \in \gothB_{\ttS\cup\{\tau, \tau^-\}}^{r-1}$ for $\ttS \cup \{\tau, \tau^-\}$.
\end{notation}

\subsection{Goren--Oort cycles}
\label{S:GO cycles}
We fix a pair $(\ttS, \ttT)$ as before.
For a periodic semi-meander $\gotha$ for $\ttS$, we define a pair $(\ttS_\gotha,\ttT_{\gotha})$ as follows: $\ttS_\gotha$ is obtained by adjoining to $\ttS$ all end-nodes of the arcs of $\gotha$ and $\ttT_\gotha$ is obtained by adjoining to $\ttT$ all the \emph{right} end-nodes (in the sense of \ref{S:semi-meanders}) of the arcs of $\gotha$.

We now construct the \emph{Goren--Oort cycle} $\Sh_{K_p}(G_{\ttS,\ttT})_\gotha$  associated to a periodic semi-meander $\gotha$ for $(\ttS, \ttT)$.\footnote{It is expected that the Goren--Oort cycle $\Sh_{K_p}(G_{\ttS,\ttT})_\gotha$  is independent of the auxiliary choices when defining the unitary Shimura variety $\Sh_{K_p''}(G_{\tilde \ttS})$. However, we do not know how to prove this.} The cycle will admit an iterated $\PP^1$-bundle morphism
\[
\pi_\gotha: \Sh_{K_p}(G_{\ttS, \ttT})_\gotha \to \Sh_{K_p}(G_{\ttS_\gotha, \ttT_\gotha})
\]
for some appropriate subsets $\ttS_\gotha$ and $\ttT_\gotha$ of $\Sigma_\infty$. The resulting correspondence 
\begin{equation}
\label{E:GO cycle correspondence} \Sh_{K_p}(G_{\ttS_\gotha, \ttT_\gotha}) \xleftarrow{\pi_\gotha}\Sh_{K_p}(G_{\ttS, \ttT})_\gotha \hookrightarrow \Sh_{K_p}(G_{\ttS, \ttT})
\end{equation}
will be constructed using the unitary Shimura varieties via Construction~\ref{S:from unitary to quaternionic} depending on a shift $\boldsymbol t_{\gotha} = \boldsymbol t_{\emptyset, \gotha}  \in E^{\times, \cl} \backslash \AAA_E^{\infty, \times}/\cO_{E_{\gothp}}^{\times}$, which is canonical up to $F^{\times, \cl}\backslash\AAA_F^{\infty,\times}/\cO_{F_{\gothp}}^{\times}$,  as explained in Construction~\ref{S:from unitary to quaternionic}.

We first define a partial order on the set of all periodic semi-meander for our fixed pair $(\ttS, \ttT)$ by setting $\gotha \prec \gothb$ if all arcs of $\gotha$ appears (up to graphic deformation) as arcs of $\gothb$.
Clearly, the periodic semi-meander consisting of only semi-lines is the unique minimal element for this partial order.
We will define the Goren--Oort cycles correspondence \eqref{E:GO cycle correspondence} inductively for  the partial order above.

Suppose that we have defined this for all periodic semi-meanders $\gotha' \prec \gotha$.
Now, for $\gotha$, among its arcs, there exists at least one arc that does not lie below any other arc in the sense of Notation~\ref{N:basic arc}. We fix  such an  arc $\delta$. Let $\gotha \backslash \delta$ denote the periodic semi-meander obtained by removing the arc $\delta$ from $\gotha$ and adjoin semi-lines to the end nodes of $\delta$.
By the inductive hypothesis, we have a correspondence
\begin{equation}
\label{E:GO cycle correspondence induction} \Sh_{K_p}(G_{\ttS_{\gotha \backslash \delta}, \ttT_{\gotha \backslash \delta}}) \xleftarrow{\pi_{\gotha \backslash \delta}}\Sh_{K_p}(G_{\ttS, \ttT})_{\gotha \backslash \delta} \hookrightarrow \Sh_{K_p}(G_{\ttS, \ttT})
\end{equation}
with shift $\boldsymbol t_{\gotha \backslash \delta}$.
Let $\tau$ (resp. $\tau_-$) denote the right (resp. left) end-node of $\delta$.
Applying  Proposition~\ref{P:GO-fibration}(1) to the Shimura variety $\Sh_{K_p}(G_{\ttS_{\gotha \backslash \delta}, \ttT_{\gotha \backslash \delta}})$, 
we deduce a natural $\PP^1$-bundle morphism
\[
\pi_\tau: \Sh_{K_p}(G_{\ttS_{\gotha \backslash \delta},\ttT_{\gotha \backslash \delta}})_{\tau} \longto \Sh_{K_p}(G_{\ttS_{\gotha \backslash \delta} \cup\{\tau, \tau_-\}, \ttT_{\gotha \backslash \delta}\cup\{\tau\}})
\]
with some shift $\boldsymbol t_{\gotha\backslash \delta, \gotha}$ (and we fix such a choice).
We define the \emph{Goren--Oort cycle} $\Sh_{K_p}(G_{\ttS,\ttT})_\gotha$ to be
\[
\Sh_{K_p}(G_{\ttS,\ttT})_\gotha: =
\pi_{\gotha \backslash \delta}^{-1}\big(\Sh_{K_p}(G_{\ttS_{\gotha \backslash \delta}, \ttT_{\gotha\backslash \delta}})_{\tau} \big),
\]
namely it fits the following commutative diagram where the square is Cartesian.
\[
\xymatrix{
\Sh_{K_p}(G_{\ttS,\ttT})_{\gotha } \ar[d] \ar@{^{(}->}[r] & \ar[d]_{\pi_{\gotha\backslash \delta}} \Sh_{K_p}(G_{\ttS,\ttT})_{\gotha \backslash \delta} \ar@{^{(}->}[r] \ar[d] & \Sh_{K_p}(G_{\ttS,\ttT})
\\
\Sh_{K_p}(G_{\ttS_{\gotha \backslash \delta},\ttT_{\gotha \backslash \delta}})_\tau \ar@{^{(}->}[r] \ar[d]^{\pi_\tau} & \Sh_{K_p}(G_{\ttS_{\gotha \backslash \delta},\ttT_{\gotha \backslash \delta}})
\\
\Sh_{K_p}(G_{\ttS_\gotha,\ttT_\gotha}).
}
\]
The induced
 correspondence \[
\Sh_{K_p}(G_{\ttS_\gotha, \ttT_\gotha}) \xleftarrow{\pi_\gotha := \pi_\tau\circ \pi_{\gotha \backslash \delta}}\Sh_{K_p}(G_{\ttS, \ttT})_\gotha \hookrightarrow \Sh_{K_p}(G_{\ttS, \ttT})
\]
has   shift 
\[
\boldsymbol{t}_{\gotha} = \boldsymbol t_{\emptyset, \gotha}: = \boldsymbol t_{\emptyset, \gotha \backslash \delta} \cdot \boldsymbol t_{\gotha\backslash \delta, \gotha}.
\]
This completes the inductive construction of the Goren--Oort cycles. Using Theorem~\ref{T:GO description}(1), it is easy to see inductively that such a definition of $\Sh_{K_p}(G_{\ttS,\ttT})_\gotha$ does not depend on the choice of the basic arc $\delta$.

Note that, in the process of the inductive construction, by Theorem~\ref{T:GO description}(1), 
if $\delta'$ is another arc of $\gotha$ with left (resp. right) end-nodes $\tau'$ (resp. $\tau'_-$) that does not lie below any other arc, then we have a correspondence
\[
\Sh_{K_p}(G_{\ttS_\gotha, \ttT_\gotha}) \xleftarrow{\pi_{\delta'}} \Sh_{K_p}(G_{\ttS_{\gotha \backslash \delta'}, \ttT_{\gotha \backslash \delta'}})_{\tau'} \hookrightarrow \Sh_{K_p}(G_{\ttS_{\gotha \backslash \delta'}, \ttT_{\gotha \backslash \delta'}})
\]
such that $\pi_{\delta'}$ is a $\PP^1$-bundle and, by Remark~\ref{R:composition of shifts}, the shift of this correspondence is
\[
\boldsymbol t_{\gotha \backslash \delta', \gotha}: = \boldsymbol t_\gotha \boldsymbol t_{\gotha \backslash \delta'}^{-1}.
\]
So  applying inductively this discussion, we have, for every periodic semi-meander $\gotha \prec \gothb$, a correspondence
\begin{equation}
\label{E:correspondence ab}
\Sh_{K_p}(G_{\ttS_\gotha, \ttT_\gotha}) \xleftarrow{\pi_{\gothb, \gotha}} \Sh_{K_p}(G_{\ttS_{\gothb}, \ttT_{\gothb}})_{\gotha \backslash \gothb} \hookrightarrow \Sh_{K_p}(G_{\ttS_{\gothb}, \ttT_{\gothb}})
\end{equation}
of shift $
\boldsymbol t_{\gothb, \gotha} := \boldsymbol t_\gotha \boldsymbol t_\gothb^{-1}.$

We point out that a key feature of our construction is that \emph{the  dimension of fibers of $\Sh_{K_p}(G_{\ttS,\ttT})_{\gotha}$ over $\Sh_{K_p}(G_{\ttS_{\gotha},\ttT_{\gotha}})$ is the same as the codimension of $\Sh_{K_p}(G_{\ttS,\ttT})_{\gotha}$ in $\Sh_{K_p}(G_{\ttS,\ttT})$}, which is $r$.

We fix a regular multiweight $(\underline k,w)$. Recall that $\calL_{\ttS, \ttT}^{(\underline k, w)}$ denotes the automorphic $\ell$-adic local system on $\Sh_{K_p}(G_{\ttS,\ttT})$.
The same construction above also gives rise to an isomorphism
\[
\pi_\gotha^\sharp:\pi_\gotha^*( \calL^{(\underline k, w)}_{\ttS_\gotha, \ttT_\gotha}) \xrightarrow{\ \cong\ }
\calL_{\ttS, \ttT}^{(\underline k, w)}|_{\Sh_{K_p}(G_{\ttS, \ttT})_\gotha}.
\]

\begin{remark}
\label{R:GO cycles are NP stratification}
It was pointed out  to us by X. Zhu that the union of all Goren--Oort cycles associated to periodic semi-meanders with $r$ arcs is exactly the closure of certain \emph{Newton strata} of the unitary Shimura variety, transported to the quaternionic side.
In the case of Hilbert modular varieties, the union of all codimension $r$ Goren--Oort cycles are exactly the closed Newton stratum of the Newton polygon with slopes  $\frac{r}{g}$ and $\frac{g-r}{g}$ both with multiplicity $g$. 
So maybe the name ``Goren--Oort" is slightly misleading, as it usually refers to stratification given by the $p$-torsion subgroup of the universal abelian varieties.
\end{remark} 

\begin{example}
\label{Ex:GO cycles}
Let $F$ be of degree $6$ over $\QQ$ and $\ttS = \ttT = \emptyset$. Then $\Sh_K(G_{\emptyset,\emptyset})$ is (the special fiber of) the Hilbert modular variety for $F$. 
  Let $\tau_0, \dots, \tau_5$ denote the embeddings of $\calO_F$ into $\ZZ_p^\ur$ so that $\tau_i = \tau_{i\pmod 6}$ and $\tau_{i+1} = \sigma \tau_i$. 
We have a universal abelian variety $A$ over $\Sh_K(G_{\emptyset, \emptyset})$ equipped with an $\calO_F$-action.

 We consider the periodic semi-meander $\gotha = \psset{unit=0.6}\begin{pspicture*}(-0.25,-0.1)(2.75,0.75)
\psset{linewidth=1pt}
\psset{linecolor=red}
\psarc{-}(0.75,0){0.25}{0}{180}
\psbezier{-}(0,0)(0,0.75)(1.5,0.75)(1.5,0)
\psarc{-}(2.25,0){0.25}{0}{180}
\psset{linecolor=black}
\psdots(0,0)(0.5,0)(1,0)(1.5,0)(2,0)(2.5,0)
\end{pspicture*}$.
For each $\overline \FF_p$-point $x \in \Sh_K(G_{\emptyset,\emptyset})$, 
the Dieudonn\'e module $\calD_x$ of the universal abelian variety $A_x$ at $x$ decomposes as 
$\calD_x = \bigoplus_{i=0}^5 \calD_{x, i}$, where $\calO_F$ acts on the $i$-th factor via $\tau_i$. Let $V_i: \calD_{x,i+1}\ra \calD_{x,i}$ denote the Verschiebung map for $i\in \Z/5\Z$.
Then $x \in \Sh_K(G_{\emptyset,\emptyset})_\gotha$ if and only if
\[
V_1 \circ V_2(\calD_{x, 3})\subseteq p \calD_{x, 1}, \quad V_4 \circ V_5(\calD_{x, 0})\subseteq p \calD_{x, 4}, \quad \textrm{and } V_0\circ V_1 \circ V_2 \circ V_3(\calD_{x, 4})\subseteq p^2 \calD_{x, 0}.
\]
In fact, these inclusions are forced to be equalities.
In this case, $\Sh_{K_p}(G_{\emptyset,\emptyset})_\gotha$ is just an iterated $\PP^1$-fibration over  the discrete Shimura variety $\Sh_K(G_{\Sigma_\infty, \{\tau_2, \tau_3, \tau_5\}})$.  Moreover, one can prove that each geometric connected component is isomorphic to the product of $\PP^1$ with the projective bundle $\PP(\calO_{\PP^1}(-p) \oplus \calO_{\PP^1}(p)) $ over $\PP^1$.\footnote{More canonically, it is the product of $\PP^1$ with the  projective bundle attached to a rank two bundle $E$ over $\PP^1$ which sits inside an exact sequence $0 \to \calO_{\PP^1}(p) \to E \to \calO_{\PP^1}(-p) \to 0$, which  splits (non-canonically).}
\end{example}

\section{Cohomology of Goren--Oort cycles}
Let $\Sh_{K_p}(G_{\ttS,\ttT})$ be the special fiber of a  quaternionic Shimura variety as in Section~\ref{S:quaternoinic Shimura varieties}.
Using Gysin maps, the cohomology of the Goren--Oort cycles gives rise to part of the cohomology of the big Shimura variety $\Sh_{K_p}(G_{\ttS,\ttT})$.

\subsection{Generalities on \'etale cohomology}\label{S:etale-cohomology}
We recall first some generalities on Gysin maps and \'etale cohomology of iterated $\PP^1$-bundles. 
Let $\ell$ be a fixed prime number, and $k$ be an algebraically closed field of characteristic different from  $\ell$.

Consider a closed immersion $i: Y\hra X$ of smooth varieties over $k$ of codimension $r$. 
The functor of direct image $i_*$ has a right adjoint, denoted by $i^!$.  For an $\ell$-adic \'etale sheaf $\calF$ on $X$, $i^!\calF$ is the sheaf of sections of $\calF$ with support in $Y$. 
This is a left exact functor, and let $R^{q}i^!$ denote its $q$-th derived functor.
 Then by the relative cohomological purity \cite[XVI, Th\'eor\`eme 3.7]{SGA4}, we have $R^{q}i^!\overline \Q_{\ell}=0$ for $q\neq 2r$, and a canonical isomorphism $R^{2r}i^{!}\overline\Q_{\ell}\xra{\cong}\overline\Q_{\ell}(-r)$. 
 Explicitly,  the inverse isomorphism $\overline \Q_{\ell}\xra{\cong}R^{2r}i^{!}\overline \Q_{\ell}(r)$ is given by  the fundamental class of $Y$ in $X$:  $\cl_{X}(Y)\in H^{2r}_{\et, Y}(X,\overline\Q_{\ell}(r))\cong H^{0}_{\et}(Y, R^{2r}i^{!} \overline \QQ_\ell)$. 
 Now for any lisse $\overline {\QQ}_{\ell}$-sheaf $\calF$ on $X$, we define the Gysin map as the composition
 \begin{equation}\label{E:Gysin-general}\mathrm{Gysin}\colon H^{q}_{\et}(Y, i^*\calF)\xra{\cup\, \cl_X(Y)} H^{q+2r}_{\et, Y}(X,\calF(r))\ra H^{q+2r}_{\et}(X,\calF(r)),
 \end{equation}
 where the second map is the canonical morphism from cohomology supported in $Y$ to the usual cohomology group. 
  If $i_{Z}:Z\hra X$ is another closed immersion of smooth varieties  such that $Y$ intersects with $Z$ transversally, then one has  $i_Z^*\cl_X(Y)=\cl_Z(Y\cap Z)$. It follows that the following diagram is commutative
  \begin{equation}\label{E:gysin-restriction}
  \xymatrix{H^q_{\et}(Y, i^*\calF)\ar[r]^{\mathrm{Gysin}} \ar[d]^{\mathrm{Restr.}} & H^{q+2r}_{\et}(X, \calF(r))\ar[d]^{\mathrm{Restr.}}\\
H^{q}_{\et}(Y\cap Z, i_{Y\cap Z}^*\calF)\ar[r]^{\mathrm{Gysin}} & H^{q+2r}_{\et}(Z,i^*_Z\calF(r)),}
  \end{equation}
  where the vertical maps are given by natural restrictions, and $i_{Y\cap Z}:Y\cap Z\hra X$ is the natural embedding.

 Let $\pi: X\ra Y$ be an $r$-th  iterated  $\PP^1$-bundle of proper and  smooth $k$-varieties, i.e. $\pi$ admits a factorization 
 \begin{equation}\label{E:factorization-iterated-bundle}
 \pi\colon X_0:=X\xra{\pi_1}X_1\xra{\pi_2} X_2\ra \cdots\xra{\pi_r} X_{r}:=Y,
 \end{equation}
 where each $\pi_i: X_{i-1}\ra X_{i}$ is a $\PP^1$-fibration for $1\leq i\leq r$. 
 Then the trace map 
 \[\Tr_{\pi}\colon  R^{2r}\pi_{*}(\overline\Q_{\ell}(r))\xra{\cong} \overline \Q_{\ell}\]
 is an isomorphism.
 We denote by $\cl_{\pi}\in  H^0(Y, R^{2r}\pi_*\overline\Q_{\ell}(r))$ with $\Tr_{\pi}(\cl_{\pi})=1$, and call it \emph{fundamental class of the fibration $\pi$}. 
 For any $\overline \Q_{\ell}$-lisse sheaf $\calF$ on $Y$ and any integer $q\geq 0$,  the isomorphism $\Tr_{\pi}$ induces a map 
 \[
 \pi_{!}\colon H^{q}_{\et}(X, \pi^*\calF(r))\ra H^{q-2r}_\et(Y, \calF\otimes R^{2r}\pi_{*}(\overline \Q_{\ell}(r)))\xra{\Tr_{\pi}} H^{q-2r}_{\et}(Y, \calF),
 \]
where  the first morphism comes from the Leray spectral sequence $E_2^{a,b}=H^a_{\et}(Y, R^{b}\pi_* \pi^*\calF(r))\Rightarrow H^{a+b}_\et(X, \pi^*\calF(r)).$
Explicitly, $\pi_!$ admits the following description. 
 Put  $\pi_{[0,i]}:=\pi_{i}\circ\pi_{i-1}\circ\cdots\circ\pi_1$ for $1\leq i\leq r$.
  Let $\cO_{\pi_i}(1)$ be the tautological quotient line bundle of the $\PP^1$-bundle $\pi_i$, and $c_1(\cO_{\pi_i}(1))\in H^2_\et(X_{i-1},\overline \Q_{\ell}(1))$ be its first Chern class. 
  Put $\xi_{i}=\pi_{[0,i-1]}^*c_1(\cO_{\pi_i}(1))\in H^2_\et(X, \overline \Q_{\ell}(1))$. 
  By induction on $r$,  one  deduces easily from \cite[VII, Corollaire 2.2.6]{SGA5}  a decomposition 
  \[
  H^q_{\et}(X,\pi^*\calF(r))\cong \bigoplus_{0\leq j\leq r}\bigg(\bigoplus_{1\leq i_{1}<\cdots< i_j\leq r}\pi^*H^{q-2j}_{\et}(Y, \calF(r-j))\cup\xi_{i_1}\cup\cdots \cup\xi_{i_j}\bigg).
  \]
 Then for an element $x=\sum_{j}\sum_{1\leq i_1<\cdots <i_{j}\leq r } \pi^*(y_{i_1,\dots,i_j})\cup\xi_{i_1}\cup\cdots \cup\xi_{i_j}$,  one has 
 \begin{equation}\label{E:simple-formula}
 \pi_{!}(x)=y_{1,\dots, r}.
 \end{equation}
 In particular, the fundamental class $\cl_{\pi}$ is  the image of $\xi_{1}\cup\cdots \cup\xi_{r}$ in $H^0_\et(X, R^{2r}\pi_*\overline\Q_{\ell}(r))$.
  
\subsection{Gysin and restriction maps}
\label{S:Gysin and restriction maps}
We keep the notation of Section~\ref{S:GO cycles}.
The pair of morphisms $(\pi_\gotha, \pi_\gotha^\sharp)$ induces  the following sequence of natural homomorphisms, whose composition we denote by $\mathrm{Gys}_\gotha$,
\begin{eqnarray*}
H_\et^{d-2r} \big(\Sh_{K_p}(G_{\ttS_\gotha, \ttT_\gotha})_{\overline \FF_p}, \calL_{\ttS_\gotha, \ttT_\gotha}^{(\underline k, w)} \big) &
\xrightarrow{\quad \pi_{\gotha}^*,\, \cong \quad}
& H_\et^{d-2r}\big(\Sh_{K_p}(G_{\ttS, \ttT})_{\gotha,\overline \FF_p}, \pi_\gotha^*(\calL_{\ttS_\gotha, \ttT_\gotha}^{(\underline k, w)}) \big)
\\&
\xrightarrow{\quad \pi_\gotha^\sharp,\, \cong\ \quad }
&
 H_\et^{d-2r}\big(\Sh_{K_p}(G_{\ttS, \ttT})_{\gotha,\overline \FF_p}, \calL_{\ttS, \ttT}^{(\underline k, w)}|_{\Sh_{K_p}(G_{\ttS, \ttT})_\gotha}\big)
 \\&
\xrightarrow{\mathrm{Gysin},\,\eqref{E:Gysin-general}}
& H_\et^d\big(\Sh_{K_p}(G_{\ttS, \ttT})_{\overline \FF_p}, \calL_{\ttS, \ttT}^{(\underline k, w)}(r)\big).
\end{eqnarray*}

We can also consider the dual picture,
defining the morphism $\mathrm{Res}_\gotha$ to be the composition of the following homomorphisms:
\begin{eqnarray*}
\Res_\gotha: H_\et^d\big(\Sh_{K_p}(G_{\ttS, \ttT})_{\overline \FF_p}, \calL_{\ttS, \ttT}^{(\underline k, w)}(r)\big)
&
\xrightarrow{ \mathrm{Restriction}}
&
H_\et^d\big(\Sh_{K_p}(G_{\ttS, \ttT})_{\gotha,\overline \FF_p}, \calL_{\ttS, \ttT}^{(\underline k, w)}|_{\Sh_{K_p}(G_{\ttS, \ttT})_\gotha}(r)\big)
\\
&
\xrightarrow{ (\pi_\gotha^\sharp)^{-1}, \, \cong}&
H_\et^d\big(\Sh_{K_p}(G_{\ttS, \ttT})_{\gotha,\overline \FF_p}, \pi_\gotha^*\calL_{\ttS_\gotha, \ttT_\gotha}^{(\underline k, w)}(r)\big)
\\
&
\xrightarrow{ \quad \pi_{\gotha,!}\quad }
&
H_\et^{d-2r}\big(\Sh_{K_p}(G_{\ttS_\gotha, \ttT_\gotha})_{\overline \FF_p}, \calL_{\ttS_\gotha, \ttT_\gotha}^{(\underline k, w)}\big).
\end{eqnarray*}
It is clear from the construction that both morphisms $\Gys_\gotha$ and $\Res_\gotha$ are equivariant for the action of the prime-to-$p$ Hecke algebra $\calH_{K^p}=\overline \QQ_{\ell}[K^p\backslash G(\AAA^{\infty, p})/K^p]$.

The following theorem is the key to proving our main result.  We defer its proof to the next section.

\begin{theorem}
\label{T:intersection combinatorics}
Fix $\pi \in \scrA_{(\underline k, w)}$, and fix a choice of system of shifts $\boldsymbol t_{\gotha}$ of the correspondences $\Sh_{K_p}(G_{\ttS_\gotha, \ttT_\gotha}) \xleftarrow{\pi_\gotha}\Sh_{K_p}(G_{\ttS, \ttT})_\gotha \hookrightarrow \Sh_{K_p}(G_{\ttS, \ttT})$ as in Subsection~\ref{S:GO cycles}.
For $\gotha, \gothb \in \gothB_\ttS^r$, we have the following description of the composition
\begin{align*}
H_\et^{d-2r} \big(\Sh_{K_p}(G_{\ttS_\gothb, \ttT_\gothb})_{\overline \FF_p}, \calL_{\ttS_\gothb, \ttT_\gothb}^{(\underline k, w)}\big)&
\xrightarrow{\Gys_\gothb}
H_\et^d\big(\Sh_{K_p}(G_{\ttS, \ttT})_{\overline \FF_p}, \calL_{\ttS, \ttT}^{(\underline k, w)}(r)\big)
\\ &
\xrightarrow{\Res_\gotha}
H_\et^{d-2r} \big(\Sh_{K_p}(G_{\ttS_\gotha, \ttT_\gotha})_{\overline \FF_p}, \calL_{\ttS_\gotha, \ttT_\gotha}^{(\underline k, w)} \big).
\end{align*}
\begin{itemize}
\item[(1)]
When $\langle \gotha |\gothb \rangle =0$, the $\pi$-isotypical component of the composed map $\Res_\gotha \circ \Gys_\gothb$ factors through the $\pi$-isotypical component of the cohomology group 
$H^{d-{2(r+1)}}_\et (\Sh_{K_p}(G_{\ttS',\ttT'})_{\overline \FF_p}, \calL^{(\underline k, w)}_{\ttS',\ttT'})(-1)$
of some quaternionic Shimura variety of dimension $d-2(r+1)$ with $\#\ttT' = \#\ttT + (r+1)$ and $\ttS'$ having the same set of finite places as $\ttS$. 
\item[(2)]
When $r<\frac d2$ and $\langle \gotha |\gothb \rangle=(-2)^{m_0(\gotha, \gothb)} v^{m_v(\gotha,\gothb)}$, we can define the induced link $\eta_{\ttS_ \gotha, \ttS_ \gothb}: \ttS_\gotha \to \ttS_\gothb$ as in Subsection~\ref{S:gram matrix}.
Then there exists a normalized link morphism in the sense of Subsection~\ref{S:link morphism II}
\[
\eta_{\ttS_ \gotha, \ttS_ \gothb,(z)}^\star \colon H^{d-2r}_\et\big(\Sh_{K_p}(G_{\ttS_{\gothb},\ttT_{\gothb}})_{\overline \FF_p} \calL_{\ttS_\gothb, \ttT_\gothb}^{(\underline k, w)} \big)\ra H^{d-2r}_\et\big(\Sh_{K_p}(G_{\ttS_{\gotha},\ttT_{\gotha}})_{\overline \FF_p} \calL_{\ttS_\gotha, \ttT_\gotha}^{(\underline k, w)} \big),
\]
associated to $\eta_{\ttS_{\gotha},\ttS_{\gothb}}$ with shift $\boldsymbol t_{\gotha}\boldsymbol t_{\gothb}^{-1}$ and indentation degree 
\[z=\begin{cases}\ell(\gotha)-\ell(\gothb) & \text{if $\gothp$ splits in $E/F$},\\
0 & \text{if $\gothp$ is inert in $E/F$.}\end{cases}\] 
Moreover, we have an equality 
\[
\Res_\gotha \circ \Gys_\gothb=(-2)^{m_0(\gotha, \gothb)}\cdot p ^{(\ell(\gotha) + \ell(\gothb))/2} \eta_{\ttS_ \gotha, \ttS_ \gothb, (z)}^\star.
\]

\item[(3)]
When $r=\frac d2$ and $\langle \gotha |\gothb \rangle=(-2)^{m_0(\gotha, \gothb)}  T^{m_T(\gotha,\gothb)}$, we have   
\[
\Res_\gotha \circ \Gys_\gothb=(-2)^{m_0(\gotha,\gothb)}\cdot p^{(\ell(\gotha)+ \ell(\gothb))/2} (T_\gothp/p^{g/2})^{m_T(\gotha,\gothb)} \circ \eta_{\ttS_\gotha, \ttS_\gothb,(z)}^\star,
\]
where $\eta_{\ttS_\gotha,\ttS_\gothb}$ is the trivial link from $\ttS_\gotha$ to $\ttS_\gothb$ and \[
\eta_{\ttS_\gotha,\ttS_\gothb,(z)}^\star:H^{d-2r}_\et(\Sh_{K_p}(G_{\ttS_\gothb,\ttT_\gothb})_{\overline \FF_p}, \calL_{\ttS, \ttT}^{(\underline k, w)}) \longto H^{d-2r}_\et(\Sh_{K_p}(G_{\ttS_\gotha,\ttT_\gotha})_{\overline \FF_p}, \calL_{\ttS, \ttT}^{(\underline k, w)})\] is the associated normalized link morphism with shift $\boldsymbol t_\gotha \boldsymbol t_\gothb^{-1} \varpi_{\bar \gothq}^{-m_T(\gotha,\gothb)}$ and indentation degree  $z = \ell(\gotha) - \ell(\gothb)-m_T(\gotha,\gothb) g$.
\end{itemize}
\end{theorem}

We now assume Theorem~\ref{T:intersection combinatorics} and deduce the main theorem of this paper.

\begin{theorem}
\label{T:Tate}
Fix a positive integer $r \leq \frac d2$, and keep the notation of Theorem~\ref{T:intersection combinatorics}.
\begin{itemize}
\item[(1)] For each periodic semi-meander $\gotha \in \gothB_\ttS^r$, the Goren--Oort cycle $\Sh_{K_p}(G_{\ttS, \ttT})_{\gotha}$ of the Shimura variety $\Sh_{K_p}(G_{\ttS, \ttT})$ is a subvariety of codimension $r$, stable under the action of the tame Hecke action of $G(\AAA^{\infty,p})$.
Moreover,  it admits a natural proper smooth morphism 
$$\pi_\gotha: \Sh_{K_p}(G_{\ttS, \ttT})_\gotha \to \Sh_{K_p}(G_{\ttS_\gotha, \ttT_\gotha})
$$ to another quaternionic Shimura variety (in  characteristic $p$), such that the fibers of $\pi_\gotha$ are $r$-times iterated $\PP^1$-bundles.
The morphism is equivariant for the tame Hecke action.
\item[(2)]
We fix  a cuspidal automorphic representation $\pi \in \scrA_{(\underline k, w)}$ so that its associated Galois representation $\rho_\pi$ is unramified at $p$.
Let $\alpha_\pi$ and $\beta_\pi$ denote the (generalized) eigenvalues of $\rho_{\pi, \gothp}(\Frob_{p^g})$.
Suppose that $\alpha_\pi / \beta_\pi$ is not a $2n$-th root of unity for any $n \leq d$ so that $\alpha^{2i}_\pi \beta^{2(d-i)}_\pi$ are distinct from each other for $1\leq i\leq d$.
Then the action of $\Frob_{p^{2g}}$ on the generalized eigenspace of $H^{d}_\et\big(\Sh_{K_p}(G_{\ttS, \ttT})_{\overline \FF_p}, \calL_{\ttS, \ttT}^{(\underline k, w)}(r)\big)[\pi]$ with eigenvalue $\alpha^{2(d-r)}_\pi \beta^{2r}_\pi (\alpha_\pi \beta_\pi / p^g)^{2\#\ttT}p^{-2gr}$ is semisimple (so that the generalized eigenspace is a genuine eigenspace), and  the direct sum of the Gysin morphisms
\begin{equation}
\label{E:sum of Gysin map}
\bigoplus_{\gotha \in \gothB_\ttS^r} H^{d-2r}_\et(\Sh_{K_p}\big(G_{\ttS, \ttT})_{\gotha, \overline \FF_p}, \calL^{(\underline k, w)}_{\ttS, \ttT}\big)[\pi] \xra{\sum_{\gotha}\Gys_{\gotha}} H^{d}_\et\big(\Sh_{K_p}(G_{\ttS, \ttT})_{\overline \FF_p}, \calL_{\ttS, \ttT}^{(\underline k, w)}(r)\big)[\pi]
\end{equation}
induces an \emph{isomorphism} on  the $\Frob_{p^{2g}}$-eigenspaces  with  eigenvalue $\alpha^{2(d-r)}_\pi \beta^{2r}_\pi (\alpha_\pi \beta_\pi / p^g)^{2\#\ttT}p^{-2gr}$. 
\item[(2')]
Keep the notation in (2)  but assume that $r = \frac d2$ (so $d$ is even) and $(\underline k, w) = \underline 2$.  Suppose that $\alpha_\pi / \beta_\pi$ is not a $2n$-th root of unity for $n \leq \frac d2$.
Then the $\Frob_{p^{2g}}$-invariant subspace of $H^d_\et (\Sh_{K_p}(G_{\ttS, \ttT})_{\overline \FF_p}, \overline \QQ_\ell(\frac d2))[\pi]$ is generated by the cycle classes of  $\Sh_{K_p}(G_{\ttS, \ttT})_\gotha$ for $\gotha \in \gothB_\ttS^{d/2}$.
\end{itemize}
\end{theorem}
\begin{proof}
Statement (1)  follows from the construction of Goren--Oort cycles in Subsection~\ref{S:GO cycles}.
(2') is clearly a special case of (2).
We now focus on the proof of (2).  
By Proposition~\ref{P:cohomology of sh}, the Frobenius semi-simplification of the morphism \eqref{E:sum of Gysin map} is the same as 
\begin{equation}
\label{E:pi component morphism}
\bigoplus_{\gotha \in \gothB_\ttS^r} \rho_{\pi, \gothp}^{\otimes (d - 2r)} \otimes(\det\rho_{\pi, \gothp}(1))^{\otimes(\#\ttT +r)}\longrightarrow\rho_{\pi, \gothp}^{\otimes d} \otimes(\det\rho_{\pi, \gothp}(1))^{\otimes\#\ttT}(r).
\end{equation}
Thus the  generalized eigenspace for the action of $\Frob_{p^{2g}}$ with eigenvalue 
\begin{equation}
\label{E:correct eigenvalue}
\alpha_\pi^{2(d-2r)} (\alpha_\pi \beta_\pi / p^g)^{2(\#\ttT +r)} =\alpha_\pi^{2(d-r)}\beta_\pi^{2r} (\alpha_\pi \beta_\pi / p^g)^{2\#\ttT}p^{-2gr} 
\end{equation} 
has  dimension exactly equal to $\binom dr$ for both sides of \eqref{E:sum of Gysin map}. Thanks to the assumption on the ratio of Satake parameters,   the generalized eigenspace on left hand side  is a genuine eigenspace (since it is the direct sum of $\binom{d}{r}$-copies of one-dimensional generalized eigenspace).  
Thus, the proof of (2) and (2') will be finished if we show that \eqref{E:sum of Gysin map} is injective on the corresponding generalized eigenspace. 

We consider the composition of the Gysin morphisms \eqref{E:sum of Gysin map} with the Restriction morphisms:
\begin{align}
\label{E:composition of Gysin and Restriction}
\bigoplus_{\gothb \in \gothB_\ttS^r} H^{d-2r}_\et(\Sh_{K_p}(G_{\ttS_\gothb, \ttT_\gothb})_{\overline \FF_p},& \calL^{(\underline k, w)}_{\ttS_\gothb, \ttT_\gothb})[\pi] \xrightarrow{\sum \Gys_\gothb}
H^d_\et(\Sh_{K_p}(G_{\ttS,\ttT})_{\overline \FF_p}, \calL^{(\underline k, w)}_{\ttS, \ttT})(r)[\pi] \\
\nonumber&
\xrightarrow{\oplus \Res_\gotha}
\bigoplus_{\gotha \in \gothB_\ttS^r} H^{d-2r}_\et(\Sh_{K_p}(G_{\ttS_\gotha, \ttT_\gotha})_{\overline \FF_p}, \calL^{(\underline k ,w)}_{\ttS_\gotha, \ttT_\gotha})[\pi].
\end{align}
Here, we switched the first sum from over $\gotha$ (as in \eqref{E:sum of Gysin map}) to over $\gothb$.
Taking a basis of  the generalized eigenspace for $\Frob_{p^{2g}}$ acting on \eqref{E:composition of Gysin and Restriction}  with the eigenvalue \eqref{E:correct eigenvalue} and using the description Proposition~\ref{P:cohomology of sh}, we arrive at the following linear map
\begin{equation}
\label{E:Res Gysin simplified}
\bigoplus_{\gothb \in\gothB_\ttS^r} \overline \QQ_\ell \to \bigoplus_{\gotha \in\gothB_\ttS^r} \overline \QQ_\ell
\end{equation}
of vector spaces, which  is represented by a $\binom dr \times \binom dr$-matrix $A$ with coefficients in $\overline \QQ_\ell $. The proof of (2) will be finished if we can show that $\det(A)$ is nonzero.

We explain how this matrix $A$ is related to the Gram matrix $\gothG^r_{\ttS}$ for the periodic semi-meanders (See Theorem~\ref{T:determinant}).
Let $\Lambda$ be the diagonal matrix, whose  $(\gotha, \gotha)$-entry with $\gotha\in \gothB_{\ttS}^r$ is  $p^{-\ell(\gotha)/2}$.  Let   $B$ be the product matrix $\Lambda A\Lambda$. Then dropping  the auxiliary factor $p^{(\ell(\gotha)+\ell(\gothb))/2}$ from formulae in Theorem~\ref{T:intersection combinatorics} gives the entries of $B$. 
We will prove that 
\[
\det B = \begin{cases}
\det \gothG_\ttS^{d/2}|_{T^2 = T_p^\mathrm{n}}, & \textrm{if }r = d/2\\
\det \gothG_\ttS^r|_{v^g = \eta_\univ^\star}, & \textrm{if }r < d/2,
\end{cases}
\]
where $|_{T^2 = T_p^\rmn}$ and $|_{v^g = \eta_\univ^\star}$ are formal substitutions, and $T_p^\rmn$ and $\eta_\univ^\star$ are some formal symbols we define later.

We first compare the entries of $B$ with the entries of $\gothG_\ttS^r$ when $\langle \gotha |\gothb\rangle =0$. In this case, by Theorem~\ref{T:intersection combinatorics}(1), the $\pi$-isotypical component of $\Res_\gotha \circ \Gys_\gothb$ factors through 
\[
H^{d-2(r+1)}_{\et}(\Sh_{K_p}(G_{\ttS',\ttT'})_{\Fpb}, \calL^{(\kb,w)}_{\ttS',\ttT'})(-1)[\pi]
\]
for some quaternionic Shimura variety $\Sh_{K_p}(G_{\ttS',\ttT'})$ of dimension $d-2(r+1)$.
Thanks to the assumption on the ratio of Satake parameters, we see that $\alpha_\pi^{2i} \beta_\pi^{2(d-i)}$ are distinct. The cohomology group above does not contain any generalized $\Frob_{p^{2g}}$-eigenspaces with eigenvalue \eqref{E:correct eigenvalue}.
  Thus the $(\gotha, \gothb)$-entry of $B$ is zero.

We separate the discussion for $r <\frac d2$ and $r = \frac d2$.
First, suppose that $r< \frac d2$.
A subtle point of our argument is that we can not directly identify the matrix $B$ with $\gothS_\ttS^r$ entry by entry,  because there is no canonical choice of basis on each of the factor of \eqref{E:Res Gysin simplified}.
 The proof resembles to that of Theorem~\ref{T:determinant}.
The determinant of $B$ is equal to the sum over all permutations $s$ of the set $\gothB_\ttS^r$, of the product of the signature of $s$, and, for every cycle $(\gotha_1\cdots \gotha_t)$ of the permutation $s$, the product 
\begin{equation}
\label{E:product of gysin and restriction}
p^{-(\ell(\gotha_1) + \cdots + \ell(\gotha_t))} \cdot
(\Res_{\gotha_1} \circ \Gys_{\gotha_t}) \cdot (\Res_{\gotha_t} \circ \Gys_{\gotha_{t-1}}) \cdots 
(\Res_{\gotha_2} \circ \Gys_{\gotha_1}).
\end{equation}
Let $m_0 = m_0(\gotha_1, \gotha_2)+\cdots+ m_0(\gotha_t, \gotha_1)$ be the sum of  total number of contractible loops in the diagrams $D(\gotha_{1},\gotha_2)$, $D(\gotha_2, \gotha_3)$, $\dots$, $D(\gotha_t,\gotha_1)$.
 Then by Theorem~\ref{T:intersection combinatorics}(2), \eqref{E:product of gysin and restriction} is  of the form $(-2)^{m_0}$ times the following composition of link morphisms on the cohomology groups:
\begin{equation}
\label{E:composition of links pullback}
\eta_{\ttS_{ \gotha_1},\ttS_{ \gotha_2}, z( \gotha_1 , \gotha_2)}^\star \circ \eta_{\ttS_{\gotha_2},\ttS_{ \gotha_3}, z(\gotha_2, \gotha_3)}^\star\circ \cdots \circ
\eta_{\ttS_{\gotha_t}, \ttS_{\gotha_1}, z( \gotha_t, \gotha_1)}^\star,
\end{equation}
of shift $\prod_{i=1}^{t-1} (\boldsymbol t_{\gotha_i} \boldsymbol t_{\gotha_{i+1}}^{-1}) \boldsymbol t_{\gotha_{t}} \boldsymbol t_{\gotha_1}^{-1} = \boldsymbol 1$ and indentation degree 
\[
\sum_{i=1}^{t-1} z( \gotha_i,  \gotha_{i+1}) + z( \gotha_t,  \gotha_1) = \begin{cases}
\sum_{i=1}^{t-1} (\ell( \gotha_i)- \ell(\gotha_{i+1})) + \ell( \gotha_n)- \ell(\gotha_1) = 0, & \textrm{ if }\gothp\textrm{ splits in }E/F,
\\
0+\cdots + 0 = 0,& \textrm{ if }\gothp\textrm{ is inert in }E/F.
\end{cases}\]
So this composition \eqref{E:composition of links pullback} is the same link morphism associated to some $n$-th power of the \emph{fundamental link} $\eta_{\ttS_{\gotha_1}}$ for $\ttS_{\gotha_1}$, with trivial shift and indentation degree $0$ (no matter $\gothp$ splits or not in $E/F$).  

Note that the link morphism $(\eta_{\ttS_{\gotha_1}}^n)_{(0)}^\star$ acting on the  one-dimensional $\Frob_{p^{2g}}$-eigenspace  
\begin{equation}\label{E:eigenspace}
\big(H^{d-2r}_{\et}(\Sh_{K_p}(G_{\ttS_{\gotha_1},\ttT_{\gotha_1}}),\calL^{(\underline k, w)}_{\ttS_{\gotha_1},\ttT_{\gotha_1}})[\pi]\big)^{\Frob_{p^{2g}}=\alpha_\pi^{2(d-2r)} (\alpha_\pi \beta_\pi / p^g)^{2(\#\ttT +r)}}
\end{equation} is just the multiplication by  a scalar which we denote by $\lambda_{\gotha_1,n}$. We claim that \emph{$\lambda_{\gotha_1,n}$ does not depend on $\gotha_1\in \gothB_{\ttS}^r$}.
 Indeed, for $\gotha, \gotha' \in \gothB_\ttS^r$ with $\langle \gotha|\gotha'\rangle \neq 0$, Theorem~\ref{T:intersection combinatorics}(2) gives a normalized link morphism 
\[
\eta_{\ttS_\gotha, \ttS_\gotha', (z)}^\star: H_\et^{d-2r}
\big(\Sh_{K_p}(G_{\ttS_{\gotha'},\ttT_{\gotha'}}), \calL_{\ttS_{\gotha'}, \ttT_{\gotha'}}^{(\underline k, w)} \big)\ra \big(\Sh_{K_p}(G_{\ttS_{\gotha},\ttT_{\gotha}}), \calL_{\ttS_{\gotha}, \ttT_{\gotha}}^{(\underline k, w)} \big)
\] 
with some indentation degree $z$ and some shift, then 
\[
(\eta_{\ttS_\gotha}^n)^\star_{(0)} =  \eta_{\ttS_{\gotha},\ttS_{ \gotha'}, (z)}^\star \circ (\eta_{\ttS_{\gotha'}}^n )_{(0)}^\star
\circ ( \eta_{\ttS_{\gotha},\ttS_{ \gotha'},(z)}^\star)^{-1}
\]
provided one of $(\eta_{\ttS_\gotha}^n)^\star_{(0)}$ or $(\eta_{\ttS_{\gotha'}}^n)^\star_{(0)}$ exists.
When this happens, we must have $\lambda_{\gotha, n}=\lambda_{\gotha',n}$.
For general $\gotha$ and $\gotha'$, we can always find $\gotha_1 =\gotha, \dots, \gotha'  = \gotha_t \in \gothB_\ttS^r$ such that $\langle \gotha_i| \gotha_{i+1} \rangle \neq 0$. So if for some $n$ the link morphism  $(\eta_{\ttS_\gotha}^n)^\star$ exists, then it does not depend on $\gotha$. In the sequel, we put $\lambda_n=\lambda_{\gotha,n}$ as long as $(\eta^{n}_{\ttS_\gotha})^{\star}_{(0)}$ exists for some $\gotha\in \gothB_\ttS^r$. The element $\lambda_n$ is clearly multiplicative in $n$.

We can thus introduce the formal symbol $\eta_\univ^\star$ such that $(\eta_\univ^\star)^n = \lambda_{n}$ whenever $(\eta_{\ttS_{\gotha}}^n)^\star_{(0)}$ exists for an integer $n$.
Comparing this computation with $\det \gothG_\ttS^r$ in the proof of Theorem~\ref{T:determinant}, we see that $\det B$ is obtained by  replacing every $v^g$ in $\det \gothG_\ttS^r$ by $\eta_\univ^\star$.
By Theorem~\ref{T:determinant}, this means that
\[
\det B = \pm (\eta_\univ^\star - (\eta_\univ^\star)^{-1})^{2t_{d,r}}.
\]
In particular, $(\eta_{\univ}^\star)^2$ appears in the determinant and hence $(\eta_{\ttS_{\gotha}}^2)^\star_{(0)}$ exists.

Finally, it follows from Proposition~\ref{P:self composition of eta} that
 $(\eta^\star_\univ)^{2(d-2r)}=\lambda_{2(d-2r)}=(\alpha_\pi / \beta_\pi)^{d-2r}$.  
Our assumption implies that $(\alpha_\pi/\beta_\pi)^{d-2r} \neq 1$; so $(\eta_\univ^\star)^2 \neq 1$ and hence $\det B \neq 0$.
This concludes (2).

We now treat the case of $r=\frac d2$. 
Similarly to the discussion above, $\det B$ is equal to the sum over all permutations $s$ of the set $\gothB_\ttS^r$, of the product of the signature of $s$, and, for every cycle $(\gotha_1 \cdots \gotha_t)$ of the permutation $s$, the product \eqref{E:product of gysin and restriction}.
By Theorem~\ref{T:intersection combinatorics}(3), \eqref{E:product of gysin and restriction} in this case is of the form $(-2)^{m_0} \cdot (T_\gothp/p^{g/2})^{m_T}$ times the   link morphism from $(\ttS_{\gotha_1}, \ttT_{\gotha_1})$ to itself with
shift $\prod_{i=1}^{t-1} (\boldsymbol t_{\gotha_i} \boldsymbol t_{\gotha_{i+1}}^{-1}\varpi_{\bar \gothq}^{-m_{T}(\gotha_i, \gotha_{i+1})}) \boldsymbol t_{\gotha_t} \boldsymbol t_{\gotha_1}^{-1}\varpi_{\bar \gothq}^{-m_{T}(\gotha_n, \gotha_1)} =  \varpi_{\bar \gothq}^{-m_{T}}$ and indentation degree
\[
\sum_{i=1}^t \big(\ell(\gotha_i) - \ell(\gotha_{i-1}) - m_{T, i}g \big)= -m_T g.
\]
Here $m_{0} = m_{0}(\gotha_1, \gotha_2) + \cdots+ m_0(\gotha_t, \gotha_1)$ (resp. $m_{T} = m_{T}(\gotha_1, \gotha_2) + \cdots+ m_T(\gotha_t, \gotha_1)$) is the total number of contractible (resp. non-contractible) loops in $D(\gotha_1, \gotha_2), \dots, D(\gotha_t,\gotha_1)$. 
By Example~\ref{Ex:Sq} and the uniqueness of link morphisms (Lemma~\ref{L:uniqueness-link}), this link morphism is equal to the one associated to  $S_\gothq^{-m_T/2}$ with shift $\varpi_{\bar \gothq}^{-m_T}$. This in particular says that $m_T$ is even.
By the second part of Example~\ref{Ex:Sq}, we see that this link morphism is exactly $S_\gothp^{-m_T/2}$.
Therefore, \eqref{E:product of gysin and restriction} is given by
\[
(-2)^{m_0} (T_\gothp/p^{g/2})^{m_T} (S_\gothp)^{-m_T/2}= (-2)^{m_0} \big( (\alpha_\pi + \beta_\pi)^2 / (\alpha_\pi \beta_\pi) \big)^{m_T/2}.
\]
Comparing this with the computation of $\det \gothG_\ttS^r$, we see that $\det B$ is nothing but replacing every $T^2$ by $T_{\gothp}^\rmn \colon = (\alpha_\pi + \beta_\pi)^2 / \alpha_\pi \beta_\pi$.  By Theorem~\ref{T:determinant}, we see that
\[
\det B = \pm \big( (\alpha_\pi + \beta_\pi)^2 / \alpha_\pi \beta_\pi - 4 \big)^{t_{d, d/2}}  = \pm \big( (\alpha_\pi - \beta_\pi) / \alpha_\pi \beta_\pi \big)^{t_{d, d/2}}.
\]
It is nonzero as long as $\alpha_\pi \neq \beta_\pi$.\footnote{Note that we still need $\alpha_\pi / \beta_\pi$ to avoid certain roots of unity to get \eqref{E:Res Gysin simplified}.} This concludes the proof of Theorem~\ref{T:Tate}.
\end{proof}

Before giving a more detailed discussion of the case $\alpha_\pi = \beta_\pi$, we first give some general remarks.
\begin{remark}
\label{Remark after the theorems}
\begin{enumerate}
\item
We discuss the possibility of generalizing this main theorem to the case when $p$ is only assumed to be unramified, namely, $p\calO_F= \gothp_1 \cdots \gothp_h$.
In this case, one can construct the twisted partial Frobenius $\gothF''_{\gothp_i^2}$ for each prime ideal $\gothp_i$ as in \cite[\S 3.22]{tian-xiao1}. Roughly speaking, on the level of moduli space, this is to send the abelian variety $A$ to $A/\Ker_{\gothp^2_i} \otimes_{\calO_F} \gothp_i$, where $\Ker_{\gothp^2_i}$ is the $\gothp_i$-component of the kernel of $\mathrm{Fr}^2: A \to A^{(p^2)}$.

Suppose that one can describe the action of $\gothF''_{\gothp^2}$ on the cohomology of the unitary Shimura variety as in Proposition~\ref{P:self composition of eta}(3), or more precisely   \cite[Conjecture 5.18]{tian-xiao2} holds true.
Then the same argument above can generalize the Theorem in complete generality to the case when $p$ is only assumed to be unramified, an every prime ideal $\gothp_i$ behaves ``in an independent way". More precisely, we fix $r_i \leq \frac {d_i}2$ for all $i$, where $d_i = \#(\ttS_\infty^c \cap \Sigma_{\gothp_i})$ and $\Sigma_{\gothp_i}$ is the subset of $p$-adic embeddings that induce the prime $\gothp_i$.
Then the Goren--Oort cycles would be parameterized by $h$-tuples whose $i$th component is a semi-meander with $d_i$ nodes and $r_i$ arcs. Under the genericity condition: the eigenvalues of $\rho_\pi(\Frob_{\gothp_i})$ avoid certain roots of unity, the cohomology of the  Goren--Oort cycles generate the subspace of the cohomology $H^d_\et(\Sh_{K_p}(G_{\ttS, \ttT})_{\overline \FF_p}, \calL_{\ttS, \ttT}^{(\underline k, w)})[\pi]$ where certain analogs of $\gothF''_{\gothp_i}$ acts with appropriate eigenvalues determined by $r_i$ and $\Frob_{\gothp_i}$.

Without \cite[Conjecture~5.18]{tian-xiao2}, we can only prove the analogous statement when $r_i = \frac{d_i}2$ for $i$, that is, in the case for Tate cycles.\footnote{If $r_i < \frac{d_i}2$ for some $i$, the determinant of the intersection matrix would involve the knowledge of different powers of the action of $\gothF''_{\gothp_i}$. But we only have the information of their product $S_p^{-1} \cdot F^2: = \prod_{i = 1}^h \gothF''_{\gothp_i^2}$.  On the other hand, the case $r_i = \frac{d_i}2$ is fine, because we only use the Hecke operators, whose action on the cohomology is known.}
Moreover, since we can not distinguish the actions of each $\gothF''_{\gothp_i}$, we would have to assume that the eigenvalues of $\rho_\pi(\Frob_{\gothp_i})$ are ``generic", so that all eigenvalues of $\Frob_p^\mathtt{g}$ acting on $H^2_\et(\Sh_{K_p}(G_{\ttS, \ttT})_{\overline \FF_p}, \calL_{\ttS, \ttT}^{(\underline k, w)})$ are ``as distinct as possible", where $\mathtt{g}$ stands for the least common multiple of the inertia degrees of the $\gothp_i$s. For example, this excludes the case when both $\gothp_1$ and $\gothp_2$ have inertia degree $2$ and $\Frob_{\gothp_1}$ and $\Frob_{\gothp_2}$ have the same set of eigenvalues (which would be okay if \cite[Conjecture~5.18]{tian-xiao2} is known).
\item
It would be interesting to know, when $p$ is ramified in $F/\QQ$, whether one can prove a similar result for the special fiber of the splitting model of the Hilbert modular variety of Pappas and Rapoport. The construction of the corresponding Goren--Oort divisors is discussed in \cite{Reduzzi-Xiao}.
\item
The construction of these Goren--Oort cycles uses the CM extension $E$ of $F$. Even though we think these cycles should be independent of the choice of $E$, we do not know how to prove this. 
This auxiliary CM extension is also responsible for avoiding $2n$th roots of unity as opposed to just $n$th roots of unity.
We think these this issues is purely technical, as our current technique relies very much on the PEL moduli interpretation.
\item
In the case of $r=d/2$ (namely the case for Tate classes), the map \eqref{E:sum of Gysin map} is injective as long as $\alpha_\pi \neq \beta_\pi$. We need $\alpha_\pi/\beta_\pi$ to avoid more roots of unities so that both sides of \eqref{E:sum of Gysin map} have the same dimension.

\item
It is attempting to ask the question: to what extend does this imply the semisimplicity of the Frobenius action on $H^{d}(\Sh_{K_p}(G_{\ttS, \ttT})_{\overline \FF_p}, \overline \QQ_{\ell})[\pi]$?
Unfortunately, our theorem is in its strongest form only when $\alpha_\pi \neq \beta_\pi$, where $\rho_\pi(\Frob_{p^g})$ is automatically  semi-simple. Thus if $\otimes\Ind_{\Gal_F}^{\Gal_{\QQ}}\rho_\pi$ is irreducible as a representation of $\Gal_\QQ$, then the $\Gal_{\QQ}$ representation $H^{d}(\Sh_{K_p}(G_{\ttS, \ttT})_{\overline \QQ}, \overline \QQ_{\ell})[\pi]$ is isomorphic to $\otimes\Ind_{\Gal_F}^{\Gal_{\QQ}}\rho_\pi$ up to characters, so that  $\Frob_{p^{g}}$ is  semi-simple. 
However,  $\otimes\Ind_{\Gal_F}^{\Gal_{\QQ}}\rho_\pi$ might be reducible (e.g. when $\pi$ is CM). In this case, our theorem might provide some insight into the semi-simplicity of $H^d_\et(X_{\overline \QQ}, \overline \QQ_\ell)[\pi]$ as a representation of $\Gal_\QQ$. See also \cite{nekovar}.

\item
It is also attempting to ask: in the case of $r = d/2$ (the Tate classes case), is the determinant of the intersection matrix related to the higher derivatives of the local L-function (to get certain local version of the Beilinson--Bloch Conjecture)?
We feel the answer might be negative.
Note that the intersection matrix is always a power of $(\alpha_\pi - \beta_\pi)$, but the higher derivatives of the local L-functions can involve factors of the form $\alpha^s -\beta^s$ for $s<d/2$.
We also point out that, in the recent preprint \cite{YZ} of Z. Yun and W. Zhang, they seem to suggest a new philosophy for higher derivatives of \emph{global} L-functions. We do not know how to compare the determinant of our intersection matrix to their formulation.
\end{enumerate}

\end{remark}

\begin{remark}
\label{Ex:alpha=beta}
It is a very interesting question to understand the case when $\alpha_\pi = \beta_\pi$. We explain this in the quadratic case.
Let $F$ be a real quadratic field in which $p$ is inert.  Let $\pi \in \scrA_{(\underline 2, 2)}$ be an automorphic representation with trivial central character, defined over $\RR$.
Suppose that $\pi$ appears in the cohomology of the quaternionic Shimura variety $X = \Sh_K(G_{\{v_1, v_2\}, \emptyset})$, where $v_1$ and $v_2$ are two finite prime-to-$p$ places of $F$ (so that $X$ is proper for simplicity). Suppose that the unramified  $\alpha_\pi = \beta_\pi = \pm p$.
For instance, when $\pi$ comes from the base change of a usual  modular form  corresponding to an elliptic curve  over $\Q$ which has  supersingular  (good) reduction at $p$, then the local Satake parameters  of  $\pi$ at $p$ are $\alpha_\pi = \beta_\pi =  p$.

In this case, $H^2_\et(X_{\overline \FF_p}, \overline \QQ_\ell(1))[\pi]$ is $4$-dimensional on which $\Frob_{p^2}$ acts trivially.
More precisely, as pointed out by Prasanna, the action of $\Frob_p$ on this four dimensional subspace has two eigenvalues: $\alpha_\pi/p$ (with multiplicity $3$) and $-\alpha_\pi/p$ (with multiplicity $1$).
There are two Goren--Oort cycles, both given by a collection of $\PP^1$'s. The $\pi$-isotypical components of their cycle classes contribute non-trivially to the subspace with $\Frob_p$-eigenvalue $-\alpha_\pi/p$.
We claim that the $\pi$-isotypical component of the their cycles classes does not contribute to the subspace with $\Frob_p$-eigenvalue $\alpha_\pi/p$.
Indeed, the intersection matrix $B$ given above is degenerate (having rank $1$). Note that, for any cuspidal $\pi$, the $\pi$-isotypical component of the rational N\'eron--Severi group of $X$ is orthogonal to the subspace of ample line bundles, and  Hodge index theorem implies that the intersection pairing on the $\pi$-isotypical component is non-degenerate.
So the degeneracy of the intersection matrix means that the contribution from the Goren--Oort cycles is indeed a one-dimensional subspace of $H^2_\et(X_{\overline \FF_p}, \overline \QQ_\ell(1))[\pi]$, namely the subspace with $\Frob_p$-eigenvalue $-\alpha_\pi/p$.

We think this phenomenon is comparable to the case of Heegner points: when the rank of the elliptic curve is one (``generic rank''), the Heegner point gives a canonical generator of the Mordell--Weil group tensored with $\QQ$; however, when the rank of the elliptic curve is strictly bigger than one (``generic rank''),  the Heegner point becomes torsion.
In our case, the classes of the Goren--Oort cycles are similar to Heegner points. When the dimension of the corresponding Frobenius (generalized) eigenspace is ``generic", the classes of the Goren--Oort cycles give a canonical basis, but when the dimension is bigger than the generic one, the contribution from the  Goren--Oort cycles tends to degenerate.
\end{remark}

\section{Computation of the intersection matrix}
The aim of this section is to establish Theorem~\ref{T:intersection combinatorics} and hence to finish the proof of the main theorems.
We keep the notation from the previous section.

\begin{notation}
To simplify notation, we suppress the automorphic sheaf $\calL_{\ttS, \ttT}^{(\underline k, w)}$, the level structure ${K_p}$, the change of base to $\overline \FF_p$, and the subscript et from the notation of cohomology groups, as they are all fixed throughout this section. 
For example, we write 
\[
\textrm{\ }
H^\star(\Sh(G_{\ttS, \ttT})_\gotha)(r) \textrm{ \  for \ }H^\star_\et\big(\Sh_{K_p}(G_{\ttS, \ttT})_{\gotha,\overline \FF_p}, \calL_{\ttS, \ttT}^{(\underline k, w)}(r)|_{\Sh_{K_p}(G_{\ttS, \ttT})_\gotha}\big).
\]
This should not cause any confusion because all the automorphic sheaves are compatible on the Goren--Oort cycles.
As in Theorem~\ref{T:intersection combinatorics}, we fix a choice of system of shifts $\boldsymbol t_{\gotha}$ of the correspondences $\Sh_{K_p}(G_{\ttS_\gotha, \ttT_\gotha}) \xleftarrow{\pi_\gotha}\Sh_{K_p}(G_{\ttS, \ttT})_\gotha \hookrightarrow \Sh_{K_p}(G_{\ttS, \ttT})$ as in Subsection~\ref{S:GO cycles}.
\end{notation}

Before going into the intricate induction, we first handle a few simple but essential cases. The general case will be essentially reduced to these cases.

\subsection{The case of $r=1$ and $\gotha = \gothb$}
\label{S:case r=1 a=b}
This is the case where the corresponding periodic semi-meanders are given as 
\[
\gotha = \gothb = \psset{unit=0.3}
\begin{pspicture*}(-9,-1.6)(14,2)
\psset{linewidth=1pt}
\psset{linecolor=red}
\psbezier{-}(0,0)(0,2)(4.8, 2)(4.8,0)
\psline{-}(9.6,0)(9.6,2)
\psline{-}(-4.8,0)(-4.8,2)
\psset{linecolor=black}
\psdots(0,0)(4.8,0)(-4.8,0)(9.6,0)
\psdots[dotstyle=+](1,0)(3.8,0)(-1,0)(5.8,0)(-3.8,0)(-5.8,0)(8.6,0)(10.6,0)
\psset{linewidth=.1pt}
\psdots(1.7,0)(2.4,0)(3.1,0)(-1.7,0)(-2.4,0)(-3.1,0)(6.5,0)(7.2,0)(7.9,0)(-6.5,0)(-7.2,0)(-7.9,0)(11.3,0)(12,0)(12.7,0)
\rput(0,-1){\psframebox*{\begin{tiny}$\tau^{-}$\end{tiny}}}
\rput(4.8,-1.2){\psframebox*{\begin{tiny}$\tau$\end{tiny}}}
\end{pspicture*}
\]
(or their shifts), linking $\tau$ with $\tau^- = \sigma^{-n_\tau} \tau$. 

Unwinding the definition, we have the following commutative diagram
\[
\xymatrix{
H^{d-2}(\Sh(G_{\ttS_\gotha, \ttT_\gotha})) \ar[rr]^{\Res_\gotha \circ \Gys_\gotha}
\ar[d]^{\pi_\gotha^*}
&& H^{d-2}(\Sh(G_{\ttS_\gotha, \ttT_\gotha}))
\\
H^{d-2}(\Sh(G_{\ttS, \ttT})_\gotha) \ar[r]^{\textrm{Gysin}}
&
H^d(\Sh(G_{\ttS, \ttT}))(1)
\ar[r]^{\textrm{Restr.}}
&
H^d(\Sh(G_{\ttS, \ttT})_\gotha)(1)
\ar[u]^{\pi_{\gotha,!}}.
}
\]
Recall that $\Sh(G_{\ttS,\ttT})_{\gotha}$ is a $\PP^1$-bundle over $\Sh(G_{\ttS_{\gotha},\ttT_{\gotha}})$, hence $\pi_{\gotha}^*$ and $\pi_{\gotha,!}$ are both isomorphisms. 
By the excessive intersection formula \cite[\S 6.3]{fulton}, the composition of the bottom line is given by  the cup product with the first Chern class of the normal bundle of the embedding 
$\Sh(G_{\ttS, \ttT})_\gotha \hookrightarrow\Sh(G_{\ttS, \ttT})$, which  is isomorphic to  $-2p^{n_\tau}$ times the canonical quotient ample line  bundle for the $\PP^1$-bundle given by $\pi_\gotha$, according to  Proposition~\ref{P:GO-fibration}(2).
Therefore, the top line morphism $\Res_\gotha\circ \Gys_\gotha$ is nothing but multiplication by $-2p^{n_\tau} = -2p^{\ell(\gotha)}$.

\subsection{The case of $d=2$ and $r=1$ with $\gotha \neq \gothb$}
\label{S:case of g=2 Tp}
This is the case where the corresponding periodic semi-meanders are given as 
\[
\gotha = \psset{unit=0.3}\begin{pspicture*}(-2.6,-1.6)(8.1,2)
\psset{linewidth=1pt}
\psset{linecolor=red}
\psbezier{-}(0,0)(0,2)(4.8, 2)(4.8,0)
\psset{linecolor=black}
\psdots(0,0)(4.8,0)
\psdots[dotstyle=+](1,0)(3.8,0)(-1,0)(5.8,0)
\psset{linewidth=.1pt}
\psdots(1.7,0)(2.4,0)(3.1,0)(-1.7,0)(-2.4,0)(-3.1,0)(6.5,0)(7.2,0)(7.9,0)
\rput(0,-1){\psframebox*{\begin{tiny}$\tau^{-}$\end{tiny}}}
\rput(4.8,-1.2){\psframebox*{\begin{tiny}$\tau$\end{tiny}}}
\end{pspicture*} \quad \textrm{and} \quad
\gothb =\begin{pspicture*}(-2.6,-1.6)(8.1,2)
\psset{linewidth=1pt}
\psset{linecolor=red}
\psbezier{-}(0,0)(0,2)(-6, 2)(-6,0)
\psbezier{-}(10.8,0)(10.8,2)(4.8, 2)(4.8,0)
\psset{linecolor=black}
\psdots(0,0)(4.8,0)
\psdots[dotstyle=+](1,0)(3.8,0)(-1,0)(5.8,0)
\psset{linewidth=.1pt}
\psdots(1.7,0)(2.4,0)(3.1,0)(-1.7,0)(-2.4,0)(-3.1,0)(6.5,0)(7.2,0)(7.9,0)
\rput(0,-1){\psframebox*{\begin{tiny}$\tau^{-}$\end{tiny}}}
\rput(4.8,-1.2){\psframebox*{\begin{tiny}$\tau$\end{tiny}}}
\end{pspicture*}
\]
(or their  simultaneous shifts).
Let $\tau^-$ denote the left end-node of the arc of $\gotha$, and $\tau$ is the right end-node. We have $\tau^+ = \tau^-$. Here the meaning of left and right refers to the $xy$-plane presentation of $\gotha$, as explained in \ref{S:semi-meanders}.

Unwinding the definition, the morphism $\Res_\gotha \circ \Gys_\gothb$ is the composition of the following commutative diagram from the upper-left to the lower-right (first rightward and then downward):
\[
\xymatrix{
H^0(\Sh(G_{\ttS_\gothb, \ttT_\gothb})) \ar[r]^{\pi_\gothb^*}
\ar[dr]
&
H^0(\Sh(G_{\ttS, \ttT})_{\tau^-})
\ar[r]^{\textrm{Gysin}}
\ar[d]^{\textrm{Restr.}}
& 
H^2(\Sh(G_{\ttS, \ttT}))(1)
\ar[d]^{\textrm{Restr.}}
\\
&
H^0(\Sh(G_{\ttS, \ttT})_{\{\tau^-,\tau\}}) \ar[r]^-{\textrm{Gysin}}
\ar[dr]_-{\Tr_{\pi_\gotha}}
&
H^2(\Sh(G_{\ttS, \ttT})_{\tau})(1)
\ar[d]^{\pi_{\gotha,!}}
\\
&&
H^0(\Sh(G_{\ttS_\gotha, \ttT_\gotha})).
}
\]
Here, the commutativity of the square follows from \eqref{E:gysin-restriction} the fact that   $\Sh(G_{\ttS,\ttT})_{\{\tau, \tau^-\}}$ is the transversal intersection of $\Sh(G_{\ttS,\ttT})_{\tau^-}$ and $\Sh(G_{\ttS,\ttT})_{\tau}$ in $\Sh(G_{\ttS,\ttT})$, 
 and  $\Tr_{\pi_\gotha}$ is the trace map   induced by  the finite \'etale map (of zero-dimensional Shimura varieties)
\[
\Sh(G_{\ttS,\ttT})_{\{\tau^-,\tau\}}\hra \Sh(G_{\ttS,\ttT})_{\tau}\xra{\pi_{\gotha}}\Sh(G_{\ttS_{\gotha},\ttT_{\gotha}}),
\]
and the natural isomorphism  between  the pullback of the  automorphic sheaf on $\Sh(G_{\ttS_\gotha, \ttT_{\gotha}})$ with that on $\Sh(G_{\ttS,\ttT})_{\tau}$.

By Theorem~\ref{T:GO description}(3), the diagonal composition from upper-left to lower-right, or equivalently the morphism $\Res_\gotha\circ \Gys_\gothb$,  is $T_\gothp \circ (\eta^\star_{\ttS_\gotha,\ttS_\gothb,(n)})^{-1}$, where $\eta^\star_{\ttS_\gotha,\ttS_\gothb,(n)}$ is the link morphism associated to the trivial link $\eta_{\ttS_\gothb,\ttS_\gotha}: \ttS_\gothb \to \ttS_\gotha$ with indentation degree $n =2n_{\tau^-} = -(\ell(\gotha) - \ell(\gothb) - g)$ and shift $\varpi_{\bar \gothq} \boldsymbol t_\gotha^{-1} \boldsymbol t_\gothb$.
Thus the inverse $(\eta^\star_{\ttS_\gothb,\ttS_\gotha,(n)})^{-1} = (\eta^{-1}_{\ttS_\gotha,\ttS_\gothb,(-n)})^\star$ is the link morphism associated to the link $\eta^{-1}$ with indention degree $\ell(\gotha) - \ell(\gothb) - g$ and shift $\varpi_{\bar \gothq}^{-1} \boldsymbol t_\gotha \boldsymbol t_\gothb^{-1}$.
This proves Theorem~\ref{T:intersection combinatorics}(3) for the given case.

\subsection{The case of $r=1$, $d>2$, and $\langle \gotha, \gothb\rangle = v^{m_v}$}
\label{S:case r=1 link}
Assume that $m_v >0$ first.
In this situation, the corresponding periodic semi-meanders, up to shifting, are given by 
\begin{eqnarray}
\label{E:a and b when r=1}
&&\gotha = \psset{unit=0.3}\begin{pspicture*}(-8.5,-1.6)(18,2)
\psset{linewidth=1pt}
\psset{linecolor=red}
\psbezier{-}(9.6, 0)(9.6, 2)(4.8, 2)(4.8,0)
\psline{-}(0,0)(0,2)
\psline{-}(-4.8,0)(-4.8,2)
\psline{-}(14.4,0)(14.4,2)
\psset{linecolor=black}
\psdots(0,0)(4.8,0)(9.6,0)(-4.8,0)(14.4,0)
\psdots[dotstyle=+](1,0)(3.8,0)(5.8,0)(8.6,0)(-1,0)(-3.8,0)(-5.8,0)(10.6,0)(15.4,0)(13.4,0)
\psset{linewidth=.1pt}
\psdots(1.7,0)(2.4,0)(3.1,0)(6.5,0)(7.2,0)(7.9,0)(-1.7,0)(-2.4,0)(-3.1,0)(-6.5,0)(-7.2,0)(-7.9,0)(11.3,0)(12,0)(12.7,0)(16.1,0)(16.8,0)(17.5,0)
\rput(0,-1){\psframebox*{\begin{tiny}$\tau^{-}$\end{tiny}}}
\rput(4.8,-1.2){\psframebox*{\begin{tiny}$\tau$\end{tiny}}}
\rput(9.6,-1){\psframebox*{\begin{tiny}$\tau^{+}$\end{tiny}}}
\end{pspicture*}
\quad \textrm{and}
\\
\nonumber &&
\psset{unit=0.3}
\gothb = 
\begin{pspicture*}(-8.5,-1.6)(18,2)
\psset{linewidth=1pt}
\psset{linecolor=red}
\psbezier{-}(0,0)(0,2)(4.8, 2)(4.8,0)
\psline{-}(-4.8,0)(-4.8,2)
\psline{-}(14.4,0)(14.4,2)
\psline{-}(9.6, 0)(9.6, 2)
\psset{linecolor=black}
\psdots(0,0)(4.8,0)(9.6,0)(-4.8,0)(14.4,0)
\psdots[dotstyle=+](1,0)(3.8,0)(5.8,0)(8.6,0)(-1,0)(-3.8,0)(-5.8,0)(10.6,0)(15.4,0)(13.4,0)
\psset{linewidth=.1pt}
\psdots(1.7,0)(2.4,0)(3.1,0)(6.5,0)(7.2,0)(7.9,0)(-1.7,0)(-2.4,0)(-3.1,0)(-6.5,0)(-7.2,0)(-7.9,0)(11.3,0)(12,0)(12.7,0)(16.1,0)(16.8,0)(17.5,0)\rput(0,-1){\psframebox*{\begin{tiny}$\tau^{-}$\end{tiny}}}
\rput(4.8,-1.2){\psframebox*{\begin{tiny}$\tau$\end{tiny}}}
\rput(9.6,-1){\psframebox*{\begin{tiny}$\tau^{+}$\end{tiny}}}
\end{pspicture*}.
\end{eqnarray}
Note that the two arcs in $\gotha$ and $\gothb$ must be adjacent, otherwise, $\langle\gotha, \gothb\rangle =0$.
Let $\tau$ denote the left end-node of the (unique) arc in $\gotha$, as shown in the pictures above. Then $\tau^-$ is the left end-node of the arc in $\gothb$, and $\tau^+$ is the right end-node of the arc in $\gotha$.
So if $\tau = \sigma^{-n_{\tau^+}} \tau^+$ and $\tau^- = \sigma^{-n_\tau} \tau$, then $m_v = m_v(\gotha, \gothb) = n_\tau  + n_{\tau^+}$.

Unwinding the definition, the morphism $\Res_\gotha \circ \Gys_\gothb$ is the composition of the following commutative diagram from the upper-left to the lower-right:
\begin{equation}
\label{E:CD r=1 g>2}
\xymatrix{
H^{d-2}(\Sh(G_{\ttS_\gothb, \ttT_\gothb})) \ar[r]^{\pi_\gothb^*}
\ar[dr]^{\pi_\gothb^*}
&
H^{d-2}(\Sh(G_{\ttS, \ttT})_{\tau})
\ar[r]^{\textrm{Gysin}}
\ar[d]^{\textrm{Restr.}}
& 
H^d(\Sh(G_{\ttS, \ttT}))(1)
\ar[d]^{\textrm{Restr.}}
\\
&
H^{d-2}(\Sh(G_{\ttS, \ttT})_{\{\tau, \tau^+\}}) \ar[r]^-{\textrm{Gysin}}\ar[dr]_-{ (\theta^{-1})^*,\cong }
&
H^d(\Sh(G_{\ttS, \ttT})_{\tau^+})(1)
\ar[d]^{\pi_{\gotha,!}}
\\
&&
H^{d-2}(\Sh(G_{\ttS_\gotha, \ttT_\gotha})).
}
\end{equation}
By Theorem~\ref{T:GO description}(2), the morphism
$$
\theta: \Sh(G_{\ttS, \ttT})_{\{\tau, \tau^+\}} \hra \Sh(G_{\ttS,\ttT})_{\tau^+}\xra{\pi_{\gotha}} \Sh(G_{\ttS_\gotha, \ttT_\gotha})
$$
 is an isomorphism, and the composition
\[
\Sh(G_{\ttS_\gotha, \ttT_\gotha}) \xleftarrow{\ \theta^{-1}\ }
\Sh(G_{\ttS, \ttT})_{\{\tau, \tau^+\}}\hra \Sh(G_{\ttS,\ttT})_{\tau} \xrightarrow{\ \pi_\gothb\ } \Sh(G_{\ttS_\gothb, \ttT_\gothb})
\]
is exactly the link  morphism
\[
\eta_{\gotha,\gothb, (z), \sharp}: \Sh(G_{\ttS_\gotha, \ttT_\gotha}) \longto \Sh(G_{\ttS_\gothb, \ttT_\gothb}),
\]
associated to the link $\eta_{\gotha,\gothb}\colon \ttS_{\gotha}\ra \ttS_{\gothb}$ given by 
\begin{equation}
\label{E:eta a to b when r=1}
\psset{unit=0.3}
\begin{pspicture*}(-3.5,-0.2)(23,4)
\psset{linecolor=red}
\psset{linewidth=1pt}
\psline{-}(0,0)(0,2)
\psline{-}(19.2,0)(19.2,2)
\psbezier(4.8,2)(4.8,0)(14.4,2)(14.4,0)
\psset{linecolor=black}
\psdots(0,0)(0,2)(4.8,2)(14.4,0)(19.2,0)(19.2,2)
\psdots[dotstyle=+](1,0)(-1,0)(-1,2)(1,2)(3.8,0)
(3.8,2)(4.8,0)(5.8,0)(5.8,2)(8.6,0)(8.6,2)(10.6,0)(9.6,0)(9.6,2)(10.6,2)(13.4,0)(13.4,2)(15.4,0)(14.4,2)(15.4,2)(18.2,0)(18.2,2)(20.2,0)(20.2,2)
\psset{linewidth=.1pt}
\psdots(1.7,0)(1.7,2)(2.4,0)(2.4,2)(3.1,0)(3.1,2)(6.5,0)(7.2,0)(7.9,0)(6.5,2)(7.2,2)(7.9,2)(11.3,0)(12,0)(12.7,0)(11.3,2)(12,2)(12.7,2)(16.1,0)(16.8,0)(17.5,0)(16.1,2)(16.8,2)(17.5,2)(-1.7,0)(-1.7,2)(-2.4,0)(-2.4,2)(-3.1,0)(-3.1,2)(20.9,0)(21.6,0)(22.3,0)(20.9,2)(21.6,2)(22.3,2)
\rput(5,3.3){\psframebox*{\begin{tiny}$\tau^{-}$\end{tiny}}}
\rput(9.6,3.1){\psframebox*{\begin{tiny}$\tau$\end{tiny}}}
\rput(14.5,3.3){\psframebox*{\begin{tiny}$\tau^{+}$\end{tiny}}}
\end{pspicture*},
\end{equation}
with shift $\boldsymbol t_\gotha \boldsymbol t_\gothb^{-1}$ and indentation degree $z$ equal to $\ell(\gotha) -\ell(\gothb)$ if $\gothp$ splits in $E/F$ and to $0$ if $\gothp$ is inert in $E/F$.
Therefore, $\Res_\gotha\circ \Gys_\gothb$ is exactly $p^{v(\eta_{\gotha, \gothb})/2} \eta_{\gotha, \gothb, (z)}^{\star} = p^{m_v/2}\eta_{\gotha, \gothb, (z)}^\star$ (note the normalization in \eqref{E:normalized-link}), verifying Theorem~\ref{T:intersection combinatorics}(2). 

We now come to the case where $m_v$ is negative.
In this case, the picture of $\gotha$ and $\gothb$ in  \eqref{E:a and b when r=1} are swapped.
 Then we have a commutative  diagram similar to \eqref{E:CD r=1 g>2}:
 \[
 \xymatrix{
H^{d-2}(\Sh(G_{\ttS_\gothb, \ttT_\gothb})) \ar[r]^{\pi_\gothb^*}
\ar[dr]^{ \cong}
&
H^{d-2}(\Sh(G_{\ttS, \ttT})_{\tau^+})
\ar[r]^{\textrm{Gysin}}
\ar[d]^{\textrm{Restr.}}
& 
H^d(\Sh(G_{\ttS, \ttT}))(1)
\ar[d]^{\textrm{Restr.}}
\\
&
H^{d-2}(\Sh(G_{\ttS, \ttT})_{\{\tau, \tau^+\}}) \ar[r]^-{\textrm{Gysin}}\ar[dr]
&
H^d(\Sh(G_{\ttS, \ttT})_{\tau})(1)
\ar[d]^{\pi_{\gotha,!}}
\\
&&
H^{d-2}(\Sh(G_{\ttS_\gotha, \ttT_\gotha})),
}
 \]
and the composed diagonal morphism gives $\Res_{\gotha}\circ\Gys_{\gothb}$.
Let $\eta_{\gothb,\gotha}: \ttS_{\gothb}\ra \ttS_{\gotha}$ denote the inverse link of $\eta_{\gotha,\gothb}$.
Since $\gotha$ and $\gothb$ are obtained by swapping with each other from the previous case, the link morphism $\eta_{\gothb, \gotha, (-z), \sharp}: \Sh(G_{\ttS_{\gothb}, \ttT_{\gothb}})\ra \Sh(G_{\ttS_{\gotha}, \ttT_{\gotha}})$ with shift $\boldsymbol t_{\gothb}\boldsymbol t_{\gotha}^{-1}$ exists, where $z=\ell(\gotha)-\ell(\gothb)$ if $\gothp$ splits in $E$ and $z=0$ if $\gothp$ is inert in $E$.
Note also that $\eta_{\gothb, \gotha, (-z), \sharp}$ is finite flat of degree $p^{-m_{v}}=p^{v(\eta_{\gothb, \gotha})}$  by Theorem~\ref{T:GO description}. 
One sees easily that  $\Res_{\gotha}\circ\Gys_{\gothb}=\Tr_{\eta_{\gothb, \gotha, (-z),\sharp}}$. 
   By Lemma~\ref{L:inverse of link morphism}(3), this is exactly   $p^{-m_{v}/2} (\eta_{\gothb, \gotha, (-z)}^\star)^{-1}=p^{(\ell(\gotha)+\ell(\gothb))/2}\eta^{\star}_{\gotha,\gothb, (z)}$.   This proves Theorem~\ref{T:intersection combinatorics}(2) in this case.

\subsection{Decomposition of periodic semi-meanders}
\label{SS:decomposition of semi-meanders}
Before proceeding to the inductive proof, we discuss certain ways to ``decompose" periodic semi-meanders appearing in the induction.
Let $\gotha \in \gothB_\ttS^r$ be a periodic semi-meander.  
We call a subset $\Delta$  of $r'$ arcs ($r' \leq r$) in $\gotha$ \emph{saturated}, if for each arc $\delta$ belonging to $ \Delta$, any arc that lies below $\delta$ in the sense of Notation~\ref{N:basic arc} belongs to $\Delta$. 
For example, if 
$\psset{unit=0.3}
\gotha =
\begin{pspicture*}(-.5,-0.2)(9.5,3)
\psset{linewidth=1pt}
\psset{linecolor=red}
\psline{-}(0,0)(0,3)
\psline{-}(9,0)(9,3)
\psbezier{-}(1,0)(1,3.5)(8,3.5)(8,0)
\psbezier{-}(2,0)(2,1.5)(5,1.5)(5,0)
\psarc{-}(3.5,0){0.5}{0}{180}
\psarc{-}(6.5,0){0.5}{0}{180}
\psset{linecolor=black}
\psdots(0,0)(1,0)(2,0)(3,0)(4,0)(9,0)(8,0)(7,0)(6,0)(5,0)
\end{pspicture*}$, the subset $\Delta=\psset{unit=0.3}
\begin{pspicture*}(-.5,-0.2)(9.5,1.6)
\psset{linewidth=1pt}
\psset{linecolor=red}
\psbezier{-}(2,0)(2,1.5)(5,1.5)(5,0)
\psarc{-}(3.5,0){0.5}{0}{180}
\psarc{-}(6.5,0){0.5}{0}{180}
\psset{linecolor=black}
\psdots(0,0)(1,0)(2,0)(3,0)(4,0)(9,0)(8,0)(7,0)(6,0)(5,0)
\end{pspicture*}$\footnote{Here $\Delta$ is only the set of the arcs, not including the nodes in the picture.} is saturated, but  
$\psset{unit=0.3}
\begin{pspicture*}(-.5,-0.2)(9.5,1.6)
\psset{linewidth=1pt}
\psset{linecolor=red}
\psbezier{-}(2,0)(2,1.5)(5,1.5)(5,0)
\psarc{-}(6.5,0){0.5}{0}{180}
\psset{linecolor=black}
\psdots(0,0)(1,0)(2,0)(3,0)(4,0)(9,0)(8,0)(7,0)(6,0)(5,0)
\end{pspicture*}$ 
is not.

Now fix a saturated $\Delta$. We use $\gotha^\flat$ to denote the periodic semi-meander for $\ttS$ given by all the arcs in $\Delta$ and then adjoining semi-lines to the rest of  nodes.
Then  $\ttS_{\gotha^\flat}$ is the union of $\ttS$ and all nodes connected to an arc in $\Delta$.
We use $\gotha_\res = \gotha\backslash \Delta$ to denote the periodic semi-meander for $\ttS_{\gotha^\flat}$ obtained by removing all the arcs in $\Delta$ and replacing their end-nodes by plus signs.  In the example above, $ \gotha^\flat = \psset{unit=0.3}
\begin{pspicture*}(-.5,-0.2)(9.5,1.6)
\psset{linewidth=1pt}
\psset{linecolor=red}
\psline(0,0)(0,1.4)
\psline(1,0)(1,1.4)
\psline(8,0)(8,1.4)
\psline(9,0)(9,1.4)
\psbezier{-}(2,0)(2,1.5)(5,1.5)(5,0)
\psarc{-}(3.5,0){0.5}{0}{180}
\psarc{-}(6.5,0){0.5}{0}{180}
\psset{linecolor=black}
\psdots(0,0)(1,0)(2,0)(3,0)(4,0)(9,0)(8,0)(7,0)(6,0)(5,0)
\end{pspicture*}$
and
$\gotha_\res = \psset{unit=0.3}
\begin{pspicture*}(-.5,-0.2)(9.5,3)
\psset{linewidth=1pt}
\psset{linecolor=red}
\psline{-}(0,0)(0,3)
\psline{-}(9,0)(9,3)
\psbezier{-}(1,0)(1,3.5)(8,3.5)(8,0)
\psset{linecolor=black}
\psdots(0,0)(1,0)(9,0)(8,0)
\psdots[dotstyle=+](2,0)(3,0)(4,0)(7,0)(6,0)(5,0)
\end{pspicture*}$, where the plus signs indicates points corresponding to $\ttS_{\gotha^\flat, \infty}$.

By the construction of the Goren--Oort cycles, we have the following commutative diagram, where the middle square is Cartesian.
\begin{equation}
\label{E:factor of gothb}
\xymatrix{
\Sh(G_{\ttS,\ttT})_{\gotha} \ar@{^{(}->}[r] \ar[d] 
\ar@/_50pt/[dd]_{\pi_\gotha}
& 
\Sh(G_{\ttS,\ttT})_{\gotha^\flat} \ar@{^{(}->}[r]
\ar[d]^{\pi_{\gotha^\flat}}
& \Sh(G_{\ttS,\ttT})
\\
\Sh(G_{\ttS_{\gotha^\flat},\ttT_{\gotha^\flat}})_{\gotha_\res} \ar@{^{(}->}[r]
\ar[d]^{\pi_{\gotha_\res}}
&
\Sh(G_{\ttS_{\gotha^\flat},\ttT_{\gotha^\flat}})
\\
\Sh(G_{\ttS_{\gotha},\ttT_{\gotha}})
}
\end{equation}
Since the construction of this diagram comes from the unitary Shimura varieties, we point out that, as explained near \eqref{E:correspondence ab}, the shift of the correspondence
$\Sh(G_{\ttS_\gotha, \ttT_\gotha}) \xleftarrow{ \pi_{\gotha_\res}} \Sh(G_{\ttS_{\gotha^\flat}, \ttT_{\gotha^\flat}})_{\gotha_\res} \hookrightarrow \Sh(G_{\ttS_{\gotha^\flat}, \ttT_{\gotha^\flat}})$ is $\boldsymbol t_{\gotha^\flat, \gotha} = \boldsymbol{t}_{\gotha} \boldsymbol t_{\gotha^\flat}^{-1}$.
From the commutative diagram, we can decompose the morphisms $\Res_\gotha$ and $\Gys_\gotha$ as follows:
\begin{align*}
\Gys_\gotha: H^{d-2r}(\Sh(G_{\ttS_\gotha, \ttT_\gotha}))& \xrightarrow{\pi_{\gotha^*_\res}} H^{d-2r}(\Sh(G_{\ttS_{\gotha^\flat}, \ttT_{\gotha^\flat}})_{\gotha_\res})
\xrightarrow{\textrm{Gysin}}
H^{d-2r'}(\Sh(G_{\ttS_{\gotha^\flat}, \ttT_{\gotha^\flat}}))(r-r')
\\
&
\xrightarrow{\pi_{\gotha^\flat}^*}
H^{d-2r'}(\Sh(G_{\ttS, \ttT})_{\gotha^\flat})(r-r')
\xrightarrow{\textrm{Gysin}}
H^d(\Sh(G_{\ttS,\ttT}))(r)
\end{align*}
and 

\begin{align*}
\Res_\gotha:
H^d(\Sh(G_{\ttS,\ttT}))(r)
&
\xrightarrow{\textrm{Restr.}}
H^d(\Sh(G_{\ttS, \ttT})_{\gotha^\flat})(r)
\xrightarrow{\pi_{\gotha^\flat, !}}
H^{d-2r'}(\Sh(G_{\ttS_{\gotha^\flat}, \ttT_{\gotha^\flat}}))(r-r')
\\
&
\xrightarrow{\textrm{Restr.}}
H^{d-2r'}(\Sh(G_{\ttS_{\gotha^\flat}, \ttT_{\gotha^\flat}})_{\gotha_\res})(r-r')
\xrightarrow{\pi_{\gotha_\res, !}}
H^{d-2r}(\Sh(G_{\ttS_\gotha, \ttT_\gotha})).
\end{align*}
Here, to get the decomposition for $\Res_\gotha$, we  have used the fact that  the trace map $\Tr_{\pi_{\gotha}}$  can be factorized as 
\[
R^{2r}\pi_{\gotha*}\overline \QQ_\ell(r) \cong
R^{2r-2r'}\pi_{\gotha_\res*} \big(R^{2r'}\pi_{\gotha^\flat*}\overline\QQ_\ell\big)(r)\xra{\Tr_{\pi_{\gotha^\flat}}}R^{2r-2r'}\pi_{\gotha_\res*}(\overline\QQ_\ell)(r-r')\xra{\Tr_{\pi_{\gotha_{\res}}}}\overline\QQ_\ell.
\]
Summing up everything in short, we obtain thus 
\[
\Gys_\gotha = \Gys_{\gotha^\flat} \circ \Gys_{\gotha_\res} , \quad
\Res_\gotha = \Res_{\gotha_\res} \circ \Res_{\gotha^\flat}, \quad \textrm{and } \boldsymbol t_\gotha = \boldsymbol t_{\gotha^\flat} \boldsymbol t_{\gotha^\flat, \gotha}.
\]

We will apply this to appropriate $\Delta$'s to reduce the calculation to $\Sh(G_{\ttS_{\gotha^\flat}, \ttT_{\gotha^\flat}})$ and  reduce the inductive proof essentially to the cases considered above.

\subsection{Decomposition of periodic semi-meanders continued}
\label{SS:decomposition of semi-meanders contd}
We will also encounter the following situation:  assume that the set of arcs in a periodic semi-meander $\gotha$ is the disjoint union of two \emph{saturated} subsets $\Delta$ and $\Delta'$. Put $s = \#\Delta$ and $s' = \#\Delta'$ so that $r=s+s'$.
We will show that $\Delta$ and $\Delta'$ ``behave" independently.

We write $\gotha^\flat$ (resp. $\gotha^{\flat \prime}$) for the periodic semi-meander for $\ttS$ given by all arcs in $\Delta$ (resp. $\Delta'$) and then adjoining semi-lines to the rest of the nodes.
We put $\gotha_\res$ (resp. $\gotha'_\res$) for the periodic semi-meander for $\ttS_{\gotha^\flat}$ (resp. $\ttS_{\gotha^{\flat\prime}}$) obtained by removing all arcs in $\Delta$ (resp. $\Delta'$) and replacing all their end-nodes by plus signs.

In this case, in view of the construction of the Goren--Oort cycle $\Sh(G_{\ttS, \ttT})_\gotha$, we could either go through the arcs in $\Delta$ first, or the arcs in $\Delta'$ first.  So we have the following commutative \emph{Cartesian} diagram:
\begin{equation}
\label{E:cartesian decomposition diagram}
\xymatrix{
\Sh(G_{\ttS, \ttT})
&
\Sh(G_{\ttS, \ttT})_{\gotha^\flat}
\ar@{_{(}->}[l]
\ar[r]^-{\pi_{\gotha^\flat}}
&
\Sh(G_{\ttS_{\gotha^\flat}, \ttT_{\gotha^\flat}})
\\
\Sh(G_{\ttS, \ttT})_{\gotha^{\flat\prime}}
\ar@{^{(}->}[u]
\ar[d]^{\pi_{\gotha^{\flat\prime}}}
&
\Sh(G_{\ttS, \ttT})_{\gotha}
\ar[r]^-{\pi_\Delta}
\ar@{^{(}->}[u]\ar@{_{(}->}[l]
\ar[d]^{\pi_{\Delta'}}
&
\Sh(G_{\ttS_{\gotha^\flat}, \ttT_{\gotha^\flat}})_{\gotha_\res}
\ar[d]^{\pi_{\gotha_\res}}
\ar@{^{(}->}[u]
\\
\Sh(G_{\ttS_{\gotha^{\flat\prime}}, \ttT_{\gotha^{\flat\prime}}})
&
\Sh(G_{\ttS_{\gotha^{\flat\prime}}, \ttT_{\gotha^{\flat\prime}}})_{\gotha'_\res}
\ar@{_{(}->}[l]
\ar[r]^-{\pi_{\gotha'_\res}}
&
\Sh(G_{\ttS_{\gotha}, \ttT_{\gotha}}),
}
\end{equation}
where $\pi_\Delta$ and $\pi_{\Delta'}$ are the morphisms defined by the natural pull-back of upper-right and lower-left Cartesian squares, respectively.
By Remark~\ref{R:composition of shifts} the shifts satisfies the following equality
\begin{equation}
\label{E:decomposition meander shift}
\boldsymbol t_{\gotha^\flat}
\boldsymbol t_{\gotha^\flat, \gotha} = \boldsymbol t_\gotha = \boldsymbol t_{\gotha^{\flat\prime}}
\boldsymbol t_{\gotha^{\flat\prime}, \gotha} \quad \textrm{ in } E^{\times, \cl} \backslash \AAA_E^{\infty,\times}/\calO_{E_{\gothp}}^\times.
\end{equation}

This implies that both $\pi_\Delta$ and $\pi_{\Delta'}$ are iterated $\PP^1$-bundles of relative dimensions $s$ and $s'$ respectively.
We use $\pi_{\Delta,!}$ to denote the natural morphism
\begin{align*}
\pi_{\Delta, !}:
H^\star_\et\big(\Sh_{K_p}(G_{\ttS, \ttT})_{\gotha,\overline \FF_p}, \calL_{\ttS, \ttT}^{(\underline k, w)}(s)\big)
& \xrightarrow{\cong} H^{\star}_\et\big(\Sh_{K_p}(G_{\ttS_{\gotha^\flat}, \ttT_{\gotha^\flat}})_{\gotha_\res, \overline \FF_p}, \calL_{\ttS_{\gotha^\flat}, \ttT_{\gotha^\flat}}^{(\underline k, w)}
\otimes 
R \pi_{\Delta*}\overline \QQ_\ell(s)\big)
\\
&\xrightarrow{ \Tr_{\pi_{\Delta}}}
H^{\star-2s}_\et \big(\Sh_{K_p}(G_{\ttS_{\gotha^\flat}, \ttT_{\gotha^{\flat}}} )_{\gotha_\res, \overline \FF_p}, \calL_{\ttS_{\gotha^\flat}, \ttT_{\gotha^{\flat}}}^{(\underline k, w)}\big),
\end{align*}
where the last map is induced by the trace isomorphism $R^{2s}\pi_{\Delta,*}(\overline \Q_\ell(s))\cong \overline \Q_\ell$.

As a consequence of the Cartesian property and Theorem~\ref{T:GO description}(1),
we have the following commutative diagram (which is placed into \eqref{E:cartesian decomposition diagram} vertically on the right)
\begin{equation}
\label{E:cartesian decomposition}
\xymatrix{
H^{d-2s'}(\Sh(G_{\ttS_{\gotha^{\flat\prime}}, \ttT_{\gotha^{\flat\prime}}})_{\gotha'_{\res}})(s)
\ar[r]^-{\pi_{\Delta'}^*}
\ar[d]^{\pi_{\gotha'_\res,!}}
&
H^{d-2s'}(\Sh(G_{\ttS, \ttT})_{\gotha})(s)
\ar[d]^{\pi_{\Delta,!}}
\ar[r]^-{\textrm{Gysin}}
&
H^d(\Sh(G_{\ttS, \ttT})_{\gotha^\flat})(s+s')
\ar[d]^{\pi_{\gotha^\flat,!}}
\\
H^{d-2s-2s'}(\Sh(G_{\ttS_\gotha, \ttT_\gotha}))
\ar[r]^-{\pi^*_{\gotha_\res}}
&
H^{d-2s-2s'}(\Sh(G_{\ttS_{\gotha^\flat}, \ttT_{\gotha^\flat}})_{\gotha_\res})
\ar[r]^-{\textrm{Gysin}}
&
H^{d-2s}(\Sh(G_{\ttS_{\gotha^\flat}, \ttT_{\gotha^\flat}}))(s').
}
\end{equation}

\subsection{Inductive proof of Theorem~\ref{T:intersection combinatorics}}
We now start the proof of Theorem~\ref{T:intersection combinatorics} by induction on $d=\#\ttS^c_{\infty}$ or equivalently the dimension of the Shimura variety $\Sh(G_{\ttS,\ttT})$ (and also on $r$ by keeping $d-2r$ fixed throughout the induction).  The base case $d=0$ and $d=1$ are trivial (as there is no nontrivial periodic semi-meander).

We now assume that Theorem~\ref{T:intersection combinatorics} holds for all Shimura varieties $\Sh_K(G_{\ttS,\ttT})$ with $\#\ttS_\infty^c < d$.
We now fix $\ttS, \ttT$ so that $\#\ttS_\infty^c =d$. 
The case of $r=0$ is clear.  We henceforth assume that $r>0$.

Let $\gotha, \gothb \in \gothB_\ttS^r$ be as in Theorem~\ref{T:intersection combinatorics}.
We fix a \emph{basic} arc $\delta_\gothb$ of $\gothb$, with right end-node $\tau \in \ttS_\infty^c$ (and left end-node $\tau^- \in \ttS_\infty^c$).
As in Subsection~\ref{SS:decomposition of semi-meanders}, we use $\gothb_\res \in \gothB_{\ttS\cup\{\tau, \tau^-\}}^{r-1}$ to denote the periodic semi-meander $\gothb \backslash\delta_\gothb$ obtained by removing $\delta_\gothb$ from $\gothb$ and replacing its end-nodes by plus signs.
We will use $\delta_\gothb$ itself to denote the corresponding $\gothb^\flat$, that is, we also view $\delta_\gothb$ as a periodic semi-meander for $\ttS$ with only one arc $\delta_\gothb$ (and  $d-2$ semi-lines).

The basic 
idea is to factor the Gysin map $\Gys_\gothb$ using $\delta_\gothb$, in the sense of Subsection~\ref{SS:decomposition of semi-meanders}, and to factor the restriction map $\Res_\gotha$ according to the following list of four cases.
\begin{itemize}
\item[(i)]
The two nodes $\tau, \tau^-$ are both linked to semi-lines in $\gotha$. This forces us to fall into the case (1) of Theorem~\ref{T:intersection combinatorics}.
\item[(ii)]
There is a (basic) arc $\delta_\gotha$ in $\gotha$ linking $\tau^-$ to $\tau$ from left to right, so that $\delta_\gotha$ and $\delta_\gothb$ form a contractible loop in $D(\gotha, \gothb)$. In other words,   $\delta_\gotha$ and $\delta_\gothb$ are the same (up to deformation of the arcs).
We shall reduce the proof of Theorem~\ref{T:intersection combinatorics} to the case for $\ttS' = \ttS\cup\{\tau, \tau^-\}$, $\ttT'= \ttT\cup \{\tau\}$, $\gotha_\res = \gotha \backslash \delta_\gotha$ and $\gothb_\res$ and it hence follows from the inductive hypothesis.  In particular, we shall see that the contractible loop $\delta_\gotha$ and $\delta_\gothb$ contributes a factor of $-2p^{\ell(\delta_\gothb)}$.
\item[(iii)]
There is an arc $\delta_\gotha$ in $\gotha$ connecting $\tau$ to $\tau^-$ wrapped around the cylinder from right to left. In other words, $\delta_\gotha$ and $\delta_\gothb$ together form a non-contractible loop in $D(\gotha, \gothb)$. This can only happen if $r=d/2$. We will show that the composition $\Res_\gotha\circ \Gys_\gothb$ is essentially the $T_\gothp$-operator composed with $\Res_{\gotha \backslash\delta_\gotha} \circ \Gys_{\gothb \backslash \delta_\gothb}$ for the Shimura variety with $\ttS' = \ttS\cup\{\tau, \tau^-\}$ and $\ttT'= \ttT\cup \{\tau\}$, up to some link morphism which we make explicit later.

\item[(iv)]
Neither of above happens.
Then, in $\gotha$, either $\tau$ is connected by an arc  whose other end-node is not $\tau^-$, and/or $\tau^-$ is connected by an arc whose other end-node is not $\tau$.
In either case, we will reduce to a case with the two nodes $\tau$ and $\tau^-$ removed, after composing with a certain link morphism.
\end{itemize}
We now treat each of the cases separately.

\subsection{Case (i)}
This is the case when $\tau$ and $\tau^-$ are connected to semi-lines in $\gotha$. This implies that $\langle \gotha |\gothb \rangle = 0$. So we are in the situation of Theorem~\ref{T:intersection combinatorics}(1). We need to show that the $\pi$-isotypical component  of $\Res_\gotha \circ \Gys_\gothb$ factors through the cohomology of a Shimura variety of smaller dimension. 
Let $\gotha^*$ denote the periodic semi-meander for $\ttS$ given by removing the two semi-lines of $\gotha$ connected to $\tau$ and $\tau^-$ and reconnecting $\tau$ and $\tau^-$ by a (basic) arc. Note that this is possible because $\delta_\gothb$ is a basic arc, so $\tau$ and $\tau^-$ are adjacent nodes in the band for $\ttS$.
In particular, $\gotha^* \in \gothB_\ttS^{r+1}$.

By the discussion of Subsection~\ref{SS:decomposition of semi-meanders}, we see that the morphism $\Res_\gotha\circ \Gys_\gothb$ is the  composition from top-left to the bottom-right of the following commutative diagram by going first downwards and then rightwards.
\[
\xymatrix{
H^{d-2r}(\Sh(G_{\ttS_\gothb, \ttT_\gothb})) 
\ar[d]^{\pi^*_{\delta_\gothb} \circ \Gys_{\gothb_\res}}
\\
H^{d-2}(\Sh(G_{\ttS, \ttT})_{\delta_\gothb})(r-1)
\ar[d]^{\textrm{Gysin}}
\ar[r]^{\textrm{Restr.}}
&
H^{d-2}(\Sh(G_{\ttS, \ttT})_{\gotha^*})(r-1)
\ar[d]^{\textrm{Gysin}}
\\
H^d(\Sh(G_{\ttS, \ttT}))(r)
\ar[r]^{\textrm{Restr.}}
&
H^d(\Sh(G_{\ttS, \ttT})_\gotha)(r)
\ar[r]^-{\pi_{\gotha,!}}
&
H^{d-2r}(\Sh(G_{\ttS_\gotha, \ttT_\gotha})).
}
\]
Here, the square is commutative because the corresponding morphisms on the Shimura varieties form a Cartesian square.
  The diagram implies that the $\pi$-component of $\Res_\gotha \circ \Gys_\gothb$ factors through the cohomology group
$$H^{d-2}(\Sh(G_{\ttS, \ttT})_{\gotha^*})(r-1)[\pi] \cong H^{d-2(r+1)}(\Sh(G_{\ttS_{\gotha^*}, \ttT_{\gotha^*}}))(-1)[\pi],$$
which is the $\pi$-isotypical component of the cohomology of a quaternionic Shimura variety of dimension $d-2(r+1)$.
This means that the conclusion of Theorem~\ref{T:intersection combinatorics}(1) holds if we ever arrive in case (i) during the inductive proof.
  
\subsection{Case (ii)}
This is the case when there is a basic arc $\delta_\gotha$ in $\gotha$ linking $\tau^-$ to $\tau$ from left to right, and hence $\delta_\gotha$ and $\delta_\gothb$ are the same (up to deformation of the arcs). We write $\delta$ for the periodic semi-meander for $\ttS$ with only one arc $\delta_\gotha$.
We write $\gotha_\res =\gotha \backslash\delta$ for the periodic semi-meander for $\ttS_\delta$ obtained by removing $\delta_\gotha$ from $\gotha$ and placing its end-nodes by plus signs.

Using the discussion of Subsection~\ref{SS:decomposition of semi-meanders}, the morphism $\Res_\gotha \circ \Gys_\gothb$ is the composition from the upper-left to the upper-right going all the way around: first downwards to the bottom, then all the way to the right, and finally upwards.
\[
\xymatrix{
H^{d-2r}(\Sh(G_{\ttS_\gotha, \ttT_\gotha}))
\ar[d]^{\Gys_{\gothb_\res}}
&&
H^{d-2r}(\Sh(G_{\ttS_\gothb, \ttT_\gothb}))
\\
H^{d-2}(\Sh(G_{\ttS_\delta, \ttT_\delta}))(r-1) \ar@{-->}[rr]
\ar[d]^{\pi_\delta^*}
&& 
H^{d-2}(\Sh(G_{\ttS_\delta, \ttT_\delta}))(r-1)
\ar[u]_{\Res_{\gotha_\res}}
\\
H^{d-2}(\Sh(G_{\ttS, \ttT})_\delta)(r-1) \ar[r]^-{\textrm{Gysin}}
&
H^d(\Sh(G_{\ttS, \ttT}))(r)
\ar[r]^{\textrm{Restr.}}
&
H^d(\Sh(G_{\ttS, \ttT})_\delta)(r)
\ar[u]^{\pi_{\delta,!}}
}
\]
As in Subsection~\ref{S:case r=1 a=b}, the composition of the bottom line is given by the excessive intersection formula, that is to take the cup product with the first Chern class of the normal bundle of the embedding 
$\Sh(G_{\ttS, \ttT})_\delta \hookrightarrow\Sh(G_{\ttS, \ttT})$, which is $-2p^{\ell(\delta)}$ times the class of the canonical quotient bundle for the $\PP^1$-bundle given by $\pi_\delta$, according to Proposition~\ref{P:GO-fibration}(2).
Therefore, the dotted arrow in the middle  is simply multiplication by $-2p^{\ell(\delta)}$.
From this, we deduce that
\begin{equation}
\label{E:proof case ii}
\Res_{\gotha} \circ \Gys_{\gothb} = -2p^{\ell(\delta)} \cdot \Res_{\gotha_\res} \circ \Gys_{\gothb_\res},
\end{equation}
where the latter morphism is constructed over the Shimura variety $\Sh(G_{\ttS_\delta, \ttT_\delta})$ of lower dimension.  (Here we choose the shift $\boldsymbol t'_{\gotha'}$ for a periodic semi-meander $\gotha'$ for $(\ttS_\delta, \ttT_\delta)$ to be $\boldsymbol t_{\delta, \tilde \gotha'}$, where $\tilde \gotha'$ is a periodic semi-meander of $(\ttS, \ttT)$ consisting of all the arcs and semi-lines of $\gotha'$ together with the arc $\delta$.)

We can now complete the induction in this case, since we have already known Theorem~\ref{T:intersection combinatorics} for $\Res_{\gotha_\res} \circ \Gys_{\gothb_\res}$ by induction hypothesis.
\begin{enumerate}
\item
If $\langle \gotha, \gothb\rangle = 0$, then $\langle \gotha_\res, \gothb_\res\rangle = 0$ for simple combinatorics reasons. Then the $\pi$-isotypical component of  $\Res_{\gotha_\res} \circ \Gys_{\gothb_\res}$ factors through the cohomology of a lower dimensional Shimura variety, so the same is true for $\Res_{\gotha} \circ \Gys_{\gothb}$.
\item[(2) or (3)]
We have $\langle \gotha|\gothb\rangle = (-2)^{m_0}v^{m_v}$ or $(-2)^{m_0}T^{m_T}$. The picture $D(\gotha_\res, \gothb_\res)$ is given by removing from $D(\gotha, \gothb)$ the contractible loop consisting of $\delta_\gotha$ and $\delta_\gothb$.
So we have 
\[
\langle \gotha_\res| \gothb_\res \rangle = (-2)^{-1} \langle \gotha|\gothb\rangle = \begin{cases} (-2)^{m_0-1}v^{m_v} & \textrm{if }r<\frac d2,\\ (-2)^{m_0-1}T^{m_T} & \textrm{if }r = \frac d2.
\end{cases}
\]
Since we have
$\ell(\gotha) - \ell(\gotha_\res) = \ell(\gothb) - \ell(\gothb_\res) = \ell(\delta)$ and $\boldsymbol t'_{\gotha_\res} \boldsymbol t'^{-1}_{\gothb_\res} = \boldsymbol t_\gotha \boldsymbol t^{-1}_\gothb$, we see that $\eta_{\ttS_\gotha, \ttS_\gothb}$ gives the same link morphism as $\eta_{\ttS_{\delta, \gotha_\res}, \ttS_{\delta, \gothb_\res}}$ (with the same indentation and shift). By the inductive hypothesis and \eqref{E:proof case ii},
\begin{align*}
\Res_{\gotha} &\circ \Gys_{\gothb} = -2p^{\ell(\delta)} \cdot \Res_{\gotha_\res} \circ \Gys_{\gothb_\res}
\\
& =
\begin{cases}
-2p^{\ell(\delta)} \cdot (-2)^{m_0-1}\cdot p ^{(\ell(\gotha_\res) + \ell(\gothb_\res))/2} \eta_{\ttS_{\delta, \gotha_\res}, \ttS_{\delta, \gothb_\res}, (z)}^\star, & \textrm{if }r < \frac d2,
\\
-2p^{\ell(\delta)} \cdot (-2)^{m_0-1}\cdot p ^{(\ell(\gotha_\res) + \ell(\gothb_\res))/2} (T_\gothp/p^{g/2})^{m_T} \eta_{\ttS_{\delta, \gotha_\res}, \ttS_{\delta, \gothb_\res}, (z)}^\star
& \textrm{if }r = \frac d2,
\end{cases}
\\
& =
\begin{cases}
(-2)^{m_0}\cdot p ^{(\ell(\gotha) + \ell(\gothb))/2} \eta_{\ttS_\gotha, \ttS_\gothb, (z)}^\star, & \textrm{if }r < \frac d2,
\\ (-2)^{m_0}\cdot p ^{(\ell(\gotha) + \ell(\gothb))/2} (T_\gothp/p^{g/2})^{m_T} \eta_{\ttS_\gotha, \ttS_\gothb, (z)}^\star
& \textrm{if }r = \frac d2.
\end{cases}
\end{align*}

\end{enumerate}

\subsection{Case (iii)}
This is the case when there is an arc $\delta_\gotha$ in $\gotha$ connecting $\tau$ and $\tau^-$ wrapped around the cylinder from right to left, and hence $\delta_\gotha$ and $\delta_\gothb$ together form a non-contractible loop in $D(\gotha,\gothb)$.  We are forced to have $d=2r$ in this case (and hence $p$ splits in $E/F$).
Moreover, the arc $\delta_\gotha$  must lie over all other arcs of $\gotha$ (if there is any).
We now define a list of notations followed by an example.
\begin{itemize}
\item
Let $\delta_{\gotha, \bullet}$ (resp. $\delta_{\gothb, \bullet}$) denote the periodic semi-meander of two nodes obtained from $\gotha$ (resp.  $\gothb$) by keeping $\delta_\gotha$ (resp. $\delta_{\gothb}$) and its end-nodes and replacing  the other nodes of $\gotha$ replaced by plus signs.
\item
Let $\gotha^\flat_\bullet = \gotha \backslash \delta_\gotha$ denote the periodic semi-meander for $\ttS_\gotha$ given by removing the arc $\delta_\gotha$ from $\gotha$ and replacing the nodes at $\tau$ and $\tau^-$ by plus signs.
\item Let $\gotha^\flat$ denote the periodic semi-meander for $\ttS$ given by removing the arc $\delta_\gotha$ and adjoining two semi-lines attached to both $\tau$ and $\tau^-$.
\item 
Let $\gotha^\star$ denote the semi-meander for $\ttS$ obtained by replacing the arc $\delta_\gotha$ in $\gotha$ with $\delta_\gothb$ instead.
\end{itemize}
For example, if $
\psset{unit=0.3}
\gotha =
\begin{pspicture*}(-.5,-1.5)(5.5,1.5)
\psset{linewidth=1pt}
\psset{linecolor=red}
\psbezier{-}(0,0)(0,2)(5,2)(5,0)
\psarc{-}(3.5,0){0.5}{0}{180}
\psarc{-}(1.5,0){0.5}{0}{180}
\psset{linecolor=black}
\psdots(0,0)(1,0)(2,0)(3,0)(4,0)(5,0)
\rput(5.2,-.92){\psframebox*{\begin{tiny}$\tau^{-}$\end{tiny}}}
\rput(0,-.8){\psframebox*{\begin{tiny}$\tau$\end{tiny}}}
\end{pspicture*}$ and $
\psset{unit=0.3}
\gothb =
\begin{pspicture*}(-0.5,-1.5)(5.5,1.5)
\psset{linewidth=1pt}
\psset{linecolor=red}
\psarc{-}(5.5,0){0.5}{90}{180}
\psbezier{-}(1,0)(1,1.5)(4,1.5)(4,0)
\psarc{-}(2.5,0){0.5}{0}{180}
\psarc{-}(-0.5,0){0.5}{0}{90}
\psset{linecolor=black}
\psdots(0,0)(1,0)(2,0)(3,0)(4,0)(5,0)
\rput(5.2,-.92){\psframebox*{\begin{tiny}$\tau^{-}$\end{tiny}}}
\rput(0,-.8){\psframebox*{\begin{tiny}$\tau$\end{tiny}}}
\end{pspicture*}$, and we choose $\delta_\gothb$ to be the arc of $\gothb$ linking the first and the last nodes ($\tau$ and $\tau^-$ respectively in the pictures), then $\delta_\gotha$ is the arc linking the first and the last nodes (but ``over" all other arcs).
In this case, we have
\[
\psset{unit=0.3}
\delta_{\gotha, \bullet} =
\begin{pspicture*}(-.5,-1.5)(5.5,1.5)
\psset{linewidth=1pt}
\psset{linecolor=red}
\psbezier{-}(0,0)(0,2)(5,2)(5,0)
\psset{linecolor=black}
\psdots[dotstyle=+](1,0)(2,0)(3,0)(4,0)
\psdots(0,0)(5,0)
\rput(5.2,-.92){\psframebox*{\begin{tiny}$\tau^{-}$\end{tiny}}}
\rput(0,-.8){\psframebox*{\begin{tiny}$\tau$\end{tiny}}}
\end{pspicture*}
,\quad
\delta_{\gothb, \bullet} =
\begin{pspicture*}(-.5,-1.5)(5.5,1.5)
\psset{linewidth=1pt}
\psset{linecolor=red}
\psarc{-}(5.5,0){0.5}{90}{180}
\psarc{-}(-0.5,0){0.5}{0}{90}
\psset{linecolor=black}
\psdots[dotstyle=+](1,0)(2,0)(3,0)(4,0)
\psdots(0,0)(5,0)
\rput(5.2,-.92){\psframebox*{\begin{tiny}$\tau^{-}$\end{tiny}}}
\rput(0,-.8){\psframebox*{\begin{tiny}$\tau$\end{tiny}}}
\end{pspicture*}
,\quad
\gotha^\flat_\bullet =
\begin{pspicture*}(-.5,-1.5)(5.5,1)
\psset{linewidth=1pt}
\psset{linecolor=red}
\psarc{-}(3.5,0){0.5}{0}{180}
\psarc{-}(1.5,0){0.5}{0}{180}
\psset{linecolor=black}
\psdots(1,0)(2,0)(3,0)(4,0)
\psdots[dotstyle=+](0,0)(5,0)
\rput(5.2,-.92){\psframebox*{\begin{tiny}$\tau^{-}$\end{tiny}}}
\rput(0,-.8){\psframebox*{\begin{tiny}$\tau$\end{tiny}}}
\end{pspicture*},
\]
\[
\psset{unit=.3}
\gotha^\flat=
\begin{pspicture*}(-.5,-1.5)(5.5,1)
\psset{linewidth=1pt}
\psset{linecolor=red}
\psline(0,0)(0,1)
\psline(5,0)(5,1)
\psarc{-}(3.5,0){0.5}{0}{180}
\psarc{-}(1.5,0){0.5}{0}{180}
\psset{linecolor=black}
\psdots(0,0)(1,0)(2,0)(3,0)(4,0)(5,0)
\rput(5.2,-.92){\psframebox*{\begin{tiny}$\tau^{-}$\end{tiny}}}
\rput(0,-.8){\psframebox*{\begin{tiny}$\tau$\end{tiny}}}
\end{pspicture*}
, \quad
\gotha^\star = 
\begin{pspicture*}(-.5,-1.5)(5.5,1.5)
\psset{linewidth=1pt}
\psset{linecolor=red}
\psarc{-}(3.5,0){0.5}{0}{180}
\psarc{-}(1.5,0){0.5}{0}{180}
\psarc{-}(5.5,0){0.5}{90}{180}
\psarc{-}(-0.5,0){0.5}{0}{90}
\psset{linecolor=black}
\psdots(0,0)(1,0)(2,0)(3,0)(4,0)(5,0)
\rput(5.2,-.92){\psframebox*{\begin{tiny}$\tau^{-}$\end{tiny}}}
\rput(0,-.8){\psframebox*{\begin{tiny}$\tau$\end{tiny}}}
\end{pspicture*}
,\quad  \textrm{and }
\gothb_\res =
\begin{pspicture*}(-.5,-1.5)(5.5,1.2)
\psset{linewidth=1pt}
\psset{linecolor=red}
\psbezier{-}(1,0)(1,1.5)(4,1.5)(4,0)
\psarc{-}(2.5,0){0.5}{0}{180}
\psset{linecolor=black}
\psdots(1,0)(2,0)(3,0)(4,0)
\psdots[dotstyle=+](0,0)(5,0)
\rput(5.2,-.92){\psframebox*{\begin{tiny}$\tau^{-}$\end{tiny}}}
\rput(0,-.8){\psframebox*{\begin{tiny}$\tau$\end{tiny}}}
\end{pspicture*}
.
\]

Our goal is to prove an equality
\begin{equation}
\label{E:proof case iii}
\Res_\gotha \circ \Gys_\gothb = T_\gothp \circ \eta^\star_{\ttS_{\gotha^\star}, \ttS_\gotha} \circ \Res_{\gotha^\flat_\bullet} \circ \Gys_{\gothb_\res},
\end{equation}
where $\eta^\star_{\ttS_{\gotha^\star}, \ttS_\gotha}$ is a certain link morphism associated to the trivial link $\eta_{\ttS_{\gotha^\star}, \ttS_\gotha}: \ttS_{\gotha^\star} \to \ttS_\gotha$ which we specify later.

Using the discussion of Subsection~\ref{SS:decomposition of semi-meanders}, we see that the morphism $\Res_\gotha \circ \Gys_\gothb$ is the composition from top-left to bottom-left of the following diagram, by going first rightward to the end, then downwards to the bottom, and finally to the left by the long arrow:
\begin{equation}
\label{E:commutative diagram case iii}
\xymatrix{
H^0\big(\Sh(G_{\ttS_\gothb, \ttT_\gothb})\big) \ar[d]_{\Gys_{\gothb_\res}}
\\
H^{d-2}\big(\Sh(G_{\ttS_{\delta_{\gothb}}, \ttT_{\delta_{\gothb}}}) \big)(\frac d2-1) \ar[r]^-{\pi_{\delta_\gothb}^*}
\ar[d]_{\mathrm{Restr.}}
&
H^{d-2}\big(\Sh(G_{\ttS, \ttT})_{\delta_\gothb}\big)(\frac d2 -1)
\ar[r]^-{\mathrm{Gysin}}
\ar[d]^{\mathrm{Restr.}}
&
H^{d}\big(\Sh(G_{\ttS, \ttT})\big)(\frac{d}{2})
\ar[d]^{\mathrm{Restr.}}
\\
H^{d-2}\big(\Sh(G_{\ttS_{\delta_\gothb}, \ttT_{\delta_\gothb}})_{\gotha^\flat_\bullet} \big)(\frac{d}{2}-1)
\ar[d]^-{\pi_{\gotha_\bullet^\flat, !}}
\ar[r]^{\pi_{\delta_\gothb}^*}
&
H^{d-2}\big(\Sh(G_{\ttS, \ttT})_{\gotha^\star}\big)(\frac{d}{2}-1)
\ar[r]^-{\mathrm{Gysin}}
&
H^{d}\big(\Sh(G_{\ttS, \ttT})_{\gotha^\flat}\big)(\frac d2)
\ar[d]^{\pi_{\gotha^\flat, !}}
\\
H^0
\big(\Sh(G_{\ttS_{\gotha^\star}, \ttT_{\gotha^\star}})
\big)
\ar[r]^-{\pi^*_{\delta_{\gothb, \bullet}}}
\ar@{-->}[d]
&
H^0\big(\Sh(G_{\ttS_{\gotha^\flat}, \ttT_{\gotha^\flat}})_{\delta_{\gothb, \bullet}}
\big)
\ar[r]^{\mathrm{Gysin}}
&
H^2
\big(\Sh(G_{\ttS_{\gotha^\flat}, \ttT_{\gotha^\flat}})
\big)(1)
\ar[d]^{\mathrm{Restr.}}
\\
H^{0}\big(\Sh(G_{\ttS_\gotha, \ttT_\gotha})
\big)
&&
H^2
\big(\Sh(G_{\ttS_{\gotha^\flat}, \ttT_{\gotha^\flat}})_{\delta_{\gotha,\bullet}}
\big)(1)
\ar[ll]^{\pi_{\delta_{\gotha, \bullet}, !}}
}
\end{equation}
The commutativity of the left  top square in the diagram above follows from the commutative diagram of morphisms of varieties, and that of the right top  square follows from \eqref{E:gysin-restriction} and the fact that $\Sh(G_{\ttS,\ttT})_{\gotha^{\star}}$ is the transversal intersection of $\Sh(G_{\ttS,\ttT})_{\delta_{\gothb}}$ and $\Sh(G_{\ttS,\ttT})_{\gotha^{\flat}}$ in $\Sh(G_{\ttS,\ttT})$.

The middle rectangle of \eqref{E:commutative diagram case iii} is commutative by \eqref{E:cartesian decomposition} (applied with our $\gotha^\star$, $\gotha^\flat$, and $\delta_\gothb$ being the $\gotha$, $\gotha^{\flat}$, and $\gotha^{\flat\prime}$ therein, respectively).
Now, for the bottom rectangle, we are simply working with the Shimura variety $\Sh(G_{\ttS_{\gotha^\flat}, \ttT_{\gotha^\flat}})$ and hence are reduced to the case of $d=2$.  

Using Subsection~\ref{S:case of g=2 Tp}, we see that the dotted downward arrow on the left is exactly the operator $T_\gothp$ times a link morphism $\eta^\star_{\ttS_\gotha, \ttS_{\gotha^\star}}$ associated to the link $\eta_{\ttS_\gotha, \ttS_{\gotha^\star}}: \ttS_\gotha \to \ttS_{\gotha^\star}$, with indentation degree $-2 \ell(\delta_{\gothb, \bullet})$ and shift 
\begin{equation}
\label{E:shift case 3-1}
\varpi_{\bar \gothq}^{-1} 
\boldsymbol t_{\gotha^\flat, \gotha^\star}^{-1} \boldsymbol t_{\gotha^\flat, \gotha}
= \varpi_{\bar \gothq}^{-1} \boldsymbol t_{\gotha^\star}^{-1}
\boldsymbol t_\gotha.
\end{equation}

To sum up, the morphism $\Res_\gotha \circ \Gys_\gothb$ is the same as the composition of the downward arrows on the left in \eqref{E:commutative diagram case iii}. So we have proved \eqref{E:proof case iii}.

We now complete the inductive proof of Theorem~\ref{T:intersection combinatorics}. The condition for case (iii) implies that we are in the setup of Theorem~\ref{T:intersection combinatorics}(3).
Assume that we have $\langle \gotha|\gothb\rangle = (-2)^{m_0}T^{m_T}$. The picture $D(\gotha^\flat_\bullet, \gothb_\res)$ is given by removing from $D(\gotha, \gothb)$ the non-contractible loop consisting of $\delta_\gotha$ and $\delta_\gothb$.
So we have 
\[
\langle \gotha^\flat_\bullet| \gothb_\res \rangle = T^{-1} \langle \gotha|\gothb\rangle =  (-2)^{m_0}T^{m_T-1}.
\]
By inductive hypothesis applied to the Shimura variety $\Sh_{K_p}(G_{\ttS_{\delta_\gothb}, \ttT_{\delta_\gothb}})$ of lower dimension (where the shift $\boldsymbol t'_{\gotha'}$ for a periodic semi-meander $\gotha'$ for $(\ttS_{\delta_\gothb}, \ttT_{\delta_\gothb})$ is taken to be $\boldsymbol t_{\delta_\gothb, \tilde \gotha'}$, where $\tilde \gotha'$ is a periodic semi-meander of $(\ttS, \ttT)$ consisting of all the arcs and semi-lines of $\gotha'$ together with the arc $\delta_\gothb$),
\begin{equation}
\label{E:proof case iii2}
\Res_{\gotha^\flat_\bullet} \circ \Gys_{\gothb_\res} = 
(-2)^{m_0}\cdot p^{(\ell(\gotha^\flat_\bullet)+ \ell(\gothb_\res))/2} (T_\gothp/p^{g/2})^{m_T-1} \circ \eta_{\ttS_{\gotha^\flat_\bullet}, \ttS_{\gothb_\res},(z')}^\star,
\end{equation}
where $\eta_{\ttS_{\gotha^\flat_\bullet}, \ttS_{\gothb_\res},(z')}^\star$ is the trivial link morphism with shift 
\begin{equation}
\label{E:shift case 3-2}
\boldsymbol t'_{\gotha^\flat_\bullet} \boldsymbol t'^{-1}_{\gothb_\res} \varpi_{\bar \gothq}^{-m_T+1}
= \boldsymbol{t}_{\delta_\gothb, \gotha^\star} \boldsymbol t^{-1}_{\delta_\gothb, \gothb}
\varpi_{\bar \gothq}^{-m_T+1}
= \boldsymbol t_{\gotha^\star} \boldsymbol t_{\gothb}^{-1}\varpi_{\bar \gothq}^{-m_T+1}
\end{equation}
and indentation degree  $z = \ell(\gotha^\flat_\bullet) - \ell(\gothb_\res)-(m_T-1) g$.
Combining \eqref{E:proof case iii} and \eqref{E:proof case iii2} with the numerical equalities
\[
\ell(\gotha^\flat_\bullet) = \ell(\gotha) - g + \ell(\delta_{\gothb,\bullet}) \quad \textrm{and} \quad \ell(\gothb_\res) = \ell(\gothb) - \ell(\delta_{\gothb,\bullet}),
\]
we deduce that
\[
\Res_\gotha \circ \Gys_\gothb = (-2)^{m_0} p^{(\ell(\gotha) + \ell(\gothb))/2}(T_\gothp / p^{g/2})^{m_T} \circ \eta_{\ttS_\gotha, \ttS_{\gotha^\star}}^\star \circ \eta_{\ttS_{\gotha^\flat_\bullet}, \ttS_{\gothb_\res},(z')}^\star.
\]
The composition of the last two link morphism is a link morphism $\ttS_\gotha \to \ttS_{\gotha^\star} = \ttS_{\gotha_\bullet^\flat} \to \ttS_{\gothb_\res} = \ttS_\gothb$, whose indentation degree is
\[
-2 \ell(\delta_{\gothb, \bullet}) + z = \ell(\gotha) - \ell(\gothb) - m_Tg
\]
and whose shift is equal to the product of \eqref{E:shift case 3-1} and \eqref{E:shift case 3-2}, or explicitly,
\[
\varpi_{\bar \gothq}^{-1} \boldsymbol t_{\gotha^\star}^{-1}
\boldsymbol t_\gotha  \cdot \boldsymbol t_{\gotha^\star} \boldsymbol t_{\gothb}^{-1} \varpi_{\bar \gothq}^{-m_T+1} = \boldsymbol t_{\gotha}
\boldsymbol t_\gothb^{-1} \varpi_{\bar \gothq}^{-m_T}.
\]

This completes Theorem~\ref{T:intersection combinatorics}(3) in this case.

\subsection{Case (iv)}
Recall that $\delta_\gothb$ is a basic arc of $\gothb$ linking $\tau$ with $\tau^-$ from right to left.
We are looking at the situation when at least one of  $\tau$ and $\tau^-$ is linked to an arc in $\gotha$ that is not connected to the other node.
We start with a long list of combinatorics construction, followed by two examples.
\begin{itemize}
\item
Let $\gotha^\circ$ be the periodic semi-meander for $\ttS_{\delta_\gothb}$ given by first replacing the nodes $\tau, \tau^-$ by plus signs, then adjoining the basic arc $\delta_\gothb$ to $\gotha$ from \emph{underneath} the band to connect to the arcs or links that are already linked to the nodes $\tau^-,\tau$, and finally continuously deform the picture so that all arcs are above the band and all semi-lines are straight.
Intuitively, one can view the last step as ``pulling the strings to tighten the drawing."

\item
Let $\gotha^\star$ denote the periodic semi-meander for $\ttS$ modified from $\gotha^{\star}$ by replacing the plus signs $\tau, \tau^-$ by nodes and adjoining them by the arc $\delta_\gothb$.

\item
Let $\gotha^\flat$ denote the periodic semi-meander for $\ttS$ that consists of two semi-lines at both $\tau$ and $\tau^-$,  all arcs in $\gotha$ that do not intersect with these two semi-lines, and semi-lines at the nodes that are not connected to anything above.
Let $r'$ ($<r$) denote the number of arcs in $\gotha^\flat$ so that $\gotha^\flat \in \gothB_\ttS^{r'}$.
\item
Let $\gotha^\flat_+$ denote the periodic semi-meander for $\ttS_{\delta_\gothb}$ obtained by removing the two semi-lines at both $\tau$ and $\tau^-$ from $\gotha^\flat$ and replacing the nodes at $\tau, \tau^-$ by plus signs.
\item
We use $\gotha^\flat_\tau$ to denote the periodic semi-meander for $\ttS$ given by replacing in $\gotha^\flat$ the two semi-lines connected to $\tau$ and $\tau^-$ by $\delta_\gothb$.
\item
We use  $\delta_{\gothb/\gotha^\flat}$ to denote the periodic semi-meander for $\ttS_{\gotha^\flat}$ consisting of only one arc  $\delta_\gothb$ (and all semi-lines of $\gotha^\flat_+$).
\item
We choose and fix an arc $\delta_{\gotha}$ of $\gotha$ such that
\begin{itemize}
\item {\it Case (a)}. either  $\tau$ is the left end-node of $\delta_\gotha$, or
\item {\it Case (b)}.
$\tau^-$ is the right end-node of $\delta_\gotha$.
\end{itemize}
Such an arc $\delta_{\gotha}$ exists under the assumption of Case (iv) (there might be one or two such arcs).

We use $\tau'$ to denote the right endpoint of $\delta_\gotha$. Thus, $\tau'$ is neither $\tau$ nor $\tau^-$ in Case (a), and $\tau'$ is the same as $ \tau^-$ in Case (b).

\item
We use $\gotha^\flat_{\tau'}$ to denote the periodic semi-meander for $\ttS$ given by deleting from $\gotha^\flat$ the two semi-lines connected to the end-nodes of $\delta_\gotha$ and then adjoining the arc $\delta_\gotha$.
\item
We use  $\delta_{\gotha/\gotha^\flat}$ to denote the periodic semi-meander for $\ttS_{\gotha^\flat}$ consisting of only one arc $\delta_\gotha$ (and all semi-lines of $\gotha^\flat_{\tau'}$).

\item
We use $\eta_{\gotha^\flat_{\tau'}, \gotha^\flat_{\tau}}$ to denote the link from $\ttS_{\gotha^\flat_{\tau'}}$ to $\ttS_{\gotha^\flat_{\tau}}$ given by the reduction of $D(\delta_{\gothb/\gotha^\flat}, \delta_{\gotha/\gotha^\flat})$ as defined in \ref{S:semi-meanders}.

\item
We use $\gotha^\flat_\res$ to denote the periodic semi-meander for $\ttS_{\gotha^\flat_{\tau'}}$ given by deleting all arcs in $\gotha$ that has already appeared in $\gotha^\flat_{\tau'}$, and changing their end-nodes to plus signs.

\item
We use $\gotha^\circ_\res$ to  denote the periodic  semi-meander for $\ttS_{\gotha^\circ}$ with nodes given by deleting all arcs in $\gotha^\circ$ that has already appeared in $\gotha^\flat_+$, and changing their end-nodes to plus signs.

\item
We use $\eta_{\gotha, \gotha^\star}$ to denote the link from $\ttS_\gotha$  to $\ttS_{\gotha^\star} = (\ttS_{\delta_\gothb})_{\gotha^\circ}$ which is the restriction of $\eta_{\gotha^\flat_{\tau'}, \gotha^\flat_{\tau}}$ to $\ttS_{\gotha}$.

\item We use $\gothb_{\res}$ to denote the semi-meander for $\ttS_{\delta_{\gothb}}$ obtained by deleting the arc $\delta_{\gothb}$ and replacing the nodes $\tau, \tau^-$ by plus signs.
\end{itemize}

We now give two examples. In both instances, $\gothb$ has a basic arc connecting the node $1$ with $2$ (starting with node $0$ on the left). So node $1$ is $\tau^-$ and node $2$ is $\tau$.

\textbf{Example 1:} We take $\psset{unit=0.3}
\gotha =
\begin{pspicture*}(-.5,-0.2)(9.5,3)
\psset{linewidth=1pt}
\psset{linecolor=red}
\psline{-}(0,0)(0,3)
\psline{-}(9,0)(9,3)
\psbezier{-}(1,0)(1,3.5)(8,3.5)(8,0)
\psbezier{-}(2,0)(2,1.5)(5,1.5)(5,0)
\psarc{-}(3.5,0){0.5}{0}{180}
\psarc{-}(6.5,0){0.5}{0}{180}
\psset{linecolor=black}
\psdots(0,0)(1,0)(2,0)(3,0)(4,0)(9,0)(8,0)(7,0)(6,0)(5,0)
\end{pspicture*}$,
then the arc $\delta_\gotha$ has to be the one connecting nodes $2$ and $5$ and $\tau'$ is the node $5$. 
We are in Case (a), and we have
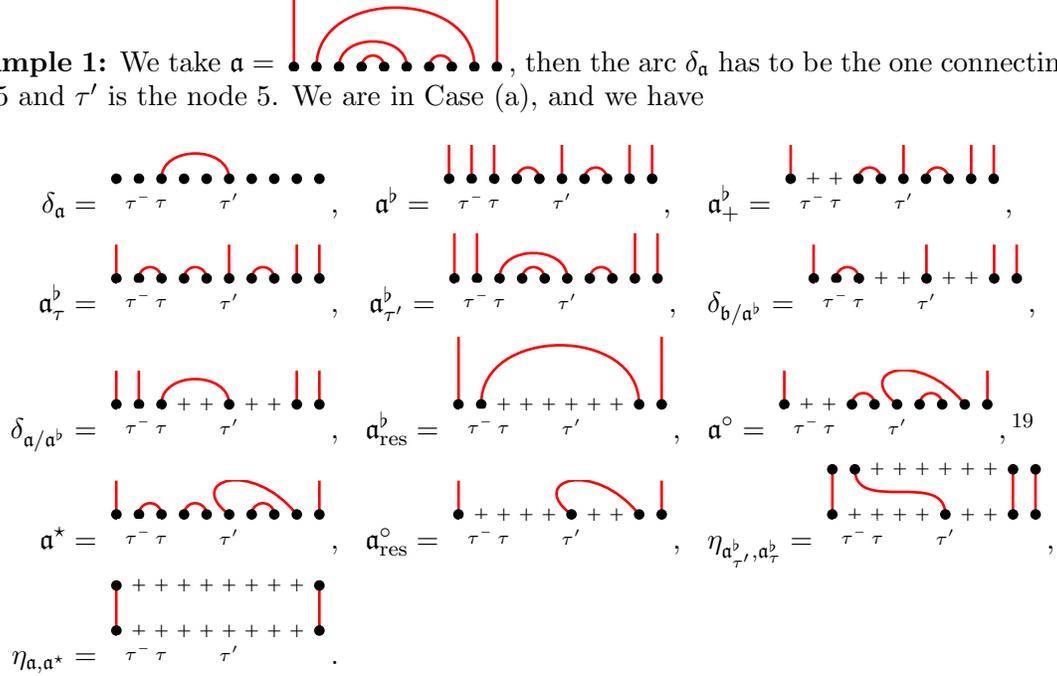
\begin{eqnarray*}
\psset{unit=0.3}
\delta_\gotha =
\begin{pspicture*}(-.5,-1.5)(9.5,2)
\psset{linewidth=1pt}
\psset{linecolor=red}
\psbezier{-}(2,0)(2,1.5)(5,1.5)(5,0)
\psset{linecolor=black}
\psdots(0,0)(1,0)(2,0)(3,0)(4,0)(9,0)(8,0)(7,0)(6,0)(5,0)
\rput(1,-.92){\psframebox*{\begin{tiny}$\tau^{-}$\end{tiny}}}
\rput(2,-1.1){\psframebox*{\begin{tiny}$\tau$\end{tiny}}}
\rput(5,-1){\psframebox*{\begin{tiny}$\tau'$\end{tiny}}}
\end{pspicture*},
&
\psset{unit=0.3}
\gotha^\flat =
\begin{pspicture*}(-.5,-1.5)(9.5,1.5)
\psset{linewidth=1pt}
\psset{linecolor=red}
\psline{-}(0,0)(0,1.5)
\psline{-}(9,0)(9,1.5)
\psline{-}(1,0)(1,1.5)
\psline{-}(2,0)(2,1.5)
\psline{-}(5,0)(5,1.5)
\psline{-}(8,0)(8,1.5)
\psarc{-}(3.5,0){0.5}{0}{180}
\psarc{-}(6.5,0){0.5}{0}{180}
\psset{linecolor=black}
\psdots(0,0)(1,0)(2,0)(3,0)(4,0)(9,0)(8,0)(7,0)(6,0)(5,0)
\rput(1,-.92){\psframebox*{\begin{tiny}$\tau^{-}$\end{tiny}}}
\rput(2,-1.1){\psframebox*{\begin{tiny}$\tau$\end{tiny}}}
\rput(5,-1){\psframebox*{\begin{tiny}$\tau'$\end{tiny}}}
\end{pspicture*}
,&
\psset{unit=0.3}
\gotha^\flat_+ =
\begin{pspicture*}(-.5,-1.5)(9.5,1.5)
\psset{linewidth=1pt}
\psset{linecolor=red}
\psline{-}(0,0)(0,1.5)
\psline{-}(9,0)(9,1.5)
\psline{-}(5,0)(5,1.5)
\psline{-}(8,0)(8,1.5)
\psarc{-}(3.5,0){0.5}{0}{180}
\psarc{-}(6.5,0){0.5}{0}{180}
\psset{linecolor=black}
\psdots(0,0)(3,0)(4,0)(9,0)(8,0)(7,0)(6,0)(5,0)
\psdots[dotstyle=+](1,0)(2,0)
\rput(1,-.92){\psframebox*{\begin{tiny}$\tau^{-}$\end{tiny}}}
\rput(2,-1.1){\psframebox*{\begin{tiny}$\tau$\end{tiny}}}
\rput(5,-1){\psframebox*{\begin{tiny}$\tau'$\end{tiny}}}
\end{pspicture*}
,
\\
\psset{unit=0.3}
\gotha^\flat_\tau =
\begin{pspicture*}(-.5,-1.5)(9.5,1.5)
\psset{linewidth=1pt}
\psset{linecolor=red}
\psline{-}(0,0)(0,1.5)
\psline{-}(9,0)(9,1.5)
\psarc{-}(1.5,0){0.5}{0}{180}
\psline{-}(5,0)(5,1.5)
\psline{-}(8,0)(8,1.5)
\psarc{-}(3.5,0){0.5}{0}{180}
\psarc{-}(6.5,0){0.5}{0}{180}
\psset{linecolor=black}
\psdots(0,0)(1,0)(2,0)(3,0)(4,0)(9,0)(8,0)(7,0)(6,0)(5,0)
\rput(1,-.92){\psframebox*{\begin{tiny}$\tau^{-}$\end{tiny}}}
\rput(2,-1.1){\psframebox*{\begin{tiny}$\tau$\end{tiny}}}
\rput(5,-1){\psframebox*{\begin{tiny}$\tau'$\end{tiny}}}
\end{pspicture*},
&
\psset{unit=0.3}
\gotha^\flat_{\tau'} =
\begin{pspicture*}(-.5,-1.5)(9.5,2)
\psset{linewidth=1pt}
\psset{linecolor=red}
\psline{-}(0,0)(0,2)
\psline{-}(9,0)(9,2)
\psline{-}(1,0)(1,2)
\psline{-}(8,0)(8,2)
\psbezier{-}(2,0)(2,1.5)(5,1.5)(5,0)
\psarc{-}(3.5,0){0.5}{0}{180}
\psarc{-}(6.5,0){0.5}{0}{180}
\psset{linecolor=black}
\psdots(0,0)(1,0)(2,0)(3,0)(4,0)(9,0)(8,0)(7,0)(6,0)(5,0)
\rput(1,-.92){\psframebox*{\begin{tiny}$\tau^{-}$\end{tiny}}}
\rput(2,-1.1){\psframebox*{\begin{tiny}$\tau$\end{tiny}}}
\rput(5,-1){\psframebox*{\begin{tiny}$\tau'$\end{tiny}}}
\end{pspicture*},
&
\psset{unit=0.3}
\delta_{\gothb/\gotha^\flat} =
\begin{pspicture*}(-.5,-1.5)(9.5,1.5)
\psset{linewidth=1pt}
\psset{linecolor=red}
\psline{-}(0,0)(0,1.5)
\psline{-}(9,0)(9,1.5)
\psline{-}(5,0)(5,1.5)
\psline{-}(8,0)(8,1.5)
\psarc{-}(1.5,0){0.5}{0}{180}
\psset{linecolor=black}
\psdots(0,0)(9,0)(8,0)(5,0)(1,0)(2,0)
\psdots[dotstyle=+](3,0)(4,0)(7,0)(6,0)
\rput(1,-.92){\psframebox*{\begin{tiny}$\tau^{-}$\end{tiny}}}
\rput(2,-1.1){\psframebox*{\begin{tiny}$\tau$\end{tiny}}}
\rput(5,-1){\psframebox*{\begin{tiny}$\tau'$\end{tiny}}}
\end{pspicture*}
,\\
\psset{unit=0.3}
\delta_{\gotha/\gotha^\flat} =
\begin{pspicture*}(-.5,-1.5)(9.5,1.5)
\psset{linewidth=1pt}
\psset{linecolor=red}
\psline{-}(0,0)(0,1.5)
\psline{-}(9,0)(9,1.5)
\psline{-}(1,0)(1,1.5)
\psline{-}(8,0)(8,1.5)
\psbezier{-}(2,0)(2,1.5)(5,1.5)(5,0)
\psset{linecolor=black}
\psdots(0,0)(9,0)(8,0)(5,0)(1,0)(2,0)
\psdots[dotstyle=+](3,0)(4,0)(7,0)(6,0)
\rput(1,-.92){\psframebox*{\begin{tiny}$\tau^{-}$\end{tiny}}}
\rput(2,-1.1){\psframebox*{\begin{tiny}$\tau$\end{tiny}}}
\rput(5,-1){\psframebox*{\begin{tiny}$\tau'$\end{tiny}}}
\end{pspicture*}
,&
\psset{unit=0.3}
\gotha^\flat_\res =
\begin{pspicture*}(-.5,-1.5)(9.5,3)
\psset{linewidth=1pt}
\psset{linecolor=red}
\psline{-}(0,0)(0,3)
\psline{-}(9,0)(9,3)\psbezier{-}(1,0)(1,3.5)(8,3.5)(8,0)
\psset{linecolor=black}
\psdots(0,0)(1,0)(9,0)(8,0)
\psdots[dotstyle=+](2,0)(3,0)(4,0)(7,0)(6,0)(5,0)
\rput(1,-.92){\psframebox*{\begin{tiny}$\tau^{-}$\end{tiny}}}
\rput(2,-1.1){\psframebox*{\begin{tiny}$\tau$\end{tiny}}}
\rput(5,-1){\psframebox*{\begin{tiny}$\tau'$\end{tiny}}}
\end{pspicture*}
,&
\psset{unit=0.3}
\gotha^\circ =
\begin{pspicture*}(-.5,-1.5)(9.5,1.5)
\psset{linewidth=1pt}
\psset{linecolor=red}
\psline{-}(0,0)(0,1.5)
\psline{-}(9,0)(9,1.5)
\psbezier{-}(5,0)(3,1.5)(6,2.5)(8,0)
\psarc{-}(3.5,0){0.5}{0}{180}
\psarc{-}(6.5,0){0.5}{0}{180}
\psset{linecolor=black}
\psdots(0,0)(3,0)(4,0)(9,0)(8,0)(7,0)(6,0)(5,0)
\psdots[dotstyle=+](1,0)(2,0)
\rput(1,-.92){\psframebox*{\begin{tiny}$\tau^{-}$\end{tiny}}}
\rput(2,-1.1){\psframebox*{\begin{tiny}$\tau$\end{tiny}}}
\rput(5,-1){\psframebox*{\begin{tiny}$\tau'$\end{tiny}}}
\end{pspicture*}
,\footnotemark
\\
\psset{unit=0.3}
\gotha^\star =
\begin{pspicture*}(-.5,-1.5)(9.5,1.5)
\psset{linewidth=1pt}
\psset{linecolor=red}
\psline{-}(0,0)(0,1.5)
\psline{-}(9,0)(9,1.5)
\psbezier{-}(5,0)(3,1.5)(6,2.5)(8,0)
\psarc{-}(1.5,0){0.5}{0}{180}
\psarc{-}(3.5,0){0.5}{0}{180}
\psarc{-}(6.5,0){0.5}{0}{180}
\psset{linecolor=black}
\psdots(0,0)(1,0)(2,0)(3,0)(4,0)(9,0)(8,0)(7,0)(6,0)(5,0)
\rput(1,-.92){\psframebox*{\begin{tiny}$\tau^{-}$\end{tiny}}}
\rput(2,-1.1){\psframebox*{\begin{tiny}$\tau$\end{tiny}}}
\rput(5,-1){\psframebox*{\begin{tiny}$\tau'$\end{tiny}}}
\end{pspicture*}, &
\psset{unit=0.3}
\gotha^\circ_\res =
\begin{pspicture*}(-.5,-1.5)(9.5,1.5)
\psset{linewidth=1pt}
\psset{linecolor=red}
\psline{-}(0,0)(0,1.5)
\psline{-}(9,0)(9,1.5)
\psbezier{-}(5,0)(3,1.5)(6,2.5)(8,0)
\psset{linecolor=black}
\psdots(0,0)(9,0)(8,0)(5,0)
\psdots[dotstyle=+](1,0)(2,0)(3,0)(4,0)(7,0)(6,0)
\rput(1,-.92){\psframebox*{\begin{tiny}$\tau^{-}$\end{tiny}}}
\rput(2,-1.1){\psframebox*{\begin{tiny}$\tau$\end{tiny}}}
\rput(5,-1){\psframebox*{\begin{tiny}$\tau'$\end{tiny}}}
\end{pspicture*}
,&
\psset{unit=0.3}
\eta_{\gotha_{\tau'}^\flat, \gotha_{\tau}^\flat}=
\begin{pspicture*}(-.5,-1.5)(9.5,2.3)
\psset{linecolor=red}
\psset{linewidth=1pt}
\psline{-}(0,2)(0,0)
\psline{-}(8,2)(8,0)
\psline{-}(9,2)(9,0)
\psbezier(5,0)(5,2)(1,0)(1,2)
\psset{linecolor=black}
\psdots(0,2)(5,0)(8,2)(9,2)
\psdots(0,0)(1,2)(8,0)(9,0)
\psdots[dotstyle=+](1,0)(2,2)(3,2)(4,2)(6,2)(7,2)
\psdots[dotstyle=+](2,0)(3,0)(4,0)(5,2)(6,0)(7,0)
\rput(1,-.92){\psframebox*{\begin{tiny}$\tau^{-}$\end{tiny}}}
\rput(2,-1.1){\psframebox*{\begin{tiny}$\tau$\end{tiny}}}
\rput(5,-1){\psframebox*{\begin{tiny}$\tau'$\end{tiny}}}
\end{pspicture*}
,\\
\eta_{\gotha, \gotha^\star}=
\psset{unit=0.3}
\begin{pspicture*}(-.5,-1.5)(9.5,2.3)
\psset{linecolor=red}
\psset{linewidth=1pt}
\psline{-}(0,2)(0,0)
\psline{-}(9,2)(9,0)
\psset{linecolor=black}
\psdots(0,2)(9,2)
\psdots(0,0)(9,0)
\psdots[dotstyle=+](1,2)(2,2)(3,2)(4,2)(6,2)(7,2)(5,2)(8,2)
\psdots[dotstyle=+](2,0)(3,0)(4,0)(5,0)(6,0)(7,0)(1,0)(8,0)
\rput(1,-.92){\psframebox*{\begin{tiny}$\tau^{-}$\end{tiny}}}
\rput(2,-1.1){\psframebox*{\begin{tiny}$\tau$\end{tiny}}}
\rput(5,-1){\psframebox*{\begin{tiny}$\tau'$\end{tiny}}}
\end{pspicture*}.
\end{eqnarray*}
\footnotetext{We give special shape to the arc linking nodes $5$ and $8$ here to remind the reader that this arc is obtained by ``pulling the strings".}

\textbf{Example 2:} We take $\psset{unit=0.3}
\gotha =
\begin{pspicture*}(-.5,-0.2)(9.5,2)
\psset{linewidth=1pt}
\psset{linecolor=red}
\psline{-}(2,0)(2,2)
\psline{-}(7,0)(7,2)
\psbezier{-}(1,0)(1,1.5)(-2,1.5)(-2,0)
\psbezier{-}(8,0)(8,1.5)(11,1.5)(11,0)
\psbezier{-}(3,0)(3,1.5)(6,1.5)(6,0)
\psarc{-}(4.5,0){0.5}{0}{180}
\psarc{-}(-0.5,0){0.5}{0}{180}
\psarc{-}(9.5,0){0.5}{0}{180}
\psset{linecolor=black}
\psdots(0,0)(1,0)(2,0)(3,0)(4,0)(9,0)(8,0)(7,0)(6,0)(5,0)
\end{pspicture*}$,
then the arc $\delta_\gotha$ has to be the one connecting nodes $1$ and $8$ through the imaginary boundary at $x=-1/2$ and $x = g-1/2$. We are in Case (b), so $\tau' = \tau^-$ is the node $1$.  We have
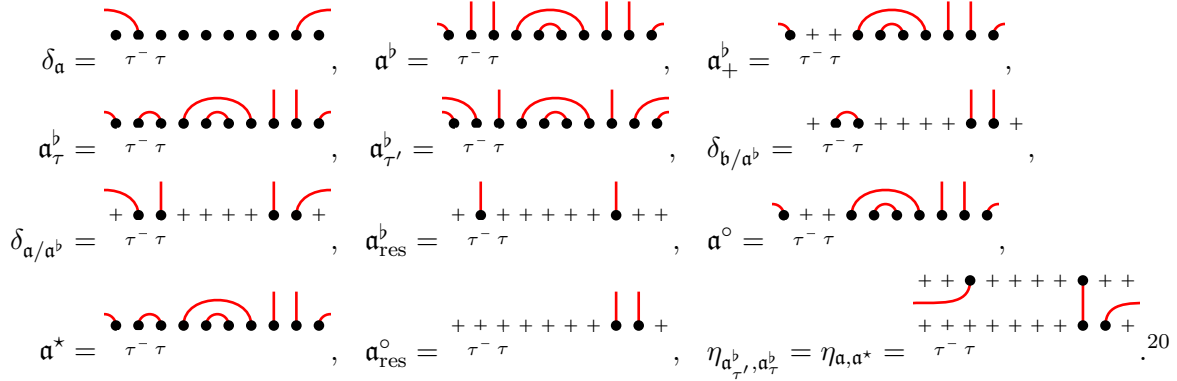
\begin{eqnarray*}
\psset{unit=0.3}
\delta_\gotha =
\begin{pspicture*}(-.5,-1.5)(9.5,2)
\psset{linewidth=1pt}
\psset{linecolor=red}
\psbezier{-}(1,0)(1,1.5)(-2,1.5)(-2,0)
\psbezier{-}(8,0)(8,1.5)(11,1.5)(11,0)
\psset{linecolor=black}
\psdots(0,0)(1,0)(2,0)(3,0)(4,0)(9,0)(8,0)(7,0)(6,0)(5,0)
\rput(1,-.92){\psframebox*{\begin{tiny}$\tau^{-}$\end{tiny}}}
\rput(2,-1.1){\psframebox*{\begin{tiny}$\tau$\end{tiny}}}
\end{pspicture*},
&
\psset{unit=0.3}
\gotha^\flat =
\begin{pspicture*}(-.5,-1.5)(9.5,1.5)
\psset{linewidth=1pt}
\psset{linecolor=red}
\psline{-}(2,0)(2,1.5)
\psline{-}(7,0)(7,1.5)
\psline{-}(1,0)(1,1.5)
\psline{-}(8,0)(8,1.5)
\psbezier{-}(3,0)(3,1.5)(6,1.5)(6,0)
\psarc{-}(4.5,0){0.5}{0}{180}
\psarc{-}(-0.5,0){0.5}{0}{180}
\psarc{-}(9.5,0){0.5}{0}{180}
\psset{linecolor=black}
\psdots(0,0)(1,0)(2,0)(3,0)(4,0)(9,0)(8,0)(7,0)(6,0)(5,0)
\rput(1,-.92){\psframebox*{\begin{tiny}$\tau^{-}$\end{tiny}}}
\rput(2,-1.1){\psframebox*{\begin{tiny}$\tau$\end{tiny}}}
\end{pspicture*}
,&
\psset{unit=0.3}
\gotha^\flat_+ =
\begin{pspicture*}(-.5,-1.5)(9.5,1.5)
\psset{linewidth=1pt}
\psset{linecolor=red}
\psline{-}(8,0)(8,1.5)
\psline{-}(7,0)(7,1.5)
\psbezier{-}(3,0)(3,1.5)(6,1.5)(6,0)
\psarc{-}(4.5,0){0.5}{0}{180}
\psarc{-}(-0.5,0){0.5}{0}{180}
\psarc{-}(9.5,0){0.5}{0}{180}
\psset{linecolor=black}
\psdots(0,0)(3,0)(4,0)(9,0)(8,0)(7,0)(6,0)(5,0)
\psdots[dotstyle=+](1,0)(2,0)
\rput(1,-.92){\psframebox*{\begin{tiny}$\tau^{-}$\end{tiny}}}
\rput(2,-1.1){\psframebox*{\begin{tiny}$\tau$\end{tiny}}}
\end{pspicture*}
,
\\
\psset{unit=0.3}
\gotha^\flat_\tau =
\begin{pspicture*}(-.5,-1.5)(9.5,1.5)
\psset{linewidth=1pt}
\psset{linecolor=red}
\psline{-}(7,0)(7,1.5)
\psline{-}(8,0)(8,1.5)
\psbezier{-}(3,0)(3,1.5)(6,1.5)(6,0)
\psarc{-}(4.5,0){0.5}{0}{180}
\psarc{-}(-0.5,0){0.5}{0}{180}
\psarc{-}(9.5,0){0.5}{0}{180}
\psarc{-}(1.5,0){0.5}{0}{180}
\psset{linecolor=black}
\psdots(0,0)(1,0)(2,0)(3,0)(4,0)(9,0)(8,0)(7,0)(6,0)(5,0)
\rput(1,-.92){\psframebox*{\begin{tiny}$\tau^{-}$\end{tiny}}}
\rput(2,-1.1){\psframebox*{\begin{tiny}$\tau$\end{tiny}}}
\end{pspicture*},
&
\psset{unit=0.3}
\gotha^\flat_{\tau'} =
\begin{pspicture*}(-.5,-1.5)(9.5,1.5)
\psset{linewidth=1pt}
\psset{linecolor=red}
\psline{-}(7,0)(7,1.5)
\psline{-}(2,0)(2,1.5)
\psbezier{-}(3,0)(3,1.5)(6,1.5)(6,0)
\psarc{-}(4.5,0){0.5}{0}{180}
\psarc{-}(-0.5,0){0.5}{0}{180}
\psarc{-}(9.5,0){0.5}{0}{180}
\psbezier{-}(1,0)(1,1.5)(-2,1.5)(-2,0)
\psbezier{-}(8,0)(8,1.5)(11,1.5)(11,0)
\psset{linecolor=black}
\psdots(0,0)(1,0)(2,0)(3,0)(4,0)(9,0)(8,0)(7,0)(6,0)(5,0)
\rput(1,-.92){\psframebox*{\begin{tiny}$\tau^{-}$\end{tiny}}}
\rput(2,-1.1){\psframebox*{\begin{tiny}$\tau$\end{tiny}}}
\end{pspicture*},
&
\psset{unit=0.3}
\delta_{\gothb/\gotha^\flat} =
\begin{pspicture*}(-.5,-1.5)(9.5,1.5)
\psset{linewidth=1pt}
\psset{linecolor=red}
\psline{-}(7,0)(7,1.5)
\psline{-}(8,0)(8,1.5)
\psarc{-}(1.5,0){0.5}{0}{180}
\psset{linecolor=black}
\psdots(8,0)(7,0)(1,0)(2,0)
\psdots[dotstyle=+](0,0)(3,0)(4,0)(5,0)(6,0)(9,0)
\rput(1,-.92){\psframebox*{\begin{tiny}$\tau^{-}$\end{tiny}}}
\rput(2,-1.1){\psframebox*{\begin{tiny}$\tau$\end{tiny}}}
\end{pspicture*}
,\\
\psset{unit=0.3}
\delta_{\gotha/\gotha^\flat} =
\begin{pspicture*}(-.5,-1.5)(9.5,1.5)
\psset{linewidth=1pt}
\psset{linecolor=red}
\psline{-}(2,0)(2,1.5)
\psline{-}(7,0)(7,1.5)
\psbezier{-}(1,0)(1,1.5)(-2,1.5)(-2,0)
\psbezier{-}(8,0)(8,1.5)(11,1.5)(11,0)
\psset{linecolor=black}
\psdots(8,0)(7,0)(1,0)(2,0)
\psdots[dotstyle=+](0,0)(3,0)(4,0)(5,0)(6,0)(9,0)
\rput(1,-.92){\psframebox*{\begin{tiny}$\tau^{-}$\end{tiny}}}
\rput(2,-1.1){\psframebox*{\begin{tiny}$\tau$\end{tiny}}}
\end{pspicture*}
,&
\psset{unit=0.3}
\gotha^\flat_\res =
\begin{pspicture*}(-.5,-1.5)(9.5,1.5)
\psset{linewidth=1pt}
\psset{linecolor=red}
\psline{-}(1,0)(1,1.5)
\psline{-}(7,0)(7,1.5)
\psset{linecolor=black}
\psdots(1,0)(7,0)
\psdots[dotstyle=+](2,0)(3,0)(4,0)(0,0)(6,0)(5,0)(9,0)(8,0)
\rput(1,-.92){\psframebox*{\begin{tiny}$\tau^{-}$\end{tiny}}}
\rput(2,-1.1){\psframebox*{\begin{tiny}$\tau$\end{tiny}}}
\end{pspicture*}
,&
\psset{unit=0.3}
\gotha^\circ =
\begin{pspicture*}(-.5,-1.5)(9.5,1.5)
\psset{linewidth=1pt}
\psset{linecolor=red}
\psline{-}(8,0)(8,1.5)
\psline{-}(7,0)(7,1.5)
\psbezier{-}(3,0)(3,1.5)(6,1.5)(6,0)
\psarc{-}(4.5,0){0.5}{0}{180}
\psarc{-}(-0.5,0){0.5}{0}{180}
\psarc{-}(9.5,0){0.5}{0}{180}
\psset{linecolor=black}
\psdots(0,0)(3,0)(4,0)(9,0)(8,0)(7,0)(6,0)(5,0)
\psdots[dotstyle=+](1,0)(2,0)
\rput(1,-.92){\psframebox*{\begin{tiny}$\tau^{-}$\end{tiny}}}
\rput(2,-1.1){\psframebox*{\begin{tiny}$\tau$\end{tiny}}}
\end{pspicture*},\\
\psset{unit=0.3}
\gotha^\star =
\begin{pspicture*}(-.5,-1.5)(9.5,1.5)
\psset{linewidth=1pt}
\psset{linecolor=red}
\psline{-}(8,0)(8,1.5)
\psline{-}(7,0)(7,1.5)
\psbezier{-}(3,0)(3,1.5)(6,1.5)(6,0)
\psarc{-}(4.5,0){0.5}{0}{180}
\psarc{-}(1.5,0){0.5}{0}{180}
\psarc{-}(-0.5,0){0.5}{0}{180}
\psarc{-}(9.5,0){0.5}{0}{180}
\psset{linecolor=black}
\psdots(0,0)(1,0)(2,0)(3,0)(4,0)(9,0)(8,0)(7,0)(6,0)(5,0)
\rput(1,-.92){\psframebox*{\begin{tiny}$\tau^{-}$\end{tiny}}}
\rput(2,-1.1){\psframebox*{\begin{tiny}$\tau$\end{tiny}}}
\end{pspicture*}
,&
\psset{unit=0.3}
\gotha^\circ_\res =
\begin{pspicture*}(-.5,-1.5)(9.5,1.5)
\psset{linewidth=1pt}
\psset{linecolor=red}
\psline{-}(7,0)(7,1.5)
\psline{-}(8,0)(8,1.5)
\psset{linecolor=black}
\psdots(8,0)(7,0)
\psdots[dotstyle=+](1,0)(2,0)(3,0)(4,0)(5,0)(6,0)(0,0)(9,0)
\rput(1,-.92){\psframebox*{\begin{tiny}$\tau^{-}$\end{tiny}}}
\rput(2,-1.1){\psframebox*{\begin{tiny}$\tau$\end{tiny}}}
\end{pspicture*}
,&
\psset{unit=0.3}
\eta_{\gotha_{\tau'}^\flat, \gotha_{\tau}^\flat}=\eta_{\gotha, \gotha^\star}=
\begin{pspicture*}(-.5,-1.5)(9.5,2.3)
\psset{linecolor=red}
\psset{linewidth=1pt}
\psline{-}(7,2)(7,0)
\psbezier(2,2)(2,0)(-2,2)(-2,0)
\psbezier(12,2)(12,0)(8,2)(8,0)
\psset{linecolor=black}
\psdots(7,0)(8,0)(7,2)(2,2)
\psdots[dotstyle=+](0,2)(1,2)(5,2)(3,2)(4,2)(6,2)(8,2)(9,2)
\psdots[dotstyle=+](0,0)(1,0)(2,0)(3,0)(4,0)(5,0)(6,0)(9,0)
\rput(1,-.92){\psframebox*{\begin{tiny}$\tau^{-}$\end{tiny}}}
\rput(2,-1.1){\psframebox*{\begin{tiny}$\tau$\end{tiny}}}
\end{pspicture*}
.\footnotemark
\end{eqnarray*}
\footnotetext{When either $\tau$ or $\tau^-$ is connected to a semi-line, a lot of the new periodic semi-meanders constructed are either ``simple" or ``similar" to $\gotha$.}

Using the discussion in Subsection~\ref{SS:decomposition of semi-meanders}, we see that the morphism $\Res_\gotha \circ \Gys_\gothb$ is the composition of the following diagram from the top-left to the bottom through first $\Gys_{\gothb_\res}$ and then all the way to the right and then all the way downwards, and finally through $\pi_{\delta_{\gotha/\gotha^\flat}, !}$ and $\Res_{\gotha^\flat_\res}$.
\begin{equation}
\label{E:commutative diagram case iv}
\xymatrix{
H^{d-2r}\big(\Sh(G_{\ttS_\gothb, \ttT_\gothb})\big) \ar[d]_{\Gys_{\gothb_\res}}
\\
H^{d-2}\big(\Sh(G_{\ttS_{\delta_\gothb}, \ttT_{\delta_\gothb}}) \big) \ar[r]^-{\pi_{\delta_\gothb}^*}
\ar[d]_{\mathrm{Restr.}}
&
H^{d-2}\big(\Sh(G_{\ttS, \ttT})_{\delta_\gothb}\big)
\ar[r]^-{\mathrm{Gysin}}
\ar[d]^{\mathrm{Restr.}}
&
H^{d}\big(\Sh(G_{\ttS, \ttT})\big)
\ar[d]^{\mathrm{Restr.}}
\\
H^{d-2}\big(\Sh(G_{\ttS_{\delta_\gothb}, \ttT_{\delta_\gothb}})_{\gotha^\flat_+} \big)
\ar[d]^-{\pi_{\gotha_+^\flat,!}}
\ar[r]^{\pi_{\delta_\gothb}^*}
&
H^{d-2}\big(\Sh(G_{\ttS, \ttT})_{\gotha^\flat_\tau}\big)
\ar[r]^-{\mathrm{Gysin}}
&
H^{d}\big(\Sh(G_{\ttS, \ttT})_{\gotha^\flat}\big)
\ar[d]^{\pi_{\gotha^\flat,!}}
\\
H^{d-2r'-2}
\big(\Sh(G_{\ttS_{\gotha^\flat_\tau}, \ttT_{\gotha^\flat_\tau}})\big)
\ar[r]^-{\pi^*_{\delta_{\gothb/\gotha^\flat}}}
\ar[dr]_{p^{y/2}\eta^\star_{\gotha^\flat_{\tau'}, \gotha^\flat_{\tau}}}
\ar[d]^{\Res_{\gotha^\circ_\res}}
&
H^{d-2r'-2}\big(\Sh(G_{\ttS_{\gotha^\flat}, \ttT_{\gotha^\flat}})_{\delta_{\gothb/\gotha^\flat}}
\big)
\ar[r]^-{\mathrm{Gysin}}
&
H^{d-2r'}
\big(\Sh(G_{\ttS_{\gotha^\flat}, \ttT_{\gotha^\flat}})
\big)
\ar[d]^{\mathrm{Restr.}}
\\
H^{d-2r}(\Sh(G_{\ttS_{\gotha^\star}, \ttT_{\gotha^\star}}))
\ar[dr]_{p^{(x+y)/2}\eta^\star_{\gotha, \gotha^\star}}
&
H^{d-2r'-2}\big(\Sh(G_{\ttS_{\gotha^\flat_{\tau'}}, \ttT_{\gotha^\flat_{\tau'}}})
\big)
\ar[d]^{\Res_{\gotha^\flat_\res}}
&
H^{d-2r'}
\big(\Sh(G_{\ttS_{\gotha^\flat}, \ttT_{\gotha^\flat}})_{\delta_{\gotha/\gotha^\flat}}
\big)
\ar[l]_-{\pi_{\delta_{\gotha/\gotha^\flat},!}}
\\
&
H^{d-2r}\big(\Sh(G_{\ttS_\gotha, \ttT_\gotha})\big).
}
\end{equation}
Here, the numbers $x,y$ and the link morphisms $\eta^\star_{\gotha^\flat_{\tau'}, \gotha^\flat_\tau}$ and $\eta^\star_{\gotha, \gotha^\star}$ will be defined explicitly later. For simplicity, we have omitted the Tate twists from the notation, and each cohomology group $H^{a}(\star)$ should be understood as $H^a(\star)(b)$ with $a-2b=d-2r$; for instance, $H^{d-2r'-2}\big(\Sh(G_{\ttS_{\gotha^\flat_{\tau'}}, \ttT_{\gotha^\flat_{\tau'}}})
\big)$ should be understood as $H^{d-2r'-2}\big(\Sh(G_{\ttS_{\gotha^\flat_{\tau'}}, \ttT_{\gotha^\flat_{\tau'}}})
\big)(r-r'-1)$.

We now explain the commutativity of this diagram.  The top two squares are commutative because the corresponding morphisms on the Shimura varieties form Cartesian squares.
The commutativity of the middle rectangle follows from that of \eqref{E:cartesian decomposition} (applied with our $\gotha^\flat_\tau$ and $\gotha^\flat$ being the $\gotha$ and $\gotha^\flat$ therein, respectively).
The commutativity of the lower trapezoid will follow from applying Subsection~\ref{S:case r=1 link} (applied to the Shimura variety $\Sh(G_{\ttS_{\gotha^\flat}, \ttT_{\gotha^\flat}})$ with our $\delta_{\gotha/\gotha^\flat}$ and $\delta_{\gothb/\gotha^\flat}$ being the $\gotha$ and $\gothb$ therein) once we have clarified the meaning of $y$ and $\eta^\star_{\gotha^\flat_{\tau'}, \gotha^\flat_\tau}$  in \ref{SS:proof case iv link morphism}.
Finally, the commutativity of  the bottom parallelogram will be justified in Subsection~\ref{SS:proof case iv commutative diagram} and  Lemma~\ref{L:restriction gysin compatibility} later.

To sum up, the morphism $\Res_\gotha \circ \Gys_\gothb  $ will be the composition of \eqref{E:commutative diagram case iv} from the top-left to the bottom by first going all the way down and through $\eta^\star_{\gotha, \gotha^\star}$.
This gives the following equality
\begin{equation}
\label{E:proof case iv}
\Res_\gotha \circ \Gys_\gothb = p^{(x+y)/2} \eta^\star_{\gotha , \gotha^\star} \circ \Res_{\gotha^\star} \circ \Gys_{\gothb_\res},
\end{equation}
using which we will complete the inductive proof of Theorem~\ref{T:intersection combinatorics} in Case (iv), as we shall explain in \ref{SS:proof of case iv finish}.

\subsection{Link morphism $\eta_{\gotha^{\flat}_{\tau'}, \gotha^{\flat}_{\tau}}$}
\label{SS:proof case iv link morphism}
Now, let us get to the details, starting with the link morphism associated with $\eta_{\gotha^{\flat}_{\tau'}, \gotha^{\flat}_{\tau}}$. We distinguish the two cases:
\begin{itemize}
\item[\emph{Case (a).}] Suppose that the left node of $\delta_{\gotha}$ is $\tau$. Example 1 above falls into this case. 
All the curves in the  link $\eta_{\gotha^{\flat}_{\tau'}, \gotha^{\flat}_{\tau}}$ are semi-lines, except for one which turns to the right and we denote by $\xi$. The curve $\xi$ sends $\tau^-$ to $\tau'$. 
 By Theorem~\ref{T:GO description}(2), there exists a link morphism $\eta_{\gotha^{\flat}_{\tau'}, \gotha^{\flat}_{\tau}, \sharp}$ with indentation degree $\ell(\delta_\gotha)-\ell(\delta_{\gothb})$ and shift $\boldsymbol t_{\gotha^{\flat}_{\tau'}}\boldsymbol t_{\gotha^{\flat}_{\tau}}^{-1}$ (and also a link morphism on the local system as in Theorem~\ref{T:GO description}(2)(d)), which fits into the following commutative diagram
\[
\xymatrix{
\Sh(G_{\ttS_{\gotha^{\flat}}, \ttT_{\gotha^\flat}})_{\delta_{\gotha}} \ar[d]_{\pi_{\delta_{\gotha}}}
&
\Sh(G_{\ttS_{\gotha^\flat}, \ttT_{\gotha^{\flat}}})_{\{\tau',\tau\}}
\ar@{_{(}->}[l]
\ar@{^{(}->}[r]
\ar@{->}[ld]_\cong
&
\Sh(G_{\ttS_{\gotha^\flat}, \ttT_{\gotha^\flat}})_{\delta_{\gothb}}
\ar[d]^{\pi_{\delta_\gothb}}
\\
\Sh(G_{\ttS_{\gotha^{\flat}_{\tau'}}, \ttT_{\gotha^{\flat}_{\tau'}}}) \ar[rr]^{\eta_{\gotha^{\flat}_{\tau'},\gotha^{\flat}_{\tau},  \sharp}}
&&
\Sh(G_{\ttS_{\gotha^{\flat}_{\tau}}, \ttT_{\gotha^{\flat}_{\tau}}}).
}
\]
By Theorem~\ref{T:GO description}(2)(c) $\eta_{\gotha^{\flat}_{\tau'},\gotha^{\flat}_{\tau},  \sharp}$ is finite flat of degree $p^y$ with $y:=v(\eta_{\gotha^{\flat}_{\tau'},\gotha^{\flat}_{\tau},  \sharp})=\ell(\delta_{\gotha})+\ell(\delta_{\gothb})$.
Let $\eta_{\gotha^{\flat}_{\tau'},\gotha^{\flat}_{\tau}}^\star: H^{d-2r'-2}\big(\Sh(G_{\ttS_{\gotha^\flat_\tau}, \ttT_{\gotha^\flat_\tau}})\big) \to H^{d-2r'-2}\big(\Sh(G_{\ttS_{\gotha^\flat_{\tau'}}, \ttT_{\gotha^\flat_{\tau'}}})\big)$ denote the induced link homomorphism on the cohomology groups.
By the same argument as in \ref{S:case r=1 link}, we see that the trapezoid in the diagram \eqref{E:commutative diagram case iv} is commutative.

\item[\emph{Case (b).}]
Suppose now that the right node of $\delta_{\gotha}$ is $\tau^-$. Example 2 above falls into this case. Then the only genuine curve in the link $\eta_{\gotha^{\flat}_{\tau'},\gotha^{\flat}_{\tau}}$ is turning to the left with displacement $y=\ell(\delta_{\gotha})+\ell(\delta_{\gothb})$.
Let $\eta_{\gotha_{\tau}^\flat, \gotha_{\tau'}^{\flat}}$ be the inverse link of $\eta_{\gotha^{\flat}_{\tau'},\gotha^{\flat}_{\tau}}$. 
Applying the discussion in \emph{Case (a)} to $\eta_{\gotha^{\flat}_{\tau},\gotha^{\flat}_{\tau'}}$, one gets a link morphism $\eta_{\gotha^{\flat}_{\tau},\gotha^{\flat}_{\tau'}, \sharp}: \Sh(G_{\ttS_{\tau}^{\flat}, \ttT_{\tau}^{\flat}})\ra \Sh(G_{\ttS_{\tau'}^{\flat}, \ttT_{\tau'}^{\flat}})$ of indentation degree $\ell(\delta_{\gothb})-\ell(\delta_{\gotha})$ and shift $\boldsymbol t_{\gotha^{\flat}_{\tau}}\boldsymbol t_{\gotha^{\flat}_{\tau'}}^{-1}$.
By Lemma~\ref{L:inverse of link morphism}, we get a well-defined  link morphism on the cohomology groups
\begin{equation}\label{E:link-traces-1}
\eta^{\star}_{\gotha^{\flat}_{\tau'},\gotha^{\flat}_{\tau}}=(\eta^{\star}_{\gotha^{\flat}_{\tau},\gotha^{\flat}_{\tau'}})^{-1}=p^{-y/2}\Tr_{\eta_{\gotha^{\flat}_{\tau},\gotha^{\flat}_{\tau'},\sharp}}\colon
 H^{d-2r'-2}_{\et}\big(\Sh(G_{  \ttS_{\gotha^{\flat}_{\tau}}, \ttT_{\gotha^{\flat}_{\tau}}})\big) \ra H^{d-2r'-2}_{\et}\big(\Sh(G_{  \ttS_{\gotha^{\flat}_{\tau'}}, \ttT_{\gotha^{\flat}_{\tau'}}})\big)
 \end{equation}
of indentation degree $\ell(\gotha)-\ell(\gothb)$ associated to the link $\eta_{\gotha^{\flat}_{\tau'},\gotha^{\flat}_{\tau}}$ and shift $\boldsymbol t_{\gotha^{\flat}_{\tau'}}\boldsymbol t_{\gotha^{\flat}_{\tau}}^{-1}$.
Now the argument as in \ref{S:case r=1 link} proves the commutativity of the trapezoid in the diagram \eqref{E:commutative diagram case iv}.
\end{itemize}

\subsection{Commutativity of the parallelogram in \eqref{E:commutative diagram case iv}}
\label{SS:proof case iv commutative diagram}
We continue the discussion above by separating the two cases.

\begin{itemize}

\item[\emph{Case (a).}]
Consider the $(r-r'-1)$-th iterated $\PP^1$-bundle $\pi_{\gotha^{\flat}_{\res}}\colon \Sh(G_{ \ttS_{\gotha^{\flat}_{\tau'}}, \ttT_{\gotha^{\flat}_{\tau'}}})_{\gotha^\flat_{\res}}\ra  \Sh(G_{ \ttS_\gotha, \ttT_\gotha})$. 
By  applying repeatedly \cite[Proposition~7.17]{tian-xiao1} and Construction~\ref{S:from unitary to quaternionic}, one produces a commutative diagram:
\begin{equation}\label{E:factorization-link-bundles}
\xymatrix{
\Sh(G_{ \ttS_{\gotha^{\flat}_{\tau'}}, \ttT_{\gotha^{\flat}_{\tau'}}})_{\gotha^\flat_{\res}}\ar[r]^-{\pi_1^{\flat}}\ar[d]^{\eta_{\gotha^{\flat}_{\tau'},\gotha^{\flat}_{\tau},\sharp}} 
& X_1 \ar[r]^{\pi^{\flat}_2}\ar[d]^{\eta_{1, \sharp}} & X_2\ar[r]\ar[d]^{\eta_{2,\sharp}} & \cdots\ar[r] & X_{r-r'-2}\ar[r]^-{\pi^\flat_{r-r'-1}}\ar[d]^-{\eta_{r-r'-2,\sharp}} &
\Sh(G_{ \ttS_\gotha, \ttT_\gotha})\ar[d]^-{\eta_{\gotha,\gotha^{\star}, \sharp}}\\
\Sh(G_{ \ttS_{\gotha^{\flat}_{\tau}}, \ttT_{\gotha^{\flat}_{\tau}}})_{\gotha^\circ_{\res}}
\ar[r]^-{\pi^{\circ}_1} &Y_1 \ar[r]^{\pi^{\circ}_2} & Y_2\ar[r] &\cdots\ar[r]& Y_{r-r'-2} \ar[r]^-{\pi_{r-r'-1}^{\circ}}&
\Sh(G_{\ttS_{\gotha^{\star}},\ttT_{\gotha^{\star}}}),
}
\end{equation}
where  $\pi^{\flat}_i$ and $\pi^{\circ}_{i}$ are all $\PP^1$-fibrations,  the vertical arrows are link morphisms (associated to certain links),  and the composition of top (resp. bottom) horizontal arrows is $\pi_{\gotha^{\flat}_{\res}}$ (resp. $\pi_{\gotha_{\res}^{\circ}}$). 
There exist at the same time link morphisms $\eta_i^\sharp$ and $\eta_{\gotha,\gotha^{\circ}}^{\sharp}$ on the \'etale local systems satisfying  a similar commutative diagram.
We explain now how to construct $\eta_{1,\sharp}\colon X_1\ra Y_1$, one chooses a basic arc $\delta_{\gothc}$ in $\gotha^{\flat}_{\res}$. 
Let $\gotha_{\res, 1}^{\flat}$ be the periodic semi-meander obtained by removing $\delta_{\gothc}$ from $\gotha_{\res}^{\flat}$ and replacing the end-nodes of $\delta_\gothc$ by  plus signs, and $\gotha^{\flat}_{\tau',1}$ be the periodic semi-meander obtained by removing from $\gotha_{\tau'}^\flat$ the semi-lines at the end-nodes of $\delta_\gothc$ and adjoining $\delta_\gothc$.
Put  $X_1:=\Sh(G_{ \ttS_{\gotha^{\flat}_{\tau', 1}}, \ttT_{\gotha^{\flat}_{\tau',1}}})_{\gotha_{\res,1}^{\flat}}$, and  denote by 
$$
\pi^{\flat}_1\colon \Sh(G_{ \ttS_{\gotha^{\flat}_{\tau'}}, \ttT_{\gotha^{\flat}_{\tau'}}})_{\gotha^\flat_{\res}}\longto X_1
$$
the $\PP^1$-fibration given by the arc $\delta_\gothc$.
Let $\delta_{\gothc^\circ}$denote the arc $\eta_{\gotha^{\flat}_{\tau'}, \gotha^{\flat}_{\tau}}(\delta_\gothc)$ obtained by extending $\delta_\gothc$ using the curves of $\eta_{\gotha^{\flat}_{\tau'}, \gotha^{\flat}_{\tau}}$ at the end-nodes of $\delta_\gothc$. This $\delta_{\gothc^\circ}$ is a basic arc in $\gotha^{\circ}_{\res}$. 
We define periodic semi-meanders $\gotha^{\flat}_{\tau, 1}$  and $\gotha_{\res,1}^{\circ}$ in the same way as $\gotha^{\flat}_{\tau', 1}$ and $\gotha_{\res,1}^{\flat}$ with $\delta_\gothc$ replaced by $\delta_{\gothc^{\circ}}$. Then we have a $\PP^1$-fibration
 $$
 \pi_1^{\circ}\colon \Sh(G_{ \ttS_{\gotha^{\flat}_{\tau}}, \ttT_{\gotha^{\flat}_{\tau}}})_{\gotha^\circ_{\res}} \longto Y_1:=\Sh(G_{ \ttS_{\gotha^{\flat}_{\tau, 1}}, \ttT_{\gotha^{\flat}_{\tau,1}}})_{\gotha_{\res,1}^{\circ}}.
 $$ 
If $\eta_1: \ttS_{\gotha^{\flat}_{\tau', 1}}\ra \ttS_{\gotha^{\flat}_{\tau, 1}}$ denotes the link induced by $\eta_{\gotha^{\flat}_{\tau'},\gotha^{\flat}_{\tau}}$, then 
 \cite[Proposition~7.17]{tian-xiao1} and Construction~\ref{S:from unitary to quaternionic} implies the existence of the link morphism $\eta_{1,\sharp}$ which fits into the left commutative square of \eqref{E:factorization-link-bundles}. 
 This finishes the construction of $X_1$ and $Y_1$.
The induced link $\eta_1$ has the same property as $\eta_{\gotha^{\flat}_{\tau'},\gotha^{\flat}_{\tau}}$, namely all the curves of $\eta_1$ are semi-lines except possibly for one turning to the right. 
The rest of \eqref{E:factorization-link-bundles} can be constructed inductively in a similar way.

Since we require the diagram \eqref{E:factorization-link-bundles} to be commutative, by Remark~\ref{R:composition of shifts}, the link morphism $\eta_{\gotha,\gotha^{\star}, \sharp}$ has shift 
\[
\boldsymbol t_{\gotha^{\flat}_{\tau'}, \gotha} \cdot \boldsymbol t_{\gotha^{\flat}_{\tau}, \gotha^\star}^{-1} \cdot \textrm{(shift of }\eta_{\gotha^\flat_{\tau'}, \gotha^\flat_{\tau}}) = 
\boldsymbol t_{\gotha} \boldsymbol t_{\gotha^\star}^{-1}.
\]
Moreover, the indentation degree of $\eta_{\gotha,\gotha^{\star}, \sharp}$ is $\ell(\delta_{\gotha})-\ell(\delta_{\gothb})+\ell(\gotha^{\flat}_{\res})-\ell(\gotha_{\res}^{\circ})$ if $\gothp$ splits in $E/F$ and degree $0$ if $\gothp$ is inert in $E/F$.
Note  that  even though each $\eta_{i,\sharp}$ is not unique (since there are many ways to choose a basic arc of $\gotha^{\flat}_{\res}$ for instance), the final link morphism $\eta_{\gotha, \gotha^\star, \sharp}$ is uniquely determined by the uniqueness of link morphisms.
  By \cite[Proposition~7.17(3)]{tian-xiao1} and Construction~\ref{S:from unitary to quaternionic},   $\eta_{\gotha,\gotha^{\star},\sharp}$ is  finite flat of degree $p^{v(\eta_{\gotha, \gotha^{\star}})}$.
We have thus the  normalized link morphisms $\eta^{\star}_{\gotha^{\flat}_{\tau'},\gotha^{\flat}_{\tau}}$ and $\eta^{\star}_{\gotha,\gotha^{\star}}$ on the corresponding cohomology groups as defined in \eqref{E:normalized-link} induced by 
$(\eta_{\gotha^{\flat}_{\tau'},\gotha^{\flat}_{\tau}, \sharp}, \eta^\sharp_{\gotha^{\flat}_{\tau'},\gotha^{\flat}_{\tau}})$ and $(\eta_{\gotha,\gotha^{\star}}^{\sharp}, \eta_{\gotha, \gotha^{\star}}^{\sharp})$ respectively.

\item[\emph{Case (b)}] Suppose now that the right node of $\delta_{\gotha}$ is $\tau^-$. 
Applying the discussion in \emph{Case (a)} to the inverse link $\eta_{\gotha^{\flat}_{\tau},\gotha^{\flat}_{\tau'}}$, one gets a link morphism $\eta_{\gotha^{\star}, \gotha, \sharp}: \Sh(G_{\ttS_{\gotha^{\star}},\ttT_{\gotha^\star}})\ra \Sh(G_{\ttS_{\gotha}, \ttT_{\gotha}})$ associated to the inverse link of $\eta_{\gotha,\gotha^{\star}}$ of indentation degree $\ell(\delta_{\gothb})-\ell(\delta_{\gotha})+\ell(\gotha_{\res}^{\circ})-\ell(\gotha_{\res}^{\flat})$,  and shift $\boldsymbol t_{\gotha}^{-1}\boldsymbol t_{\gotha^\star}$.
By Lemma~\ref{L:inverse of link morphism}, we get a well-defined  link morphism on the cohomology groups
\begin{equation}\label{E:link-traces-2}
\eta^{\star}_{\gotha,\gotha^{\star}}=( \eta^{\star}_{\gotha^{\star},\gotha})^{-1}=p^{v(\eta_{\gotha,\gotha^{\star}})/2}\Tr_{\eta_{\gotha^{\star},\gotha, \sharp}}\colon H^{d-2r}(\Sh_{G_{\ttS,\ttT}})\ra H^{d-2r}(\Sh_{G_{\ttS_{\gotha^{\star}},\ttT_{\gotha^{\star}}}})
\end{equation}
of indentation degree  $\ell(\delta_\gotha)-\ell(\delta_{\gothb})+\ell(\gotha^{\flat}_{\res})-\ell(\gotha_{\res}^{\circ})$ and shift $\boldsymbol t_{\gotha} \boldsymbol t_{\gotha^\star}^{-1}$ associated to the link $\eta_{\gotha,\gotha^{\star}}$.

\end{itemize}

\begin{lemma}\label{L:restriction gysin compatibility}
Under the above notation,  put $x=\ell(\gotha^{\flat}_{\res})-\ell(\gotha^{\circ}_{\res})$ and $y = \ell(\delta_\gotha) + \ell(\delta_\gothb)$. 
Then in both  {Case (a)} and {Case (b)} above,  one has a commutative diagram of cohomology groups:
\[
\xymatrix@C=40pt{
H_\et^{d-2r'-2}(\Sh(G_{  \ttS_{\gotha^{\flat}_{\tau}}, \ttT_{\gotha^{\flat}_{\tau}}}))
\ar[r]^-{\mathrm{Rest.}}
\ar[d]^{p^{y/2}\eta^\star_{\gotha^{\flat}_{\tau'},\gotha^{\flat}_{\tau}}}
&
H_\et^{d-2r'-2}(\Sh( G_{ \ttS_{\gotha^{\flat}_{\tau}}, \ttT_{\gotha^{\flat}_{\tau}}})_{\gotha^{\star}_{\res}})
\ar[r]^-{\pi_{\gotha_{\res}^{\circ},!}}
\ar[d]^{p^{y/2}\eta_{\gotha^{\flat}_{\tau'},\gotha^{\flat}_{\tau}}^{\star}}
&
H_\et^{d-2r}(\Sh( G_{\ttS_{\gotha^{\star}}, \ttT_{\gotha^{\star}}}))
\ar[d]^{p^{(x+y)/2}\eta_{\gotha, \gotha^{\star}}^\star}
\\
H_\et^{d-2r'-2}(\Sh(G_{\ttS_{\gotha^{\flat}_{\tau'}}, \ttT_{\gotha^{\flat}_{\tau'}}}))
\ar[r]^-{\mathrm{Rest.}}
&
H_\et^{d-2r'-2}(\Sh(G_{\ttS_{\gotha^{\flat}_{\tau'}}, \ttT_{\gotha^{\flat}_{\tau'}}})_{\gotha_{\res}^{\flat}})\ar[r]^-{\pi_{\gotha^{\flat}_{\res},!}}
&
H_\et^{d-2r}(\Sh(G_{\ttS_\gotha, \ttT_\gotha})).
}
\]
Note that the composite of the top (resp. bottom) two horizontal morphisms above is exactly    $\Res_{\gotha^{\circ}_{\res}}$ (resp.  $\Res_{\gotha^{\flat}_{\res}}$). So this verifies the commutativity of the parallelogram in \eqref{E:commutative diagram case iv}.
\end{lemma}

\begin{proof}
The commutativity of the left square is evident. We check the  commutativity of the right hand side square case by case.

Suppose first we are in \emph{Case (a)}, i.e. the left end-node of $\delta_{\gotha}$ is $\tau$. We distinguish three subcases:
\begin{enumerate}

\item[\emph{Case (a1)}] $\tau^-$ is linked to a semi-line in $\gotha$. 
Then both $\gotha^{\flat}_{\res}$ and $\gotha^{\circ}_{\res}$ contain no arcs. It follows that $x=0$,  and  $\pi_{\gotha^{\flat}_{\res}}$ and $\pi_{\gotha^{\circ}_{\res}}$ are isomorphisms. In this case, the commutativity of the right hand side square is trivial.

\item[\emph{Case (a2)}] $\tau^-$ is the left end-node of an arc in $\gotha$. Example 1 above falls into this case. It is easy to see that $x=y$, and that the link $\eta_{\gotha, \gotha^{\star}}$ contains only semi-lines. 
By \cite[Proposition~7.17(3)]{tian-xiao1} and Construction~\ref{S:from unitary to quaternionic}, $\eta_{\gotha, \gotha^{\star},\sharp}$ is an isomorphism. 
 Consider the commutative diagram  \eqref{E:factorization-link-bundles}. 
 Both top and bottom rows are factorizations of $(r-r'-1)$-th iterated $\PP^1$-bundles as in \eqref{E:factorization-iterated-bundle}. 
For each $1\leq i\leq r-r'-1 $, let $\xi_i'\in H^2_{\et}(\Sh_{K_p}\big(G_{ \ttS_{\gotha^{\flat}_{\tau'}}, \ttT_{\gotha^{\flat}_{\tau'}}})_{\overline \FF_p}, \overline \Q_{\ell}(1)\big)$ (resp. $\xi_i\in H^2_{\et}\big(\Sh_{K_p}(G_{ \ttS_{\gotha^{\flat}_{\tau}}, \ttT_{\gotha^{\flat}_{\tau}}})_{\overline \FF_p}, \overline \Q_{\ell}(1)\big)$) be the inverse image of the first Chern class of the tautological quotient line bundle of $\pi_{i}^{\flat}$ (resp. $\pi^{\circ}_i$) as considered in Subsection~\ref{S:etale-cohomology}. 
Note that the only curve in $\eta_{\gotha_{\tau'}^{\flat}, \gotha^{\flat}_{\tau}}$ links the \emph{left} end-node of an arc of $\gotha_{\res}^{\flat}$ to the \emph{left} end-node of an arc of $\gotha_{\res}^{\circ}$.
Then  by applying iteratively \cite[Proposition~7.17(3)]{tian-xiao1} and Construction~\ref{S:from unitary to quaternionic},  there exists a unique integer $i_0$ with $1\leq i_0\leq r-r'-1$ such that $\eta^{*}_{\gotha_{\tau'}^{\flat}, \gotha^{\flat}_{\tau}, \sharp}(\xi_{i_0})=p^{y}\xi_{i_0}'$, and $\eta^{*}_{\gotha_{\tau'}^{\flat}, \gotha^{\flat}_{\tau}, \sharp}(\xi_i)=\xi_{i}'$ for all $i\neq i_0$. 
  Let 
  $$
  z=\sum_{1\leq j\leq r-r'-1}\bigg(\sum_{1\leq i_1<\cdots <i_{j}\leq r-r'-1} \pi^{*}_{\gotha_{\res}^{\circ}}(z_{i_1,\dots, i_{j}})\cup \xi_{i_1}\cup \cdots \cup \xi_{i_j}\bigg) 
  $$
be an element of $H^{d-2r'-2}(\Sh(G_{\ttS_{\gotha^{\flat}_{\tau}}, \ttT_{\gotha^{\flat}_{\tau}}})_{\gotha_{\res}^{\circ}})$ with $z_{i_1, \dots, i_j}\in H^{d-2r'-2-2j}(\Sh(G_{\ttS_{\gotha^\star}, \ttT_{\gotha^{\star}}}))$. Then one has 
\begin{align*}
\pi_{\gotha_{\res}^{\flat}, !}(p^{y/2}\eta^{\star}_{\gotha_{\tau'}^{\flat},\gotha_{\tau}^{\flat} }(z))&= \pi_{\gotha_{\res}^{\flat}, !} \eta^{*}_{\gotha_{\tau'}^{\flat},\gotha_{\tau}^{\flat}, \sharp}(z)
\\
&=p^y \pi_{\gotha_{\res}^{\flat}, !}\Big(\sum\big(\eta^{*}_{\gotha_{\tau'}^{\flat},\gotha_{\tau}^{\flat}, \sharp}(\pi^{*}_{\gotha_{\res}^{\circ}}(z_{i_1, \dots, i_j}))\cup \xi'_{i_1}\cup \cdots \cup \xi'_{i_j} \big)\Big)
\\
&=p^y \pi_{\gotha_{\res}^{\flat}, !}\Big(\sum\big(\pi^{*}_{\gotha_{\res}^{\flat}}(\eta^{*}_{\gotha,\gotha^{\star}, \sharp}(z_{i_1, \dots, i_j}))\cup \xi'_{i_1}\cup \cdots \cup \xi'_{i_j}\big)\Big)\\
&=p^y\eta^{*}_{\gotha,\gotha^{\star}, \sharp}(z_{1, \dots, r-r'-1})=p^{(x+y)/2}\eta^{\star}_{\gotha,\gotha^{\star}}(\pi_{\gotha^{\circ}_{\res},!}(z)),
\end{align*}
where the forth and fifth equalities use the formula \eqref{E:simple-formula}. This shows the commutativity of the right square in the Lemma.

\item[\emph{Case (a3)}] $\tau^-$ is the right end-node of an arc in $\gotha$. Then $x=-y$ and  $\eta_{\gotha, \gotha^{\star}}$ contains only semi-lines.  Hence, $\eta_{\gotha, \gotha^{\star},\sharp}$ is an isomorphism as in \emph{Case (a2)}. We want to show 
$$\eta_{\gotha,\gotha^{\star}}^{\star}\circ\pi_{\gotha^{\circ}_{\res}, !}= \pi_{\gotha^{\flat}_{\res},!}\circ (p^{y/2}\eta^{\star}_{\gotha^{\flat}_{\tau'},\gotha^{\flat}_{\tau}}) .$$
The argument is quite similar to that of \emph{Case (a2)}.
 Let $\xi_i, \xi_i'$ be as defined  in \emph{Case (a2)}  for $1\leq i\leq r-r'-1$.
  Then by \cite[Proposition~7.17(3)]{tian-xiao1}, we have $\eta^{*}_{\gotha_{\tau'}^{\flat}, \gotha^{\flat}_{\tau}, \sharp}(\xi_i)=\xi_{i}'$ for all $1\leq i\leq r-r'-1$ (this differs from the situation of  \emph{Case (a2)} because the unique curve in $\eta_{\gotha_{\tau'}^{\flat}, \gotha^{\flat}_{\tau}}$ links the \emph{right} end-node of an arc of $\gotha_{\res}^{\flat}$ to the \emph{right} end-node of an arc of $\gotha_{\res}^{\circ}$). Then the rest of the computation is the same as in \emph{Case (a2)}. 
\end{enumerate}

Consider now \emph{Case (b)}, i.e. the right end-node of $\delta_{\gotha}$ is $\tau^-$. Symmetrically, we have three subcases:
\begin{enumerate}
\item[\emph{Case (b1)}] $\tau$ is linked to a semi-line in $\gotha$. Then as in \emph{Case (a1)}, 
we have $x=0$,  and  $\pi_{\gotha^{\flat}_{\res}}$ and $\pi_{\gotha^{\circ}_{\res}}$ are both isomorphisms. The commutativity of the right hand side square is trivial.

\item[\emph{Case (b2)}]   $\tau$ is the left end-node of an arc in $\gotha$. Then $x=-y$, and $\eta_{\gotha, \gotha^{\star}}$ contains only semi-lines. Hence, $\eta_{\gotha^{\star}, \gotha,\sharp}$ is an isomorphism as in \emph{Case (a2)}. By \eqref{E:link-traces-1} and \eqref{E:link-traces-2}, the desired commutativity is equivalent to 
$$
\Tr_{\eta_{\gotha^{\star},\gotha,\sharp}}\circ \pi_{\gotha^{\circ}_{\res}, !}=\pi_{\gotha^{\flat}_{\res},!}\circ \Tr_{\eta_{\gotha^{\flat}_{\tau},\gotha^{\flat}_{\tau'},\sharp}},
$$
which is an easy consequence of the compatibility of trace maps with composition.

\item[\emph{Case (b3)}] $\tau$ is the right end-node of an arc in $\gotha$. Then $x=y$, and $\eta_{\gotha^{\star}, \gotha,\sharp}$ is an isomorphism as in \emph{Case (a2)}. The desired commutativity is equivalent to 
\[
\pi_{\gotha^{\flat}_{\res}, !}\circ p^{y/2}(\eta^{\star}_{\gotha^{\flat}_{\tau},\gotha^{\flat}_{\tau'}})^{-1}=p^y(\eta_{\gotha^{\star}, \gotha}^{\star})^{-1}\circ \pi_{\gotha^{\circ}_{\res}, !}\quad \Longleftrightarrow \quad \eta_{\gotha^{\star}, \gotha}^{\star}\circ \pi_{\gotha^{\flat}_{\res}, !}= \pi_{\gotha^{\circ}_{\res}, !}\circ (p^{y/2}\eta^{\star}_{\gotha^{\flat}_{\tau},\gotha^{\flat}_{\tau'}}).
\]
Thus the situation is exactly the same as \emph{Case (a3)} above (except for switching the roles of $\Sh(G_{\ttS_{\gotha^{\flat}_{\tau'}}, \ttT_{\gotha_{\tau'}^{\flat}}})$ and $\Sh(G_{\ttS_{\gotha^{\flat}_{\tau}}, \ttT_{\gotha_{\tau}^{\flat}}})$), and we conclude by the same arguments. \qedhere
\end{enumerate}\end{proof}

\subsection{Finish of the proof in Case (iv)}
\label{SS:proof of case iv finish}
We are now in position to complete the inductive proof of Theorem~\ref{T:GO description} in Case (iv). We have shown the commutativity of the diagram~\eqref{E:commutative diagram case iv}, from which we deduce \eqref{E:proof case iv}:
\[
\Res_\gotha \circ \Gys_\gothb = p^{(x+y)/2} \eta^\star_{\gotha , \gotha^\star} \circ \Res_{\gotha^\circ} \circ \Gys_{\gothb_\res},
\]
where $\eta_{\gotha, \gotha^\star}^\star$ is the link homomorphism associated to the link $\eta_{\gotha, \gotha^\star}: \ttS_{\gotha} \to \ttS_{\gotha^\star}$ with
\begin{itemize}
\item
indentation $\ell(\delta_{\gotha}) - \ell(\delta_{\gothb}) + \ell(\gotha_\res^\flat) - \ell(\gotha_\res^\circ)$ if $p$ splits in $E/F$ and trivial if $p$ is inert in $E/F$,
\item and
shift $\boldsymbol t_{\gotha} \boldsymbol t_{\gotha^\star}^{-1}
$.
\end{itemize}

Before proceeding, 
we point out the following equality of shifts which we shall use later:
\begin{equation}
\label{E:equality of shifts}
\boldsymbol t_{\gotha} \boldsymbol t_{\gotha^\star}^{-1} \cdot \boldsymbol t_{\delta_{\gothb}, \gotha^\star} \boldsymbol t_{\delta_\gothb, \gothb}^{-1} = \boldsymbol t_\gotha \boldsymbol t_\gothb^{-1}.
\end{equation}
Also, we point out that
our decomposition of periodic semi-meanders gives a numerical equalities of spans:
\begin{equation}
\label{E:equalities of spans}
\ell(\gotha) = \ell(\gotha_+^\flat) + \ell(\delta_{\gotha/\gotha^\flat}) + \ell(\gotha_\res^\flat) , \quad \ell(\gotha^\circ)
= \ell(\gotha^\flat_+) + \ell(\gotha_\res^\circ) ,\quad \text{and} \quad
\ell(\gothb) = \ell(\delta_\gothb)  +\ell(\gothb_\res)
.
\end{equation}
This (and the trivial equality $\ell(\delta_\gotha) = \ell(\delta_{\gotha/\gotha^\flat})$) implies that the indentation degree of $\eta_{\gotha, \gotha^\star}^\star$ when $p$ splits in $E/F$, is equal to
\begin{equation}
\label{E:case iv indentation equality}
\ell(\delta_\gotha) - \ell(\delta_\gothb)+ \big( \ell(\gotha_\res^\flat)- \ell(\gotha_\res^\circ)\big) = \ell(\gotha) -\ell(\gothb) - \big( \ell(\gotha^\circ) -\ell(\gothb_\res)\big).
\end{equation}
Similarly, \eqref{E:equalities of spans} also implies that
\begin{equation}
\label{E:case iv x+y}
x+y =
\ell(\gotha_\res^\flat) - \ell(\gotha^\circ_\res) + \ell(\delta_\gotha) + \ell(\delta_\gothb)= \ell(\gotha) + \ell(\gothb) - (\ell(\gotha^\circ) + \ell(\gothb_\res)).
\end{equation}

Now we separate the discussion according to $\langle \gotha|\gothb\rangle$.

\begin{enumerate}
\item
If $\langle \gotha, \gothb\rangle = 0$, then $\langle \gotha^\circ| \gothb_\res\rangle = 0$ for simple combinatorics reasons. Then the $\pi$-isotypical component of  $\Res_{\gotha^\circ} \circ \Gys_{\gothb_\res}$ factors through the cohomology of a lower dimensional Shimura variety, so the same is true for $\Res_{\gotha} \circ \Gys_{\gothb}$.
\item[(2) or (3)]
We have $\langle \gotha|\gothb\rangle = (-2)^{m_0}v^{m_v}$ or $(-2)^{m_0}T^{m_T}$. The picture $D(\gotha^\circ, \gothb_\res)$ can be identified with the picture  $D(\gotha, \gothb)$ after deforming some curves (``pulling strings").
In particular, we have $\langle \gotha^\circ| \gothb_\res\rangle = \langle \gotha|\gothb\rangle$. By the inductive hypothesis for the pair $(\ttS_{\delta_\gothb}, \ttT_{\delta_\gothb})$\footnote{Here, as before, the shift $\boldsymbol t'_{\gotha'}$ for a periodic semi-meander $\gotha'$ for $(\ttS_{\delta_\gothb}, \ttT_{\delta_\gothb})$ is taken to be $\boldsymbol t_{\delta_\gothb, \tilde \gotha'}$, where $\tilde \gotha'$ is a periodic semi-meander of $(\ttS, \ttT)$ consisting of all the arcs and semi-lines of $\gotha'$ together with the arc $\delta_\gothb$.} and \eqref{E:proof case iv}, we have
\begin{align*}
\Res_{\gotha} &\circ \Gys_{\gothb} =  p^{(x+y)/2} \eta^\star_{\gotha , \gotha^\star} \circ \Res_{\gotha^\circ} \circ \Gys_{\gothb_\res}
\\
& =
\begin{cases}
p^{(x+y)/2} \eta^\star_{\gotha , \gotha^\star} \circ (-2)^{m_0}\cdot p ^{(\ell(\gotha^\circ) + \ell(\gothb_\res))/2} \eta_{\ttS_{\delta_\gothb, \gotha^\circ}, \ttS_{\delta_\gothb, \gothb_\res}}^\star, & \textrm{if }r < \frac d2,
\\
p^{(x+y)/2} \eta^\star_{\gotha , \gotha^\star} \circ (-2)^{m_0}\cdot p ^{(\ell(\gotha^\circ) + \ell(\gothb_\res))/2} (T_\gothp/p^{g/2})^{m_T}\eta_{\ttS_{\delta_\gothb, \gotha^\circ}, \ttS_{\delta_\gothb, \gothb_\res}}^\star
& \textrm{if }r = \frac d2,
\end{cases}
\\
& \stackrel{\eqref{E:case iv x+y}}=
\begin{cases}
(-2)^{m_0}\cdot p ^{(\ell(\gotha) + \ell(\gothb))/2}
\eta^\star_{\gotha , \gotha^\star} \circ \eta_{\ttS_{\delta_\gothb, \gotha^\circ}, \ttS_{\delta_\gothb, \gothb_\res}}^\star, & \textrm{if }r < \frac d2,
\\ (-2)^{m_0}\cdot p ^{(\ell(\gotha) + \ell(\gothb))/2} (T_\gothp/p^{g/2})^{m_T}\eta^\star_{\gotha , \gotha^\star} \circ \eta_{\ttS_{\delta_\gothb, \gotha^\circ}, \ttS_{\delta_\gothb, \gothb_\res}}^\star
& \textrm{if }r = \frac d2.
\end{cases}
\end{align*}
The composition of the two links is exactly $\eta^\star_{\ttS_\gotha, \ttS_\gothb}$ of the needed indentation degree by \eqref{E:case iv indentation equality} an of the required shift by \eqref{E:equality of shifts}.
\end{enumerate}
This concludes the proof of Theorem~\ref{T:intersection combinatorics}.
\appendix

\section{Cohomology of quaternionic Shimura varieties}
We include the proof of Proposition~\ref{P:cohomology of sh} regarding the cohomology of our ``slightly twisted" quaternionic Shimura varieties. It is based on comparing the cohomology with the known case when $\ttT=\emptyset$.  This is certainly  known to the experts, but we could not find the exact version in the literature.

\subsection{Discrete Shimura varieties for $F^\times$}
\label{S:shimura for F times}
Consider a Deligne homomorphism for $T_{F, \ttT}:=\Res_{F/\QQ}(\GG_m)$ given by
\[
\xymatrix@R=0pt{
h_\ttT: \SSS(\RR) = \CC^\times \ar[r] & T_F(\RR)=(\RR^\times)^\ttT \times (\RR^\times)^{\Sigma_\infty - \ttT}
\\
z \ar@{|->}[r] &  \big((|z|^2, \dots, |z|^2), (1, \dots,1) \big).
}
\]
Under this choice of Deligne homomorphism,
we can define a discrete Shimura variety $\calS h_{K_{T,p}}(T_{F, \ttT})$ for $K_{T,p} = \calO_\gothp^\times$ whose complex points are given by
\[
\calS h_{K_{T,p}}(T_{F, \ttT})(\CC) = F^{\times,\cl} \backslash \AAA_F^{\infty,\times}/ \calO_\gothp^\times.
\] 
It admits an integral canonical model with special fiber $\Sh_{K_{T,p}}(T_{F, \ttT})$ over $\FF_{p^g}$ (in the sense of \cite[Section 2.8]{tian-xiao1}), which is determined by the Shimura reciprocity map
\[
\Rec_{T, \ttT,p}: \Gal_{\FF_{p^g}} \longrightarrow F^{\times,\cl} \backslash \AAA_F^{\infty,\times}/ \calO_\gothp^\times.
\]
Explicitly, $\Rec_{T,\ttT,p}$ sends the geometric Frobenius $\Frob_{p^g}$ to the finite idele $(\underline p_F)^{\#\ttT}$.

Fix a prime number $\ell\neq p$.  The algebraic representation $\rho^w_{T, \ttT}$ of $ T_{F, \ttT} \times \CC \cong \prod_{\tau \in \Sigma_\infty} \GG_{m,\tau}$ sending $x$ to $(x^{2-w}, \dots, x^{2-w})$ gives a lisse $\overline \QQ_{\ell}$-\'etale sheaf $\calL^w_{T, \ttT}$ of pure weight $2(w-2)\#\ttT$ on $\Sh_{K_{T,p}}(T_{F, \ttT})$.

\subsection{Changing $\ttT$}
We need to compare the Shimura varieties $\Sh_{K_p}(G_{\ttS, \ttT})$ and  $\Sh_{K_p}(G_{\ttS, \emptyset})$.
Note that the natural product morphism
\[
\pr: G_{\ttS, \emptyset} \times \Res_{F/\QQ} \GG_m \to G_{\ttS, \ttT}
\]
is compatible with the Deligne homomorphism $h_{\ttS, \emptyset} \times h_{\ttT}$ on the source and $h_{\ttS, \ttT}$ on the target, i.e. $\pr \circ (h_{\ttS, \emptyset} \times h_{\ttT}) = h_{\ttS, \ttT}$.
This gives a natural morphism of Shimura varieties
\begin{equation}
\label{E:pr bullet}
\pr_\bullet:
\Sh_{K_p}(G_{\ttS, \emptyset}) \times \Sh_{K_{T,p}}(T_{F, \ttT}) \longrightarrow\Sh_{K_p}(G_{\ttS, \ttT}).
\end{equation}
Moreover, the product morphism is compatible with the algebraic representations in the sense that
\[\rho^{(\underline k,w)}_{\ttS, \ttT} \circ
\pr   \cong \rho^{(\underline k,w)}_{\ttS, \emptyset}\boxtimes  \rho^w_{T, \ttT} .
\]
So we have a natural isomorphism of sheaves
\begin{equation}
\label{E:equality of pullback sheaves in quaternion case}
\pr_\bullet^*(\calL_{\ttS, \ttT}^{(\underline k, w)}) \cong \calL_{\ttS, \emptyset}^{(\underline k, w)} \boxtimes \calL_{T, \ttT}^w.
\end{equation}

\begin{prop}
\label{P:cohomology of sh appendix}
Let $\pi \in\scrA_{(\underline k, w)}$ be an automorphic representation appearing in the cohomology of the Shimura variety $\calS h_K(G_{\ttS, \ttT})$.
Then we have a canonical isomorphism 
\[
H^{i}_\et(\Sh_K(G_{\ttS, \ttT})_{\overline \FF_p}, \calL_{\ttS, \ttT}^{(\underline k, w)})[\pi]^{\Fr-s.s.} =
\begin{cases}
 \rho_{\pi, \gothp}^{\otimes d} \otimes [\det(\rho_{\pi,\gothp})(1)]^{\otimes \#\ttT} & \textrm{if }i =d,\\
 0 & \textrm{if }i \neq d;
 \end{cases}
\]
  equivariant for the action of the geometric Frobenius $\Frob_{p^{g}}$. Here, the superscript $\Fr-s.s.$ means taking the semi-simplification as $\Frob_{p^g}$-modules.
\end{prop}
\begin{proof}
The proposition is known when $\ttT =\emptyset$ by  \cite[\S 3.2]{brylinski-labesse} (note that 
we have the tensor product instead of tensorial induction because $\rho_{\pi,\gothp}$ is unramified at $\gothp$.)
For general $\ttT$, the morphism \eqref{E:pr bullet} induces an isomorphism
\begin{align*}
H^\star_\et(\Sh_K(G_{\ttS, \ttT})_{\overline \FF_p}, \calL_{\ttS, \ttT}^{(\underline k, w)}) &\cong H^\star_\et \big( \Sh_K (G_{\ttS, \ttT})_{\overline \FF_p} \times \Sh_{K_{T,p}}(T_{F, \ttT})_{\overline \FF_p}, \pr_\bullet^*(\calL_{\ttS, \ttT}^{(\underline k, w)}) \big)^{\AAA_F^{\infty, \times}}
\\
&\stackrel{\eqref{E:equality of pullback sheaves in quaternion case}}{\cong}
 \Big(H^\star_\et \big(\Sh_K (G_{\ttS, \emptyset})_{\overline \FF_p}, \calL_{\ttS, \emptyset}^{(\underline k, w)} \big) \otimes  H^0_\et \big( \Sh_{K_{T,p}}(T_{F, \ttT})_{\overline \FF_p}, \calL_{T, \ttT}^{w} \big) \Big)^{\AAA_F^{\infty, \times}},
\end{align*}
where the superscript $\AAA_F^{\infty, \times}$ means to take the invariant part for the \emph{anti-diagonal action} of this group, i.e. $z \in \AAA_F^{\infty, \times}$ acts by $(z,z^{-1})$.
So, if $\omega_\pi$ denotes the central character of $\pi$, we have
\[
H^\star_\et(\Sh_K(G_{\ttS, \ttT})_{\overline \FF_p}, \calL_{\ttS, \ttT}^{(\underline k, w)})[\pi] \cong H^\star_\et(\Sh_K(G_{\ttS, \emptyset})_{\overline \FF_p}, \calL_{\ttS, \emptyset}^{(\underline k, w)})[\pi] \otimes 
H^0_\et \big( \Sh_{K_{T,p}}(T_{F, \ttT})_{\overline \FF_p}, \calL_{T, \ttT}^{w} \big)[\omega_\pi],
\]
where the last factor is the one-dimensional subspace where $\AAA_F^{\infty, \times}$ acts through $\omega_\pi$.
By the Shimura reciprocity map $\Rec_{T, \ttT,p}$  recalled in Subsection~\ref{S:shimura for F times} and the Eichler--Shimura relation \eqref{E:Eichler-Shimura}, the geometric Frobenius $\Frob_{p^g}$ acts on this one-dimensional space by multiplication by 
\[
\omega_\pi(\underline p_F)^{-\#\ttT}
= \big(\det(\rho_{\pi, \gothp}(\Frob_{p^g}))/p^g)\big)^{\#\ttT}.
\]
This concludes the proof of this Proposition.
\end{proof}

\end{document}